\documentclass[a4paper,10pt, oneside]{article}
\usepackage{bm,amsmath,amssymb,amsthm,mathrsfs,graphicx}
\usepackage[font=small,labelfont=md,textfont=it]{caption}
\usepackage[top=2.0cm,bottom=2.15cm,left=1.98cm,right=1.98cm]{geometry}
\usepackage[pdftex, pdfpagelabels, bookmarks, hyperindex, hyperfigures,colorlinks=red, pagebackref]{hyperref}
\usepackage{diagbox}
\usepackage{amssymb}
\usepackage{amsmath,bm}
\usepackage{multirow}
\usepackage{makecell}
\usepackage{graphicx}
\usepackage{subfigure}
\usepackage{float}
\usepackage[titletoc, title]{appendix}
\usepackage{color}
\usepackage{lipsum}
\usepackage{amsfonts}
\usepackage{graphicx}
\usepackage{epstopdf}
\usepackage{algorithmic}
\usepackage{amsopn}
\usepackage{amsthm}
\usepackage{cleveref}
\usepackage{amssymb}

\usepackage{authblk}

\numberwithin{equation}{section}
\numberwithin{figure}{section}

\newtheorem{lemma}{Lemma}[section]
\newtheorem{theorem}{Theorem}[section]
\newtheorem{remark}{Remark}[section]
\newtheorem{definition}{Definition}[section]

\newtheorem{exam}{Example}[section]
\begin{document}

\title{\Large
Robust globally divergence-free weak Galerkin methods for   unsteady incompressible convective Brinkman-Forchheimer equations
\thanks
{
	This work was supported in part by Natural Science Foundation of Sichuan Province, China (2023NSFSC0075) and National Natural Science Foundation of China  (12171340).
}
}
\author[a]{Xiaojuan Wang \thanks{ Email: 1271582983@qq.com}}
\author[b]{Jihong Xiao \thanks{ Email: xiaojh2752@163.com }}
\author[a]{Xiaoping Xie \thanks{  Email: xpxie@scu.edu.cn}}
\author[a]{Shiquan Zhang \thanks{Corresponding author. Email: shiquanzhang@scu.edu.cn }}

\affil[a]{School of Mathematics, Sichuan University, Chengdu 610064, China}
\affil[b]{Division of Mathematics, Sichuan University Jinjiang College, Pengshan 620860, China}

\renewcommand*{\Affilfont}{\small\it}
\renewcommand\Authands{, }
\maketitle

\begin{abstract}
This paper  develops and analyzes a class of  semi-discrete and fully discrete weak
Galerkin finite element methods for   unsteady incompressible convective Brinkman-Forchheimer equations.  For the spatial discretization, the methods adopt the piecewise polynomials of degrees  $m\ (m\geq1)$ and $m-1$  respectively  to approximate the velocity  and pressure inside the elements, and piecewise polynomials of degree  $m$ to approximate  their numerical  traces on the interfaces of elements.  In the fully discrete method,  the backward Euler difference scheme is used to approximate the time derivative. The methods are shown to yield globally divergence-free velocity  approximation.  Optimal a priori error estimates in  the energy norm and $L^2$ norm are established.
A convergent linearized iterative algorithm is designed for solving the fully discrete system.
Numerical experiments are provided to verify the  theoretical results.
\end{abstract}

%\begin{keyword}
% \color{black}	
	
%words here, in the form: keyword \sep keyword
\text{\bf AMS 2010} \; 65M60, 65N30

%\text{AMS 2010} \; 65M60, 65N30

%\end{keyword}

%\end{frontmatter}


\color{black}
\section{Introduction}
Let $\Omega\subset \mathbb{R}^n$ $(n=2,3)$ \color{black} be a   polygonal/polyhedral domain.  We consider the following unsteady incompressible convective Brinkman-Forchheimer  model:  find \color{black} \color{black} the velocity $\bm{u}(\bm{x},t):\Omega \times [0,T]\rightarrow \mathbb{R}^{n}$ and  the pressure $p(\bm{x},t):\Omega \times [0,T]\rightarrow \mathbb{R}$ such that
\begin{eqnarray}\label{BF0}
\left\{
\begin{aligned}
\bm{u} _{t}-\nu \Delta \bm{u}+\nabla\cdot(\bm{u}\otimes\bm{u})+\alpha |\bm{u}|^{r-2}\bm{u}+\nabla p&=\bm{f},&\text{in} \ \Omega\times [0,T],\label{BF0}\\
\nabla\cdot\bm{u}&=0,&\text{in} \ \Omega\times [0,T],\\
\bm{u}(\bm{x},0)&=\bm{u}_{0},&\text{in} \  \Omega,\\
\bm{u}&=\bm{0},&\text{on} \ \partial \Omega. \\
\end{aligned}
\right.
\end{eqnarray}
Here,
 $\bm{f}(\bm{x},t) $ denotes a given forcing term, $\nu$  the Brinkman coefficient, $\alpha > 0 $  the Forchheimer coefficient, $3\leq r <\infty$ for $n=2$ and $3\leq r\leq 6$ for $n=3$. $\bm{u}_{0}(\bm{x})$ is the initial data satisfying $\bm{u}_{0}|_{\partial\Omega}=0$.
The operator $\otimes$ is defined by $\bm{u}\otimes\bm{v}=(u_iv_j)_{n\times n}$   for $\bm{v}=(v_1,\cdots,v_n)^T$.

The Brinkman-Forchheimer model, which  can be viewed as the Navier-Stokes equations with a nonlinear damping term, is  used to modelling fast flows in highly porous media \cite{D1982Nonlinear,K1981Boundary}. 
In recent years there have developed many numerical algorithms   for  stationary Brinkman-Forchheimer equations,  such as  conforming mixed finite element  methods \cite{2019Mixed,YangHuaijun2023Saot}, stabilized mixed methods \cite{LiZhenzhen2019Smfe}, multi-level mixed methods \cite{LiMinghao2019Tmfe,QiuHailong2019Msaf,ZhengBo2021Tdsa,ZhengBo2023Tsaf}, parallel finite element algorithms \cite{WassimEid2022Lapf,WassimEid2023Aptm} and weak Galerkin methods (WG) \cite{TaiYue2024WGmf,WangX.J2024Rgdw}.
  For unsteady  Brinkman-Forchheimer equations, there are limited numerical studies in the literature.  In \cite{QianLiu2021Saoa}  a superconvergence analysis was presented for a  fully discrete nonconforming mixed finite element   method with  a backward Euler temporal scheme. In  \cite{2022Unconditional}    error estimates were derived for  semi-discrete  and    linearized fully discrete conforming mixed finite element schemes. We mention that there are several mixed finite element formulations for the system \eqref{BF0} with the term  $\nabla\cdot(\bm{u}\otimes\bm{u})$ being replaced by  a Darcy  term, e.g. the pseudostress-velocity mixed formulation \cite{Caucao;2021}, the velocity-velocity gradient-pseudostress mixed formulation  \cite{CaucaoSergio2022AtBs} and the pressure stabilized mixed formulation  \cite{Louaked;2017}.

It is well-known that  the  divergence constraint  $\nabla\cdot\bm{u}=0$ corresponds to the conservation of   mass for  incompressible fluid flows, and   that   numerical methods   with poor conservation usually  suffer from instabilities   \cite{bookin1, JLMNR2017, add4,bookin3,  add2}.
Meanwhile, the  numerical schemes  with  the exactly divergence-free velocity  approximation may automatically lead to pressure-robustness in the sense that the velocity approximation  error   is independent of    the pressure approximation \cite{JLMNR2017,MR3564690,MuLin2023ApwG}.
We refer to \cite{CFX2016,Chen-Xie2023,CKS2007,HLX2019,HX2019,MR3511719,MuLin2018Addf,XZ2010,ZhangLi2019Agdw,ZCX2017} for some divergence-free finite element methods for  the incompressible fluid flows.
	
In \cite{WangX.J2024Rgdw}, robust globally divergence-free WG methods were developed and analyzed for   stationary incompressible convective Brinkman-Forchheimer equations.
	In this paper, we shall consider the weak Galerkin   discretization of the unsteady Brinkman-Forchheimer model \eqref{BF0}. The WG framework was first proposed  in \cite{WangJunping2013AwGf,WangJunping2013AWGM}  for  second-order elliptic problems.  It  allows the use of totally discontinuous functions on meshes with arbitrary shape of polygons/polyhedra due to the introduction of weakly defined gradient/divergence operators over functions with discontinuity, and has  the local elimination property, i.e. the unknowns defined in the interior of elements can be locally eliminated by using the numerical traces  defined on the interfaces of  elements. We refer to \cite{CFX2016, Chen-Xie2023, HLX2019,HX2019,HuXiaozhe2019AwGf,MuLin2023ApwG,MuLin2018Addf,PengHui2022WGmf,WangJunping2016AwGf,WangRuishu2016AwGf,WangX.J2024Rgdw, XZ2010,ZZX2023,ZCX2017} some developments and applications of  the WG methods for fluid flow  problems.  Particularly,  a class of robust globally divergence-free weak Galerkin methods were developed  in \cite{CFX2016} for Stokes equations, and later were extended  to solve incompressible quasi-Newtonian Stokes equations \cite{ZCX2017}, natural convection equations \cite{HLX2019,HX2019}, incompressible Magnetohydrodynamics flow equations \cite{ZZX2023} and stationary incompressible convective Brinkman-Forchheimer equations \cite{WangX.J2024Rgdw}.

 The goal of this  contribution is to extend the WG methods of  \cite{HLX2019} to the semi-discretization and full discretization of  the unsteady Brinkman-Forchheimer model  \eqref{BF0}.
The main features of our  WG discretizations  for the   model \eqref{BF0} are as follows:
\begin{itemize}

\item  The proposed WG methods are arbitrary order ones.  In the  spatial discretization, we adopt   piecewise polynomials of degrees $m$ $ (m\geq1)$ and $m-1$  to approach the velocity and pressure in the interiors of elements, respectively, and use polynomials of degree  $m$ to approximate the numerical traces of velocity and pressure on the interfaces of elements.
In the full discretization, the backward Euler difference scheme is used for the temporal discretization.
\item The WG methods yield     globally divergence-free velocity approximation, which  automatically leads to   pressure-robustness.	

\item The schemes  are  ``parameter-friendly", i.e.,   the stabilization parameters in the schemes
are not required  to be  ``sufficiently large".

\item The well-posedness of the discrete schemes and optimal a priori error estimates in the energy norm and $L^2$ norm   are established.

\end{itemize}

The rest of this paper is organized as follows.
Section 2 gives  weak formulations, the semi-discrete WG scheme and some preliminary results.
Sections  3 and 4 derive the well-posedness  and optimal a priori error estimates of the semi-discrete   scheme, respectively.
Section 5 gives the backward Euler fully discrete WG scheme  and shows its  existence and uniqueness  results.
Section 6 presents optimal error estimations for the fully discrete scheme.
Section 7 proposes a linearized  iteration algorithm for solving the nonlinear fully discrete  system.
Section 8 provides numerical tests to verify the theoretical results.
Finally, some concluding remarks are made in Section 9.
\color{black}

\section{Semi-discrete WG formulations}
\subsection{Notation and weak problem}
For any bounded domain $\Lambda \subset \mathbb{R}^k$ $( k=n,n-1 )$, nonnegative  integer $s$ and real number $1\leq q<\infty$, let $W^{s,q}(\Lambda)$ and $W_0^{s,q}(\Lambda)$ be the standard Sobolev spaces defined on $\Lambda$ with norm $||\cdot||_{s,q, \Lambda}$ and semi-norm $|\cdot|_{s,q,\Lambda}$. In particular,   $H^s(\Lambda):=W^{s,2}(\Lambda)$, $L^2(\Lambda)=H^{0}(\Lambda)$ and $H^s_0(\Lambda):=W^{s,2}_0(\Lambda)$, with $||\cdot||_{s,\Lambda}:=||\cdot||_{s,2, \Lambda}$ and $|\cdot|_{s,\Lambda}:=|\cdot|_{s,2, \Lambda}$.   We use $(\cdot,\cdot)_{s,\Lambda}$ to denote the inner product of $H^s(\Lambda)$, with $(\cdot,\cdot)_\Lambda :=(\cdot,\cdot)_{0,\Lambda}$.
	When $\Lambda = \Omega$, we set $||\cdot||_s := ||\cdot||_{s,\Omega},|\cdot|_s := |\cdot|_{s,\Omega}$, and $(\cdot,\cdot):= (\cdot,\cdot)_{\Omega}$.
	Especially, when $\Lambda \subset \mathbb{R}^{n-1}$ we use $\langle\cdot,\cdot\rangle_\Lambda$ to replace $(\cdot,\cdot)_{\Lambda}$.
For a nonnegative integer $m$, let $P_m(\Lambda)$  be the set of all polynomials defined  on $\Lambda$ with degree   no more than $m$.

We introduce the following spaces:
\begin{align*}
&\bm{V}:=[H^1_0(\Omega)]^n, \quad \bm{V}_{0}:=\{\bm{v}\in\bm{V}:\nabla\cdot \bm{v}=0\},\\
&Q:=L^2_0(\Omega):=\{q\in L^2(\Omega)\big | \ (q,1)=0\},\\
&\bm{H}({\rm div};\Lambda):=\left\{\bm v\in [L^2(\Lambda)]^n\big | \  \color{black}\nabla\cdot\color{black}\bm{v}\in L^2(\Lambda)\right\},\\
&L^{p}(0,T; H^{s}(\Lambda) ):=\{ v:[0,T]\rightarrow H^{s}(\Lambda) \big |\ 
 \int_{0}^{T}\|v\|_{s,\Lambda}^{p}dx <\infty\},  1\leq p<\infty,\\
&L^{\infty}(0,T; H^{s}(\Lambda) ):=\{ v:[0,T]\rightarrow H^{s}(\Lambda)\big | \  %\text{strongly measurable:}
\sup_{t\in [0,T] }\|v\|_{s,\Lambda}<\infty\}.
\end{align*}
For simplicity, $ L^{p}(0,T; H^{s}(\Lambda) )$ and $L^{\infty}(0,T; H^{s}(\Lambda) )$ are denoted by $L^{p}(H^{s})$ and $L^{\infty}(H^{s})$, respectively.  %

Let $\mathcal{T}_h$ be a shape regular  partition of $\Omega$ into closed simplexes,
and let $\mathcal{E}_h$ be the set of all edges (faces) of all the elements in $\Omega$.
	For any $K\in \mathcal{T}_h$, $e\in \mathcal{E}_h$, we denote by $h_K$ the diameter of $K$ and by $h_e$ the diameter of $e$.
Let $h=\max_{K\in\mathcal{T}_h}h_K$.
 For all $K\in\mathcal{T}_h$ and $e\in\mathcal{E}_h$, $\color{black}\bm{n}_K\color{black}$ and $\bm{n}_e$ stand for the outward unit normal vectors along with the boundary $\partial K$ and $e$, respectively.
 We may abbreviate $\bm{n}_K$ as $\bm{n}$ when  there is no ambiguity.
	We use  $\nabla_h$ and $\nabla_h\cdot$  to denote  respectively  the operators of piecewise-defined gradient and divergence with respect to the partition $\mathcal{T}_h$.
In addition,  introduce the following  inner product and mesh-dependent norm:
\begin{align*}
\langle \cdot,\cdot \rangle_{\partial\mathcal{T}_h}:=&\sum_{K\in\mathcal{T}_h}\langle \cdot,\cdot \rangle_{\partial K},\ \ \ \ \
\|\cdot\|_{0,\partial \mathcal{T}_h}:=(\sum_{K\in\mathcal{T}_h}\|\cdot\|^2_{0,\partial K})^{1/2}.
\end{align*}

Let $N > 0$ be an integer and $0=t_0<t_1<...<t_N =T$ be a uniform division of time domain $[0,T]$ with the time step $\Delta t =\frac{T}{N}$ and  $t_k=k\Delta t$.

For convenience, throughout the paper we use $x\lesssim y$ ($x\gtrsim y$)
to denote $x\le Cy$ ($x\ge Cy$), where   $C$  is a positive constant independent of the mesh size $h$ and time step $\Delta t$.

We define the following bilinear and trilinear forms: for $\bm{\kappa}$, $\bm{u}$, $\bm{v}\in \bm{V}$ and $q\in Q,$
%\begin{subequations}
\begin{align*}
&a(\bm{u},\bm{v}):=\nu ( \nabla \bm{u},\nabla \bm{v}),\quad
b(\bm{v},q):=-( q, \nabla\cdot\bm{v}),\\\quad%\nonumber\\
&c(\bm{\kappa};\bm{u},\bm{v}):=\alpha(|\bm{\kappa}|^{r-2}\bm{u},\bm{v}),\quad %\nonumber\\
d(\bm{\kappa};\bm{u},\bm{v}):=\frac{1}{2}(\nabla\cdot(\bm{u}\otimes\bm{\kappa}),\bm{v})
-\frac{1}{2}(\nabla\cdot(\bm{v}\otimes\bm{\kappa}),\bm{u}).
\end{align*}
%\end{subequations}
Then the weak form of \eqref{BF0} is given as follows: for all $t\in (0,T]$, seek $(\bm{u},p )\in \bm{V}\times Q$ such that
\begin{subequations}\label{weak}
\begin{align}
(\bm{u}_{t},\bm{v})+a(\bm{u},\bm{v})+b(\bm{v},p )+c(\bm{u};\bm{u},\bm{v} )+d(\bm{u};\bm{u},\bm{v} )&=(\bm{f},\bm{v}),&\forall \bm{v}\in\bm{V},\label{weak1}\\
b(\bm{u},q )&=0,&\forall q\in Q. \label{weak1b}
\end{align}
\end{subequations}

\subsection{Semi-discrete WG scheme}

In order to give the WG scheme to the system \eqref{BF0}, for integer $\gamma\geq 0$,  we introduce the
discrete gradient operator $\nabla_{w,\gamma}$ and  the discrete weak divergence operator $\nabla_{w,\gamma}\cdot$ as follows.

\begin{definition}
For all $K\in \mathcal{T}_{h}$ and $v\in \mathcal{V}(K):=\{v=\{v_i,v_b\}:v_i\in L^2(K),v_b\in H^{1/2}(\partial K)\}$, the discrete weak gradient   $\nabla_{w,\gamma,K}v\in [P_{\gamma}(K)]^{n}$ of v on $K$ is defined by
\begin{align}\label{weak gradient}
(\nabla_{w,\gamma,K}v, \bm{\varsigma})_K=-(v_i,\nabla \cdot \bm{\varsigma})_K+\langle v_b,\bm{\varsigma}\cdot \bm{n}_{K} \rangle_{\partial K}, \quad\forall  \bm{\varsigma}\in [P_{\gamma}(K)]^{n}.
\end{align}
Then  the global discrete weak gradient operator $\nabla_{w,\gamma}$ is defined as
$$\nabla_{w,\gamma}|_{K}:=\nabla_{w,\gamma,K}, \quad\forall K\in \mathcal{T}_h. $$
Moreover, given a vector  $\bm{v}=(v_1,v_2,...,v_n)^{T}$ with $v_j|_K\in \mathcal{V}(K)$ for $j=1,...,n$,
the discrete weak gradient $\nabla_{w,\gamma}\bm{v}$ is defined as
$$ \nabla_{w,\gamma}\bm{v}:= (\nabla_{w,\gamma}v_1,\nabla_{w,\gamma}v_2,..., \nabla_{w,\gamma}v_n )^{T}.$$
\end{definition}

\begin{definition}
For all  $K\in \mathcal{T}_{h}$ and $\bm{w}\in \bm{\mathcal{W}}(K):=\{\bm{w}=\{\bm{w}_i,\bm{w}_b\}:\bm{w}_i\in [L^2(K)]^{n},\bm{w}_b\cdot\bm{n}_{K}\in H^{-1/2}(\partial K)\} $, the discrete weak divergence   $\nabla_{w,\gamma,K}\cdot\bm{w}\in P_{\gamma}(K)$ of $\bm{w}$ on $K$ is defined by
\begin{align}
(\nabla_{w,\gamma,K}\cdot\bm{w} , \varsigma)_K=-(\bm{w}_i,\nabla  \varsigma)_K+\langle \bm{w}_b\cdot \bm{n},\varsigma \rangle_{\partial K}, \quad\forall \varsigma\in P_{\gamma}(K).
\end{align}
Then the global discrete weak divergence operator $\nabla_{w,\gamma}\cdot $ is defined as $$\nabla_{w,\gamma}\cdot|_{K}:=\nabla_{w,\gamma,K}\cdot,\quad \forall K\in \mathcal{T}_h. $$
Moreover, given a tensor  $\widetilde{\bm{w}}=(\bm{w}_1,...,\bm{w}_n)^{T}$ with $\bm{w}_j|_K\in  \bm{\mathcal{W}}(K)$ for $j=1,...,n$, the discrete weak divergence $\nabla_{w,\gamma}\cdot \widetilde{\bm{w}}$ is defined as $$\nabla_{w,\gamma}\cdot \widetilde{\bm{w}}:=(\nabla_{w,\gamma}\cdot \bm{w}_{1},...,\nabla_{w,\gamma}\cdot \bm{w}_n)^{T}. $$
\end{definition}

For any integer $m\ge 1$,
we introduce the following finite dimensional spaces:
\begin{align*}
\bm{V}_h:=&\{\bm{v}_h=\{\bm{v}_{hi},\bm{v}_{hb}\}: \bm{v}_{hi}|_K\in [P_m(K)]^n,\bm{v}_{hb}|_e\in [P_{m}(e)]^n ,
\forall K\in\mathcal{T}_h,\forall e\in\mathcal{E}_h\},\\
\bm{V}_h^{0}:=&\{\bm{v}_h=\{\bm{v}_{hi},\bm{v}_{hb}\}\in\bm{V}_h: \bm{v}_{hb}|_{\partial\Omega}=\bm{0} \},\\
Q_h:=&\{q_h=\{q_{hi},q_{hb}\}: q_{hi}|_K\in P_{m-1}(K),q_{hb}|_e\in P_{m}(e),
\forall K\in\mathcal{T}_h,\forall e\in\mathcal{E}_h\},\\
Q_h^{0}:=&\{q_h=\{q_{hi},q_{hb}\}\in Q_h, q_{hi}\in L_0^2(\Omega)\}.
\end{align*}

For any $\bm{u}_{h}=\{\bm{u}_{hi},\bm{u}_{hb}\},\bm{v}_{h}=\{\bm{v}_{hi},\bm{v}_{hb}\},
\bm{\kappa}_{h}=\{\bm{\kappa}_{hi},\bm{\kappa}_{hb}\}\in \bm{V}_{h}^{0}$, and $p_h=\{p_{hi}, p_{hb}\}\in Q_h^{0},$
we shall define bilinear and trilinear terms as follows:
\begin{align*}
a_h(\bm{u}_h,\bm{v}_h)&:=\nu(\nabla _{w,l}\bm{u}_h,\nabla _{w,l}\bm{v}_h)
+s_h(\bm{u}_h,\bm{v}_h), \\
s_h(\bm{u}_h,\bm{v}_h)&:=\nu \langle  \eta( \bm{u}_{hi}-\bm{u}_{hb}), \bm{v}_{hi}-\bm{v}_{hb} \rangle_{\partial \mathcal{T}_h},\\
b_h(\bm{v}_h,q_h)&:=(\nabla _{w,m}q_h,\bm{v}_{hi}), \\
c_h(\bm{\kappa}_h ;\bm{u}_h,\bm{v}_h )&:=\alpha( |\bm{\kappa}_{hi}|^{r-2}\bm{u}_{hi}, \bm{v}_{hi} ), \\
d_h(\bm{\kappa}_h ;\bm{u}_h,\bm{v}_h )
&:=\frac{1}{2}( \nabla _{w,m}\cdot\{\bm{u}_{hi}\otimes\bm{\kappa}_{hi}, \bm{u}_{hb}\otimes\bm{\kappa}_{hb}\},\bm{v}_{hi}  )%\\&\qquad
    -\frac{1}{2}( \nabla _{w,m}\cdot\{\bm{v}_{hi}\otimes\bm{\kappa}_{hi}, \bm{v}_{hb}\otimes\bm{\kappa}_{hb}\},\bm{u}_{hi}  ),
\end{align*}
where the stabilization parameter $\eta|_{\partial K}=h_K^{-1} ,$ $\forall K\in\mathcal{T}_h$,
and  integer $l= m-1,m.$

Based on the above definitions, the semi-discrete WG scheme for \eqref{BF0} reads: for any $t\in [0,T]$, seek $\bm{u}_h=\{\bm{u}_{hi},\bm{u}_{hb} \}\in \bm{V}_h^{0}$, $p_h=\{p_{hi}, p_{hb}\}\in Q_h^{0} $ such that
\begin{subequations}\label{WG}
\begin{align}
(\frac{\partial\bm{u}_{hi}}{\partial t},\bm{v}_{hi} )+a_h(\bm{u}_h,\bm{v}_h)+b_h(\bm{v}_{h},p_h)%
+c_h(\bm{u}_h ;\bm{u}_h,\bm{v}_h )%&\nonumber\\%&&
+d_h(\bm{u}_h ;\bm{u}_h,\bm{v}_h )
=&(\bm{f},\bm{v}_{hi}),  \forall \bm{v}_h \in \bm{V}_h^{0},\label{WG1}\\
b_h(\bm{u}_{h},q_h)=&0, \forall q_h\in  Q_h^{0}, \label{WG2}
\end{align}
\end{subequations}
	%%where the initial data $\bm{u}_{h}(\bm{x},0 )=\bm{u}_{h}^{0}=\bm{\Pi}_m^\ast \bm{u}_0$.

By following a similar process as in the proofs of \cite[Theorem 2.1]{WangX.J2024Rgdw}, we can obtain that the scheme \eqref{WG} yields   globally divergence-free velocity approximation.
\begin{theorem}\label{TH2.2}
Let $\bm{u}_h=\{\bm{u}_{hi},\bm{u}_{hb} \}\in \bm{V}_h^{0}$ be the velocity solution of the WG scheme \eqref{WG}. Then there hold
 \begin{align}\label{divergence-free}
	&\bm{u}_{hi}\in \bm{H}({\rm div};\Omega),\quad  \nabla\cdot\bm{u}_{hi}=0.
\end{align}
\end{theorem}

\subsection{Preliminary results}

We first introduce   two semi-norms, $|||\cdot|||_V$ and $|||\cdot|||_Q$,  on the spaces $ \bm{V}_h$ and $ Q_h$, respectively, as follows:
\begin{align*}
|||\bm{v}_h|||_V^2&:= \|\nabla _{w,l}\bm{v}_h\|^2_0
+\|\eta^{\frac 1 2}(\bm{v}_{hi}-\bm{v}_{hb})\|^2_{0,\partial \mathcal{T}_h}, \quad \forall \bm{v}_h\in \bm{V}_h,\\
|||q_h|||_Q^2&:=\|{q}_{hi}\|^2_0+\sum_{K\in \mathcal{T}_h}h_{K}^{2}\|\nabla _{w,m} q_h\|_{0,K}^{2}, \quad \forall q_h \in Q_h,
\end{align*}
where $\|\cdot\|_{0,\partial \mathcal{T}_h}:=(\sum_{K\in\mathcal{T}_h}\|\cdot\|^2_{0,\partial K})^{1/2}$,
 $\eta|_{\partial K}=h_{K}^{-1}$.
It is easy to see that $|||\cdot|||_V$ and $|||\cdot|||_Q$ are norms on $\bm{V}_h^{0}$ and $Q_h^{0}$, respectively (cf. \cite[Lemma 3.3]{  CFX2016}).

In view of the definitions of the discrete weak gradient operator,   Green's formula,
  the projection operator,  the Cauchy-Schwarz inequality, the inverse inequality and the trace inequality,   the following lemma holds (cf. \cite[{\color{black}Lemma 3.2}]{CFX2016}).
\begin{lemma}\label{Lemma 2.3}
For any $K\in \mathcal{T}_h$ and $\bm{\omega}_{h}=\{\bm{\omega}_{hi},\bm{\omega}_{hb} \}\in [P_{m}(K)]^{n}\times[P_{k}(\partial K) ]^{n}$ with $0\leq m-1\leq l\leq m$, there hold
\begin{subequations}
\begin{align}
\|\nabla \bm{\omega}_{hi}\|_{0,K}\lesssim \|\nabla_{w,l}\bm{\omega}_{h}\|_{0,K}+h_K^{-\frac{1}{2}}\| \bm{\omega}_{hi}-\bm{\omega}_{hb}\|_{0,\partial K},\label{2.10a}\\
\|\nabla_{w,l} \bm{\omega}_{h}\|_{0,K}\lesssim \|\nabla \bm{\omega}_{hi}\|_{0,K}+h_K^{-\frac{1}{2}}\| \bm{\omega}_{hi}-\bm{\omega}_{hb}\|_{0,\partial K}.\label{2.10b}
\end{align}
\end{subequations}
\end{lemma}

By the definition of the norm  $|||\cdot|||_V$, we
further have the following conclusion (cf. \cite[(3.3)]{HX2019}, \cite[Lemma 2.4]{WangX.J2024Rgdw}, \cite[Lemma 3.3]{ZZX2023}):
\begin{lemma} \label{Lemma 2.4}
For any $ \bm{v}_{h}\in \bm{V}_{h}^{0}$, there hold
\begin{align}
 \|\nabla_{h} \bm{v}_{hi}\|_{0}\lesssim ||| \bm{v}_{h}|||_V
\end{align}
and
\begin{eqnarray}\label{Cr-est}
		\|\bm{v}_{hi}\|_{0,\widetilde{r}}\leq C_{\widetilde{r}} |||\bm{v}_{h}||| _V
	\end{eqnarray}
	for   $\widetilde{r}$ satisfying
	\begin{eqnarray*}
		\left\{
		\begin{aligned}
			&2\le \widetilde{r}< \infty,&\text{ if } \ n=2,\\
			&2\le \widetilde{r}\le 6 ,&\text{ if } \ n=3,
		\end{aligned}
		\right.
	\end{eqnarray*}
	where $C_{\widetilde{r}}>0$ is a constant only  depending on $\widetilde{r}$.
\end{lemma}

For any $K\in\mathcal{T}_h$, $ e\in\mathcal{E}_h$ and nonnegative integer $j$, %here $\zeta$ is a integer,
let $\Pi_j^\ast: L^2(K)\rightarrow P_{j}(K)$ and $\Pi_j^{B}: L^2(e)\rightarrow P_{j}(e)$ be the usua $L^2$-projection operators. We shall adopt $\bm{\Pi}_j^\ast$ to denote $\Pi_j^\ast$ for the vector form.

For  any  integer $ j\geq 0$, we introduce the local Raviart-Thomas (RT) element space
\begin{eqnarray}
	\bm{RT}_j(K)=[P_j(K)]^n+\bm{x}P_j(K), \ \forall K\in \mathcal{T}_{h}\nonumber
\end{eqnarray}
and the RT  projection operator $\bm{P}^{RT}_j: [H^1(K)]^n\rightarrow  \bm{RT}_j(K)$ (cf. \cite{RT-1}) defined by
	\begin{align}
			\langle\bm{P}^{RT}_j\bm{\omega}\cdot \bm{n}_e,\sigma \rangle_e&=\langle \bm{\omega}\cdot \bm{n}_e,\sigma\rangle_e, \quad\forall \sigma\in P_j(e),  e\in\mathcal{E}_{h} \text{ and } e\subset \partial K, \text{ for } j\geq 0,\label{RT1}\\
			(\bm{P}^{RT}_j\bm{\omega},\bm{\sigma})_K&=(\bm{\omega},\bm{\sigma})_K,  \quad\forall \bm{\sigma} \in [P_{j-1}(K)]^n, \text{ for } j\geq 1.\label{RT2}
		\end{align}
Then we have the following commutativity properties for the   RT projection, the $L^2$ projections and the discrete weak operators.
\begin{lemma}(cf. \cite[Lemma 3.7]{CFX2016})\label{Lemma 2.8}
 For $m\geq 1$, there hold
% \begin{subequations}
\begin{align*}
&\nabla _{w,l}\{ \bm{P}_m^{RT}\bm{\omega}, \bm{\Pi}_{m}^{B}\bm{\omega} \}=\bm{\Pi}_{l}^{\ast}(\nabla\bm{\omega} ),\quad \forall \bm{\omega}\in [H^{1}(\Omega)]^{n}, l=m, m-1,\\
&\nabla _{w,m}\{\Pi_{m-1}^{\ast}q,\Pi _m ^{B}q\}=\bm{\Pi}_{m}^{\ast}(\nabla q ),\quad\forall q\in H^{1}(\Omega).
\end{align*}
%\end{subequations}
 \end{lemma}

 Finally, we give several inequalities which will be used later (cf. \cite[Theorem 5.3.3]{CiarletP.G1978TFEM}, \cite[Lemma 1]{LiMinghao2019Tmfe}).
\begin{lemma} \label{Lemma 2.9}
For any $  \lambda,\mu\in \mathbb{R}^{n}$ and $r\geq2$, there hold
\begin{align*}
&\left| | \lambda|^{r-2}-|\mu |^{r-2}\right|\leq (r-2) (|\lambda|^{r-3}+|\mu|^{r-3})|\lambda-\mu|,\\
&\left|| \lambda|^{r-2}\lambda-|\mu |^{r-2}\mu\right|\leq C_{r} (|\lambda|+|\mu|)^{r-2}|\lambda-\mu|,\\
&\left(| \lambda|^{r-2}\lambda-|\mu |^{r-2}\mu\right)\cdot( \lambda-\mu)\gtrsim |\lambda-\mu|^{r},
\end{align*}
where  $|\cdot|$ denotes the Euclid norm and $C_{r}$ is a positive constant,  depending only on $r$, to be used in the latter analysis.
\end{lemma}

\section{Well-posedness of the semi-discrete WG scheme}

\subsection{Stability results }
Denote
\begin{eqnarray}\label{V0h}
\bm{V}_{0h}:=\{\bm{\kappa}_h\in\bm{V}_h^{0}:b_h(\bm{\kappa}_{h},q_h)=0,\forall q_h\in Q_h^{0}\}.
\end{eqnarray}
%\begin{remark}
From the proof of Theorem \ref{TH2.2}, it is trival to see that
\begin{eqnarray*}
\bm{V}_{0h}=\{\bm{\kappa}_h\in\bm{V}_h^{0}:\bm{\kappa}_{hi}\in H(div, \Omega), \nabla\cdot \bm{\kappa}_{hi}=0\}.
\end{eqnarray*}
%\end{remark}

Lemmas \ref{Lemma 3.1} and \ref{theoremLBB} give some stability  conditions for the Semi-discrete scheme \eqref{WG}.
\begin{lemma}\label{Lemma 3.1}
For any $ \bm{\kappa}_h=\{\bm{\kappa}_{hi},\bm{\kappa}_{hb}\},
\bm{u}_h=\{\bm{u}_{hi},\bm{u}_{hb}\},
\bm{v}_h=\{\bm{v}_{hi},\bm{v}_{hb}\}
\in\bm{V}_h^0$,  $\bm{f}\in L^{2}( [ L^{2}(\Omega) ]^{n})$, there hold
\begin{align}
a_h(\bm{u}_h,\bm{v}_h)&\lesssim \nu|||\bm{u}_{h}|||_{V}\cdot|||\bm{v}_{h}|||_{V},\label{a1}\\
a_h(\bm{v}_h,\bm{v}_h)&= \nu|||\bm{v}_{h}|||^{2}_{V},\label{a2}\\
c_{h}(\bm{v}_h;\bm{v}_h,\bm{v}_h)&= \alpha\|\bm{v}_{hi}\|_{0,r}^{r},\label{c2}\\
c_{h}(\bm{\kappa}_h;\bm{u}_h,\bm{v}_h)&\leq \alpha C_{\widetilde{r}}^{r} |||\bm{\kappa}_{h}|||^{r-2}_{V}|||\bm{u}_{h}|||_{V}\cdot|||\bm{v}_{h}|||_{V},\label{c1}\\
d_h(\bm{\kappa}_h;\bm{v}_h,\bm{v}_h)&=0 , \label{d0}\\
d_h(\bm{\kappa}_h;\bm{u}_h,\bm{v}_h)&\leq \mathcal{N}_h|||\bm{\kappa}_h|||_{V} \cdot |||\bm{u}_h|||_{V}\cdot |||\bm{v}_h|||_{V}%\nonumber \\&
\lesssim |||\bm{\kappa}_h|||_{V} \cdot |||\bm{u}_h|||_{V}\cdot |||\bm{v}_h|||_{V},\label{dh1}\\
(\bm f,\bm v_{hi})&\leq \|\bm{f}\|_{*,h} |||\bm{v}_h|||_{V}\lesssim  \|\bm{f}\|_0|||\bm{v}_h|||_{V},\label{fh}
\end{align}
where $C_{\widetilde{r}}$ is the same as  in \eqref{Cr-est} and
$$\mathcal{N}_h:=\sup_{\bm 0\neq\bm{\kappa}_h,\bm{u}_h,\bm{v}_h\in\bm{V}_{0 h}}
\frac{d_h(\bm{\kappa}_h;\bm{u}_h,\bm{v}_h)}{|||\bm{\kappa}_h|||_{V}\cdot |||\bm{u}_h|||_{V}\cdot|||\bm{v}_h|||_{V} },\quad \|\bm{f}\|_{*,h}:=  \sup\limits_{\bm 0\neq \bm{v}_h\in\bm{V}_{0h}}\frac{(\bm f,\bm v_{hi})}{|||\bm{v}_h|||_{V}} .$$
\end{lemma}
\begin{proof}
The results \eqref{a1} - \eqref{dh1} follow from \cite[Lemma 3.1]{WangX.J2024Rgdw}.
 From the definition of   $|||\cdot|||_V$,  H\"{o}lder's inequality and Lemma \ref{Lemma 2.4}, we obtain \eqref{fh}.
\end{proof}

We also have the following discrete inf-sup inequality and stability results.
 \begin{lemma}(\cite[Theorem 3.1]{CFX2016})\label{theoremLBB}
 There holds
\begin{eqnarray*}
\sup_{ \bm{v}_{h}\in \bm{V}_h^{0}}\frac{b_h(\bm{v}_{h},p_h)}{|||\bm{v}_{h}|||_{V} }\gtrsim |||p_h|||_{Q}, \quad \forall p_h\in Q_h^{0}.
\end{eqnarray*}
\end{lemma}

\begin{theorem} \label{Theorem 32}
Assume that $\bm{f}\in L^{2}( [ L^{2}(\Omega) ]^{n})$ and $\bm{u}_{hi}^{0}\in [ L^{2}(\Omega) ]^{n}.$
Then, for the semi-discrete scheme \eqref{WG},   the following  stability result holds
%\begin{subequations}
\begin{align}
\| \bm{u}_{hi}(t)\|_{0}^{2}+\nu\int_{0}^{t} |||\bm{u}_{h}(\tau)|||^{2}_{V}d\tau
+2\alpha\int_{0}^{t}\|\bm{u}_{hi}(\tau) \|_{0,r}^{r}d\tau%\nonumber\\&
&\lesssim \| \bm{u}_{hi}^{0}\|_{0}^{2}+\frac{1}{\nu}\int_{0}^{t}\|\bm{f}(\tau) \|_{*,h}^{2}d\tau.\label{3.25}
\end{align}
\end{theorem}
\begin{proof}
Taking $(\bm{v}_{h},q_{h})=(\bm{u}_{h},p_{h})$ in \eqref{WG} and utilizing  H\"{o}lder's inequality, we gain
\begin{align}
&\frac{1}{2}\frac{d}{dt}\| \bm{u}_{hi}(t)\|_{0}^{2}+\nu |||\bm{u}_{h}(t)|||^{2}_{V}+\alpha\|\bm{u}_{hi}(t) \|_{0,r}^{r}%\nonumber\\&
=(\bm{f},\bm{u}_{hi}(t))
\leq \frac{1}{2\nu}\|\bm{f} \|_{*,h}^{2}+\frac{\nu}{2}|||\bm{u}_{h}(t)|||^{2}_{V}.\label{3.27}
\end{align}
Integrating the inequality \eqref{3.27} with respect to $t$, we obtain \eqref{3.25}.
\end{proof}
%\section{Existence and uniqueness of the semi-discrete solution}
\subsection{Existence and uniqueness }

To obtain the existence and uniqueness of the semi-discrete solution for the scheme \eqref{WG}, we first introduce the following lemma.
\begin{lemma}(cf.\cite[Theorem 3.2]{TeschlGerald2012Odea}) \label{Lemma 41}
For the initial value problem
\begin{align}
\bm{y}^{'}(t)=f(t,\bm{y}(t)), \quad \bm{y}(t_{0})=\bm{y}_{0},
\end{align}
assume that $f:\Omega\times[0,T]\rightarrow \mathbb{R}^{n}$ is continuous and that $f (t, \bm{y})$ is  locally Lipschitz continuous in $\bm{y}$. Then there exists  $\epsilon>0$ such that for every $(t_{0}, \bm{y}_{0})$ the initial problem has a unique local solution defined on the interval $|t-t_{0}|<\epsilon$.
\end{lemma}

%\subsection{Existence result}
To prove the existence and uniqueness to the scheme \eqref{WG}, we introduce the following auxiliary system: find $\bm{u}_h\in \bm{V}_{0h}$ such that
\begin{align}\label{wg1}
		(\frac{\partial\bm{u}_{hi}}{\partial t},\bm{v}_{hi})+a_h(\bm{u}_h,\bm{v}_h)+c_h(\bm{u}_h ;\bm{u}_h,\bm{v}_h )
		+d_h(\bm{u}_h ;\bm{u}_h,\bm{v}_h )=&(\bm{f},\bm{v}_{hi}), \quad\forall \bm{v}_h\in \bm{V}_{0h},
\end{align}

then we have the following equivalent result.
\begin{lemma} \label{equivalent2}
The semi-discrete problems \eqref{WG} and \eqref{wg1} are equivalent in the sense that both $(i)$  and $(ii)$  hold:\\
$(i)$ If $(\bm{u}_h,p_h)\in \bm{V}_h^{0}\times Q_h^{0}$ solves \eqref{WG}, then $\bm{u}_h\in \bm{V}_{0h}$ solves  \eqref{wg1};\\
$(ii)$ If $\bm{u}_h\in \bm{V}_{0h}$ solves  \eqref{wg1}, then  $(\bm{u}_h,p_h)$ solves \eqref{WG}, where $p_h\in Q_h^{0} $ is determined by
\begin{align*}
		b_h(\bm{v}_{h},p_h)=&(\bm{f},\bm{v}_{hi})-(\frac{\partial\bm{u}_{hi}}{\partial t},\bm{v}_{hi})-a_{h}(\bm{u}_h, \bm{v}_h )%\nonumber\\&
		-c_h(\bm{u}_h ;\bm{u}_h ,\bm{v}_h)-d_h(\bm{u}_h ;\bm{u}_h ,\bm{v}_h), \quad\forall \bm{v}_{h}\in \bm{V}_h^{0}.
	\end{align*}
\end{lemma}
According to  Lemma \ref{equivalent2}, we only need to take into consideration the system \eqref{wg1} for the existence
and uniqueness of the semi-discrete WG solution to the system \eqref{WG}.

Taking $\bm{v}_{h}=\{\bm{v}_{hi},\bm{0}\}\in \bm{V}_{0h}$ in \eqref{wg1}, we derive
\begin{align}\label{4.3}
&(\frac{\partial\bm{u}_{hi}}{\partial t},\bm{v}_{hi})-\nu ( \nabla\cdot\nabla _{w,l}\bm{u}_{h},\bm{v}_{hi} )
+\nu \langle \eta ( \bm{u}_{hi}-\bm{u}_{hb}), \bm{v}_{hi}\rangle_{\partial\mathcal{T}_{h}}
+\alpha (|\bm{u}_{hi}|^{r-2}\bm{u}_{hi},\bm{v}_{hi} )\nonumber\\
&-\frac{1}{2}(\bm{u}_{hi}\otimes \bm{u}_{hi},\nabla_{h} \bm{v}_{hi})
+\frac{1}{2}\langle(\bm{u}_{hb}\otimes \bm{u}_{hi})\bm{n}_{K},\bm{v}_{hi}\rangle_{\partial\mathcal{T}_{h}}
+\frac{1}{2}(\bm{v}_{hi}\otimes \bm{u}_{hi},\nabla_{h} \bm{u}_{hi})
=( \bm{f},\bm{v}_{hi}).
\end{align}
And then, taking $\bm{v}_{h}=\{\bm{0},\bm{v}_{hb}\}\in \bm{V}_{0h}$ in \eqref{wg1}, we gain
\begin{align}\label{4.4}
\nu \langle \nabla _{w,l}\bm{u}_{h}\bm{n}_{K},\bm{v}_{hb} \rangle_{\partial\mathcal{T}_{h}}
-\nu \langle \eta ( \bm{u}_{hi}-\bm{u}_{hb}), \bm{v}_{hb}\rangle_{\partial\mathcal{T}_{h}}
-\frac{1}{2}\langle(\bm{v}_{hb}\otimes \bm{u}_{hi})\bm{n}_{K},\bm{u}_{hi}\rangle_{\partial\mathcal{T}_{h}}=0.
\end{align}

\begin{theorem} \label{TH4.1}
Assume that $\bm{f}(t,\cdot)$ is continuous with respect to $t$, then there exists  a $\hat{t}\in (0,T]$
such that the  system \eqref{wg1} admits a unique solution $\bm{u}_h\in \bm{V}_{0h}$ for all $t\in (0,\hat{t}]$.
\end{theorem}
\begin{proof}
Let $\bm{\Phi}:=\{\bm{\Phi}_{i},\bm{\Phi}_{b}\}$ and $\bm{\Phi}_{w}$ be  the bases of $\bm{V}_{h}|_{K}$ and $[P_{l}(K)]^{n}$, respectively, with
\begin{align}
\bm{\Phi}_{i}=\{\bm{\Phi}_{i1},...,\bm{\Phi}_{il_{1}}\},\quad
\bm{\Phi}_{b}=\{\bm{\Phi}_{b1},...,\bm{\Phi}_{bl_{2}}\},\quad
\bm{\Phi}_{w}=\{\bm{\Phi}_{w1},...,\bm{\Phi}_{wl_{3}}\}.\
\end{align}
Let $\bm{u}_{h}(t)|_{K}:=\{\bm{u}_{hi}(t)|_{K} ,\bm{u}_{hb}(t)|_{\partial K}\}=\{\bm{\Phi}_{i}\bm{\beta}_{i}(t),\bm{\Phi}_{b}\bm{\beta}_{b}(t) \}$ and $\big( \nabla_{w,l}u_{h}(t)\big)|_{K}:=\bm{\Phi}_{w}\bm{\beta}_{w}(t)$ with
 \begin{align}
 \bm{u}_{hi}(t)=&(u_{hi,1}(t),...,u_{hi,n}(t))^{T},\quad
  \bm{u}_{hb}(t)=(u_{hb,1}(t),...,u_{hb,n}(t))^{T},\nonumber\\
\bm{\beta}_{i}(t)=&(\beta_{i1}(t),...,\beta_{il_{1}}(t) )^{T},\quad\quad
 \bm{\beta}_{b}(t)=(\beta_{b1}(t),...,\beta_{bl_{2}}(t) )^{T},\nonumber\\
 \bm{\beta}_{w}(t)=&(\beta_{w1}(t),...,\beta_{wl_{3}}(t) )^{T},\quad
 |\bm{u}_{hi}(t)|=(u_{hi,1}^{2}(t)+...+u_{hi,n}^{2}(t) )^{1/2}.
\end{align}

Denote
\begin{align*}
\mathcal{R}_{0}&=\sum_{K\in\mathcal{T}_{h}}\int_{K}\bm{\Phi}_{w}^{T}\bm{\Phi}_{w} dx,\quad\quad\quad\quad
\mathcal{R}_{1}=\sum_{K\in\mathcal{T}_{h}}\int_{K}(\nabla\cdot\bm{\Phi}_{w})^{T}\bm{\Phi}_{i} dx,\\
\mathcal{R}_{2}&=\sum_{K\in\mathcal{T}_{h}}\int_{\partial K}(\bm{\Phi}_{w}\bm{n}_{e})^{T}\bm{\Phi}_{b}ds,\quad\quad
\mathcal{R}_{3}=\sum_{K\in\mathcal{T}_{h}}\int_{K}\bm{\Phi}_{i}^{T}\bm{\Phi}_{i} dx,\\
\mathcal{R}_{4}&=\sum_{K\in\mathcal{T}_{h}}\int_{\partial K}\eta\bm{\Phi}_{i}^{T}\bm{\Phi}_{i}ds,\quad\quad\quad\quad
\mathcal{R}_{5}=\sum_{K\in\mathcal{T}_{h}}\int_{\partial K}\eta\bm{\Phi}_{i}^{T}\bm{\Phi}_{b}ds,\\
\mathcal{R}_{6}&=\sum_{K\in\mathcal{T}_{h}}\int_{K}(\nabla\bm{\Phi}_{i})^{T}A_{u,\Phi} dx,\quad\quad
\mathcal{R}_{7}=\sum_{K\in\mathcal{T}_{h}}\int_{\partial K}\bm{\Phi}_{i}^{T}\widehat{A}_{u,\Phi} ds,\\
\mathcal{R}_{8}&=\sum_{K\in\mathcal{T}_{h}}\int_{\partial K}\eta\bm{\Phi}_{b}^{T}\bm{\Phi}_{b}ds,\quad\quad\quad\quad
\mathcal{R}_{9}=\sum_{K\in\mathcal{T}_{h}}\int_{K}|\bm{u}_{hi}(t)|^{r-2} \bm{\Phi}_{i}^{T}\bm{\Phi}_{i} dx,\\%\widetilde{ A}_{u,\Phi} dx,\\
F(t)&=\sum_{K\in\mathcal{T}_{h}}\int_{ K}\bm{\Phi}_{i}^{T} \bm{f}(t)dx,
\end{align*}
where, $\bm{n}_{e}=(\widehat{n}_1,\widehat{n}_2,...,\widehat{n}_n)^{T}$,
\begin{align*}
A_{u,\Phi}=&
\begin{pmatrix}%{ccc}
 \bm u_{hi,1}\bm{\Phi}_{i1}& \cdots&\bm u_{hi,1}\bm{\Phi}_{i l_{1}}\\
 \vdots                & \ddots& \vdots\\
 \bm u_{hi,n}\bm{\Phi}_{i1}& \cdots&\bm u_{hi,n}\bm{\Phi}_{i l_{1}}\\
 \end{pmatrix}
 ,\quad
 % \quad\text{in}\quad \mathcal{R}_{6},\\
 \widehat{A}_{u,\Phi}=
\begin{pmatrix}%{ccc}
 \bm u_{hi,1}\bm{\Phi}_{b1}\widehat{n}_1& \cdots&\bm u_{hi,1}\bm{\Phi}_{i l_{2}}\widehat{n}_1\\
 \vdots                & \ddots& \vdots\\
 \bm u_{hi,n}\bm{\Phi}_{b1}\widehat{n}_n& \cdots&\bm u_{hi,n}\bm{\Phi}_{i l_{2}}\widehat{n}_n\\
 \end{pmatrix}
 ,
 % \quad\text{in}\quad \mathcal{R}_{7},
\end{align*}
and we recall that $u_{hi}=(u_{hi,1},...,u_{hi,n})^{T}=\bm{\Phi}_{i}\bm{\beta}_{i}(t).$
Then, \eqref{weak gradient}, \eqref{4.3} and \eqref{4.4} can be written as
\begin{subequations}\label{R}
\begin{align}
\mathcal{R}_{0}\widehat{\bm{\beta}}_{w}(t)+\mathcal{R}_{1}\widehat{\bm{\beta}}_{i}(t)-\mathcal{R}_{2}\widehat{\bm{\beta}}_{b}(t)=&0,\\
\mathcal{R}_{3}\frac{d \widehat{\bm{\beta}}_{i}(t)}{dt}-\nu\mathcal{R}_{1}^{T}\widehat{\bm{\beta}}_{w}(t)
+\big(\nu\mathcal{R}_{4}+\alpha\mathcal{R}_{9}-\frac{1}{2}(\mathcal{R}_{6}-\mathcal{R}_{6}^{T} )\big)\widehat{\bm{\beta}}_{i}(t)%&\nonumber\\
+(-\nu\mathcal{R}_{5}+\frac{1}{2}\mathcal{R}_{7})\widehat{\bm{\beta}}_{b}(t)=&F(t),\\
\nu \mathcal{R}_{2}^{T}\widehat{\bm{\beta}}_{w}(t)-(\nu\mathcal{R}_{5}^{T}+\frac{1}{2}\mathcal{R}_{7}^{T})\widehat{\bm{\beta}}_{i}(t)
+\nu\mathcal{R}_{8}\widehat{\bm{\beta}}_{b}(t)=&0,
\end{align}
\end{subequations}
where, $\widehat{\bm{\beta}}_{w}(t)|_{K}=\bm{\beta}_{w}(t)$, $\widehat{\bm{\beta}}_{i}(t)|_{K}=\bm{\beta}_{i}(t)$,  $\widehat{\bm{\beta}}_{b}(t)|_{K}=\bm{\beta}_{b}(t)$ and $\widehat{\bm{\beta}}(t)|_{K}=\bm{\beta}(t)$.
Since $\mathcal{R}_{0},\mathcal{R}_{3},\mathcal{R}_{4}$ and $\mathcal{R}_{8}$ are symmetric positive defined, we can eliminate $\widehat{\bm{\beta}}_{i}(t)$ and $\widehat{\bm{\beta}}_{b}(t)$ from \eqref{R} to derive
\begin{align}\label{4.8}
\mathcal{R}_{3}\frac{d \widehat{\bm{\beta}}_{i}(t)}{dt}+\widehat{\mathcal{R}}\big(\widehat{\bm{\beta}}_{i}(t)\big)\widehat{\bm{\beta}}_{i}(t)=F(t).
\end{align}
Here,
\begin{align*}
\widehat{\mathcal{R}}(\widehat{\bm{\beta}}_{i}(t))=&-\big(\nu \mathcal{R}_{1}^{T}+(-\nu\mathcal{R}_{5}+\frac{1}{2}\mathcal{R}_{7})\mathcal{R}_{8}^{-1}\mathcal{R}_{2}^{T}\big)
(\mathcal{R}_{0}+\mathcal{R}_{2}\mathcal{R}_{8}^{-1}\mathcal{R}_{2}^{T} )^{-1}%\nonumber\\&
\big(-\mathcal{R}_{1}+\mathcal{R}_{2}\mathcal{R}_{8}^{-1}(\mathcal{R}_{5}^{T}+\frac{\mathcal{R}_{7}^{T}}{2\nu} )\big)\nonumber\\&
+\big(\nu\mathcal{R}_{4}+\alpha\mathcal{R}_{9}-\frac{1}{2}(\mathcal{R}_{6}-\mathcal{R}_{6}^{T} )\big)%\nonumber\\&
+(-\nu\mathcal{R}_{5}+\frac{1}{2}\mathcal{R}_{7})\mathcal{R}_{8}^{-1}(\mathcal{R}_{5}^{T}+\frac{\mathcal{R}_{7}^{T}}{2\nu} ).
\end{align*}
Notice that $\mathcal{R}_{0}$ and $\mathcal{R}_{8}$ are symmetric positive defined. Thus, $\mathcal{R}_{8}$ is invertible, as well as, $\mathcal{R}_{0}+\mathcal{R}_{2}\mathcal{R}_{8}^{-1}\mathcal{R}_{2}^{T}$ are symmetric defined, which imply invertible.
In addition, $\mathcal{R}_{6}$, $\mathcal{R}_{7}$ and $\mathcal{R}_{9}$ rely on $\widehat{\bm{\beta}}_{i}$.
According to  the stability result in Theorem \ref{Theorem 32}, we can gain that
 $\mathcal{R}_{6}=\mathcal{R}_{6}(\widehat{\bm{\beta}}_{i}),%\quad
 \mathcal{R}_{7}=\mathcal{R}_{7}(\widehat{\bm{\beta}}_{i}),
 \mathcal{R}_{9}=\mathcal{R}_{9}(\widehat{\bm{\beta}}_{i})$
are bounded in $[0,T]$, which imply that $ \widehat{\mathcal{R}}\big(\widehat{\bm{\beta}}_{i}(t)\big)\widehat{\bm{\beta}}_{i}(t)$ is globally Lipschitz continuous with respect to $\widehat{\bm{\beta}}_{i}$.
 In addition, $\mathcal{R}_{3}$ is symmetric positive defined. From  Lemma \ref{Lemma 41}, it follows that there exists $\hat{t}\in (0,T]$ such that the system \eqref{4.8} exists a unique solution $\widehat{\bm{\beta}}_{i}$ on $(0,\hat{t}].$ And the system \eqref{R}  has one unique solution $(\widehat{\bm{\beta}}_{i}(t),\widehat{\bm{\beta}}_{b}(t),\widehat{\bm{\beta}}_{w}(t))$,
which implies that the existence and uniqueness of $\bm{u}_{h}$ in \eqref{wg1} on $(0,\hat{t}].$
we complete this proof.
\end{proof}

Combining the stability result \eqref{3.25}, the ordinary differential equations theory, Lemmas \ref{theoremLBB} and \ref{equivalent2}, and Theorem \ref{TH4.1},  we obtain the following existence and uniqueness results.
\begin{theorem} \label{TH4.2}
Under the same conditions as in Theorem \ref{TH4.1}, for any $t\in(0,T],$ the semi-discrete WG scheme \eqref{WG}  has a unique solution $(\bm{u}_h,p_h)\in \bm{V}_h^{0}\times Q_h^{0}$.
\end{theorem}

\section{A priori error estimates for the semi-discrete WG scheme}
\label{section4}
In this section, we provide the error analysis for the semi-discrete WG scheme \eqref{WG}. To this end, we first assume that the solution $(\bm{u},p)$ to \eqref{weak} satisfies the following regularity conditions for $m\geq 1:$
\begin{align}\label{regularity}
\bm{u}\in& L^{\infty}([ H^{2}(\Omega)]^{n} )\cap L^{2}(\bm{V}\cap [H^{m+1}(\Omega)]^{n} ),\quad%\\
p\in L^{2}( L^{2}_{0} \cap H^{m+1}(\Omega)),\nonumber\\
 \bm{u}_{t}\in& L^{\infty} ([H^{2}(\Omega)]^{n})\cap L^{2} ([H^{m+1}(\Omega)]^{n}),\quad
  \bm{u}_{tt}\in L^{2} [ L^{2}(\Omega)]^{n}.
\end{align}

Define
\begin{align}
\bm{\mathcal{I}}_{h}\bm{u}|_K:=\{\bm{P}_{m}^{RT}(\bm{u}|_K), \bm{\Pi} _{m}^{B}(\bm{u}|_K)\},\quad
\mathcal{P}_{h}p|_K:=\{ \Pi_{m-1}^{\ast}(p|_K),\Pi_{m}^{B}(p|_K) \},\ m\geq 1.
\end{align}
%Here we recall that $m\geq 1$.
The  initial data is given as
\begin{align}
\bm{u}_{h}^{0}:=\bm{\mathcal{I}}_{h}\bm{u}_{0}=\{\bm{P}_{m}^{RT}\bm{u}_{0}, \bm{\Pi} _{m}^{B}\bm{u}_{0}\}.
\end{align}

\subsection{Primary results}

Firstly, we give several primary results to prepare for error estimates in the next subsection.
\begin{lemma}\label{Lemma 5.3}
Assume that $(\bm{u},p)$ is the solution to \eqref{weak}, for any $(\bm{v}_h,q_h)\in \bm{V}_h^{0}\times Q_h^{0}$, there hold
	\begin{eqnarray}
	\bm{P}^{RT}_m\bm{u}|_K\in [P_{m}(K)]^n, \quad \forall K\in\mathcal{T}_h, \label{4.5}
	\end{eqnarray}
% the   equations
\begin{subequations}\label{ee}
\begin{align}
a_{0}(\bm{\mathcal{I}}_{h}\bm{u}_{t},\bm{v}_{h})
+a_h(\bm{\mathcal{I}}_{h}\bm{u},\bm{v}_h)+b_h(\bm{v}_h,\mathcal{P}_{h}p)
+c_h(\bm{\mathcal{I}}_{h}\bm{u};\bm{\mathcal{I}}_{h}\bm{u},\bm{v}_h)
+ d_h(\bm{\mathcal{I}}_{h}\bm{u};\bm{\mathcal{I}}_{h}\bm{u},\bm{v}_h)&\nonumber\\
=(\bm{f},\bm{v}_{hi}) +\xi_{I}(\bm{u};\bm{u},\bm{v}_h)+\xi_{II}(\bm{u},\bm{v}_h)%\nonumber\\
+\xi_{III}(\bm{u};\bm{u},\bm{v}_h)
+ (\bm{P}_m^{RT}\bm{u}_{t}-\bm{u}_{t},\bm{v}_{hi} ),&\label{e2}\\
	b_h(\bm{\mathcal{I}}_{h}\bm{u},q_h)=0,&\label{e3}
\end{align}
\end{subequations}
	where, $a_{0}(\bm{\mathcal{I}}_{h}\bm{u}_{t},\bm{v}_{h})=(\bm{P}_m^{RT}\bm{u}_{t},\bm{v}_{hi})$,
	\begin{align*}
\xi_{I}(\bm{u};\bm{u},\bm{v}_{h})=& -\frac{1}{2}(\bm{P}_m^{RT}\bm{u}\otimes \bm{P}_m^{RT}\bm{u}-\bm{u}\otimes\bm{u},\nabla_h\bm{v}_{hi})%\nonumber\\&
+\frac{1}{2}\langle(\bm{\Pi}_{m}^{B}\bm{u}\otimes \bm{\Pi}_{m}^{B}\bm{u}
-\bm{u}\otimes\bm{u})\bm{n},\bm{v}_{hi}\rangle_{\partial\mathcal{T}_h}\nonumber\\
&-\frac{1}{2}((\bm{u}\cdot\nabla)\bm{u}-(\bm{P}_m^{RT}\bm{u}\cdot\nabla_h)\bm{P}_m^{RT}\bm{u},\bm{v}_{hi})%\nonumber\\&
-\frac{1}{2}\langle(\bm{v}_{hb}\otimes \bm{\Pi}_{m}^{B}\bm{u})\bm{n},\bm{P}_m^{RT}\bm{u}\rangle_{\partial\mathcal{T}_h},\nonumber\\
	\xi_{II}(\bm{u},\bm{v}_h)=& \nu\langle (\nabla\bm u-\bm{\Pi}_{l}^\ast\nabla\bm u)\bm{n},\bm{v}_{hi}-\bm{v}_{hb} \rangle_{\partial\mathcal{T}_h} +\nu\langle \eta(\bm{P}^{RT}_m\bm{u}-{\bm{u}}),\bm{v}_{hi}-\bm{v}_{hb} \rangle_{\partial\mathcal{T}_h},\nonumber\\
\xi_{III}(\bm{u};\bm{u},\bm{v}_h)=&
\alpha (|\bm{P}^{RT}_m\bm{u}|^{r-2}\bm{P}^{RT}_m\bm{u}-|\bm{u}|^{r-2}\bm{u},\bm{v}_{hi}).
	\end{align*}
	
\end{lemma}
\begin{proof}
For all $ K\in\mathcal{T}_h,\varrho_m\in P_{m}(K)$, utilizing the  property of  RT projection (cf. \cite[Lemma 3.1]{RT-1}), we obtain
	\begin{eqnarray*}
		(\nabla\cdot\bm{P}^{RT}_{m}\bm{u},\varrho_m)_K=(\nabla\cdot\bm{u},\varrho_m)_K=0,
	\end{eqnarray*}
	which means that $\nabla\cdot\bm{P}^{RT}_m\bm{u}=0$,  i.e. \eqref{4.5} holds.
	
	 From the definition  of discrete weak divergence,  Green's formula and the definition of   the trilinear form $d_h(\cdot;\cdot,\cdot)$, it is trival to have
	 \begin{eqnarray*}
	d_h(\bm{\mathcal{I}}_{h}\bm{u} ;\bm{\mathcal{I}}_{h}\bm{u},\bm{v}_{h})
	=(\nabla\cdot(\bm{u}\otimes\bm{u}),\bm{v}_{hi})+\xi_{I}(\bm{u};\bm{u},\bm{v}_{h}). %\label{EN1}
	\end{eqnarray*}

  Therefore, by Lemma    \ref{Lemma 2.8}, the   projection properties and the definition of discrete weak gradient, we can obtain
\begin{align*}
&a_{0}(\bm{\mathcal{I}}_{h}\bm{u}_{t},\bm{v}_{h})+
a_h(\bm{\mathcal{I}}_{h}\bm{u},\bm{v}_h)+b_h(\bm{v}_h,\mathcal{P}_{h}p)
+c_h(\bm{\mathcal{I}}_{h}\bm{u};\bm{\mathcal{I}}_{h}\bm{u},\bm{v}_h)
+ d_h(\bm{\mathcal{I}}_{h}\bm{u};\bm{\mathcal{I}}_{h}\bm{u},\bm{v}_h)\nonumber\\
=&(\bm{P}_m^{RT}\bm{u}_{t},\bm{v}_{hi})+\nu (\nabla_{w,l}\{ \bm {P}_{m}^{RT}\bm{u},\bm {\Pi}_{k}^{B}\bm{u} \},\nabla_{w,l}\bm{v}_{h}  )
+\nu\langle \eta (\bm {P}_{m}^{RT}\bm{u}-\bm{u}), \bm{v}_{hi}-\bm{v}_{hb}\rangle_{\partial\mathcal{T}_h}\nonumber\\
&+(\bm{v}_{hi}, \nabla_{w,m}\{  \Pi_{m-1}^{\ast}p, \Pi_{m}^{B}p \} )
+\alpha(| \bm {P}_{m}^{RT}\bm{u}|^{r-2}\bm {P}_{m}^{RT}\bm{u},\bm{v}_{hi} )%\nonumber\\&
+(\nabla\cdot(\bm{u}\otimes\bm{u}),\bm{v}_{hi})+\xi_{I}(\bm{u};\bm{u},\bm{v}_{h})\nonumber\\
=&(\bm{u}_{t},\bm{v}_{hi})+\nu ( \bm {\Pi}_{l}^{\ast}(\nabla\bm{u}), \nabla_{w,l}\bm{v}_{h} )
+\nu\langle \eta (\bm {P}_{m}^{RT}\bm{u}- \bm{u}), \bm{v}_{hi}-\bm{v}_{hb}\rangle_{\partial\mathcal{T}_h}
+(\bm{v}_{hi}, \bm {\Pi}_{m}^{\ast}(\nabla p))\nonumber\\
&+\alpha(| \bm{u}|^{r-2}\bm{u},\bm{v}_{hi} )+\xi_{III}(\bm{u};\bm{u},\bm{v}_{h})%\nonumber\\
+(\nabla\cdot(\bm{u}\otimes\bm{u}),\bm{v}_{hi})+\xi_{I}(\bm{u};\bm{u},\bm{v}_{h})
+(\bm{P}_m^{RT}\bm{u}_{t}-\bm{u}_{t},\bm{v}_{hi})\nonumber\\
=&(\bm{u}_{t},\bm{v}_{hi})-\nu (\nabla_{h}\cdot \bm {\Pi}_{l}^{\ast}(\nabla\bm{u}), \bm{v}_{hi} )
+ \nu\langle\bm {\Pi}_{l}^{\ast}(\nabla\bm{u})\bm{n}, \bm{v}_{hb} \rangle_{\partial\mathcal{T}_h}%\nonumber\\&
+\nu\langle \eta (\bm {P}_{m}^{RT}\bm{u}- \bm{u}), \bm{v}_{hi}-\bm{v}_{hb}\rangle_{\partial\mathcal{T}_h}\nonumber\\
&+(\nabla\cdot(\bm{u}\otimes\bm{u}),\bm{v}_{hi})
+(\bm{v}_{hi}, \nabla p)
+\alpha(| \bm{u}|^{r-2}\bm{u},\bm{v}_{hi} )
+\xi_{I}(\bm{u};\bm{u},\bm{v}_{h})
+\xi_{III}(\bm{u};\bm{u},\bm{v}_{h})
+(\bm{P}_m^{RT}\bm{u}_{t}-\bm{u}_{t},\bm{v}_{hi}),
\end{align*}
which, together with  Green's formula, the   projection properties, the fact that $\langle (\nabla
\bm{u})\bm{n}, \bm{v}_{hb}\rangle_{\partial\mathcal{T}_h}=0$, and  the first equation of \eqref{BF0}, meams
\begin{align*}
&a_{0}(\bm{\mathcal{I}}_{h}\bm{u}_{t},\bm{v}_{h})+
a_h(\bm{\mathcal{I}}_{h}\bm{u},\bm{v}_h)+b_h(\bm{v}_h,\mathcal{P}_{h}p)
+c_h(\bm{\mathcal{I}}_{h}\bm{u};\bm{\mathcal{I}}_{h}\bm{u},\bm{v}_h)
+ d_h(\bm{\mathcal{I}}_{h}\bm{u};\bm{\mathcal{I}}_{h}\bm{u},\bm{v}_h)\nonumber\\
=&(\bm{u}_{t},\bm{v}_{hi})+\nu (\bm {\Pi}_{l}^{\ast}(\nabla\bm{u}), \nabla_{h} \bm{v}_{hi} )
+ \nu\langle\bm {\Pi}_{l}^{\ast}(\nabla\bm{u})\bm{n}, \bm{v}_{hb}-\bm{v}_{hi} \rangle_{\partial\mathcal{T}_h}%\nonumber\\&
+\nu\langle \eta (\bm {P}_{m}^{RT}\bm{u}- \bm{u}), \bm{v}_{hi}-\bm{v}_{hb}\rangle_{\partial\mathcal{T}_h}\nonumber\\
&+(\nabla\cdot(\bm{u}\otimes\bm{u}),\bm{v}_{hi})
+(\bm{v}_{hi}, \nabla p)+\alpha(| \bm{u}|^{r-2}\bm{u},\bm{v}_{hi} )
+\xi_{I}(\bm{u};\bm{u},\bm{v}_{h})+\xi_{III}(\bm{u};\bm{u},\bm{v}_{h})
+(\bm{P}_m^{RT}\bm{u}_{t}-\bm{u}_{t},\bm{v}_{hi})\nonumber\\
=&(\bm{u}_{t},\bm{v}_{hi})-\nu( \triangle\bm{u}, \bm{v}_{hi} )
+\nu\langle(\nabla \bm{u}-\bm {\Pi}_{l}^{\ast}\nabla\bm{u})\bm{n}, \bm{v}_{hi} -\bm{v}_{hb} \rangle_{\partial\mathcal{T}_h}%\nonumber\\&
+\nu\langle \eta (\bm {P}_{m}^{RT}\bm{u}- \bm{u}), \bm{v}_{hi}-\bm{v}_{hb}\rangle_{\partial\mathcal{T}_h}\nonumber\\
&+(\nabla\cdot(\bm{u}\otimes\bm{u}),\bm{v}_{hi})
+(\bm{v}_{hi}, \nabla p)+\alpha(| \bm{u}|^{r-2}\bm{u},\bm{v}_{hi} ) %\nonumber\\
+\xi_{I}(\bm{u};\bm{u},\bm{v}_{h})+\xi_{III}(\bm{u};\bm{u},\bm{v}_{h})
+(\bm{P}_m^{RT}\bm{u}_{t}-\bm{u}_{t},\bm{v}_{hi})\nonumber\\
=&(\bm{f}, \bm{v}_{hi})+\xi_{I}(\bm{u};\bm{u},\bm{v}_{h})
+\xi_{II}(\bm{u},\bm{v}_{h})+\xi_{III}(\bm{u};\bm{u},\bm{v}_{h})
+(\bm{P}_m^{RT}\bm{u}_{t}-\bm{u}_{t},\bm{v}_{hi}).
\end{align*}
This proves \eqref{e2}.

Finally, the relation \eqref{e3} follows from the definition of $\nabla_{w,m}$, the fact $\nabla\cdot\bm{P}^{RT}_{m}\bm{u}=0$ and  \eqref{RT1}.
This finishes the proof.
\end{proof}

%\color{black}
 By following a similar line as in the proofs of \cite[Lemma 4.3]{HX2019}, \cite[Lemma 5.2]{ZZX2023} and \cite[Lemma 4.2]{WangX.J2024Rgdw}, we can obtain the  estimates of  $\xi_{I}$, $\xi_{II}$, and $\xi_{III}$.
\begin{lemma} \label{Lemma 5.5}
For any $\bm{v}_h\in\bm{V}_h^{0}$, there hold
\begin{subequations}
\begin{align}
|\xi_{I}(\bm{u},\bm{u};\bm{v}_h)|&\lesssim  h^{m}
\|\bm{u}\|_{2}\|\bm{u}\|_{1+m}|||\bm{v}_h|||_{V},\label{X1}\\
|\xi_{II}(\bm{u};\bm{v}_h)|&\lesssim  h^{m}\|\bm{u}\|_{1+m}|||\bm{v}_h|||_{V},\label{X2}\\
|\xi_{III}(\bm{u},\bm{u}; \bm{v}_h)|&\lesssim  h^{m}\|\bm{u}\|_{2}^{r-2}\|\bm{u}\|_{1+m}|||\bm{v}_h|||_{V}.\label{X3}
\end{align}
\end{subequations}
\end{lemma}

\begin{lemma}\cite[Lemma 4.6]{HLX2019} \label{Lemma 5.6}
Assume $\bm{u}_{t}\in [H^{m+1}(\Omega)]^{n}$, then for all $\bm{v} \in \bm{V}_{h}^{0}$, it holds
\begin{align}
|(\bm{P}_m^{RT}\bm{u}_{t}-\bm{u}_{t},\bm{v}_{hi} )|\lesssim h^{m+1}|\bm{u}_{t}|_{m+1}\|\bm{v}_{hi}\|_{0}.
\end{align}
\end{lemma}

\subsection{Error estimate}

To obtain the a  priori error estimates, we introduce the following auxiliary problem:
find $(\widetilde{\bm{u}},\phi)$ such that
\begin{eqnarray}\label{auxiliary}
\left\{
\begin{aligned}
\widetilde{\bm{u}}_{t}-\nu \Delta \widetilde{\bm{u}}+\alpha |\widetilde{\bm{u}}|^{r-2}\widetilde{\bm{u}}+\nabla \phi&=\bm{f}-(\bm{u}\cdot\nabla)\bm{u},
&\text{in} \ \Omega\times[0,T],\\
\nabla\cdot\widetilde{\bm{u}}&=0,&\text{in} \ \Omega\times[0,T],\\
\widetilde{\bm{u}}(\bm{x},0 )&=\bm{u}_{0}(\bm{x}),& \text{in}\ \Omega,\\
\widetilde{\bm{u}}&=\bm{0},&\text{on } \partial \Omega. \\
\end{aligned}
\right.
\end{eqnarray}
And we consider the corresponding WG scheme: seek $\widetilde{\bm{u}}_{h}\in \bm{V}_{0h}$ such that
\begin{align}\label{AWG}
(\frac{\partial\widetilde{\bm{u}}_{hi}}{\partial t},\bm{v}_{h})+a_{h}(\widetilde{\bm{u}}_{h}, \bm{v}_{h})+c_{h}(\widetilde{\bm{u}}_{h};\widetilde{\bm{u}}_{h},\bm{v}_{h} )=
( \bm{f}, \bm{v}_{hi})%&\nonumber\\
-d_{h}(\bm{\mathcal{I}}_{h}\bm{u};\bm{\mathcal{I}}_{h}\bm{u},\bm{v}_{h} ),\quad\forall \bm{v}_{h}\in \bm{V}_{0h},&
\end{align}
with the initial data $\widetilde{\bm{u}}_{h}(\bm{x},0)=\widetilde{\bm{u}}_{h}^{0}=\bm{u}_{h}^{0}$. %Here, $\bm{V}_{0h}$ is defined in \eqref{V0h}.

Let
$$\bm{e}_{u}=\bm{u}-\bm{u}_{h}=(\bm{u}- \bm{\mathcal{I}}_{h}\bm{u})+( \bm{\mathcal{I}}_{h}\bm{u}-\widetilde{\bm{u}}_{h})+( \widetilde{\bm{u}}_{h}-\bm{u}_{h} )
:=\delta_{u}+\varphi_{u}+\chi_{u},$$
where
$$\delta_{u}=\bm{u}- \bm{\mathcal{I}}_{h}\bm{u}, \ \varphi_{u}= \bm{\mathcal{I}}_{h}\bm{u}-\widetilde{\bm{u}}_{h}, \ \chi_{u}= \widetilde{\bm{u}}_{h}-\bm{u}_{h}
,$$
and $\varphi_{u}^{0}=\chi_{u}^{0}=0.$

For the above error decomposition, we have the following lemmas.

\begin{lemma} \label{Lemma 5.7}
Assume that $\bm{u}\in L^{\infty} ([H^{2}]^{n})\cap L^{2} ([H^{m+1}]^{n})$ and $\bm{u}_{t}\in L^{2} ([H^{m+1}]^{n})$. Let  $\widetilde{\bm{u}}_{h}\in \bm{V}_{0h}$ be the solution to \eqref{AWG}, then there holds
\begin{align}\label{5.22}
\nu \int_{0}^{t} |||\varphi_{u}(\tau) |||^{2}_{V} d\tau   \lesssim \mathcal{C}(\bm{u})h^{2m}.
\end{align}
In addition, if $\bm{u}_{t}\in L^{\infty} ([H^{2}]^{n})\cap L^{2} ([H^{m+1}]^{n})$, then
\begin{align}\label{5.23}
\int_{0}^{t}\|\frac{\partial\varphi_{ui}}{\partial t}(\tau)\|_{0}^{2}d\tau+ \nu |||\varphi_{u}(t) |||^{2}_{V}
\lesssim \mathcal{C}(\bm{u})h^{2m}.
\end{align}
Here, $\mathcal{C}(\bm{u})>0$ is a constant depending only on $ \nu, \alpha, r$ and the regularity of $\bm{u}.$
\end{lemma}
\begin{proof}
According to  Lemma \ref{Lemma 5.3} and \eqref{AWG}, we can get
\begin{align}\label{5.24}
(\frac{\partial\varphi_{ui}}{\partial t},\bm{v}_{h} )+ a_{h}(\varphi_{u},\bm{v}_{h}  )+c_h(\bm{\mathcal{I}}_{h}\bm{u};\bm{\mathcal{I}}_{h}\bm{u},\bm{v}_h)
-c_{h}(\widetilde{\bm{u}}_{h};\widetilde{\bm{u}}_{h},\bm{v}_{h} )&\nonumber\\
=\xi_{I}(\bm{u};\bm{u},\bm{v}_h)+\xi_{II}(\bm{u},\bm{v}_h)
+\xi_{III}(\bm{u};\bm{u},\bm{v}_h)
+ (\bm{P}_m^{RT}\bm{u}_{t}-\bm{u}_{t},\bm{v}_{hi} )
\end{align}
Taking $\bm{v}_{h}=\varphi_{u}$ in the above equation and using  Lemmas \ref{Lemma 2.9}, \ref{Lemma 3.1}, \ref{Lemma 5.5} - \ref{Lemma 5.6} and  H\"{o}lder's inequality, we obtain
\begin{align}\label{5.25}
&\frac{1}{2} \frac{d}{dt} \|\varphi_{ui}\|_{0}^{2}+ \nu |||\varphi_{u}|||^{2}_{V}+ \alpha\|\varphi_{ui}\|_{0,r}^{r}\nonumber\\
\lesssim& h^{m}(\|\bm{u}\|_{2}\|\bm{u}\|_{1+m}+ \|\bm{u}\|_{1+m}
+ \|\bm{u}\|_{2}^{r-2}\|\bm{u}\|_{1+m})|||\varphi_{u}|||_{V}
+h^{m+1}|\bm{u}_{t}|_{m+1}\|\varphi_{ui}\|_{0}\nonumber\\
\lesssim & \frac{h^{2m}}{2\nu}(\|\bm{u}\|_{2}\|\bm{u}\|_{1+m}+ \|\bm{u}\|_{1+m}
+ \|\bm{u}\|_{2}^{r-2}\|\bm{u}\|_{1+m})^{2}
+\frac{\nu}{2}|||\varphi_{u}|||^{2}_{V}+h^{m+1}|\bm{u}_{t}|_{m+1}\|\varphi_{ui}\|_{0}.
\end{align}
Integrating both sides of \eqref{5.25} with respect to $t$, we derive
\begin{align*}\label{5.26}
& \|\varphi_{ui}(t)\|_{0}^{2}+ \nu\int_{0}^{t} |||\varphi_{u}(\tau)|||^{2}_{V} d\tau
+2\alpha\int_{0}^{t}\|\varphi_{ui}(\tau)\|_{0,r}^{r}d\tau \nonumber\\
\lesssim & \frac{h^{2m}}{\nu}\int_{0}^{t}(\|\bm{u}\|_{2}\|\bm{u}\|_{1+m}+ \|\bm{u}\|_{1+m}
+ \|\bm{u}\|_{2}^{r-2}\|\bm{u}\|_{1+m})^{2}d\tau%\nonumber\\&
+2\int_{0}^{t}h^{m+1}|\bm{u}_{t}|_{m+1}\|\varphi_{ui}\|_{0}d\tau\nonumber\\
\lesssim & \frac{h^{2m}}{\nu}\int_{0}^{t}(\|\bm{u}\|_{2}\|\bm{u}\|_{1+m}+ \|\bm{u}\|_{1+m}
+ \|\bm{u}\|_{2}^{r-2}\|\bm{u}\|_{1+m})^{2}d\tau%\nonumber\\&
+ h^{2m+2}\int_{0}^{t}|\bm{u}_{t}|_{m+1}^{2}d\tau
+\int_{0}^{t}\|\varphi_{ui}\|_{0}^{2}d\tau,
\end{align*}
which, together with the continuous Gronwall's inequality (cf. \cite[Lemma 3.3]{TeschlGerald2012Odea}), implies \eqref{5.22}.

Taking $\bm{v}_{h}=\frac{\partial\varphi_{u}}{\partial t}$ in \eqref{5.24} and utilizing Lemmas \ref{Lemma 2.9}, \ref{Lemma 5.5} and \ref{Lemma 5.6}, we gain
\begin{align}
& \|\frac{\partial\varphi_{ui}}{\partial t}\|_{0}^{2}+\frac{\nu}{2} \frac{d}{dt} |||\varphi_{u}|||^{2}_{V}
\nonumber\\
=&\xi_{I}(\bm{u};\bm{u},\frac{\partial\varphi_{u}}{\partial t})+\xi_{II}(\bm{u},\frac{\partial\varphi_{u}}{\partial t})
+\xi_{III}(\bm{u};\bm{u},\frac{\partial\varphi_{u}}{\partial t})
+ (\bm{P}_m^{RT}\bm{u}_{t}-\bm{u}_{t},\frac{\partial\varphi_{ui}}{\partial t} )
-c_h(\bm{\mathcal{I}}_{h}\bm{u};\bm{\mathcal{I}}_{h}\bm{u},\frac{\partial\varphi_{u}}{\partial t})
+c_{h}(\widetilde{\bm{u}}_{h};\widetilde{\bm{u}}_{h},\frac{\partial\varphi_{u}}{\partial t} )\nonumber\\
\lesssim&
(|\bm{P}_{m}^{RT}\bm{u}|^{r-2}\bm{P}_{m}^{RT}\bm{u}-|\widetilde{\bm{u}}_{hi}|^{r-2}\widetilde{\bm{u}}_{hi},\frac{\partial\varphi_{ui}}{\partial t})
+ h^{m}(\|\bm{u}\|_{2}\|\bm{u}\|_{1+m}+ \|\bm{u}\|_{1+m}
+ \|\bm{u}\|_{2}^{r-2}\|\bm{u}\|_{1+m})|||\frac{\partial\varphi_{u}}{\partial t}|||_{V}\nonumber\\
&+h^{m+1}|\bm{u}_{t}|_{m+1}\|\frac{\partial\varphi_{ui}}{\partial t}\|_{0}\nonumber\\
\lesssim& \big((|\bm{P}_{m}^{RT}\bm{u}| +|\widetilde{\bm{u}}_{hi}|)^{r-2}\varphi_{ui},\frac{\partial\varphi_{ui}}{\partial t}\big)
+ h^{m}(\|\bm{u}\|_{2}\|\bm{u}\|_{1+m}+ \|\bm{u}\|_{1+m}
+ \|\bm{u}\|_{2}^{r-2}\|\bm{u}\|_{1+m})|||\frac{\partial\varphi_{u}}{\partial t}|||_{V}\nonumber\\
&+h^{m+1}|\bm{u}_{t}|_{m+1}\|\frac{\partial\varphi_{ui}}{\partial t}\|_{0}
\end{align}
which plus  H\"{o}lder's inequality, the triangle inequality and Lemma \ref{Lemma 2.4} further imples
\begin{align}
& \|\frac{\partial\varphi_{ui}}{\partial t}\|_{0}^{2}+\frac{\nu}{2} \frac{d}{dt} |||\varphi_{u}|||^{2}_{V}
\nonumber\\
\lesssim& (\|\bm{P}_{m}^{RT}\bm{u}\|_{0,3(r-2)}+\|\widetilde{\bm{u}}_{hi}\|_{0,3(r-2)} )^{r-2}\|\varphi_{ui}\|_{0,6}\|\frac{\partial\varphi_{ui}}{\partial t} \|_0
+ h^{m}(\|\bm{u}\|_{2}\|\bm{u}\|_{1+m}+ \|\bm{u}\|_{1+m}
+ \|\bm{u}\|_{2}^{r-2}\|\bm{u}\|_{1+m})|||\frac{\partial\varphi_{u}}{\partial t}|||_{V}\nonumber\\
&+h^{m+1}|\bm{u}_{t}|_{m+1}\|\frac{\partial\varphi_{ui}}{\partial t}\|_{0}\nonumber\\
\lesssim&(\|\bm{u}\|_{0,3(r-2)}+\|\varphi_{ui}\|_{0,3(r-2)} )^{r-2}|||\varphi_{u}|||_{V}\|\frac{\partial\varphi_{ui}}{\partial t} \|_0
+ h^{m}(\|\bm{u}\|_{2}\|\bm{u}\|_{1+m}+ \|\bm{u}\|_{1+m}
+ \|\bm{u}\|_{2}^{r-2}\|\bm{u}\|_{1+m})|||\frac{\partial\varphi_{u}}{\partial t}|||_{V}\nonumber\\
&+h^{m+1}|\bm{u}_{t}|_{m+1}\|\frac{\partial\varphi_{ui}}{\partial t}\|_{0}.
\end{align}
Integrating both sides of the above inequality with respect to $t$ and using the Cauchy-Schwarz inequality and integration by parts, we have
\begin{align*}
& 2 \int_{0}^{t} \|\frac{\partial\varphi_{ui}}{\partial t}(\tau)\|_{0}^{2} d\tau+ \nu|||\varphi_{u}(t)|||^{2}_{V}
\nonumber\\
\lesssim &\int_{0}^{t}(\|\bm{u}\|_{0,3(r-2)}+\|\varphi_{ui}\|_{0,3(r-2)} )^{r-2}|||\varphi_{u}|||_{V}\|\frac{\partial\varphi_{ui}}{\partial t} \|_0d\tau\nonumber\\
&+\int_{0}^{t}h^{m}(\|\bm{u}\|_{2}\|\bm{u}\|_{1+m}+ \|\bm{u}\|_{1+m}
+ \|\bm{u}\|_{2}^{r-2}\|\bm{u}\|_{1+m})|||\frac{\partial\varphi_{u}}{\partial t}|||_{V}d\tau%\nonumber\\&
+ \int_{0}^{t}h^{m+1}|\bm{u}_{t}|_{m+1}\|\frac{\partial\varphi_{ui}}{\partial t}\|_{0}d\tau.
\nonumber\\
\lesssim &\int_{0}^{t}(\|\bm{u}\|_{2}+\|\varphi_{ui}\|_{0,3(r-2)} )^{2(r-2)}|||\varphi_{u}|||_{V}^{2}d\tau+\int_{0}^{t}\|\frac{\partial\varphi_{ui}}{\partial t} \|_0^{2}d\tau\nonumber\\
&+ h^{m}(\|\bm{u}\|_{2}\|\bm{u}\|_{1+m}+ \|\bm{u}\|_{1+m}
+ \|\bm{u}\|_{2}^{r-2}\|\bm{u}\|_{1+m})|||\varphi_{u}|||_{V}\nonumber\\
&+h^{m}|\int_{0}^{t}(\|\bm{u}\|_{2}\|\bm{u}\|_{1+m}+ \|\bm{u}\|_{1+m}
+ \|\bm{u}\|_{2}^{r-2}\|\bm{u}\|_{1+m})_{t}|||\varphi_{u}|||_{V}d\tau|%\nonumber\\&
+ \int_{0}^{t}h^{m+1}|\bm{u}_{t}|_{m+1}\|\frac{\partial\varphi_{ui}}{\partial t}\|_{0}d\tau\nonumber\\
\lesssim &\int_{0}^{t}(\|\bm{u}\|_{2}+|||\varphi_{u}|||_{V} )^{2(r-2)}|||\varphi_{u}|||_{V}^{2}d\tau+\int_{0}^{t}\|\frac{\partial\varphi_{ui}}{\partial t} \|_0^{2}d\tau
+\frac{\nu}{2} |||\varphi_{u}|||^{2}_{V}\nonumber\\
&+\frac{2h^{2m}}{\nu}(\|\bm{u}\|_{2}\|\bm{u}\|_{1+m}+ \|\bm{u}\|_{1+m}
+ \|\bm{u}\|_{2}^{r-2}\|\bm{u}\|_{1+m})^{2}
+\nu\int_{0}^{t}|||\varphi_{u}|||^{2}_{V}d\tau
\nonumber\\
&+ \frac{h^{2m}}{\nu}\int_{0}^{t}(\|\bm{u}\|_{2}\|\bm{u}\|_{1+m}+ \|\bm{u}\|_{1+m}
+ \|\bm{u}\|_{2}^{r-2}\|\bm{u}\|_{1+m})_{t}^{2}d\tau
%\nonumber\\&
+h^{2m+2}\int_{0}^{t}|\bm{u}_{t}|_{m+1}^{2}d\tau.
\end{align*}
Combining the continuous Gronwall's inequality (cf. \cite[Lemma 3.3]{TeschlGerald2012Odea}) and \eqref{5.22}, we get \eqref{5.23}.
This completes the proof.
\end{proof}

%Next, we give the following lemma to be used later.
 \begin{lemma}\label{Lemma 4.6}
 Assume that $\bm{u}\in [H^{2}(\Omega)]^{n}$, then  $\forall \varphi_{u}, \chi_{u}\in \bm{V}_{h}^{0}$ there hold
 \begin{subequations}
 \begin{align}
 d_{h}(\bm{\mathcal{I}}_{h}\bm{u};\varphi_{u},\chi_{u} )\lesssim&  (\| \varphi_{ui}\|_{0}+h|||\varphi_{u}|||_{V} )\|\bm{u}\|_{2}|||\chi_{u} |||_{V},\label{d-1}\\
d_{h}(\varphi_{u}+\chi_{u};\varphi_{u},\chi_{u})
 \lesssim&C(\bm{u})\|\varphi_{ui}+\chi_{ui}\|_{0,2}|||\chi_{u}|||_{V} ,\label{d-2}\\
 d_{h}(\varphi_{u}+\chi_{u};\bm{\mathcal{I}}_{h}\bm{u},\chi_{u})
 \lesssim&C(\bm{u})\|\varphi_{ui}+\chi_{ui}\|_{0,2}|||\chi_{u}|||_{V} ,\label{d-3}
 \end{align}
 \end{subequations}
 where $C(\bm{u})>0$ is a constant depending only on the regularity of $\bm{u}$.
 \end{lemma}
 \begin{proof}
 The first inequality has been proved in (cf.\cite[Lemma 4.8]{HLX2019}).

 %Next, we shall demonstrate the last two inequalities.
 From the definition of $d_h (\cdot;\cdot,\cdot)$,  H\"{o}lder's inequality, the projection properties and Lemma  \ref{Lemma 5.7}, it follows
 \begin{subequations}
\begin{align}
&d_{h}(\varphi_{u}+\chi_{u};\varphi_{u},\chi_{u})\nonumber\\
=&\frac{1}{2}((\varphi_{ui}+\chi_{ui})\cdot\nabla_{h}\varphi_{ui},\chi_{ui})
-\frac{1}{2}\langle((\varphi_{ui}+\chi_{ui})\cdot\bm{n})(\varphi_{ui}-\varphi_{ub}),\chi_{ui}\rangle_{\partial\mathcal{T}_{h}}\nonumber\\
&-\frac{1}{2}((\varphi_{ui}+\chi_{ui})\cdot\nabla_{h}\chi_{ui},\varphi_{ui})
+\frac{1}{2}\langle((\varphi_{ui}+\chi_{ui})\cdot\bm{n})(\chi_{ui}-\chi_{ub}),\varphi_{ui}\rangle_{\partial\mathcal{T}_{h}}\nonumber\\
\lesssim & \|\varphi_{ui}+\chi_{ui}\|_{0,2}h^{-\frac{n}{6}} |||\varphi_{u}|||_{V} \cdot|||\chi_{u}|||_{V}\nonumber\\
\lesssim & C(\bm{u})\|\varphi_{ui}+\chi_{ui}\|_{0,2}|||\chi_{u}|||_{V},\\
&d_{h}(\varphi_{u}+\chi_{u};\bm{\mathcal{I}}_{h}\bm{u},\chi_{u})\nonumber\\
=&\frac{1}{2}((\varphi_{ui}+\chi_{ui})\cdot\nabla_{h}\bm{P}^{RT}_{m}\bm{u},\chi_{ui})
-\frac{1}{2}\langle((\varphi_{ui}+\chi_{ui})\cdot\bm{n})(\bm{P}^{RT}_{m}\bm{u}-\bm{\Pi}^{B}_{m}\bm{u}),\chi_{ui}\rangle_{\partial\mathcal{T}_{h}}\nonumber\\
&-\frac{1}{2}((\varphi_{ui}+\chi_{ui})\cdot\nabla_{h}\chi_{ui},\bm{P}^{RT}_{m}\bm{u})
+\frac{1}{2}\langle((\varphi_{ui}+\chi_{ui})\cdot\bm{n})(\chi_{ui}-\chi_{ub}),\bm{P}^{RT}_{m}\bm{u}\rangle_{\partial\mathcal{T}_{h}}\nonumber\\
\lesssim & \|\varphi_{ui}+\chi_{ui}\|_{0,2} \|\bm{u}\|_{2} |||\chi_{u}|||_{V}\nonumber\\
\lesssim & C(\bm{u})\|\varphi_{ui}+\chi_{ui}\|_{0,2}|||\chi_{u}|||_{V},\label{5.32}
\end{align}
 \end{subequations}
 which imply \eqref{d-2} and \eqref{d-3}.
 \end{proof}

\begin{lemma}\label{Lemma 5.9}
Let $\bm{u}_{h}$ and $\widetilde{\bm{u}}_{h}$ be the solution to \eqref{WG} and \eqref{AWG}, respectively.
Under the same condition as in Lemma \ref{Lemma 4.6}, there holds
\begin{align}\label{5.28}
\nu \int_{0}^{t}||| \chi_{u}(\tau)|||^{2}_{V}d\tau\lesssim \mathcal{C}(\bm{u} )h^{2m}.
\end{align}
In addition, if $\bm{u}_{t}\in L^{\infty} ([H^{2}]^{n})\cap L^{2} ([H^{m+1}]^{n})$, then
\begin{align}\label{421}
\int_{0}^{t}\|\frac{\partial\chi_{ui}}{\partial t}(\tau)\|_{0}^{2}d\tau+ \nu |||\chi_{u}(t) |||^{2}_{V}
\lesssim \mathcal{C}(\bm{u})h^{2m}.
\end{align}
Here, $\mathcal{C}(\bm{u})>0$ is a  constant depending only on $ \nu, \alpha, r$ and the regularity of $\bm{u}.$
\end{lemma}
\begin{proof}
Based on \eqref{wg1} and \eqref{AWG},  $\forall \bm{v}_{h}\in \bm{V}_{0h}$, we can get
\begin{align*}
&(\frac{\partial\chi_{ui}}{\partial t},\bm{v}_{hi})+a_{h}(\chi_{u},\bm{v}_{h} )+c_{h}(\widetilde{\bm{u}}_{h};\widetilde{\bm{u}}_{h},\bm{v}_{h} )
-c_{h}(\bm{u}_{h};\bm{u}_{h},\bm{v}_{h} )%\nonumber\\&
=d_{h}(\bm{u}_{h};\bm{u}_{h},\bm{v}_{h} )-d_{h}(\bm{\mathcal{I}}_{h}\bm{u};\bm{\mathcal{I}}_{h}\bm{u},\bm{v}_{h} ).
\end{align*}
Taking $\bm{v}_{h}=\chi_{u}$ in the above equation and utilizing Lemmas  \ref{Lemma 2.9}, \ref{Lemma 3.1} and \ref{Lemma 4.6}, we obtain
\begin{align*}
&\frac{1}{2} \frac{d}{dt} \|\chi_{ui}\|_{0}^{2}+ \nu |||\chi_{u}|||^{2}_{V}+\| \chi_{ui} \|_{0,r}^{r}\nonumber\\
\lesssim& d_{h}(\bm{u}_{h};\bm{u}_{h},\chi_{u} )-d_{h}(\bm{\mathcal{I}}_{h}\bm{u};\bm{\mathcal{I}}_{h}\bm{u},\chi_{u})\nonumber\\
=&d_{h}(\varphi_{u}+\chi_{u};\varphi_{u},\chi_{u})
-d_{h}(\varphi_{u}+\chi_{u};\bm{\mathcal{I}}_{h}\bm{u},\chi_{u})
-d_{h}(\bm{\mathcal{I}}_{h}\bm{u}; \varphi_{u},\chi_{u})\nonumber\\
\lesssim& C(\bm{u})(\|\varphi_{ui}+\chi_{ui}\|_{0,2}
+\| \varphi_{ui}\|_{0}+h|||\varphi_{u}|||_{V})|||\chi_{u} |||_{V}.
\end{align*}
Integrating both sides of the  above inequality with respect to $t$ and  using  H\"{o}lder's inequality and the triangle inequality, we gain
\begin{align*}
&\|\chi_{ui}(t)\|_{0}^{2}+2\nu \int_{0}^{t}||| \chi_{u}(\tau)|||^{2}_{V}d\tau
+2  \alpha\int_{0}^{t} \| \chi_{ui}(\tau) \|_{0,r}^{r}d\tau\nonumber\\
\lesssim &C(\bm{u} )  \int_{0}^{t}(\|\varphi_{ui}+\chi_{ui}\|_{0,2}
+\| \varphi_{ui}\|_{0}+h|||\varphi_{u}|||_{V})|||\chi_{u} |||_{V}d\tau \nonumber\\
\lesssim &\frac{C(\bm{u} )^{2} }{\nu} \int_{0}^{t}(\|\varphi_{ui}(\tau)\|_{0}^{2}+\|\chi_{ui}(\tau)\|_{0}^{2}
+\| \varphi_{ui}(\tau)\|_{0}^{2}+h^{2}|||\varphi_{u}(\tau)|||^{2}_{V})d\tau
+ \nu \int_{0}^{t}|||\chi_{u}(\tau) |||_{V}^{2}d\tau,
\end{align*}
which, together with the continuous Gronwall's inequality (cf. \cite[Lemma 3.3]{TeschlGerald2012Odea}) and Lemma \ref{Lemma 5.7},  indicates \eqref{5.28}.

Taking $\bm{v}_{h}=\frac{\partial\chi_{u}}{\partial t}$ in \eqref{5.24} and utilizing \eqref{a2}, \eqref{c2} and Lemmas \ref{Lemma 2.4}, \ref{Lemma 2.9}, \ref{Lemma 4.6}, we gain
\begin{align}
& \|\frac{\partial\chi_{ui}}{\partial t}\|_{0}^{2}+\frac{\nu}{2} \frac{d}{dt} |||\chi_{u}|||^{2}_{V}
\nonumber\\
=& c_{h}(\bm{u}_{h};\bm{u}_{h},\frac{\partial\chi_{u}}{\partial t})
-c_{h}(\widetilde{\bm{u}}_{h};\widetilde{\bm{u}}_{h},\frac{\partial\chi_{u}}{\partial t} )
+d_{h}(\bm{u}_{h};\bm{u}_{h},\frac{\partial\chi_{u}}{\partial t} )-d_{h}(\bm{\mathcal{I}}_{h}\bm{u};\bm{\mathcal{I}}_{h}\bm{u},\frac{\partial\chi_{u}}{\partial t})\nonumber\\
=&\alpha (|\bm{u}_{hi}|^{r-2}\bm{u}_{hi}-|\widetilde{\bm{u}}_{hi}|^{r-2}\widetilde{\bm{u}}_{hi}, \frac{\partial\chi_{ui}}{\partial t}  )
+d_{h}(\varphi_{u}+\chi_{u};\varphi_{u},\frac{\partial\chi_{u}}{\partial t})
-d_{h}(\varphi_{u}+\chi_{u};\bm{\mathcal{I}}_{h}\bm{u},\frac{\partial\chi_{u}}{\partial t})
-d_{h}(\bm{\mathcal{I}}_{h}\bm{u}; \varphi_{u},\frac{\partial\chi_{u}}{\partial t})\nonumber\\
\lesssim& (\|\bm{u}_{hi}\|_{0,3(r-2)}+ \|\widetilde{\bm{u}}_{hi}\|_{0,3(r-2)})^{r-2}\|\chi_{ui}\|_{0,6} \|\frac{\partial\chi_{ui}}{\partial t} \|_{0}
+C(\bm{u})(\|\varphi_{ui}+\chi_{ui}\|_{0}
+\| \varphi_{ui}\|_{0}+h|||\varphi_{u}|||_{V})|||\frac{\partial\chi_{u}}{\partial t}|||_{V}\nonumber\\
\lesssim& (|||\bm{u}_{h}|||_{V}+ |||\chi_{u}|||_{V})^{r-2}|||\chi_{u}|||_{V} \|\frac{\partial\chi_{ui}}{\partial t} \|_{0}
+C(\bm{u})(\|\varphi_{ui}+\chi_{ui}\|_{0}
+\| \varphi_{ui}\|_{0}+h|||\varphi_{u}|||_{V})|||\frac{\partial\chi_{u}}{\partial t}|||_{V}.
\end{align}
Integrating both sides of the above inequality with respect to $t$ and using the Cauchy-Schwarz inequality and integration by parts, we have
\begin{align*}
&  2\int_{0}^{t} \|\frac{\partial\chi_{ui}}{\partial t}(\tau)\|_{0}^{2} d\tau+ \nu|||\chi_{u}(t)|||^{2}_{V}
\nonumber\\
\lesssim &\int_{0}^{t}(|||\bm{u}_{h}|||_{V}+ |||\chi_{u}|||_{V})^{2(r-2)}|||\chi_{u}|||_{V}^{2}d\tau
+\int_{0}^{t} \|\frac{\partial\chi_{ui}}{\partial t} \|_{0}^{2}d\tau
+C(\bm{u})\int_{0}^{t}(\|\varphi_{ui}+\chi_{ui}\|_{0}
+\| \varphi_{ui}\|_{0}+h|||\varphi_{u}|||_{V})|||\frac{\partial\chi_{u}}{\partial t}|||_{V}d\tau\nonumber\\
\lesssim&\int_{0}^{t}(|||\bm{u}_{h}|||_{V}+ |||\chi_{u}|||_{V})^{2(r-2)}|||\chi_{u}|||_{V}^{2}d\tau
+\int_{0}^{t} \|\frac{\partial\chi_{ui}}{\partial t} \|_{0}^{2}d\tau
+C(\bm{u})(\|\varphi_{ui}+\chi_{ui}\|_{0}
+\| \varphi_{ui}\|_{0}+h|||\varphi_{u}|||_{V})|||\chi_{u}|||_{V}\nonumber\\
&+C(\bm{u})|\int_{0}^{t} (\|\varphi_{ui}+\chi_{ui}\|_{0}
+\| \varphi_{ui}\|_{0}+h|||\varphi_{u}|||_{V})_{t}|||\chi_{u}|||_{V}d\tau|\nonumber\\
\lesssim&\int_{0}^{t}(|||\bm{u}_{h}|||_{V}+ |||\chi_{u}|||_{V})^{2(r-2)}|||\chi_{u}|||_{V}^{2}d\tau
+\int_{0}^{t} \|\frac{\partial\chi_{ui}}{\partial t} \|_{0}^{2}d\tau
+\frac{2C(\bm{u})^{2}}{\nu}(\|\varphi_{ui}\|_{0}^{2}+\|\chi_{ui}\|_{0}^{2}+ h^{2}|||\varphi_{u}|||_{V}^{2})\nonumber\\
&+ \frac{\nu}{2}|||\chi_{u}|||_{V}^{2}
+\frac{C(\bm{u})^{2}}{\nu}\int_{0}^{t} (\|\varphi_{ui}+\chi_{ui}\|_{0}
+\| \varphi_{ui}\|_{0}+h|||\varphi_{u}|||_{V})_{t}^{2}d\tau
+ \nu\int_{0}^{t}|||\chi_{u}|||_{V}^{2}  d\tau,
\end{align*}
which, togerher with the continuous Gronwall's inequality (cf. \cite[Lemma 3.3]{TeschlGerald2012Odea}), Lemma \ref{Lemma 5.7} and \eqref{5.28}, yields \eqref{421}.
\end{proof}

Based on  Lemmas \ref{Lemma 5.7} and \ref{Lemma 5.9}, we can derive the following conclusion.
\begin{lemma}\label{Lemma 5.10}
Let $\bm{u}_{h}$ and $\bm{\mathcal{I}}_{h}\bm{u}$ be  the solution to the systems \eqref{WG} and \eqref{ee}, respectively.
Under the regularity assumption \eqref{regularity}, there holds
\begin{align*}
\nu \int_{0}^{t} |||\bm{\mathcal{I}}_{h}\bm{u}-\bm{u}_{h} |||^{2}_{V}d\tau \lesssim \mathcal{C}( \bm{u})h^{2m}.
\end{align*}
Here, $\mathcal{C}(\bm{u})>0$ is a  constant depending only on $ \nu, \alpha, r$ and the regularity of $\bm{u}.$
\end{lemma}

Next, we derive the error estimate to the pressure.
\begin{lemma}\label{Lemma 5.11}
Let $(\bm{u}_{h},p_{h})$ and $(\bm{\mathcal{I}}_{h}\bm{u},\mathcal{P}_{h}p)$ be  the solution to \eqref{WG} and \eqref{ee}, respectively. Under the regularity assumption \eqref{regularity}, there holds
\begin{align}\label{EP-0}
  \int_{0}^{t}|||\mathcal{P}_{h}p-p_{h}|||_{Q}d\tau \lesssim \mathcal{C}( \bm{u},p)h^{m}.
\end{align}
Here, $\mathcal{C}( \bm{u},p)>0$ is a  constant depending only on $\mathcal{C}( \bm{u})$ and  $\|p\|_{m+1}.$
\end{lemma}
\begin{proof}
In light of  Lemmas \ref{theoremLBB}, \ref{Lemma 3.1} and \ref{Lemma 5.5}, \eqref{WG1}, \eqref{e2}, we have
\begin{align}\label{E23}
&|||\mathcal{P}_{h}p-p_{h} |||_{Q}\nonumber\\
\lesssim &\sup_{( \bm{v}_{h},q_{h})\in \bm{V}_{h}^{0}\times Q_{h}^{0}}\frac{b_{h}(\bm{v}_{h},\mathcal{P}_{h}p-p_{h} )}{|||\bm{v}_{h} |||_{V}}\nonumber\\
\lesssim & \frac{1}{|||\bm{v}_{h} |||_{V}}\big( \xi_{I}(\bm{u};\bm{u},\bm{v}_h)+\xi_{II}(\bm{u},\bm{v}_h)
+\xi_{III}(\bm{u};\bm{u},\bm{v}_h)-(\frac{\partial (\varphi_{ui}+\chi_{ui})}{\partial t},\bm{v}_{hi})- a_h(\bm{\mathcal{I}}_{h}\bm{u}-\bm{u}_h,\bm{v}_h)\nonumber\\
&-c_h(\bm{\mathcal{I}}_{h}\bm{u};\bm{\mathcal{I}}_{h}\bm{u},\bm{v}_h)
+c_h(\bm{u}_{h};\bm{u}_{h},\bm{v}_h)
-d_h(\bm{\mathcal{I}}_{h}\bm{u};\bm{\mathcal{I}}_{h}\bm{u},\bm{v}_h)
+d_h(\bm{u}_{h};\bm{u}_{h},\bm{v}_h)\big)\nonumber\\
\lesssim & h^{m}(\|\bm{u}\|_{2}\|\bm{u}\|_{m+1}+\|\bm{u}\|_{m+1}+\|\bm{u}\|_{2}^{r-2}\|\bm{u}\|_{m+1} )
+\|\frac{\partial(\varphi_{ui}+\chi_{ui})}{\partial t}\|_{0}+\nu ||| \bm{\mathcal{I}}_{h}\bm{u}-\bm{u}_h |||_{V}\nonumber\\
&+\alpha C_{\widetilde{r}}^{r}C_{r}(||| \varphi_{u}|||_{V}
+|||\chi_{u} |||_{V} )^{r-2}||| \bm{\mathcal{I}}_{h}\bm{u}-\bm{u}_h |||_{V}
+\frac{1}{|||\bm{v}_{h} |||_{V} }(d_h(\bm{\mathcal{I}}_{h}\bm{u}-\bm{u}_{h};\bm{\mathcal{I}}_{h}\bm{u}-\bm{u}_{h},\bm{v}_h)\nonumber\\
&+d_h(\bm{\mathcal{I}}_{h}\bm{u}-\bm{u}_{h};\bm{\mathcal{I}}_{h}\bm{u},\bm{v}_h)
+d_h(\bm{\mathcal{I}}_{h}\bm{u};\bm{\mathcal{I}}_{h}\bm{u}-\bm{u}_{h},\bm{v}_h)).
\end{align}
Denote
\begin{align}
&\mathcal{J}_{1}=d_h(\bm{\mathcal{I}}_{h}\bm{u}-\bm{u}_{h};\bm{\mathcal{I}}_{h}\bm{u}-\bm{u}_{h},\bm{v}_h),\quad
\mathcal{J}_{2}=d_h(\bm{\mathcal{I}}_{h}\bm{u}-\bm{u}_{h};\bm{\mathcal{I}}_{h}\bm{u},\bm{v}_h),\quad
\mathcal{J}_{3}=d_h(\bm{\mathcal{I}}_{h}\bm{u};\bm{\mathcal{I}}_{h}\bm{u}-\bm{u}_{h},\bm{v}_h).
\end{align}
Next, using the definitions of $d_{h}$ and weak operators, integration by parts,  H\"{o}lder's inequality, the trace inequality, the inverse inequality and the  projection properties, we estimate $\mathcal{J}_{j}(j=1,2,3)$ one by one.
%\begin{subequations}
\begin{align}
\mathcal{J}_{1}=&\frac{1}{2}(\nabla_{w, m}\cdot((\bm{\mathcal{I}}_{h}\bm{u}-\bm{u}_{h})
\otimes (\bm{P}_{m} ^{RT}\bm{u}-\bm{u}_{hi})),\bm{v}_{hi} ) %\nonumber\\&
-\frac{1}{2}(\nabla_{w, m}\cdot(\bm{v}_{h}
\otimes (\bm{P}_{m} ^{RT}\bm{u}-\bm{u}_{hi})),\bm{P}_{m} ^{RT}\bm{u}-\bm{u}_{hi} )\nonumber\\
=&\frac{1}{2}((\bm{P}_{m} ^{RT}\bm{u}-\bm{u}_{hi})\cdot(\nabla_{h}(\bm{P}_{m} ^{RT}\bm{u}-\bm{u}_{hi})),\bm{v}_{hi} )\nonumber\\&
-\frac{1}{2}\langle((\bm{P}_{m} ^{RT}\bm{u}-\bm{u}_{hi})\cdot\bm{n})((\bm{P}_{m} ^{RT}\bm{u}-\bm{u}_{hi})-(\bm{\Pi}_{m}^{B}\bm{u} -\bm{u}_{hb})),\bm{v}_{hi}\rangle_{\partial\mathcal{T}_{h}}\nonumber\\
&-\frac{1}{2}((\bm{P}_{m} ^{RT}\bm{u}-\bm{u}_{hi})\cdot(\nabla_{h}\bm{v}_{hi}),\bm{P}_{m} ^{RT}\bm{u}-\bm{u}_{hi} )%\nonumber\\&
+\frac{1}{2}\langle((\bm{P}_{m} ^{RT}\bm{u}-\bm{u}_{hi})\cdot\bm{n})(\bm{v}_{hi}-\bm{v}_{hb}),\bm{P}_{m} ^{RT}\bm{u}-\bm{u}_{hi}\rangle_{\partial\mathcal{T}_{h}}\nonumber\\
\lesssim& \sum _{K\in \mathcal{T}_{h}}\big(
  \|\bm{P}_{m} ^{RT}\bm{u}-\bm{u}_{hi} \|_{0,3} \|\nabla_{h}(\bm{P}_{m} ^{RT}\bm{u}-\bm{u}_{hi})\|_{0,2} \|\bm{v}_{hi} \|_{0,6}    %\nonumber\\&
+h^{-1}_{K}\|\bm{P}_{m} ^{RT}\bm{u}-\bm{u}_{hi} \|_{0,2}|||\bm{\mathcal{I}}_{h}\bm{u}-\bm{u}_{h}|||_{V} \cdot\|\bm{v}_{hi} \|_{0,2} \nonumber\\
  &
  + \|\bm{P}_{m} ^{RT}\bm{u}-\bm{u}_{hi} \|_{0,3} \|\bm{P}_{m} ^{RT}\bm{u}-\bm{u}_{hi}\|_{0,6} \| \nabla_{h} \bm{v}_{hi} \|_{0,2}  +h^{-1}_{K}\|\bm{P}_{m} ^{RT}\bm{u}-\bm{u}_{hi} \|_{0,2}^{2}||| \bm{v}_{h}|||_{V}                     \big)\nonumber\\
  \lesssim& C( \bm{u}) (h ^{m-\frac{n}{6}}|||\bm{\mathcal{I}}_{h}\bm{u}-\bm{u}_{h}|||_{V}+h ^{2m-\frac{1}{2}} )||| \bm{v}_{h}|||_{V}.
\end{align}
%\end{subequations}
Similarly, we have
\begin{align}\label{E26}
\mathcal{J}_{2}+\mathcal{J}_{3}\lesssim C( \bm{u})|||\bm{\mathcal{I}}_{h}\bm{u}-\bm{u}_{h}|||_{V}\cdot||| \bm{v}_{h}|||_{V}.
\end{align}
As a consequence, we obtain
\begin{align}\label{Php}
|||\mathcal{P}_{h}p-p_{h}|||_{Q}%\nonumber\\
\lesssim & h^{m}(\|\bm{u}\|_{2}\|\bm{u}\|_{m+1}+\|\bm{u}\|_{m+1}+\|\bm{u}\|_{2}^{r-2}\|\bm{u}\|_{m+1} )
+\|\frac{\partial (\varphi_{ui}+\chi_{ui})}{\partial t}\|_{0}\nonumber\\
&+\big(\nu+ \alpha C_{\widetilde{r}}^{r}C_{r}(||| \varphi_{u}|||_{V}
+|||\chi_{u} |||_{V} )+ C( \bm{u}) \big)||| \bm{\mathcal{I}}_{h}\bm{u}-\bm{u}_h |||_{V}.
\end{align}
Integrating both sides of inequality \eqref{Php} with respect to $t$ and combining \eqref{E23} - \eqref{E26} and Lemma \ref{Lemma 5.10}, we get \eqref{EP-0}, where the estimate of $\int _{0}^{t}\|\frac{\partial (\varphi_{ui}+\chi_{ui})}{\partial t}\|_{0} d \tau $ follows from  the triangle inequality, Lemmas \ref{Lemma 5.6}, \ref{Lemma 5.7}, \ref{Lemma 5.9}, \eqref{5.23} and \eqref{421}.
\end{proof}

Now we are in the position to get the a priori error estimates for the semi-discrete WG scheme \eqref{WG}.
\begin{theorem} \label{Theoremul2}
Let  $(\bm{u},p)$ and $(\bm{u}_{h},p_{h})$ be  the solution to \eqref{weak} and \eqref{WG},  respectively.
 Under the regularity assumption \eqref{regularity}, there holds
\begin{align}
\|\bm{u}-\bm{u}_{hi} \|_{0}&\lesssim \mathcal{C}( \bm{u})h^{m+1},\label{ul2}
\end{align}
where $\mathcal{C}(\bm{u})>0$ is a  constant depending only on $ \nu, \alpha, r$, $\|\bm{u}\|_2,$  $\|\bm{u}\|_{m+1},$ $\|\bm{u}_t\|_{m+1}.$
 \end{theorem}
\begin{proof}
Let $\varepsilon_{u}=\bm{u}_{h}-\bm{\mathcal{I}}_{h}\bm{u}$ and $\varepsilon_{p}=p_{h}-\mathcal{P}_{h}p .$\\
Subtracting \eqref{WG} from \eqref{weak}, we gain
\begin{align}
(\frac{\partial\varepsilon_{ui}}{\partial t},\bm{v}_{hi}  )+a_h(\varepsilon_{u},\bm{v}_{h} )+b_h(\bm{v}_{h},\varepsilon_{p} )
+\alpha(|\bm{u}_{hi}|^{r-2}\bm{u}_{hi}-|\bm{P}_{m}^{RT}\bm{u}|^{r-2}\bm{P}_{m}^{RT}\bm{u},\bm{v}_{hi}  )
&\nonumber\\
=(\bm{u}_t-\bm{P}_{m}^{RT}\bm{u}_t,\bm{v}_{hi}  )-\alpha(|\bm{P}_{m}^{RT}\bm{u}|^{r-2}\bm{P}_{m}^{RT}\bm{u}-|\bm{u}|^{r-2}\bm{u},\bm{v}_{hi}  )
-d_h(\bm{\mathcal{I}}_{h}\bm{u}-\bm{u}; \bm{u},\bm{v}_{h} )\nonumber\\
-d_h(\bm{\mathcal{I}}_{h}\bm{u};\bm{\mathcal{I}}_{h}\bm{u}- \bm{u},\bm{v}_{h} )
-d_h(\varepsilon_{u}; \bm{\mathcal{I}}_{h}\bm{u},\bm{v}_{h}  )
-d_h(\bm{u}_{h} ; \varepsilon_{u} ,\bm{v}_{h}  )&
,\\
b_h(\varepsilon_{u},q_h  )=0.&
\end{align}
Taking $\bm{v}_{h}=\varepsilon_{u}$ and $q_h= \varepsilon_{p}$ in the above equations, we obtain
\begin{align}
&\frac{1}{2}\frac{d}{d t}\|\varepsilon_{ui} \|_0^{2}+\nu|||\varepsilon_{u}|||_V^{2}+\alpha\|\varepsilon_{ui}\|_{0,r}^{2}\nonumber\\
\lesssim&\|\bm{u}_t-\bm{P}_{m}^{RT}\bm{u}_t\|_0\|\varepsilon_{ui}\|_0
+\alpha((|\bm{P}_{m}^{RT}\bm{u}|+|\bm{u}|)^{r-2}|\bm{P}_{m}^{RT}\bm{u}- \bm{u}|,\varepsilon_{ui} )\nonumber\\
&+|d_h(\bm{\mathcal{I}}_{h}\bm{u}-\bm{u}; \bm{u},\varepsilon_{u} )|
+|d_h(\bm{\mathcal{I}}_{h}\bm{u};\bm{\mathcal{I}}_{h}\bm{u}- \bm{u},\varepsilon_{u} )|
+|d_h(\varepsilon_{u}; \bm{\mathcal{I}}_{h}\bm{u},\varepsilon_{u}  )|\nonumber\\
\lesssim& h^{m+1}\|\bm{u}_t \|_{m+1}\|\varepsilon_{ui}\|_0
+\|\bm{u}\|_{2}^{r-2}\|\bm{P}_{m}^{RT}\bm{u}- \bm{u}\|_{0}\|\varepsilon_{ui}\|_{0}
+\|\bm{P}_{m}^{RT}\bm{u}-\bm{u}\|_0 \|\bm{u}\|_2|||\varepsilon_{u}|||_V
+\|\bm{u}\|_2\|\varepsilon_{ui}\|_0|||\varepsilon_{u}|||_V\nonumber\\
\lesssim& h^{2m+2}(\|\bm{u}_t \|_{m+1}^{2}+\|\bm{u} \|_{2}^{2}\|\bm{u} \|_{m+1}^{2} +\|\bm{u} \|_{2}^{2r-2}\|\bm{u} \|_{m+1}^{2} ) +\frac{1}{2}\|\varepsilon_{ui}\|_0^{2}+\frac{\nu}{2}|||\varepsilon_{u}|||_V^{2},
\end{align}
which implies
\begin{align}
&\|\varepsilon_{ui} \|_0^{2}+\nu\int_0^t|||\varepsilon_{u}|||_V^{2}d \tau+2\alpha\int_0^t \|\varepsilon_{ui}\|_{0,r}^{2} d \tau\nonumber\\
\lesssim& h^{2m+2} \int_0^t(\|\bm{u}_t \|_{m+1}^{2}+\|\bm{u} \|_{2}^{2}\|\bm{u} \|_{m+1}^{2} +\|\bm{u} \|_{2}^{2r-2}\|\bm{u} \|_{m+1}^{2} )d \tau
+ \int_0^t\|\varepsilon_{ui}\|_0^{2}d \tau.
\end{align}
According to the above inequality, the continuous Gronwall's inequality (cf. \cite[Lemma 3.3]{TeschlGerald2012Odea}), the triangle inequality and the projection properties, we gain \eqref{ul2}.
\end{proof}

Finally, based upon the projection properties, Lemmas \ref{Lemma 5.10} and \ref{Lemma 5.11}, Theorem \ref{Theoremul2},  we can derive the following main conclusion with respect to  the semi-discrete WG scheme \eqref{WG}.
 \begin{theorem} \label{Theorem 5.1} %(\color{black}A \color{black}error estimate)
Under the same conditions as in Theorem \ref{Theoremul2},  there hold
\begin{subequations}
\begin{align}
\|\bm{u}-\bm{u}_{hi} \|_{0}&\lesssim \mathcal{C}( \bm{u})h^{m+1},\\
\nu \int_{0}^{t} |||\bm{u}-\bm{u}_{h} |||^{2}_{V}d\tau &\lesssim \mathcal{C}( \bm{u})h^{2m},\label{EU}\\
  \int_{0}^{t}\|p-p_{h}\|_{0}d\tau &\lesssim \mathcal{C}( \bm{u},p)h^{m}.\label{EP}
\end{align}
\end{subequations}
Here, $\mathcal{C}(\bm{u})>0$ is a  constant depending only on $ \nu, \alpha, r$, $\|\bm{u}\|_2,$  $\|\bm{u}\|_{m+1},$ $\|\bm{u}_t\|_2,$ $\|\bm{u}_t\|_{m+1},$ as well as, $\mathcal{C}( \bm{u},p)$
 has been defined in Lemma \ref{Lemma 5.11}.
 \end{theorem}
\begin{remark}
According to  Lemma \ref{Lemma 2.3}, \eqref{EU} and \eqref{EP} in Theorem \ref{Theorem 5.1}, we also have
\begin{subequations}
\begin{align}
	\nu \int_{0}^{t} \|\nabla\bm{u}-\nabla_{h}\bm{u}_{hi} \|_{0}^{2}d\tau&\lesssim \mathcal{C}( \bm{u})h^{2m} ,\label{EUU}\\
	 \int_{0}^{t}\|p-p_{hi}\|_{0}d\tau &\lesssim \mathcal{C}( \bm{u},p)h^{m},\label{EPP}
	\end{align}
\end{subequations}
where $\mathcal{C}( \bm{u})$ and $\mathcal{C}( \bm{u},p)$ are defined in Theorem \ref{Theorem 5.1}.
\end{remark}
\begin{remark}
The result \eqref{EUU} indicates  that the velocity error estimate is independent of the pressure approximation, which means that the proposed semi-discrete WG scheme \eqref{WG} is pressure-robust.
\end{remark}

\section{ Fully discrete WG method}

\subsection{ Fully discrete scheme}
In this section, the first-order backward Euler is adopted for temporal discretization to establish the fully discrete scheme.

Recall that $N > 0$ is an integer and $0=t_0<t_1<...<t_N =T$ is a uniform division of time domain $[0,T]$ with the time step $\Delta t =\frac{T}{N}$.
We denote by $ ( \bm{u}_{h}^{k}, p_{h}^{k})$ the approximation of $ ( \bm{u}_{h}(t), p_{h}(t))$ at the discrete time $t_{k}=k \triangle t$ for $k=1,...,N.$
 Replacing the time derivative $\frac{\partial \bm{u}_{hi}}{\partial t}$ at time $t_{k}$ in \eqref{WG} by the backward difference quotient
\begin{align}
	D_{t}\bm{u}_{hi}^{k}=\frac{\bm{u}_{hi}^{k}-\bm{u}_{hi}^{k-1}}{\triangle t},
\end{align}
the fully discrete scheme for the problem \eqref{BF0} is given as follows.
For each $1\leq k\leq N$, find $(\bm{u}_{h}^{k},p_{h}^{k})=(\{\bm{u}_{hi}^{k},\bm{u}_{hb}^{k}\},\{p_{hi}^{k},p_{hb}^{k}\})\in \bm{V}_{h}^{0}\times Q_{h}^{0}$ such that
\begin{subequations}\label{fullwg}
\begin{align}
( D_{t}\bm{u}_{hi}^{k},\bm{v}_{hi} )
+a_h(\bm{u}_h^{k},\bm{v}_h)+b_h(\bm{v}_{h},p_h^{k})%&\nonumber\\
+c_h(\bm{u}_h^{k} ;\bm{u}_h^{k},\bm{v}_h )
+d_h(\bm{u}_h^{k} ;\bm{u}_h^{k},\bm{v}_h )&=(\bm{f}^{k},\bm{v}_{hi}),\label{fullw1}\\
b_h(\bm{u}_{h}^{k},q_h)&=0,
\end{align}
\end{subequations}
with the initial value $\bm{u}_h^{0}=\bm{\mathcal{I}}_{h}\bm{u}_{0}.$
\begin{remark}
Similar to Theorem \ref{TH2.2}, we can demonstrate that the fully discrete scheme \eqref{fullwg} also yields globally divergence-free approximation for velocity $\bm{u}_{hi}^{k}$ in a pointwise sense.
\end{remark}
\subsection{  Stability results }

\begin{theorem} \label{Theorem 6.1}
Assume that $\bm{f}\in L^{2}( [L^{2}]^{n})$ and $\bm{u}_{hi}^{0}\in [L^{2}(\Omega)]^{n}$.
Then, for the fully discrete WG method  \eqref{fullwg}, the following stability results hold:
for all $1 \leq \widetilde{l}\leq N,$
\begin{subequations}
\begin{align}
\|\bm{u}_{hi}^{\widetilde{l}} \|_{0}^{2}+\sum_{k=1}^{\widetilde{l}}\|\bm{u}_{hi}^{k}-\bm{u}_{hi}^{k-1} \|_{0}^{2}
+\nu\triangle t \sum_{k=1}^{\widetilde{l}} |||\bm{u}_{h}^{k} |||^{2}_{V}
 +2\alpha\triangle t \sum_{k=1}^{\widetilde{l}}\|\bm{u}_{hi}^{k}\|_{0,r}^{r}%&\nonumber\\
\lesssim\|\bm{u}_{hi}^{0} \|_{0}^{2}+\frac{1}{\nu}\triangle t\sum_{k=1}^{\widetilde{l}}\|\bm{f}^{k} \|_{\ast,h}^{2},&\label{EEU}\\
\triangle t \sum_{k=1}^{\widetilde{l}}\|p_{hi}^{k}\|_{0}\lesssim C(\bm{f}, \nu,\alpha,r,\bm{u}_{hi}^{0} )&.\label{EE-P}
\end{align}
\end{subequations}
Here,
$C(\bm{f}, \nu,\alpha,r,\bm{u}_{hi}^{0} )>0$ is  a constant depending only on $\bm{f}, \nu,\alpha,r,\bm{u}_{hi}^{0}$.
\end{theorem}
\begin{proof}
Taking $(\bm{v}_{h},q_{h})=(\bm{u}_{h}^{k},p_{h}^{k})$ in \eqref{fullwg} and using H\"{o}lder's inequality, Lemma \ref{Lemma 3.1},  and the relation $2(b-a,b)=b^2-a^2+(b-a)^2$, we get
\begin{align*}
&\frac{1}{2\triangle t}(\|\bm{u}_{hi}^{k} \|_{0}^{2}-\|\bm{u}_{hi}^{k-1} \|_{0}^{2}
+\|\bm{u}_{hi}^{k}-\bm{u}_{hi}^{k-1} \|_{0}^{2} )+ \nu  |||\bm{u}_{h}^{k}  |||^{2}_{V}
+\alpha \|\bm{u}_{hi}^{k} \|_{0,r}^{r}%\nonumber\\&
=(\bm{f}^{k},\bm{u}_{hi}^{k} )\leq\frac{1}{2\nu}\|\bm{f}^{k} \|_{\ast,h}^{2}+\frac{\nu}{2}|||\bm{u}_{h}^{k}  |||^{2}_{V}.
\end{align*}
Summing the above inequality with respect to $k$ from $1$ to $\widetilde{l}$ and multiplying by $\Delta t$,
which yields \eqref{EEU}.
Utilizing  Lemmas \ref{Lemma 3.1} - \ref{theoremLBB} and \eqref{fullw1}, we obtain
\begin{align*}
& |||p_{h}^{k}|||_{Q}\nonumber\\
\lesssim &\sup_{( \bm{v}_{h},p_{h})\in \bm{V}_{h}^{0}\times Q_{h}^{0}}\frac{b_{h}(\bm{v}_{h},p_{h}^{k} )}{|||\bm{v}_{h} |||_{V}}\nonumber\\
=&\sup_{( \bm{v}_{h},p_{h})\in \bm{V}_{h}^{0}\times Q_{h}^{0}}\frac{(\bm{f}^{k},\bm{v}_{hi})-( D_{t}\bm{u}_{hi}^{k},\bm{v}_{hi} )
-a_h(\bm{u}_h^{k},\bm{v}_h)
-c_h(\bm{u}_h^{k} ;\bm{u}_h^{k},\bm{v}_h )
-d_h(\bm{u}_h^{k} ;\bm{u}_h^{k},\bm{v}_h ) }{|||\bm{v}_{h} |||_{V}}\nonumber\\
\lesssim &(\|\bm{f}^{k}\|_{\ast,h} +\nu ||| \bm{u}_h^{k}|||_{V}+ \alpha C_{\widetilde{r}}^{r} ||| \bm{u}_h^{k}|||^{r-1}_{V} +\mathcal{N}_{h}||| \bm{u}_h^{k}|||^{2}_{V}  )
+\frac{1}{\triangle t} \|\bm{u}_{hi}^{k}-\bm{u}_{hi}^{k-1} \|_{0}.
\end{align*}
Thus, we obtain
\begin{align*}
\Delta t\sum_{k=1}^{\widetilde{l}}|||p_{h}^{k}|||_{Q}
\lesssim&  \Delta t\sum_{k=1}^{\widetilde{l}}(\|\bm{f}^{k}\|_{\ast,h} +\nu ||| \bm{u}_h^{k}|||_{V}+ \alpha ||| \bm{u}_h^{k}|||^{r-1}_{V} +\mathcal{N}_{h}||| \bm{u}_h^{k}|||^{2}_{V}  )%\nonumber\\&
+\sum_{k=1}^{\widetilde{l}}\|\bm{u}_{hi}^{k}-\bm{u}_{hi}^{k-1} \|_{0},
\end{align*}
which, together with \eqref{EEU}, implies \eqref{EE-P}.
This completes the proof.
\end{proof}

\subsection{   Existence and uniqueness results  }

In order to prove the existence of solutions to the scheme \eqref{fullwg}, we introduce the following auxiliary equation: for each $1\leq k\leq N$, seek $\bm{u}_h^{k}\in \bm{V}_{0h}$ such that
\begin{align}\label{E55}
 \mathcal{A}_h(\bm{u}^{k}_h;\bm{u}^{k}_h,\bm{v}_h )=F_{h}( \bm{v}_{hi}), \quad\forall \bm{v}_h\in \bm{V}_{0h},
\end{align}
where  the trilinear form $\mathcal{A}_h(\cdot;\cdot,\cdot ): \bm{V}_{0h}\times \bm{V}_{0h}\times \bm{V}_{0h}\rightarrow \mathbb R^{n}$ is defined by
 \begin{align*}
\mathcal{A}_h(\bm{\kappa}_h^{k};\bm{u}_h^{k},\bm{v}_h ):=&(\frac{\bm{u}_{hi}^{k}}{\Delta t},\bm{v}_{hi}  )+a_{h}(\bm{u}_h^{k}, \bm{v}_h )+c_h(\bm{\kappa}_h^{k};\bm{u}_h^{k},\bm{v}_h )+  d_h(\bm{\kappa}_h^{k};\bm{u}_h^{k},\bm{v}_h ),
\end{align*}
and $$F_{h}( \bm{v}_{hi}):=(\bm{f}^{k},\bm{v}_{hi})+(\frac{\bm{u}_{hi}^{k-1}}{\Delta t},\bm{v}_{hi}  ),$$
for any $ \bm{\kappa}_h^{k},\bm{u}_h^{k},\bm{v}_h\in\bm{V}_{0h}$.

Similar to the semi-discrete case, we have the following equivalence result.
\begin{lemma} \label{EQR}
For each $1\leq k\leq N$, the  fully discrete problems \eqref{fullwg} and \eqref{E55} are equivalent in the sense that both $(i)$ and $(ii)$ hold:
\begin{itemize}
\item[$(i)$] If $(\bm{u}_h^{k},p_h^{k})\in \bm{V}_h^{0}\times Q_h^{0}$ solves \eqref{fullwg}, then $\bm{u}_h^{k}\in \bm{V}_{0h}$ solves   \eqref{E55};
\item[$(ii)$]  If $\bm{u}_h^{k}\in \bm{V}_{0h}$ solves   \eqref{E55}, then  $(\bm{u}_h,p_h)$ solves \eqref{fullwg}, where $p_h^{k}\in Q_h^{0} $ is determined by
\begin{align}\label{QP}
		b_h(\bm{v}_{h},p_h^{k})=(\bm{f}^{k},\bm{v}_{hi})- (\frac{\bm{u}_{hi}^{k}-\bm{u}_{hi}^{k-1}}{\triangle t},\bm{v}_{hi} )-a_{h}(\bm{u}_h^{k}, \bm{v}_h )%\nonumber\\&&
		-c_h(\bm{u}_h^{k} ;\bm{u}_h^{k} ,\bm{v}_h)-d_h(\bm{u}_h^{k} ;\bm{u}_h^{k} ,\bm{v}_h), \quad\forall \bm{v}_{h}\in \bm{V}_h^{0}.
	\end{align}
\end{itemize}
\end{lemma}

Based on Lemma \ref{EQR}, we can gain the following existence  result for the full  discrete WG method.
\begin{theorem} \label{TH5.2}
For each $1\leq k\leq N$, the WG scheme \eqref{fullwg} admits at least a solution  $(\bm{u}_h^{k},p_h^{k})\in \bm{V}_h^{0}\times Q_h^{0}$.
\end{theorem}
\begin{proof}
We first show the problem \eqref{E55}  admits at least one solution $\bm{u}_{h}^{k}\in\bm{V}_{0h}$. According to
\cite[Theorem 1.2]{GiraultVivette1986FEMf}, it suffices to show that
the following two results hold:
\begin{itemize}
\item[(I)] $\mathcal{A}_h(\bm{v}_h^{k};\bm{v}_h^{k},\bm{v}_h^{k} )\geq\nu|||\bm{v}_h^{k}|||_{V}^2, \quad \forall  \bm{v}_h^{k}\in \bm{V}_{0h}$;
	
\item[(II)] $\bm{V}_{0h}$ is separable,  and the relation $\lim\limits_{\widehat{l}\to\infty}\bm{u}_h^{k,\widehat{l}}=\bm{u}_h^{k}$ (weakly in $\bm{V}_{0h}$) implies
\begin{eqnarray*}
\lim_{\widehat{l}\to\infty}\mathcal{A}_h(\bm{u}_h^{k,\widehat{l}};\bm{u}_h^{k,\widehat{l}},\bm{v}_h )
	=\mathcal{A}_h(\bm{u}_h^{k};\bm{u}_h^{k},\bm{v}_h ),\quad\forall\bm{v}_h\in \bm{V}_{0h}.
\end{eqnarray*}
\end{itemize}

In fact,   (I) follows from Lemma \ref{Lemma 3.1}    directly. We only need to show (II).
	Since $\bm{V}_{0h}$ is a finite dimensional space, we know that $\bm{V}_{0h}$ is separable and that  the weak convergence $\lim\limits_{\widehat{l}\to\infty}\bm{u}_h^{k,\widehat{l}}=\bm{u}_h^{k}$ on $\bm{V}_{0h}$ is equivalent to the strong convergence
	\begin{eqnarray}
	\lim_{\widehat{l}\to\infty}||| \bm{u}_h^{k,l\widehat{}}-\bm{u}_{h}^{k}|||_{V} =0.\label{45}
	\end{eqnarray}
On the other hand, by Lemmas \ref{Lemma 2.4}, \ref{Lemma 2.9}, \ref{Lemma 3.1} and the definition of $\mathcal{N}_{h}$, we have
	\begin{align*}
	&|\mathcal{A}_h(\bm{u}_{h}^{k,\widehat{l}};\bm{u}_{h}^{k,\widehat{l}},\bm{v}_{h} )
	-\mathcal{A}_h(\bm{u}_{h}^{k};\bm{u}_{h}^{k},\bm{v}_{h} )|\nonumber\\
	=&|(\frac{\bm{u}_{hi}^{k,\widehat{l}}-\bm{u}_{hi}^{k}}{\Delta t},\bm{v}_{hi}  )
+a_{h}(\bm{u}_{h}^{k,\widehat{l}}-\bm{u}_h^{k}, \bm{v}_h )
+\big(c_h(\bm{u}_{h}^{k,\widehat{l}} ;\bm{u}_{h}^{k,\widehat{l}} ,\bm{v}_{h} )-c_h(\bm{u}_{h}^{k} ;\bm{u}_{h}^{k} ,\bm{v}_{h})\big)
   +d_h(\bm{u}_{h}^{k,\widehat{l}}-\bm{u}_{h}^{k};\bm{u}_{h}^{k,\widehat{l}}-\bm{u}_{h}^{k},\bm{v}_{h})\nonumber\\
	&+ d_h(\bm{u}_{h}^{k};\bm{u}_{h}^{k,\widehat{l}}-\bm{u}_{h}^{k},\bm{v}_{h})
    +d_h(\bm{u}_{h}^{k,\widehat{l}}-\bm{u}_{h}^{k};\bm{u}_{h}^{k},\bm{v}_{h})|\nonumber\\
	\le&\frac{1}{\Delta t}||| \bm{u}_{hi}^{k,\widehat{l}}-\bm{u}_{hi}^{k}|||_{V}\cdot|||\bm{v}_{h}|||_{V}
 +\nu|||\bm{u}_{h}^{k,\widehat{l}}-\bm{u}_{h}^{k}|||_{V}\cdot|||\bm{v}_{h}|||_{V}
    +C_{\widetilde{r}}^{r}C_{r}\alpha|||\bm{u}^{k,\widehat{l}}_{h}-\bm{u}_{h}^{k}|||_{V} (|||\bm{u}_{h}^{k,\widehat{l}}|||_{V}+ |||\bm{u}_{h}^{k}|||_{V} )^{r-2}|||\bm{v}_{h} |||_{V}\nonumber\\
    &+\mathcal{N}_h|||\bm{u}_{h}^{k,\widehat{l}}-\bm{u}_{h}^{k}|||^2_{V}
	|||\bm{v}_{h}|||_{V}%\nonumber\\&&\qquad
	+2\mathcal{N}_h|||\bm{u}_{h}^{k,\widehat{l}}-\bm{u}_{h}^{k}|||_{V}\cdot|||\bm{u}_{h}^{k}|||_{V}\cdot |||\bm{v}_{h}|||_{V},\nonumber
	\end{align*}
which, together with \eqref{45}, yields
 $$\lim_{\widehat{l}\to\infty}\mathcal{A}_h(\bm{u}_{h}^{k,\widehat{l}};\bm{u}_{h}^{k,\widehat{l}},\bm{v}_{h} )
	=\mathcal{A}_h(\bm{u}_{h}^{k};\bm{u}_{h}^{k},\bm{v}_{h} ),\
	\forall\bm{v}_{h}\in \bm{V}_{0h},$$
	i.e.  (II)   holds.	Hence, \eqref{E55}  has at least one solution $\bm{u}_{h}^{k}\in\bm{V}_{0h}$.
	
	For a given $\bm{u}_{h}^{k}\in\bm{V}_{0h}\subset \bm{V}_h^{0}$, in light of  the  discrete inf-sup condition there is a unique $p_h^{k}\in Q_h^{0}$ satisfying \eqref{QP}.
	Thus, according to  Lemma \ref{EQR} we know that $(\bm{u}_h^{k},p_h^{k})$ is a solution of \eqref{fullwg}.
This finishes the proof.
\end{proof}

Furthermore, we have the following uniqueness result.

\begin{theorem} \label{Theorem 4.3}
Assume that condition
\begin{eqnarray}
\frac{\mathcal{N}_h}{\nu^{2}}\|\bm{f}^{k}\|_{*,h}<1 \label{uni-condi}
	\end{eqnarray}
holds, then for $1\leq k\leq N$  the scheme \eqref{fullwg} admits a unique solution $(\bm{u}_h^{k},p_h^{k})\in \bm{V}_h^{0}\times Q_h^{0}$.
\end{theorem}
\begin{proof}
	  Let   $(\bm{u}_{h1}^{k},p_{h1}^{k})$ and $(\bm{u}_{h2}^{k},p_{h2}^{k})$ be two solutions of \eqref{fullwg}, i.e.   for $j=1,2$ and $(\bm{v}_h,q_h)\in \bm{V}_h^{0}\times Q_h^{0}$ there hold
	\begin{align*}
	( D_{t}\bm{u}_{hi,j}^{k},\bm{v}_{hi} )+a_{h}(\bm{u}_{hj}^{k},\bm{v}_h)+c_h(\bm{u}_{hj}^{k};\bm{u}_{hj}^{k},\bm{v}_{h})
+d_h(\bm{u}_{hj}^{k};\bm{u}_{hj}^{k},\bm{v}_{h})
+b_h(\bm{v}_{hj},p_{hj}^{k})%\label{u-11}
&=(\bm{f}^{k},\bm{v}_{hi}),\\
b_h(\bm{u}_{hj}^{k},q_h)&=0,%\label{u-12}
\end{align*}
which give
\begin{subequations}
\begin{align}
( \frac{\bm{u}_{hi,1}^{k}- \bm{u}_{hi,2}^{k}}{\Delta t },\bm{v}_{hi})
+a_{h}(\bm{u}_{h1}^{k}-\bm{u}_{h2}^{k},\bm{v}_h)+c_h(\bm{u}_{h1}^{k};\bm{u}_{h1}^{k},\bm{v}_{h})
-c_h(\bm{u}_{h2}^{k};\bm{u}_{h2}^{k},\bm{v}_{h})-b_h(\bm{v}_{h},p_{h1}-p_{h2})&\nonumber\\
\qquad\qquad\qquad \quad \ \ = d_h(\bm{u}_{h2}^{k};\bm{u}_{h2}^{k},\bm{v}_{h})- d_h(\bm{u}_{h1}^{k};\bm{u}_{h1}^{k},\bm{v}_{h}),&\label{Ua}\\
b_h(\bm{u}_{h1}^{k}-\bm{u}_{h2}^{k},q_h)=0.&\label{Ub}
\end{align}
\end{subequations}
	Taking $\bm{v}_h=\bm{u}_{h1}^{k}-\bm{u}_{h2}^{k}$ and $q_h=p_{h1}^{k}-p_{h2}^{k} $ in the above two equations and employing \eqref{d0}, \eqref{a2} and \eqref{dh1}, we  obtain
\begin{subequations}
\begin{align*}
&\frac{\|\bm{u}_{hi,1}^{k}- \bm{u}_{hi,2}^{k}\|_0^{2}}{\Delta t }+\nu|||\bm{u}_{h1}^{k}-\bm{u}_{h2}^{k}|||_{V}^{2}+c_h(\bm{u}_{h1}^{k};\bm{u}_{h1}^{k},\bm{u}_{h1}^{k}-\bm{u}_{h2}^{k})
-c_h(\bm{u}_{h2}^{k};\bm{u}_{h2}^{k},\bm{u}_{h1}^{k}-\bm{u}_{h2}^{k}) \nonumber\\
 =& d_h(\bm{u}_{h2}^{k};\bm{u}_{h2}^{k},\bm{u}_{h1}^{k}-\bm{u}_{h2}^{k})- d_h(\bm{u}_{h1}^{k};\bm{u}_{h1}^{k},\bm{u}_{h1}^{k}-\bm{u}_{h2}^{k})\\
=& d_h(\bm{u}_{h2}^{k}-\bm{u}_{h1}^{k};\bm{u}_{h1}^{k},\bm{u}_{h1}^{k}-\bm{u}_{h2}^{k})\nonumber\\%&b_h(\bm{u}_{h1}-\bm{u}_{h2},q_h)=0.%\label{Ub}
\le&\mathcal{N}_h||| \bm{u}_{h1}^{k}|||_{V} ||| \bm{u}_{h1}^{k}-\bm{u}_{h2}^{k}|||^2_{V},
\end{align*}
\end{subequations}	
which, together with Lemma \ref{Lemma 2.9} and  the inequality
$$c_h(\bm{u}_{h1}^{k};\bm{u}_{h1}^{k},\bm{u}_{h1}^{k}-\bm{u}_{h2}^{k})-c_h(\bm{u}_{h2}^{k};\bm{u}_{h2}^{k},\bm{u}_{h1}^{k}-\bm{u}_{h2}^{k})\gtrsim \|\bm{u}_{hi1}^{k}-\bm{u}_{hi2}^{k}\|^{r}_{0,r}\geq 0,$$
   implies
\begin{align*}
\nu|||\bm{u}_{h1}^{k}-\bm{u}_{h2}^{k}|||^2_{V}
\le&\mathcal{N}_h||| \bm{u}_{h1}^{k}|||_{V} ||| \bm{u}_{h1}^{k}-\bm{u}_{h2}^{k}|||^2_{V} \nonumber\\
\le&{\color{black}\nu^{-1}\mathcal{N}_h\|\bm{f}^{k}\|_{*,h}|||\bm{u}_{h1}^{k}-\bm{u}_{h2}^{k}|||^2_{V}},
\end{align*}
i.e.
\begin{eqnarray*}
(1-\frac{\mathcal{N}_h}{\nu^{2}}\|\bm{f}^{k}\|_{*,h})|||\bm{u}_{h1}^{k}-\bm{u}_{h2}^{k}|||^2_{V} \le 0.
\end{eqnarray*}
The above  inequality plus the assumption \eqref{uni-condi} leads to $\bm{u}_{h1}^{k}=\bm{u}_{h2}^{k}$. Then from equation
 \eqref{Ua},  it follows that $b_h(\bm{v}_{h},p_{h1}^{k}-p_{h2}^{k})=0$, which  combining  the  discrete inf-sup condition, yields $p_{h1}^{k}=p_{h2}^{k}$. This finishes the proof.
\end{proof}

\section{ Error estimate for the fully discrete scheme}

This section is devoted to the error analysis for the fully discrete WG method \eqref{fullwg}  with the initial data
\begin{align*}
\bm{u}_h^{0}:=\bm{\mathcal{I}}_{h}\bm{u}_{0}=\{\bm{P}_{m}^{RT}\bm{u}_{0},\bm{\Pi}_m^{B}\bm{u}_{0} \}.
\end{align*}
For $1\leq k\leq N$, denote
%\begin{subequations}
\begin{align}
\bm{\mathcal{I}}_{h}\bm{u}^{k}\mid_{K}:=\{\bm{P}_{m}^{RT}(\bm{u}^{k}\mid_{K}), \bm{\Pi}_m^{B}(\bm{u}^{k}\mid_{K})\},\quad
\mathcal{P}_{h}p^{k}\mid_{K}:=\{\Pi_{m-1}^{\ast}(p^{k}\mid_{K}),\Pi_{m}^{B}(p^{k}\mid_{K})\}.
\end{align}
%\end{subequations}

\subsection{ Primary results}
We first provide the following Discrete Gronwall's Lemma \cite[Lemma 5.1]{HeywoodJohnG.1990FAot}.
\begin{lemma}\label{Lemma 7.0}
For integer $j\geq 0$, let $\triangle t$, $\widetilde{H}, \widetilde{a}_{j}, \widetilde{b}_{j},\widetilde{c}_{j}$ and $\widetilde{d}_{j}$ be nonnegative numbers such that
\begin{align}
\widetilde{a}_{k}+\triangle t \sum_{j=0}^{k}\widetilde{b}_{j}\leq \triangle t \sum_{j=0}^{k}\widetilde{a}_{j}\widetilde{d}_{j}+\triangle t \sum_{j=0}^{k}\widetilde{c}_{j}+\widetilde{H}, \text{ for }  k\geq 0.
\end{align}
Suppose that for all $j$, $\triangle t \widetilde{d}_{j} <1.$ Then
\begin{align}
\widetilde{a}_{k}+\triangle t \sum_{j=0}^{k}\widetilde{b}_{j}\leq exp (\triangle t \sum_{j=0}^{k}\frac{\widetilde{d}_{j}}{1- \triangle t \widetilde{d}_{j}})
(\triangle t \sum_{j=0}^{k}\widetilde{c}_{j}+\widetilde{H}),\text{ for } k\geq 0.
\end{align}
\end{lemma}

Utilizing the same route as in the proof of Lemma \ref{Lemma 5.3}, we have the following lemma.
\begin{lemma}\label{Lemma 7.1}
Assume that $(\bm{u}^k,p^k)$ is the \ solution to \eqref{weak} at time $t_k$, for any $(\bm{v}_h,q_h)\in \bm{V}_h^{0}\times Q_h^{0}$, there hold
	\begin{eqnarray}
	\bm{P}^{RT}_m\bm{u}^{k}|_K\in [P_{m}(K)]^n, \quad \forall K\in\mathcal{T}_h, \label{7.5}
	\end{eqnarray}
%and for any $(\bm{v}_{h}, q_h)\in \bm{V}_{h}^{0}\times Q_{h}^{0}$
\begin{subequations}\label{ERR}
\begin{align}
a_{0}(D_{t}\bm{\mathcal{I}}_{h}\bm{u}^{k},\bm{v}_{h})
+a_h(\bm{\mathcal{I}}_{h}\bm{u}^{k},\bm{v}_h)+b_h(\bm{v}_h,\mathcal{P}_{h}p^{k})
+c_h(\bm{\mathcal{I}}_{h}\bm{u}^{k};\bm{\mathcal{I}}_{h}\bm{u}^{k},\bm{v}_h)
+ d_h(\bm{\mathcal{I}}_{h}\bm{u}^{k};\bm{\mathcal{I}}_{h}\bm{u}^{k},\bm{v}_h)&\nonumber\\
=(\bm{f}^{k},\bm{v}_{hi}) +\xi_{I}(\bm{u}^{k};\bm{u}^{k},\bm{v}_h)+\xi_{II}(\bm{u}^{k},\bm{v}_h)%&\nonumber\\
+\xi_{III}(\bm{u}^{k};\bm{u}^{k},\bm{v}_h)
+ (D_{t}\bm{P}_m^{RT}\bm{u}^{k}-\bm{u}_{t}^{k},\bm{v}_{hi} ),\label{e-2}\\
	b_h(\bm{\mathcal{I}}_{h}\bm{u}^{k},q_h)=0,&\label{e-3}
\end{align}
\end{subequations}
where,
	\begin{align*}
a_{0}(D_{t}\bm{\mathcal{I}}_{h}\bm{u}^{k},\bm{v}_{h}):=&(D_{t}\bm{P}_{m}^{RT}\bm{u}^{k},\bm{v}_{hi}),\\
\xi_{I}(\bm{u}^{k};\bm{u}^{k},\bm{v}_{h}):=& -\frac{1}{2}(\bm{P}_m^{RT}\bm{u}^{k}\otimes \bm{P}_m^{RT}\bm{u}^{k}-\bm{u}^{k}\otimes\bm{u}^{k},\nabla_h\bm{v}_{hi})%\nonumber\\&
+\frac{1}{2}\langle(\bm{\Pi}_{m}^{B}\bm{u}^{k}\otimes \bm{\Pi}_{m}^{B}\bm{u}^{k}
-\bm{u}^{k}\otimes\bm{u}^{k})\bm{n},\bm{v}_{hi}\rangle_{\partial\mathcal{T}_h}\nonumber\\&
-\frac{1}{2}((\bm{u}^{k}\cdot\nabla)\bm{u}^{k}-(\bm{P}_m^{RT}\bm{u}^{k}\cdot\nabla_h)\bm{P}_m^{RT}\bm{u}^{k},\bm{v}_{hi})%\nonumber\\&
-\frac{1}{2}\langle(\bm{v}_{hb}\otimes \bm{\Pi}_{m}^{B}\bm{u}^{k})\bm{n},\bm{P}_m^{RT}\bm{u}^{k}\rangle_{\partial\mathcal{T}_h},\nonumber\\
	\xi_{II}(\bm{u}^{k},\bm{v}_h):=& \nu\langle (\nabla\bm u^{k}-\bm{\Pi}_{l}^\ast\nabla\bm u^{k})\bm{n},\bm{v}_{hi}-\bm{v}_{hb} \rangle_{\partial\mathcal{T}_h} %\nonumber\\&
+\nu\langle \eta(\bm{P}^{RT}_m\bm{u}^{k}-{\bm{u}^{k}}),\bm{v}_{hi}-\bm{v}_{hb} \rangle_{\partial\mathcal{T}_h},\nonumber\\
\xi_{III}(\bm{u}^{k};\bm{u}^{k},\bm{v}_h):=&
\alpha (|\bm{P}^{RT}_m\bm{u}^{k}|^{r-2}\bm{P}^{RT}_m\bm{u}^{k}-|\bm{u}^{k}|^{r-2}\bm{u}^{k},\bm{v}_{hi}).
	\end{align*}
	
\end{lemma}

\begin{lemma}\label{Lemma 7.2}(cf.\cite[Lemma 6.3]{HLX2019})
Assume that $\bm{u}_{t}\in L^{2}([H^{m+1}]^{n})$ and $\bm{u}_{tt}\in L^{2}([L^{2}]^{n})$.
 Then, for any $\bm{v}_{h}\in \bm{V}_{h}^{0}$,  there holds
\begin{align}
&|(\bm{u}_{t}^{k} - D_{t}\bm{P}^{RT}_m\bm{u}^{k}, \bm{v}_{hi}) |
\lesssim  (\frac{h^{m+1}}{\sqrt{\triangle t}}(\int_{t_{k-1}}^{t_{k}}|\bm{u}_{t} |_{m+1}^{2} d\tau)^{1/2}
+\sqrt{\triangle t} (\int_{t_{k-1}}^{t_{k}}\|\bm{u}_{tt} \|^{2}_{0} d\tau)^{1/2}  )\|\bm{v}_{hi} \|_{0}.
\end{align}
\end{lemma}

\subsection{ Error estimate}

Similar to the semi-discrete situation, we introduce the following auxiliary problem \eqref{auxiliary}:
find $\widetilde{\bm{u}}_{h}^{k}\in \bm{V}_{0h}$ such that
\begin{align}\label{AWGF}
a_{0}(D_{t}\widetilde{\bm{u}}_{hi}^{k},\bm{v}_{hi})+a_{h}(\widetilde{\bm{u}}_{h}^{k}, \bm{v}_{h})+c_{h}(\widetilde{\bm{u}}_{h}^{k};\widetilde{\bm{u}}_{h}^{k},\bm{v}_{h} )=
( \bm{f}^{k}, \bm{v}_{hi})%&\nonumber\\
-d_{h}(\bm{\mathcal{I}}_{h}\bm{u}^{k};\bm{\mathcal{I}}_{h}\bm{u}^{k},\bm{v}_{h} ),\quad\forall \bm{v}_{h}\in \bm{V}_{0h},&
\end{align}
with the initial data $\widetilde{\bm{u}}_{h}(\bm{x},0)=\widetilde{\bm{u}}_{h}^{0}=\bm{u}_{h}^{0}$. %Here, $\bm{V}_{0h}$ is defined in \eqref{V0h}.

Let
$$e_{u}^{k}=\bm{u}^{k}-\bm{u}_{h}^{k}=(\bm{u}^{k}-\bm{\mathcal{I}}_{h}\bm{u}^{k})+(\bm{\mathcal{I}}_{h}\bm{u}^{k}
-\widetilde{\bm{u}}_{h}^{k})+(\widetilde{\bm{u}}_{h}^{k}- \bm{u}_{h}^{k} )
:= \delta_{u}^{k}+  \varphi_{u}^{k}+ \chi_{u}^{k},$$
where
$$\delta_{u}^{k}=\bm{u}^{k}-\bm{\mathcal{I}}_{h}\bm{u}^{k},\ \varphi_{u}^{k}=\bm{\mathcal{I}}_{h}\bm{u}^{k}
-\widetilde{\bm{u}}_{h}^{k},\ \chi_{u}^{k}=\widetilde{\bm{u}}_{h}^{k}- \bm{u}_{h}^{k},$$
and $\varphi_{u}^{0}=\chi_{u}^{ 0}=0$.

For the above error decomposition, we have the following lemmas.

\begin{lemma}\label{Lemma 7.3}%\cite{HYH2019}
Assume that $\bm{u}\in L^{\infty}([H^{2}]^{n})\cap  L^{2}([H^{m+1}]^{n})$, $\bm{u}_{t}\in L^{2}([H^{m+1}]^{n})$ and $\bm{u}_{tt}\in L^{2}([L^{2}]^{n})$.
For $k=1,...,N$, let $\widetilde{\bm{u}}_{h}^{k}\in \bm{V}_{h}^{0}$ be the solution to \eqref{AWGF}. Then there holds
\begin{align}\label{7.7}
  |||\varphi_{u}^{k}|||^{2}_{V}
\lesssim C_{1}(\bm{u}) h^{2m}+C_{2}(\bm{u})\triangle t^{2} &.
\end{align}
Here, $ C_{1}(\bm{u})$ is a positive constant depending only on $\nu,\alpha,r$ and $\|\bm{u}\|_2.$ $\|\bm{u}\|_{m+1}$ and $\|\bm{u}_t\|_{m+1}$. \\
$ C_{2}(\bm{u})=\int_{0}^{t_{k}}\|\bm{u}_{tt} \|^{2}_{0} d\tau.$
\end{lemma}
\begin{proof}
According to the Lemma \ref{Lemma 7.1} and \eqref{AWGF}, we get
\begin{align}\label{7.8}
&a_{0}( D_{t}\varphi_{u}^{k},\bm{v}_{h} )+ a_{h}(\varphi_{u}^{k},\bm{v}_{h}  )
+c_h(\bm{\mathcal{I}}_{h}\bm{u}^{k};\bm{\mathcal{I}}_{h}\bm{u}^{k},\bm{v}_h)
-c_{h}(\widetilde{\bm{u}}_{h}^{k};\widetilde{\bm{u}}_{h}^{k},\bm{v}_{h} )\nonumber\\
=&\xi_{I}(\bm{u}^{k};\bm{u}^{k},\bm{v}_h)+\xi_{II}(\bm{u}^{k},\bm{v}_h)
+\xi_{III}(\bm{u}^{k};\bm{u}^{k},\bm{v}_h)
+ (D_{t}\bm{P}_m^{RT}\bm{u}^{k}-\bm{u}_{t}^{k},\bm{v}_{hi} )
\end{align}
Taking $\bm{v}_{h}=\varphi_{u}^{k}$ in above equation and utilizing  Lemmas \ref{Lemma 2.4}, \ref{Lemma 2.9}, \ref{Lemma 3.1}, \ref{Lemma 5.5}, \ref{Lemma 5.6}, \ref{Lemma 7.2}, the relation $2(a-b,a)=a^{2}-b^{2}+(a-b)^{2}$ and  H\"{o}lder's inequality, we obtain
\begin{align}\label{7.9}
&\frac{1}{2\triangle t}( \|\varphi_{ui}^{k}\|_{0}^{2}-\|\varphi_{ui}^{k-1}\|_{0}^{2}
+\|\varphi_{ui}^{k}-\varphi_{ui}^{k-1} \|_{0}^{2})+ \nu |||\varphi_{u}^{k}|||^{2}_{V}+ \alpha\|\varphi_{ui}^{k}\|_{0,r}^{r}\nonumber\\
\lesssim& h^{m}(\|\bm{u}^{k}\|_{2}\|\bm{u}^{k}\|_{1+m}+ \|\bm{u}^{k}\|_{1+m}+\| \bm{u}^{k}\|_{2}^{r-2}\|\bm{u}^{k}\|_{1+m} )|||\varphi_{u}|||_{V}\nonumber\\
&+(\frac{h^{m+1}}{\sqrt{\triangle t}}(\int_{t_{k-1}}^{t_{k}}|\bm{u}_{t} |_{m+1}^{2} d\tau)^{1/2}
+\sqrt{\triangle t} (\int_{t_{k-1}}^{t_{k}}\|\bm{u}_{tt} \|^{2}_{0} d\tau)^{1/2}  )\|\varphi_{ui}^{k} \|_{0}\nonumber\\
\lesssim & \frac{h^{2m}}{2\nu}(\|\bm{u}^{k}\|_{2}\|\bm{u}^{k}\|_{1+m}+ \|\bm{u}^{k}\|_{1+m}+\| \bm{u}^{k}\|_{2}^{r-2}\|\bm{u}^{k}\|_{1+m} )^{2}
+\frac{\nu}{2}|||\varphi_{u}|||^{2}_{V}\nonumber\\
 &+(\frac{h^{m+1}}{\sqrt{\triangle t}}(\int_{t_{k-1}}^{t_{k}}|\bm{u}_{t} |_{m+1}^{2} d\tau)^{1/2}
+\sqrt{\triangle t} (\int_{t_{k-1}}^{t_{k}}\|\bm{u}_{tt} \|^{2}_{0} d\tau)^{1/2}  )\|\varphi_{ui}^{k} \|_{0}.
\end{align}
Furthermore, we gain
\begin{align}\label{7.10}
&\|\varphi_{ui}^{k}\|_{0}^{2}
+\sum_{j=1}^{k}\|\varphi_{ui}^{j}-\varphi_{ui}^{j-1} \|_{0}^{2}
+ \nu \triangle t \sum_{j=1}^{k} |||\varphi_{u}^{j}|||^{2}_{V}
+ 2\alpha\Delta t \sum_{j=1}^{k} \|\varphi_{ui}^{j}\|_{0,r}^{r}\nonumber\\
\lesssim & \frac{h^{2m}}{\nu}\triangle t\sum_{j=1}^{k}(\|\bm{u}^{j}\|_{2}\|\bm{u}^{j}\|_{1+m}+ \|\bm{u}^{j}\|_{1+m}+\| \bm{u}^{j}\|_{2}^{r-2}\|\bm{u}^{j}\|_{1+m} )^{2}\nonumber\\
 &+\triangle t\sum_{j=1}^{k}(\frac{h^{m+1}}{\sqrt{\triangle t}}(\int_{t_{k-1}}^{t_{k}}|\bm{u}_{t} |_{m+1}^{2} d\tau)^{1/2}
+\sqrt{\triangle t} (\int_{t_{k-1}}^{t_{k}}\|\bm{u}_{tt} \|^{2}_{0} d\tau)^{1/2}  )\|\varphi_{ui}^{k} \|_{0},\nonumber\\
\lesssim & \frac{h^{2m}}{\nu}\triangle t\sum_{j=1}^{k}(\|\bm{u}^{j}\|_{2}\|\bm{u}^{j}\|_{1+m}+ \|\bm{u}^{j}\|_{1+m}+\| \bm{u}^{j}\|_{2}^{r-2}\|\bm{u}^{j}\|_{1+m} )^{2}\nonumber\\
 &+h^{2m+2}\int_{0}^{t_{k}}|\bm{u}_{t} |_{m+1}^{2} d\tau
 +\triangle t^{2} \int_{0}^{t_{k}}\|\bm{u}_{tt} \|^{2}_{0} d\tau
 +\frac{1}{2}\|\varphi_{ui}^{k} \|_{0}^{2}.
\end{align}
together with  Lemma \ref{Lemma 7.0}, which means \eqref{7.7}.
This finishes the proof.
\end{proof}

\begin{lemma}\label{Lemma 7.4}
For $k=1,...,N$, let $\bm{u}_{h}^{k}$ and $\widetilde{\bm{u}}_{h}^{k}$ be    the solution to \eqref{fullwg} and \eqref{AWGF}, respectively. Under the same conditions as in Lemma  \ref{Lemma 7.3},  there holds
\begin{align}\label{7.15b}
 |||\chi_{u}^{k}|||^{2}_{V}
\lesssim C_{3}(\bm{u}) (h^{2m}+\triangle t^{2}) &,
\end{align}
where  $C_{3}(\bm{u})$ is a   positive constant depending only on $\alpha,r,\nu$,  $\|\bm{u}\|_2,$  $\|\bm{u}\|_{m+1},$  $\|\bm{u}_t\|_{m+1}$ and $\|\bm{u}_{tt}\|_0$.
\end{lemma}
\begin{proof}
In view of \eqref{wg1} and \eqref{AWGF}, we get
\begin{align*}
a_{0}( D _{t} \chi_{u}^{k}, \bm{v}_{h})+a_{h}( \chi_{u}^{k}, \bm{v}_{h})
+c_h(\widetilde{\bm{u}}_{h}^{k};\widetilde{\bm{u}}_{h}^{k}, \bm{v}_h  )
-c_h(\bm{u}_{h}^{k};\bm{u}_{h}^{k}, \bm{v}_h  )%&\nonumber\\
=d_h(\bm{u}_{h}^{k};\bm{u}_{h}^{k}, \bm{v}_h)
-d_h(\bm{\mathcal{I}}_{h}\bm{u}^{k};\bm{\mathcal{I}}_{h}\bm{u}^{k},\bm{v}_h).&
\end{align*}
Taking $ \bm{v}_{h}=\chi_{u}^{k}$ in above equation and using  Lemmas \ref{Lemma 2.4}, \ref{Lemma 2.9}  and \ref{Lemma 3.1}, we have
\begin{align*}
&\frac{1}{2\triangle t} ( \| \chi_{ui}^{k}\|_{0}^{2} -\| \chi_{ui}^{k-1}\|_{0}^{2}+\| \chi_{ui}^{k}-\chi_{ui}^{k-1}\|_{0}^{2})+\nu ||| \chi_{u}^{k}|||^{2}_{V}+ \alpha\|\chi_{ui}^{k}\|_{0,r}^{r}%\nonumber\\&
\lesssim d_h(\bm{u}_{h}^{k};\bm{u}_{h}^{k}, \bm{v}_h)
-d_h(\bm{\mathcal{I}}_{h}\bm{u}^{k};\bm{\mathcal{I}}_{h}\bm{u}^{k},\bm{v}_h).
\end{align*}
By following the same route as in the proof Lemma \ref{Lemma 7.3}, we could obtain
\begin{align}
&\| \chi_{ui}^{k}\|_{0}^{2} + \sum_{j=1}^{k}\| \chi_{ui}^{j}-\chi_{ui}^{j-1}\|_{0}^{2}
+2\nu  \triangle t \sum_{j=1}^{k}||| \chi_{u}^{j}|||^{2}_{V}+2\alpha \triangle t \sum_{j=1}^{k} \|\chi_{ui}^{j}\|_{0,r}^{r}\nonumber\\
\lesssim & 2C( \bm{u}) \triangle t \sum_{j=1}^{k}(\| \varphi_{ui}^{j}+\chi_{ui}^{j}\|_{0}
+\| \varphi_{ui}^{j} \|_{0}+h||| \varphi_{u}^{j} |||_{V})||| \chi_{u}^{j} |||_{V}\nonumber\\
\lesssim &  \frac{C( \bm{u})^{2}}{\nu}\triangle t \sum_{j=1}^{k}(\| \varphi_{ui}^{j}\|_{0}^{2}+\|\chi_{ui}^{j}\|_{0}^{2}
+h^{2}||| \varphi_{u}^{j} |||^{2}_{V})+\nu\triangle t \sum_{j=1}^{k}||| \chi_{u}^{j} |||_{V} .
\end{align}
From  Lemmas \ref{Lemma 7.0} - \ref{Lemma 7.3} and \eqref{7.10}, it yields  \eqref{7.15b}.
Thus, we complete this proof.
\end{proof}

Based on Lemmas \ref{Lemma 7.3} - \ref{Lemma 7.4}, we can obtain the following result.
\begin{lemma}\label{Lemma 7.5}
For $k=1,...,N$, let $\bm{u}_{h}^{k}$ and $\bm{\mathcal{I}}_{h}\bm{u}^{k}$ be   the solution to \eqref{fullwg} and \eqref{ERR}, respectively.   Under the same conditions as in Lemma  \ref{Lemma 7.3}, there holds
\begin{align*}
 |||\varepsilon_{u}^{k}|||^{2}_{V}
 &\lesssim (C_{1}(\bm{u})+C_{2}(\bm{u}) )h^{2m}+(C_{2}(\bm{u})+C_{3}(\bm{u}) )\triangle t^{2},
\end{align*}
where $\varepsilon_{u}^{k}=\bm{\mathcal{I}}_{h}\bm{u}- \bm{u}_{h}^{k}$,
$C_{1}(\bm{u} )$, $C_{2}(\bm{u} )$ and  $C_{3}(\bm{u} )$ are defined in Lemmas \ref{Lemma 7.3} and \ref{Lemma 7.4}, respectively.
\end{lemma}

According to  the triangle inequality, the projection properties and  Lemma \ref{Lemma 7.5}, we can get the following velocity error estimate.
\begin{theorem}\label{Theorem 7.1}
For $k=1,...,N$, let $\bm{u}_{h}^{k}$ and $\bm{u}^k$ be   the solution to \eqref{fullwg} and \eqref{weak}, respectively. Under the same conditions as in Lemma  \ref{Lemma 7.3}, there holds
\begin{subequations}\label{7.18}
\begin{align}
|||\bm{u}^{k}-\bm{u}_{h}^{k}|||^{2}_{V}
\lesssim C_{4}(\bm{u})(h^{2m}+\triangle t^{2}),
\end{align}
\end{subequations}
where,
$C_{4}(\bm{u})>0$ is a constant depending on $C_{1}(\bm{u})$, $C_{2}(\bm{u} )$, $C_{3}(\bm{u} )$ and  $\|\bm{u}\|_{m+1}$.
\end{theorem}

To estimate the pressure, we introduce the following time discrete scheme of the weak problem \eqref{weak}: For $k = 1,...,N$, seek $(\bm{u}^{k},p^{k})\in \bm{V}\times Q$ such that
\begin{subequations}\label{weak3}
\begin{align}
(\frac{\bm{u}^{k}-\bm{u}^{k-1} }{\Delta t},\bm{v})+a(\bm{u}^{k},\bm{v} )
+b(\bm{v}, p^{k})+c(\bm{u}^{k};\bm{u}^{k},\bm{v} )+d(\bm{u}^{k};\bm{u}^{k},\bm{v} )&=(\bm{f}, \bm{v} ),  &\forall\bm{v}\in \bm{V},\\
b(\bm{u}^{k}, q)&=0, &\forall q\in Q.
\end{align}
\end{subequations}

\begin{lemma}\label{Lemma 7.6}
Assume that $\bm{u}\in L^{\infty}([H^{2}]^{n})\cap L^{2}([H^{m+1}]^{n})$, $\bm{u}_{t}\in L^{2}([H^{m+1}]^{n})$, $\bm{u}_{tt}\in L^{2}([L^{2}]^{n})$ and $p\in L^{2}(H^{m})$.
For $k=1,...,N$, let $(\bm{u}_{h}^{k}, p_{h}^{k} )$ and $(\bm{\mathcal{I}}_{h}\bm{u}^{k},\mathcal{P}_{h}p^{k} )$ be  respectively the solution to \eqref{fullwg} and  \eqref{ERR},
then there holds
\begin{align}\label{Jp}
&|||\mathcal{P}_{h}p^{k}-p^{k}_{h} |||_{Q}
\lesssim C_{5}(\bm{u})(h^{m}+\triangle t)+ C_{6}(\bm{u})(h^{2m}+\triangle t^{2}).
\end{align}
Here, $C_{5}(\bm{u})>0$ and  $C_{6}(\bm{u})>0$ are  constants depending only on $\alpha,\nu,r,$ $C_{4}(\bm{u})$.
\end{lemma}
\begin{proof}
Define $$e_p^k=p^k-p_h^k=(p^k- \mathcal{P}_{h}p^{k})+(\mathcal{P}_{h}p^{k} -p_h^k):=\delta_{p}^{k}+ \varepsilon_{p}^{k},$$
where $$\delta_{p}^{k}=p^k- \mathcal{P}_{h}p^{k}, \varepsilon_{p}^{k}=\mathcal{P}_{h}p^{k} -p_h^k.$$

Subtracting the equation  \eqref{fullwg} from \eqref{weak3}, we get
\begin{subequations}\label{E6.26}
\begin{align}%
(\frac{e_{ui}^{k}-e_{ui}^{k-1} }{\Delta t},\bm{v}_{hi})+a_{h}(e_{u}^{k},\bm{v}_{h} )
+b_{h}(\bm{v}_{h}, e_p^{k})+c_{h}(\bm{u}^{k};\bm{u}^{k},\bm{v}_{h} )-c_{h}(\bm{u}^{k}_h;\bm{u}^{k}_h,\bm{v}_{h} )&\nonumber\\
+d_{h}(\bm{u}^{k};\bm{u}^{k},\bm{v}_{h} )-d_{h}(\bm{u}^{k}_h;\bm{u}^{k}_h,\bm{v}_{h} )&=0,  \label{E6.26a}\\
b_{h}(e_{u}^{k}, q)&=0.\label{E6.26b}
\end{align}
\end{subequations}
Taking $\bm{v}_{h}=\varepsilon_{u}^{k}$ and $q_{h}=\varepsilon_{p}^{k}$ in the above equations and utilizing the projection properties, we derive
\begin{align}\label{ERRE}
(\frac{e_{ui}^{k}-e_{ui}^{k-1} }{\Delta t},\varepsilon_{ui}^{k})+\nu|||\varepsilon_{u}^{k}|||_{V}^{2}
+c_{h}(\bm{u}^{k};\bm{u}^{k},\varepsilon_{u}^{k} )-c_{h}(\bm{u}^{k}_h;\bm{u}^{k}_h,\varepsilon_{u}^{k} )+d_{h}(\bm{u}^{k}-\bm{u}^{k}_h;\bm{u}^{k}-\bm{u}^{k}_h,\varepsilon_{u}^{k} )&\nonumber\\
+d_{h}(\bm{u}^{k}-\bm{u}^{k}_h;\bm{u}^{k}_h,\varepsilon_{u}^{k} )
+d_{h}(\bm{u}^{k}_h;\bm{u}^{k}-\bm{u}^{k}_h,\varepsilon_{u}^{k} )&=0.
\end{align}
Utilizing \eqref{ERRE},  H\"{o}lder's inequality, Lemmas  \ref{Lemma 2.4},  \ref{Lemma 2.9} and \ref{Lemma 3.1}, we gain
\begin{align*}
&\|\frac{e_{ui}^{k}-e_{ui}^{k-1} }{\Delta t}\|_{\ast,h}
=\sup_{\bm{0}\neq \bm{\varepsilon_{u}^{k} }}\frac{(\frac{e_{ui}^{k}-e_{ui}^{k-1} }{\Delta t},\varepsilon_{ui}^{k})}{ ||| \varepsilon_{u}^{k} |||_V}\nonumber\\
\lesssim&\frac{1}{ ||| \varepsilon_{u}^{k} |||_V}\big(\nu|||\varepsilon_{u}^{k}|||_{V}^{2}
+\alpha((|\bm{u}^{k} |+|\bm{u}^{k}_{hi}|)^{r-2}|\bm{u}^{k}-\bm{u}^{k}_{hi} |,\varepsilon_{ui}^{k}   )
+ (|||\bm{u}^{k}_h|||_{V}+|||\bm{u}^{k}-\bm{u}^{k}_h|||_{V}) |||\bm{u}^{k}-\bm{u}^{k}_h|||_{V}|||\varepsilon_{u}^{k}|||_{V}    \big)\nonumber\\
\lesssim&\nu|||\varepsilon_{u}^{k}|||_{V}
+\big(\alpha(|||\bm{u}^{k}-\bm{u}^{k}_h |||_V+|||\bm{u}^{k}_h|||_V)^{r-2} +|||\bm{u}^{k}_h|||_{V} \big)|||\bm{u}^{k}-\bm{u}^{k}_h |||_V+|||\bm{u}^{k}-\bm{u}^{k}_h |||_V^{2},
\end{align*}
From the discrete inf-sup inequality and \eqref{E6.26a}, it follows that
\begin{align*}
|||\varepsilon_p^{k}|||_{Q}
\lesssim&\sup_{\bm{0}\neq \bm{v}_h\in\bm{V}_{h}^{0} }\frac{b_h(\bm{v}_h,\varepsilon_p^{k}  )}{ ||| \bm{v}_h|||_V}\nonumber\\
\lesssim&\|\frac{e_{ui}^{k}-e_{ui}^{k-1} }{\Delta t}\|_{\ast,h}
+\nu|||e_u^{k} |||_{V}+\big(\alpha(|||e_u^{k}|||_V+|||\bm{u}^{k}_h|||_V)^{r-2} +|||\bm{u}^{k}_h|||_{V} \big)|||e_u^{k}|||_V+|||e_u^{k} |||_V^{2}+\|\delta _{p}^{k}  \|_0,
\end{align*}
which, together with  Lemma \ref{Lemma 7.5}, Theorem \ref{Theorem 7.1}, the  projection properties and the boundedness of $|||\bm{u}^{k}_h|||_{V}$, yields the desired result \eqref{Jp}.

\end{proof}

\begin{theorem}\label{Theorem 7.3}
Under the same conditions as in Theorem \ref{Theorem 7.1}, there holds
\begin{align}
 \|\bm{u}^{k}-\bm{u}_{hi}^{k}\|_{0}
\lesssim C_{7}(\bm{u})(h^{m+1}+\triangle t),\label{6.22}
\end{align}
where,
$C_{7}(\bm{u})>0$ is a constant depending on $\alpha,r,\nu,\bm{f}$, $\|\bm{u}\|_2$, $\|\bm{u}\|_{m+1}$ and $\|\bm{u}_t\|_{m+1}$.
\end{theorem}
\begin{proof}
Taking $\bm{v}_{h}= \varepsilon _{u}^{k}$ and $q_{h}= \varepsilon _{p}^{k}$ in \eqref{E6.26} and utilizing the projection properties, we gain
\begin{align*}%
&(\frac{\varepsilon_{ui}^{k}-\varepsilon_{ui}^{k-1} }{\Delta t},\varepsilon _{ui}^{k})+\nu||| \varepsilon _{u}^{k} |||_V^{2}
+c_{h}(\bm{\mathcal{I}}_{h}\bm{u}^{k};\bm{\mathcal{I}}_{h}\bm{u}^{k},\varepsilon _{u}^{k} )-c_{h}(\bm{u}^{k}_h;\bm{u}^{k}_h,\varepsilon _{u}^{k} )\nonumber\\
=&-(\frac{\delta_{ui}^{k}-\delta_{ui}^{k-1} }{\Delta t},\varepsilon _{ui}^{k})
+\alpha(|\bm{P}_m^{RT}\bm{u}^{k}|^{r-2}\bm{P}_m^{RT}\bm{u}^{k}-|\bm{u}^{k}|^{r-2} \bm{u}^{k}, \varepsilon _{ui}^{k}    )
-d_h(\varepsilon _{u}^{k}; \bm{u}^{k},\varepsilon _{u}^{k} )
-d_h(\delta_u^{k}; \bm{u}^{k},\varepsilon _{u}^{k} )
-d_h(\bm{u}^{k}_{h};\delta_u^{k},  \varepsilon _{u}^{k} ),
\end{align*}
which, together with  Lemmas \ref{Lemma 2.4}, \ref{Lemma 2.9}, \ref{Lemma 3.1} and \ref{Lemma 5.6}, the projection properties, Taylor expansion, the triangle inequality and H\"{o}lder's inequality, yields
\begin{align}%
&\frac{1}{2\Delta t} (   \|\varepsilon_{ui}^{k}\|_{0}^{2}-\|\varepsilon_{ui}^{k-1}\|_{0}^{2}+\|\varepsilon _{ui}^{k}-\varepsilon_{ui}^{k-1}\|_0^{2})+\nu||| \varepsilon _{u}^{k} |||_V^{2}
+\alpha \|\varepsilon_{ui}^{k}\|_{0,r}^{r}\nonumber\\
\lesssim& \| \frac{\delta_{ui}^{k}-\delta_{ui}^{k-1} }{\Delta t}\|_0 \| \varepsilon _{ui}^{k}\|_0
+\alpha(|\bm{P}_m^{RT}\bm{u}^{k}|+|\bm{u}^{k} |)^{r-2}|\bm{P}_m^{RT}\bm{u}^{k}-\bm{u}^{k}  |,\varepsilon _{ui}^{k})
+\|\bm{u}^{k}\|_2|||\varepsilon _{u}^{k} |||_V^{2}+(|||\bm{u}^{k}_{h} |||_V+\|\bm{u}^{k}\|_{2})\| \delta_{ui}^{k} \|_0 |||\varepsilon _{u}^{k} |||_V\nonumber\\
\lesssim& \frac{1}{ \Delta t}\|\bm{P}_m^{RT}(\bm{u}^{k} -\bm{u}^{k-1})-(\bm{u}^{k}-\bm{u}^{k-1} )   \|_0\| \varepsilon _{ui}^{k}\|_0
+\alpha\|\bm{u}^{k}\|_2^{r-2}\|\bm{P}_m^{RT}\bm{u}^{k}-\bm{u}^{k}  \|_0\|\varepsilon _{ui}^{k}\|_0\nonumber\\
&+\|\bm{u}^{k}\|_2|||\varepsilon _{u}^{k} |||_V^{2}+(\|\bm{f}^{k}\|_{\ast,h}+\|\bm{u}^{k}\|_{2})\| \delta_{ui}^{k} \|_0 |||\varepsilon _{u}^{k} |||_V\nonumber\\
\lesssim& \frac{1}{ \Delta t}\|(\bm{P}_m^{RT}\bm{u}_t^{k}-\bm{u}_t^{k})\Delta t+\mathcal{O}(\Delta t^{2})   \|_0\| \varepsilon _{ui}^{k}\|_0
+\alpha\|\bm{u}^{k}\|_2^{r-2}h^{m+1}\|\bm{u}^{k}  \|_{m+1}\|\varepsilon _{ui}^{k}\|_0\nonumber\\
&+\|\bm{u}^{k}\|_2|||\varepsilon _{u}^{k} |||_V^{2}+(\|\bm{f}^{k}\|_{\ast,h}+\|\bm{u}^{k}\|_{2})h^{m+1}\| \bm{u}^{k} \|_{m+1} |||\varepsilon _{u}^{k} |||_V\nonumber\\
\lesssim&h^{m+1}(\alpha\|\bm{u}^{k}\|_2^{r-2}\|\bm{u}^{k}  \|_{m+1}+ \|\bm{u}_t^{k}  \|_{m+1})\|\varepsilon _{ui}^{k}\|_0
+\Delta t\|\varepsilon _{ui}^{k}\|_0\nonumber\\&
+(\|\bm{f}^{k}\|_{\ast,h}+\|\bm{u}^{k}\|_{2})^{2}h^{2m+2}\| \bm{u}^{k} \|_{m+1} ^{2} +\frac{\nu}{2}|||\varepsilon _{u}^{k} |||_V^{2}.
\end{align}
Thus, for $k = 1,...,N$, sum the above inequality from 1 to $k$ and multiply by $\Delta t$, then we get
\begin{align*}%
&   \|\varepsilon_{ui}^{k}\|_{0}^{2}+\sum _{j=1}^{k}\|\varepsilon _{ui}^{j}-\varepsilon_{ui}^{j-1}\|_0^{2}+\nu\Delta t\sum _{j=1}^{k}||| \varepsilon _{u}^{k} |||_V^{2}
+2\alpha\Delta t\sum _{j=1}^{k} \|\varepsilon_{ui}^{k}\|_{0,r}^{r}\nonumber\\
\lesssim& \Delta t \sum _{j=1}^{k}h^{m+1}(\alpha\|\bm{u}^{j}\|_2^{r-2}\|\bm{u}^{j}  \|_{m+1}+ \|\bm{u}_t^{j}  \|_{m+1})\|\varepsilon _{ui}^{j}\|_0
+\Delta t^{2}+ \Delta t\sum _{j=1}^{k}\|\varepsilon _{ui}^{j}\|_0^{2}
+\Delta t \sum _{j=1}^{k}\|\bm{u}^{j}\|_2|||\varepsilon _{u}^{j} |||_V^{2}\nonumber\\&
+h^{2m+2} \Delta t \sum _{j=1}^{k} (\|\bm{f}^{j}\|_{\ast,h}+\|\bm{u}^{j}\|_{2})^{2}\| \bm{u}^{j} \|_{m+1} ^{2}\nonumber\\
\lesssim&h^{2m+2}\Delta t \sum _{j=1}^{k}(\alpha\|\bm{u}^{j}\|_2^{r-2}\|\bm{u}^{j}  \|_{m+1}+ \|\bm{u}_t^{j}  \|_{m+1})^{2}+\Delta t \sum _{j=1}^{k}\|\varepsilon _{ui}^{j}\|_0^{2}+\Delta t^{2}
\nonumber\\&
+h^{2m+2}\Delta t \sum _{j=1}^{k} (\|\bm{f}^{j}\|_{\ast,h}+\|\bm{u}^{j}\|_{2})^{2}\| \bm{u}^{j} \|_{m+1} ^{2},
\end{align*}
which, together with Lemma \ref{Lemma 7.0}, Theorem \ref{Theorem 6.1}, the triangle inequality and the projection properties, indicates \eqref{6.22}.
This completes this proof.
\end{proof}

Based on Lemma \ref{Lemma 7.6}, Theorems \ref{Theorem 7.1} - \ref{Theorem 7.3} ,  the triangle inequality and the projection properties, we can gain the following main conclusion  for the  fully discrete WG scheme \eqref{fullwg}.
\begin{theorem}\label{Theorem 7.2}
Under the same conditions as in Theorem \ref{Theorem 7.1}, there hold
\begin{subequations}\label{7.19}
\begin{align}
\|\bm{u}^{k}-\bm{u}_{hi}^{k}\|_{0}
\lesssim& C_{7}(\bm{u})(h^{m+1}+\triangle t),\\
 \|\nabla\bm{u}^{k}-\nabla_{h}\bm{u}_{hi}^{k}\|_{0}^{2}
+ \|\nabla\bm{u}^{k}-\nabla_{w,l}\bm{u}_{hi}^{k}\|_{0}^{2}
\lesssim & C_{4}(\bm{u} )(h^{2m}+\triangle t^{2}),\label{7.19b}\\
\|p^{k}-p_{h}^{k}\|_{0}
\lesssim& (C_{5}(\bm{u})+\|p\|_{m})h^{m}+C_{5}(\bm{u})\triangle t+ C_{6}(\bm{u})(h^{2m}+\triangle t^{2}).
\end{align}
\end{subequations}
\end{theorem}

\begin{remark}\label{remark}
The result \eqref{7.19b}  shows that the velocity error estimate is independent of the pressure approximation, which means that the proposed fully discrete WG scheme is pressure-robust. Meanwhile, all the convergence rates in \eqref{7.19} are with the optimal order.
\end{remark}
\section{Iteration scheme}

To solve the nonlinear system of the fully  discrete WG scheme \eqref{fullwg}, we shall employ  the following linearized iteration algorithm: for $ 1\leq k\leq N$,
given $\bm{u}_{h}^{k,0}=\bm{u}_{h}^{k-1,\ast}$, then  $\forall(\bm{v}_{h}, q_h )\in  \bm{V}_{h}^{0}\times Q_{h}^{0}$, seek $(\bm{u}_{h}^{k,\widetilde{k}},p^{k,\widetilde{k}})$ with $\widetilde{k}=1,2,...$,   such that
\begin{subequations}\label{Oseen}
\begin{align}
(\frac{\bm{u}_{hi}^{k,\widetilde{k}}}{\Delta t},\bm{v}_{hi} )
+a_{h}(\bm{u}_{h}^{k,\widetilde{k}},\bm{v}_{h})+b_{h}(\bm{v}_{h},p_{h}^{k,\widetilde{k}})
+c_{h}(\bm{u}_{h}^{k,\widetilde{k}-1} ;\bm{u}_{h}^{k,\widetilde{k}},\bm{v}_{h} )%&\nonumber\\
+d_{h}(\bm{u}_{h}^{k,\widetilde{k}-1} ;\bm{u}_{h}^{k,\widetilde{k}},\bm{v}_{h} )
& =(\bm{f }^{k,\widetilde{k}},\bm{v}_{hi})+(\frac{\bm{u}_{hi}^{k-1}}{\Delta t},\bm{v}_{hi} ),\\
	b_{h}(\bm{u}_{h}^{k,\widetilde{k}},q_{h})&=0,
\end{align}
\end{subequations}
where  $\widetilde{k}$ is with regard to the the linearized iteration,
$\bm{u}_{h}^{k-1,\ast}$ denotes the obtained solution from the $(k-1){th}$ time step.

It is not difficult to know that the linear system \eqref{Oseen} is uni-solvent for given $(\bm{u}_{h}^{k,\widetilde{k}-1},p_{h}^{k,\widetilde{k}-1})$, and it holds
\begin{eqnarray}\label{bound-ul}
|||\bm{u}_h^{k,\widetilde{k}}||| _{V}\leq   \gamma,\quad \widetilde{k}=1,2,....
\end{eqnarray}
where $\gamma:=\frac 1 \nu\|\bm{f}^{k}\|_{*,h}+\frac{\|\bm{u}_{hi}^{k-1}\|_0}{\Delta t}.$

We have the following  convergence result.
\begin{theorem}\label{th73}
 For $k=1,..,N$, assume that  $(\bm{u}_{h}^{k},p_{h}^{k})\in \bm{V}_{h}^{0}\times Q_{h}^{0}$ is the solution to \eqref{fullwg}, under the condition
\begin{eqnarray}\label{CON}
2(r-2)C_{\widetilde{r}}^{r}\alpha\nu^{-1} \gamma^{r-2}+\mathcal{N}_{h}\nu^{-1}\gamma<1,
\end{eqnarray}
   the linearized iteration scheme \eqref{Oseen}  is convergent in the following sense:
\begin{align}\label{COV}
\lim_{\widetilde{k}\rightarrow\infty}|||\bm{u}_{h}^{k,\widetilde{k}}-\bm{u}_{h}^{k} |||_{V}=0,
\lim_{\widetilde{k}\rightarrow\infty}|||p_{h}^{k,\widetilde{k}}-p_{h}^{k} |||_{Q}=0.
\end{align}
\end{theorem}
\begin{proof}
Denote $e_{u}^{k,\widetilde{k}}=\{e_{u_{i}}^{k,\widetilde{k}},e_{u_{b}}^{\widetilde{k}} \}$,	 where $e_{u}^{k,\widetilde{k}}=\bm{u}_{h}^{k,\widetilde{k}}-\bm{u}_{h}^{k}$ and $e_{p}^{k,\widetilde{k}}=p_{h}^{k,\widetilde{k}}-p_{h}^{k}$.
Subtracting \eqref{fullwg} from \eqref{Oseen}, for any $(\bm{v}_{h}, q_h )\in  \bm{V}_{h}^{0}\times Q_{h}^{0}$, we obtain
\begin{subequations}\label{7.12}
\begin{align}
(\frac{e_{ui}^{k,\widetilde{k}}}{\Delta t},\bm{v}_{hi} )
+a_{h}(e_{u}^{k,\widetilde{k}},\bm{v}_{h})
+b_{h}(\bm{v}_{h},e_{p}^{k,\widetilde{k}})-
c_{h}(\bm{u}_{h}^{k} ;\bm{u}_{h}^{k},\bm{v}_{h})
+c_{h}(\bm{u}_{h}^{k,\widetilde{k}-1} ;\bm{u}_{h}^{k,\widetilde{k}},\bm{v}_{h})&\nonumber\\
-d_{h}(\bm{u}_{h}^{k} ;\bm{u}_{h}^{k},\bm{v}_{h})
+d_{h}(\bm{u}_{h}^{k,\widetilde{k}-1} ;\bm{u}_{h}^{k,\widetilde{k}},\bm{v}_{h} )&=0,\label{7.12a}\\
b_{h}(e_{u}^{k,\widetilde{k}},q_{h})&=0.\label{7.12b}
\end{align}
\end{subequations}
Taking $\bm{v}_{h}=e_{u}^{k,\widetilde{k}}$, $q_{h}=e_{p}^{k,\widetilde{k}}$ in \eqref{7.12} and  utilizing the definition of $c_{h}(\cdot;\cdot,\cdot)$ and the fact that $d_{h}(\bm{u}_{h} ;e_{u}^{k,\widetilde{k}},e_{u}^{k,\widetilde{k}})=0$, we gain
\begin{align}
 &\frac{1}{\Delta t}\|e_{ui}^{k,\widetilde{k}}\|_0^{2}
 +\nu|||e_{u}^{k,\widetilde{k}}|||^{2}_{V}\nonumber\\
 =&c_{h}(\bm{u}_{h}^{k} ;\bm{u}_{h}^{k},e_{u}^{k,\widetilde{k}})-c_{h}(\bm{u}_{h}^{k,\widetilde{k}-1} ;\bm{u}_{h}^{k,\widetilde{k}},e_{u}^{k,\widetilde{k}} )
+ d_{h}(\bm{u}_{h}^{k} ;\bm{u}_{h}^{k},e_{u}^{k,\widetilde{k}})-d_{h}(\bm{u}_{h}^{\widetilde{k}-1} ;\bm{u}_{h}^{k,\widetilde{k}},e_{u}^{k,\widetilde{k}} )\nonumber\\
 =&c_{h}(\bm{u}_{h}^{k} ;\bm{u}_{h}^{k},e_{u}^{k,\widetilde{k}})
 -c_{h}(\bm{u}_{h}^{k,\widetilde{k}-1} ;\bm{u}_{h}^{k},e_{u}^{k,\widetilde{k}} )-c_{h}(\bm{u}_{h}^{k,\widetilde{k}-1};e_{u}^{k,\widetilde{k}},e_{u}^{k,\widetilde{k}})
 -d_{h}(e_{u}^{k,\widetilde{k}-1} ;\bm{u}_{h}^{k},e_{u}^{k,\widetilde{k}})\nonumber\\
 =&\alpha ((|\bm{u}_{h}^{k} |^{r-2}- |\bm{u}_{h}^{k,\widetilde{k}-1} |^{r-2}) \bm{u}_{h}^{k}
, e_{u}^{k,\widetilde{k}})
 -\alpha (|\bm{u}_{h}^{k} |^{r-2}e_{u}^{k,\widetilde{k}},e_{u}^{k,\widetilde{k}} )-d_{h}(e_{u}^{k,\widetilde{k}-1} ;\bm{u}_{h}^{k},e_{u}^{k,\widetilde{k}}),
\end{align}
which, together with  Lemmas \ref{Lemma 2.4}, \ref{Lemma 2.9}, \ref{Lemma 3.1}, \eqref{bound-ul} and the fact $-\alpha (|\bm{u}_{h}^{k} |^{r-2}e_{u}^{k,\widetilde{k}},e_{u}^{k,\widetilde{k}} )\leq 0,$ derives
\begin{align}\label{7.13}
 &\nu|||e_{u}^{k,\widetilde{k}}|||^{2}_{V}\nonumber\\
  \leq&\alpha ((|\bm{u}_{h}^{k} |^{r-3}+|\bm{u}_{h}^{k,\widetilde{k}-1} |^{r-3})e_{u}^{k,\widetilde{k}-1}  \bm{u}_{h}^{k}, e_{u}^{k,\widetilde{k}}) )-d_{h}(e_{u}^{k,\widetilde{k}-1} ;\bm{u}_{h}^{k},e_{u}^{k,\widetilde{k}})\nonumber\\
  \leq&\big( (r-2)C_{\widetilde{r}}^{r}\alpha (|||\bm{u}_{h}^{k}|||^{r-3}_{V}+|||\bm{u}_{h}^{k,\widetilde{k}-1}|||^{r-3}_{V} )+ \mathcal{N}_{h}\big)||| \bm{u}_{h}^{k}|||_{V}\cdot|||e_{u}^{k,\widetilde{k}-1} |||_{V}\cdot
||| e_{u}^{k,\widetilde{k}}|||_{V}\nonumber\\
\leq& (2(r-2)C_{r}^{r}\alpha \gamma^{r-2}+\mathcal{N}_{h}\gamma)|||e_{u}^{k,\widetilde{k}-1} |||_{V}\cdot
||| e_{u}^{k,\widetilde{k}}|||_{V}.
\end{align}
Thus, we have
\begin{align}
|||e_{u}^{k,\widetilde{k}}|||_{V}
 \leq \mathcal{M}|||e_{u}^{k,\widetilde{k}-1}|||_{V}%\nonumber\\
  \leq...%\nonumber\\
  \leq \mathcal{M}^{\widetilde{k}}|||e_{u}^{k,0}|||_{V},
\end{align}
where $\mathcal{M}= 2(r-2)C_{\widetilde{r}}^{r}\alpha\nu^{-1} \gamma^{r-2}+\mathcal{N}_{h}\nu^{-1}\gamma$.
In view of \eqref{CON}, we know that $\mathcal{M}<1$.
Therefore, we obtain
\begin{align}\label{7.16}
\lim_{\widetilde{k}\rightarrow\infty} |||e_{u}^{k,\widetilde{k}}|||_{V}=\lim_{\widetilde{k}\rightarrow\infty} |||\bm{u}_{h}^{k,\widetilde{k}}-\bm{u}_{h}^{k} |||_{V}=0.
\end{align}
Next we shall derive the second convergence relation of \eqref{COV}.
From \eqref{7.12a}, it follows
\begin{align}\label{7.17}
&b_{h}(\bm{v}_{h},e^{k,\widetilde{k}}_{p})\nonumber\\
=&-(\frac{e_{ui}^{k,\widetilde{k}}}{\Delta t},\bm{v}_{hi} )-a_{h}(e_{u}^{k,\widetilde{k}},\bm{v}_{h})
+c_{h}(\bm{u}_{h}^{k} ;\bm{u}_{h}^{k},\bm{v}_{h})-c_{h}(\bm{u}_{h}^{k,\widetilde{k}-1} ;\bm{u}_{h}^{k,\widetilde{k}},\bm{v}_{h} )%\nonumber\\&
+d_{h}(\bm{u}_{h}^{k} ;\bm{u}_{h}^{k},\bm{v}_{h})
-d_{h}(\bm{u}_{h}^{k,\widetilde{k}-1} ;\bm{u}_{h}^{k,\widetilde{k}},\bm{v}_{h})\nonumber\\
=&-(\frac{e_{ui}^{k,\widetilde{k}}}{\Delta t},\bm{v}_{hi} )
-a_{h}(e_{u}^{k,\widetilde{k}},\bm{v}_{h})+\big(c_{h}(\bm{u}_{h}^{k} ;\bm{u}_{h}^{k},e_{u}^{k,\widetilde{k}})
 -c_{h}(\bm{u}_{h}^{k,\widetilde{k}-1} ;\bm{u}_{h}^{k},e_{u}^{k,\widetilde{k}} )\big)\nonumber\\
&-c_{h}(\bm{u}_{h}^{k,\widetilde{k}-1};e_{u}^{k,\widetilde{k}},e_{u}^{k,\widetilde{k}})-\big(d_h(\bm{u}_{h}^{k} ;e_{u}^{k,\widetilde{k}},\bm{v}_{h} )+d_h(e_{u}^{k,\widetilde{k}-1} ;e_{u}^{k,\widetilde{k}},\bm{v}_{h} )+d_h(e_{u}^{k,\widetilde{k}-1};\bm{u}_{h}^{k},\bm{v}_{h}) \big),
\end{align}
for all $\bm{v}_{h}\in \bm{V}_{h}^{0}$.
By  Lemma \ref{theoremLBB}, we gain
\begin{align}\label{7.18}
|||e^{k,\widetilde{k}}_{p}|||_{Q}
  \lesssim&\sup_{ \bm{v}_{h}\in \bm{V}_{h}^{0}}\frac{b_{h}(\bm{v}_{h},e^{k,\widetilde{k}}_{p})}{|||\bm{v}_{h}|||_{V} }\nonumber\\
 \leq& \frac{|||e_{ui}^{k,\widetilde{k}}|||_{V} ^{2}}{\triangle t}
 +\nu |||e_{u}^{k,\widetilde{k}}|||_{V} +C_{r}C_{\widetilde{r}}^{r}\alpha (|||\bm{u}_{h}^{k}|||^{r-3}_{V}%\nonumber\\&
+|||\bm{u}_{h}^{k,\widetilde{k}-1}|||^{r-3}_{V} )|||\bm{u}_{h}^{k}|||_{V}\cdot|||e_{u}^{k,\widetilde{k}-1}|||_{V} \nonumber\\& +C_{\widetilde{r}}^{r}\alpha|||\bm{u}_{h}^{k,\widetilde{k}-1}|||^{r-2}_{V} \cdot |||e_{u}^{k,\widetilde{k}}|||^{2}_{V}%\nonumber\\&
+\mathcal{N}_{h} (|||\bm{u}_{h}^{k}|||_{V}\cdot |||e_{u}^{k,\widetilde{k}}|||_{V}
+|||e_{u}^{k,\widetilde{k}-1}|||_{V}\cdot |||e_{u}^{k,\widetilde{k}}|||_{V}\nonumber\\&
+|||e_{u}^{k,\widetilde{k}-1}|||_{V} \cdot|||\bm{u}_{h}^{k}|||_{V}).
\end{align}
Combining results \eqref{7.17}, \eqref{7.18}, \eqref{7.16}, Lemmas  \ref{Lemma 2.4}, \ref{Lemma 2.9} and \ref{Lemma 3.1}, we have
\begin{align*}
\lim_{\widetilde{k}\rightarrow\infty} |||e^{k,\widetilde{k}}_{p}|||_{Q}
=\lim_{\widetilde{k}\rightarrow\infty} |||p_{h}^{k,\widetilde{k}}-p_{h}^{k}|||_{Q}=0.
\end{align*}
This completes this proof.

\end{proof}

\section{Numerical experiments}

\color{black}
In this section, we provide some numerical tests to verify the performance of the full discrete WG scheme  \eqref{fullwg} for the Brinkman-Forchheimer model \eqref{BF0} in two dimensions.
We adopt the linearized iterative algorithm \eqref{Oseen} with the initial guess $\bm{u}_{hi}^{0}=\bm{0}$ and the stopping criterion
\begin{equation*}
\|\bm{u}_{h}^{k,\widetilde{k}}-\bm{u}_{h}^{k,\widetilde{k}-1}\|_{0}<1e-8.
\end{equation*}
in all the  numerical examples.

\begin{exam} \label{EX7.1}
 Set $\Omega=  [0,1]\times[0,1]$, $\nu=1$,  $\alpha=1$, $r=5$ and $T=1$ in model \eqref{BF0}.
 The force term $\bm{f}$ is chosen such that the exact solution $(\bm{u},p)$ is as follows:
\begin{equation}
\left\{\begin{array}{ll}
u_{1}=5x^{2}(x-1)^{2}y(y-1)(2y-1)cos(t),   &\text{  in  } \Omega\times [0,T],     \\
u_{2}=-5x(x-1)(2x-1)y^{2}(y-1)^{2}cos(t) ,  &\text{  in  } \Omega\times [0,T],\\
p=10(2x-1)(2y-1)cos(t),        &\text{ in } \Omega\times [0,T].
\end{array}\right.
\end{equation}

We compute the scheme \eqref{fullwg} on uniform triangular meshes  (cf. Figure \ref{fig1:mesh}), with  $m=1,2$, $l=m-1,m$.
In order to confirm the convergence rates, according to  Theorem \ref{Theorem 7.2}, we choose $\triangle t=h^{2}$  for $m=1$, and $\triangle t=h^{3}$  for $m=2$.
Numerical results of  $\|\bm{u}-\bm{u}_{hi}\|_0$, $\|\nabla \bm{u}-\nabla_{h}\bm{u}_{hi}\|_0$, $\|p-p_{hi}\|_0$ and $\| \nabla_{h}\cdot\bm{u}_{hi} \|_{\infty}$ at the final time $T=1$ are listed in Tables \ref{tab1} and \ref{tab2}.

\begin{figure}[!t]
		  \centering
	 \includegraphics[width= 2.7in]{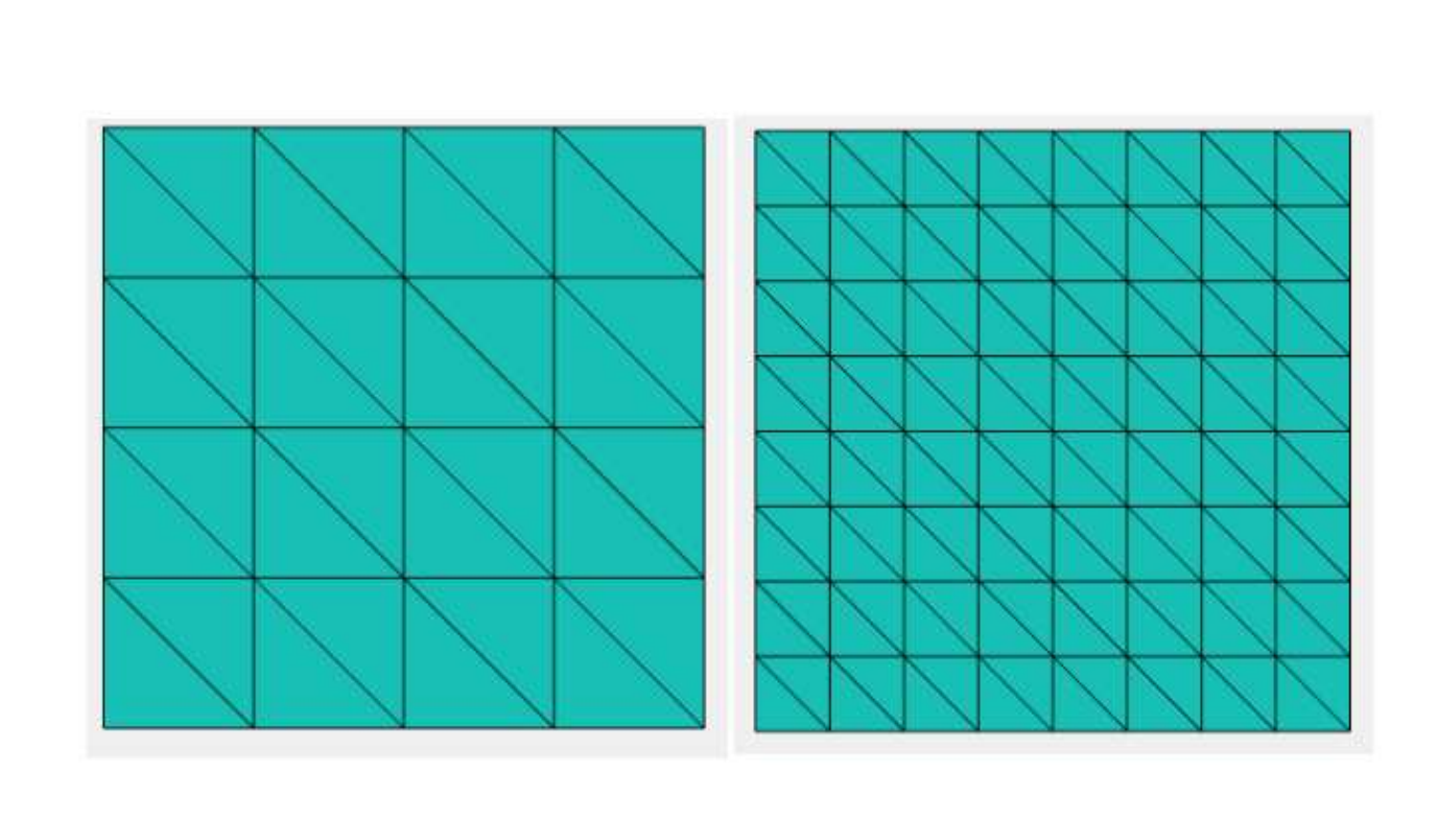}\\
		   \caption{Uniform triangular meshes:    $4\times4$ mesh (left)      and $8\times 8 $   mesh (right).}
		\label{fig1:mesh}
		\end{figure}
\end{exam}

\begin{table}[h]
	\small
	\caption{\label{tab1}
		History of convergence results at  time $T=1$  for Example \ref{EX7.1}:  $m=1$
		  }
	\centering
	
	\begin{tabular}{c|c|c|c|c|c|c|c|c}
		\Xhline{1pt}
		
		\multirow{2}{*}{$l$} &
		\multirow{2}{*}{$mesh$} &
		
		\multicolumn{2}{c|}{$\frac{\|\bm{u}-\bm{u}_{hi}\|_0}{\|\bm{u}\|_0}$} &
		\multicolumn{2}{c|}{$\frac{\|\nabla \bm{u}-\nabla_{h}\bm{u}_{hi}\|_0}{\|\nabla \bm{u}\|_0}$} &
		
		\multicolumn{2}{c| }{$\frac{\|p-p_{hi}\|_0}{\|p\|_0}$} &
		\multirow{2}{*}{$\|\nabla_{h}\cdot\bm{u}_{hi}\|_{0,\infty}$
		} \\
		\cline{3-8}
		
		& &Error &Rate  &Error &Rate  &Error &Rate  \\
		\hline
		
		\multirow{5}{*}{$0$}

&$4\times4$     &6.1174e-01   &   -   &5.1756e-01   &  -    &2.8647e-01   &    -  &2.6021e-17\\
&$8\times8$     &1.6218e-01   &1.92   &2.7342e-01   &0.92   &1.4410e-01   &0.99   &1.3531e-16\\
&$16\times16$   &4.2348e-02   &1.94   &1.3857e-01   &0.98   &7.2150e-02   &1.00   &1.2739e-17\\
&$32\times32$   &1.0847e-02   &1.97   &6.9428e-02   &1.00   &3.6085e-02   &1.00   &3.8015e-18\\
&$64\times64$   &2.7506e-03   &1.98   &3.4725e-02   &1.00   &1.8043e-02   &1.00   &5.8750e-18\\

		\Xhline{1pt}
		\multirow{5}{*}{$1$}
&$4\times4$     &5.8156e-01   &   -   &5.1353e-01   &  -    &2.8648e-01   &    -  &1.1189e-16\\
&$8\times8$     &1.5539e-01   &1.90   &2.7269e-01   &0.91   &1.4410e-01   &0.99   &7.2316e-17\\
&$16\times16$   &4.0683e-02   &1.93   &1.3845e-01   &0.98   &7.2153e-02   &1.00   &1.4637e-18\\
&$32\times32$   &1.0430e-02   &1.96   &6.9410e-02   &1.00   &3.6087e-02   &1.00   &9.0124e-19\\
&$64\times64$   &2.6448e-03   &1.98   &3.4722e-02   &1.00   &1.8045e-02   &1.00   &2.1244e-17\\

		\Xhline{1pt}
	\end{tabular}
	
\end{table}

\begin{table}[h]
	\small
	\caption{\label{tab2}
		History of convergence results at  time $T=1$  for Example \ref{EX7.1}:  $m=2$
		  }
	\centering
	
	\begin{tabular}{c|c|c|c|c|c|c|c|c}
		\Xhline{1pt}
		
		\multirow{2}{*}{$l$} &
		\multirow{2}{*}{$mesh$} &
		
		\multicolumn{2}{c|}{$\frac{\|\bm{u}-\bm{u}_{hi}\|_0}{\|\bm{u}\|_0}$} &
		\multicolumn{2}{c|}{$\frac{\|\nabla \bm{u}-\nabla_{h}\bm{u}_{hi}\|_0}{\|\nabla \bm{u}\|_0}$}&
		
		\multicolumn{2}{c| }{$\frac{\|p-p_{hi}\|_0}{\|p\|_0}$} &
		\multirow{2}{*}{$\|\nabla_{h}\cdot\bm{u}_{hi}\|_{0,\infty}$
		} \\
		\cline{3-8}
		
		& &Error &Rate  &Error &Rate  &Error &Rate  \\
		\hline
		
		\multirow{5}{*}{$1$}
&$2\times2$    &4.5784e-01   & -     &4.5504e-01   & -     &1.3276e-01   & -    &8.8753e-16\\
&$4\times4$    &6.0385e-02   &2.92   &1.3051e-01   &1.80   &3.3137e-02   &2.00  &3.9893e-16\\
&$8\times8$    &7.6249e-03   &2.99   &3.4969e-02   &1.90   &8.2790e-03   &2.00  &3.1464e-16\\
&$16\times16$  &9.5935e-04   &2.99   &8.9649e-03   &1.96   &2.0694e-03   &2.00  &3.2339e-16\\
&$32\times32$  &1.2103e-04   &2.99   &2.2716e-03   &1.98   &5.1763e-04   &2.00  &1.6534e-18\\

		\Xhline{1pt}
		\multirow{5}{*}{$2$}
&$2\times2$    &4.1899e-01   &   -   &4.5586e-01   &  -    &1.3275e-01   & -     &3.2917e-15\\
&$4\times4$    &5.7284e-02   &2.87   &1.3023e-01   &1.81   &3.3124e-02   &2.00   &6.7750e-15\\
&$8\times8$    &7.3938e-03   &2.95   &3.4901e-02   &1.90   &8.2753e-03   &2.00   &3.6799e-17\\
&$16\times16$  &9.4361e-04   &2.97   &8.9527e-03   &1.96   &2.0684e-03   &2.00   &1.0169e-15\\
&$32\times32$  &1.2018e-04   &2.97   &2.2692e-03   &1.98   &5.1739e-04   &2.00   &3.1390e-15\\

		\Xhline{1pt}
	\end{tabular}
	
\end{table}

\begin{exam} \label{EX7.2}
 Set $\Omega=  [0,1]\times[0,1]$, $\nu=1$,  $\alpha=0.1$, $r=3.5$ and $T=1$ in the model \eqref{BF0}.
 The exact solution $(\bm{u},p)$ is given as follows:
\begin{equation}\label{EX2}
\left\{\begin{array}{ll}
u_{1}=x^{2}(x-1)^{2}y(y-1)(2y-1)e^{-t},   &\text{  in  } \Omega\times [0,T],     \\
u_{2}=-x(x-1)(2x-1)y^{2}(y-1)^{2}e^{-t},  &\text{  in  } \Omega\times [0,T],\\
p=(x^2-y^2)e^{-t}, &\text{  in  } \Omega\times [0,T],\\
\end{array}\right.
\end{equation}
where $\bm{f}$ is received in the same manner as in Example \ref{EX7.1}.

We compute the scheme \eqref{fullwg} on uniform triangular meshes  (cf. Figure \ref{fig1:mesh}), with  $m=1,2$, $l=m-1,m$.
 Similar to Example \ref{EX7.1}, taking $\triangle t=h^{2}$ for $m=1$ and $\triangle t=h^{3}$  for $m=2$.
 Numerical results of  $\|\bm{u}-\bm{u}_{hi}\|_0$, $\|\nabla \bm{u}-\nabla_{h}\bm{u}_{hi}\|_0$, $\|p-p_{hi}\|_0$ and $\| \nabla_{h}\cdot\bm{u}_{hi} \|_{\infty}$ at the final time $T=1$
are listed in Tables \ref{tab3} and \ref{tab4}.

\begin{table}[h]
	\small
	\caption{\label{tab3}
		History of convergence results at  time $T=1$ for Example \ref{EX7.2}:  $m=1$
		  }
	\centering
	
	\begin{tabular}{c|c|c|c|c|c|c|c|c}
		\Xhline{1pt}
		
		\multirow{2}{*}{$l$} &
		\multirow{2}{*}{$mesh$} &
		
		\multicolumn{2}{c|}{$\frac{\|\bm{u}-\bm{u}_{hi}\|_0}{\|\bm{u}\|_0}$} &
		\multicolumn{2}{c|}{$\frac{\|\nabla \bm{u}-\nabla_{h}\bm{u}_{hi}\|_0}{\|\nabla \bm{u}\|_0}$} &
		
		\multicolumn{2}{c| }{$\frac{\|p-p_{hi}\|_0}{\|p\|_0}$} &
		\multirow{2}{*}{$\|\nabla_{h}\cdot\bm{u}_{hi}\|_{0,\infty}$
		} \\
		\cline{3-8}
		
		& &Error &Rate  &Error &Rate  &Error &Rate\\
		\hline
		
		\multirow{5}{*}{$0$}

&$4\times4$    &7.0944e-01   &-      &5.3583e-01   &-      &2.7375e-01   &-      &5.3126e-18\\
&$8\times8$    &1.7851e-01   &1.99   &2.7573e-01   &0.96   &1.3470e-01   &1.02   &5.4210e-19\\
&$16\times16$  &4.5921e-02   &1.96   &1.3885e-01   &0.99   &6.7044e-02   &1.01   &1.7686e-18\\
&$32\times32$  &1.1702e-02   &1.97   &6.9461e-02   &1.00   &3.3476e-02   &1.00   &4.1234e-18\\
&$64\times64$  &2.9561e-03   &1.99   &3.4727e-02   &1.00   &1.6731e-02   &1.00   &8.7350e-19\\

		\Xhline{1pt}
		\multirow{5}{*}{$1$}
&$4\times4$    &6.7643e-01   &   -   &5.2887e-01   &   -   &2.7369e-01   &  -    &9.2157e-19\\
&$8\times8$    &1.7150e-01   &1.98   &2.7462e-01   &0.95   &1.3469e-01   &1.02   &1.8974e-18\\
&$16\times16$  &4.4232e-02   &1.96   &1.3869e-01   &0.99   &6.7051e-02   &1.01   &3.0493e-19\\
&$32\times32$  &1.1284e-02   &1.97   &6.9438e-02   &1.00   &3.3483e-02   &1.00   &1.2705e-19\\
&$64\times64$  &2.8532e-03   &1.98   &3.4725e-02   &1.00   &1.6735e-02   &1.00   &1.4207e-18\\
		
		\Xhline{1pt}
	\end{tabular}
	
\end{table}

\begin{table}[h]
	\small
	\caption{\label{tab4}
		History of convergence results at  time $T=1$  for Example \ref{EX7.2}:  $m=2$
		  }
	\centering
	
	\begin{tabular}{c|c|c|c|c|c|c|c|c}
		\Xhline{1pt}
		
		\multirow{2}{*}{$l$} &
		\multirow{2}{*}{$mesh$} &
		
		\multicolumn{2}{c|}{$\frac{\|\bm{u}-\bm{u}_{hi}\|_0}{\|\bm{u}\|_0}$} &
		\multicolumn{2}{c|}{$\frac{\|\nabla \bm{u}-\nabla_{h}\bm{u}_{hi}\|_0}{\|\nabla \bm{u}\|_0}$}&
		
		\multicolumn{2}{c| }{$\frac{\|p-p_{hi}\|_0}{\|p\|_0}$} &
		\multirow{2}{*}{$\|\nabla_{h}\cdot\bm{u}_{hi}\|_{0,\infty}$
		} \\
		\cline{3-8}
		
		& &Error &Rate  &Error &Rate  &Error &Rate  \\
		\hline
		
		\multirow{5}{*}{$1$}
&$2\times2$    &4.5670e-01   &   -   &4.5492e-01   &   -   &4.3325e-02   &  -    &9.3838e-17\\
&$4\times4$    &6.0293e-02   &2.92   &1.3048e-01   &1.80   &1.0411e-02   &2.06   &3.9258e-17\\
&$8\times8$    &7.6173e-03   &2.98   &3.4963e-02   &1.90   &2.5611e-03   &2.02   &4.1688e-17\\
&$16\times16$  &9.5073e-04   &3.00   &8.9623e-03   &1.96   &6.3818e-04   &2.00   &2.4364e-17\\
&$32\times32$  &1.2011e-04   &2.98   &2.2663e-03   &1.98   &1.6534e-04   &1.95   &2.6648e-18\\

		\Xhline{1pt}
		\multirow{5}{*}{$2$}
&$2\times2$    &4.1828e-01   &   -   &4.5571e-01   &   -   &4.3230e-02   &   -   &2.3459e-17\\
&$4\times4$    &5.7219e-02   &2.87   &1.3020e-01   &1.81   &1.0311e-02   &2.07   &3.7610e-16\\
&$8\times8$    &7.3882e-03   &2.95   &3.4895e-02   &1.90   &2.5314e-03   &2.03   &1.4674e-16\\
&$16\times16$  &9.3551e-04   &2.98   &8.9503e-03   &1.96   &6.3057e-04   &2.01   &1.4735e-16\\
&$32\times32$  &1.1831e-04   &2.98   &2.2638e-03   &1.98   &1.6351e-04   &1.95   &4.7409e-16\\

		\Xhline{1pt}
	\end{tabular}
	
\end{table}

From the numerical results of the above two tests, we have the following observations:
\begin{itemize}
  \item The convergence rates of $\|\nabla \bm{u}-\nabla_{h}\bm{u}_{hi}\|_0$ and $\|p-p_{hi}\|_0$ for the full discrete WG scheme \eqref{fullwg}  are  $m^{th}$ orders in the cases of $m=1,2$ and $l=m, m-1$.
      Furthermore, the convergence rate of $\|\bm{u}-\bm{u}_{hi}\|_0$ is  $(m+1)^{th}$ order. These are conformable to the theoretical results in Theorem \ref{Theorem 7.2}.

  \item The results of $\|\nabla_{h}\cdot\bm{u}_{hi}\|_{\infty} $ are almost zero.
       This implies that the discrete velocity is globally divergence-free, which is consistent with   $\bm{Remark}$ \ref{remark}.
\end{itemize}
\color{black}
\end{exam}

\begin{exam}[The lid-driven cavity flow problem]\label{EX7.3}
This problem is used to test the influence of the two parameters $\alpha$  and $r$ on the solution of the fully discrete WG scheme. Take   $\Omega$ =$[0, 1]\times[0, 1]$, $\nu=0.1$  and $\bm{f}=\bm{0}$. The   boundary conditions are as follows:
$$\bm{u}|_{x=0}=\bm{u}|_{x=1}=\bm{u}|_{y=0}=\bm{0}, \quad \bm{u}|_{y=1}=(1,0)^T.$$ % u_{1}=1, u_{2}=0
We compute the fully discrete WG scheme \eqref{fullwg}  with  $m=l=2$ on the $20\times 20$ uniform triangular mesh  (cf. Figure \ref{fig1:mesh})
in the following cases:
\begin{itemize}
\item [ I]. $\alpha=0$, i.e. the case of the Navier-Stokes  equations;

\item [ II].   $r=3$ and $\alpha=1, 5, 50$;

\item [ III]. $\alpha=1$ and $r=3, 5, 10$.
\end{itemize}
The     velocity streamlines and the   pressure contours at time $T=0.5$  are displayed in Figures \ref{fig21:21} - \ref{fig233:233}.
  As a comparison,  the  referenced  numerical solutions obtained  with  the Taylor-Hood element
  are also shown   for  $\alpha=0$ at time $T=0.5$; see (a) and (c) in Figure \ref{fig21:21}.

From Figures \ref{fig21:21} - \ref{fig233:233}, we have the following observations:
\begin{itemize}
  \item From  Figure \ref{fig223:223}, we can see that the shape and   size  of the vortex  change evidently, which means that  the damping effect becomes greater for the velocity  as the damping parameter  $\alpha$ increases.   We can also  see that the  pressure approximation is not significantly affected by $\alpha$.
  \item In addition,  as shown in Figure  \ref{fig233:233},  the velocity  and pressure approximations are not  significantly effected by the number  $r$.
\end{itemize}

 \end{exam}

\begin{figure}[htbp!]%\label{Fig1}
\centering
\subfigure[ velocity (Taylor-Hood)     ]
{\includegraphics[width=3.7cm]{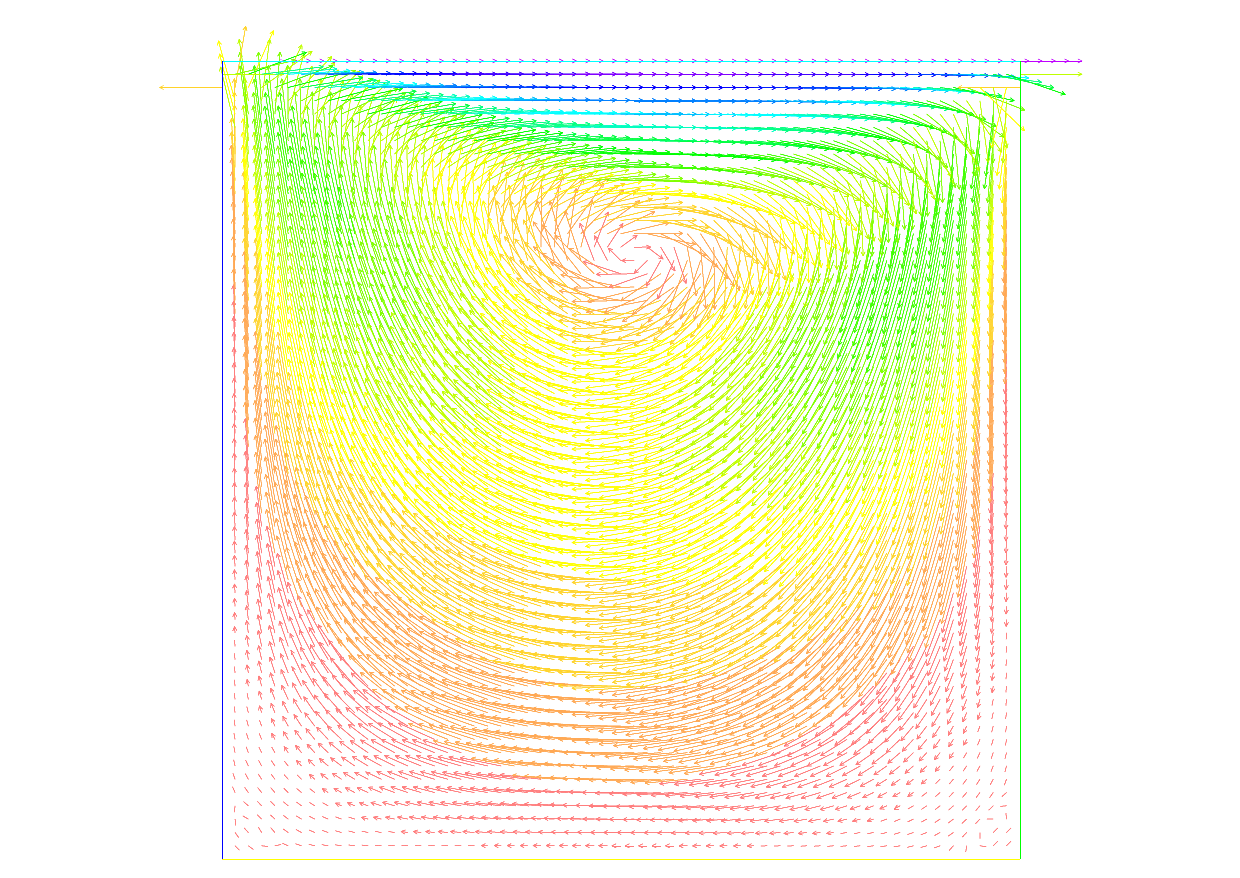}}
\quad
\subfigure[  velocity (WG) ]
{\includegraphics[width=3.7cm]{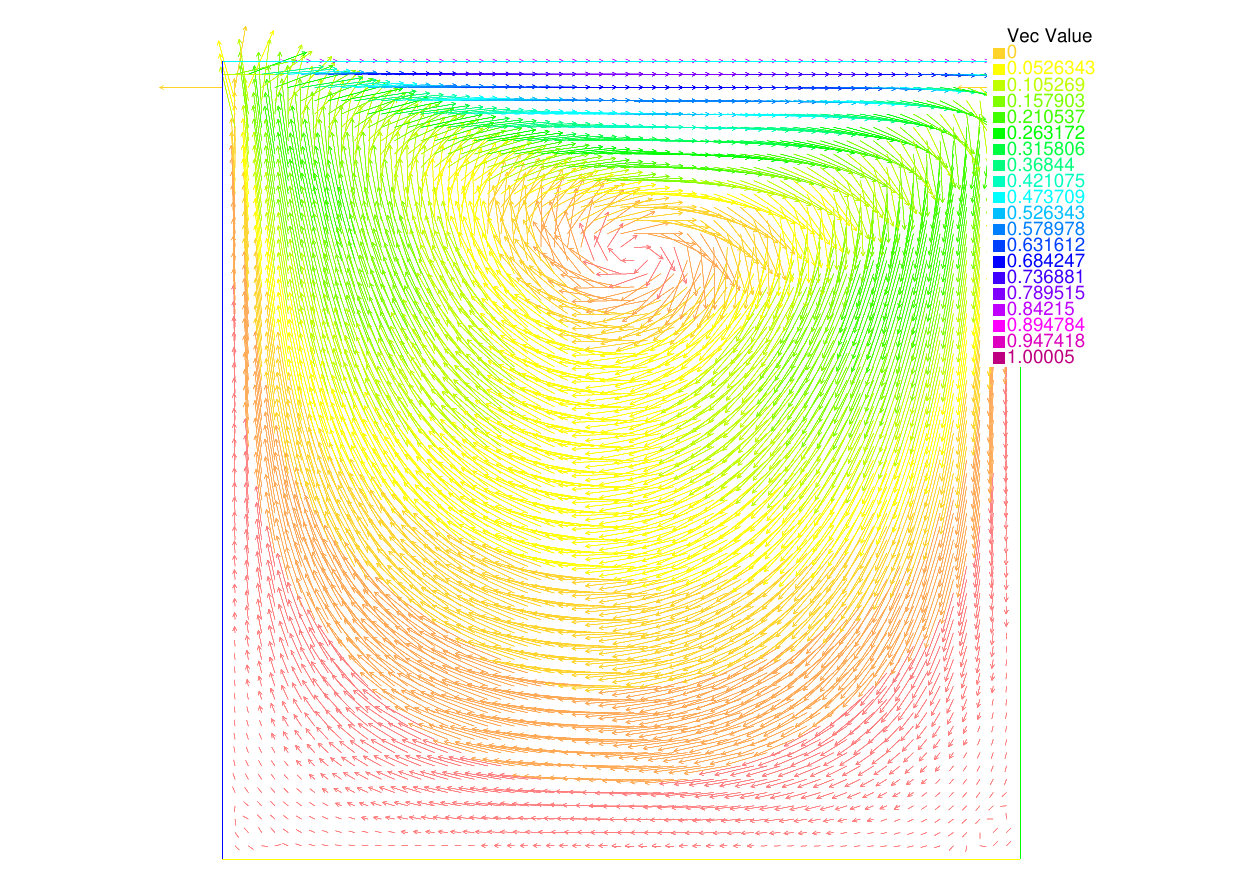}}%\\
\quad
\subfigure[ pressure  (Taylor-Hood)  ]
{\includegraphics[width=3.7cm]{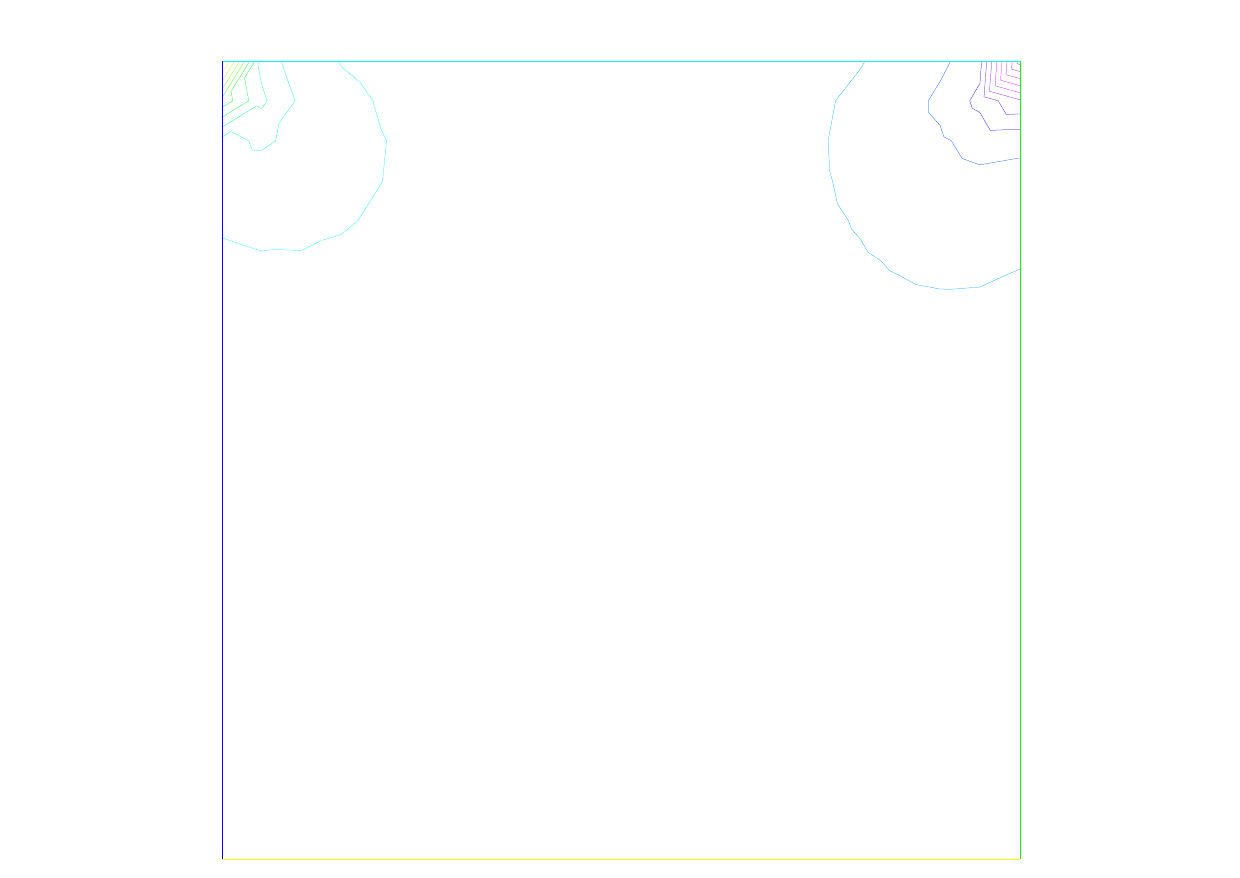}}
\quad
\subfigure[  pressure (WG) ]
{\includegraphics[width=3.7cm]{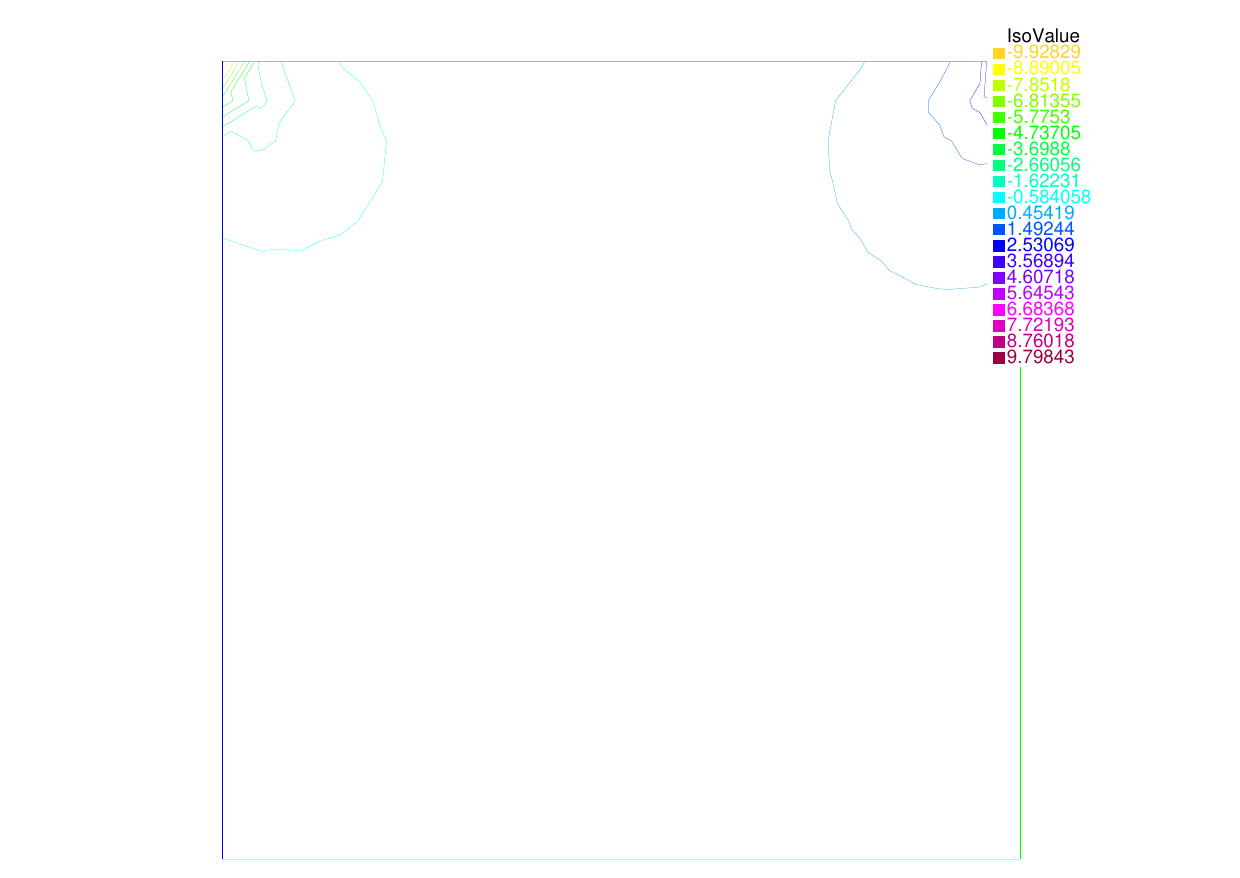}}%\\%\\\\%\\
\quad
\caption{ The velocity streamlines  and pressure contours for Example \ref{EX7.3}: $\alpha=0$ at $T=0.5$}
\label{fig21:21}
\end{figure}

\begin{figure}[htbp!]%\label{Fig1}
\centering
\subfigure[ velocity: $\alpha=1$]
{\includegraphics[width=5 cm]{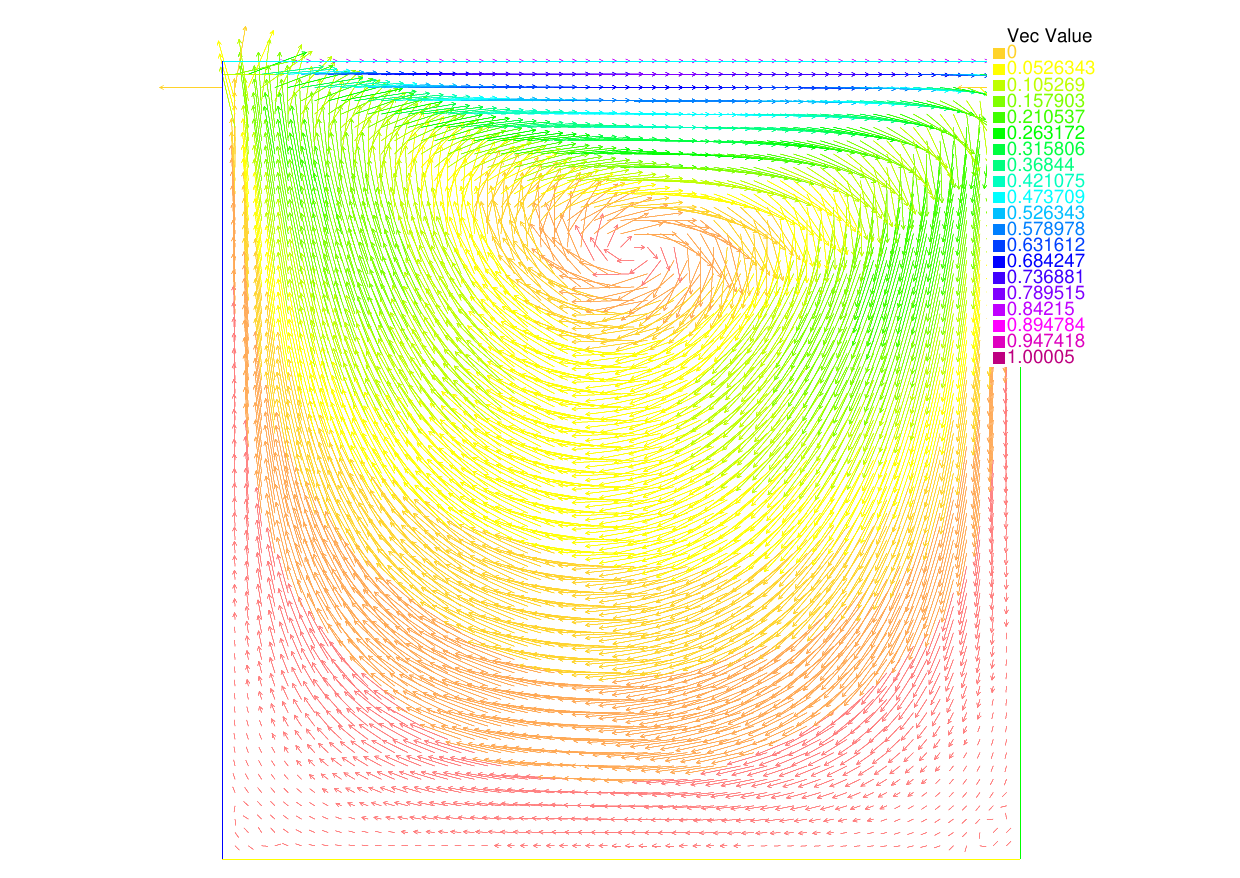}}
%\quad
\subfigure[ velocity: $\alpha=5 $]
{\includegraphics[width=5 cm]{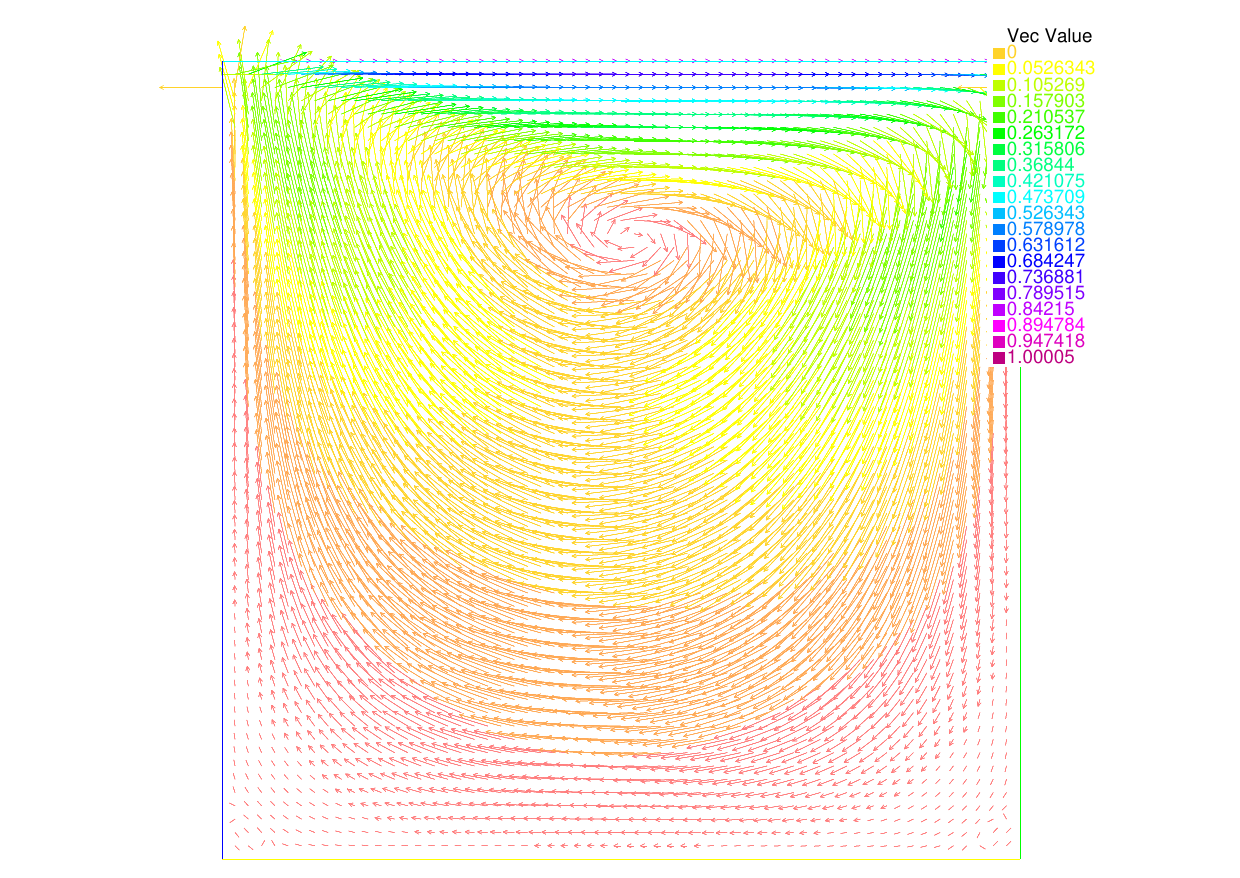}}%
%\quad
\subfigure[ velocity: $\alpha=50 $]
{\includegraphics[width=5 cm]{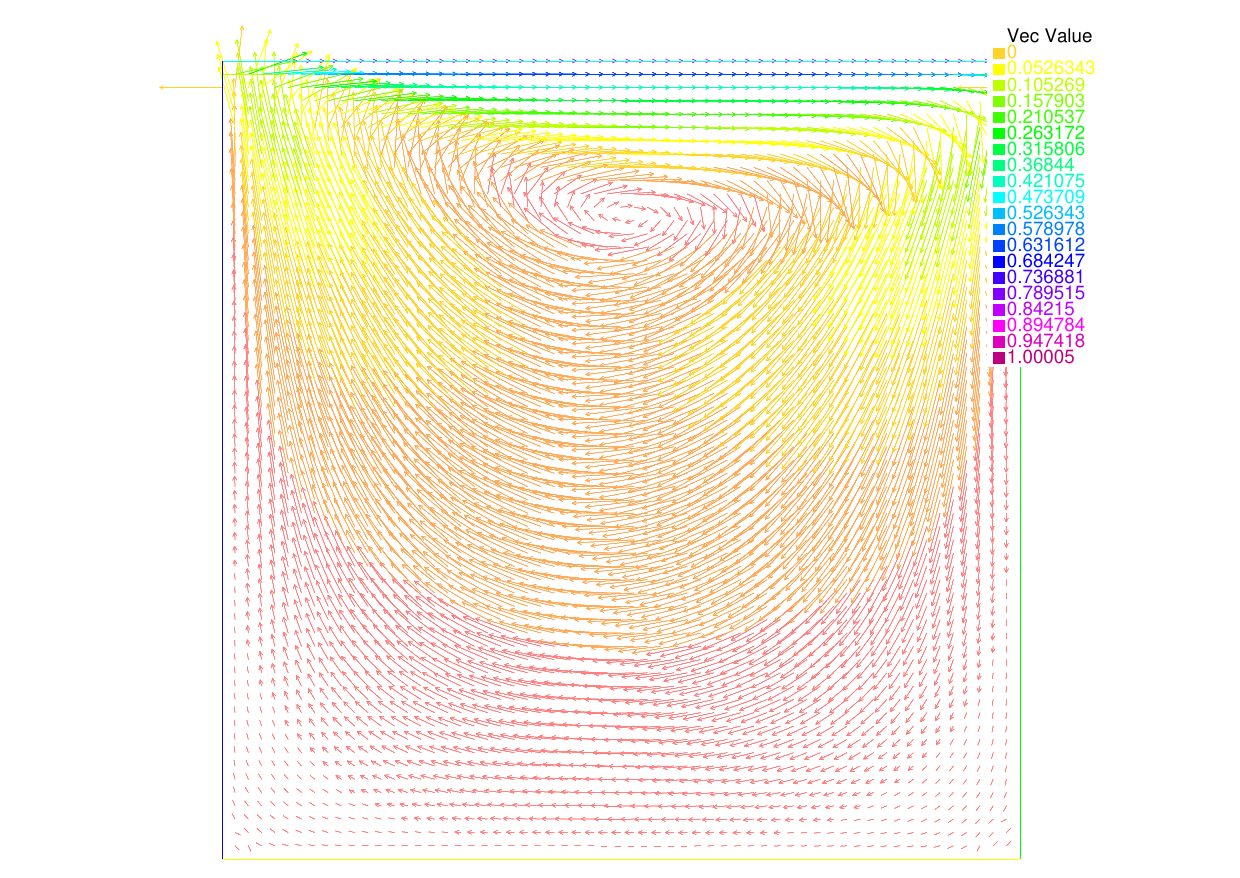}}\\
\subfigure[  pressure: $\alpha=1 $]
{\includegraphics[width=5 cm]{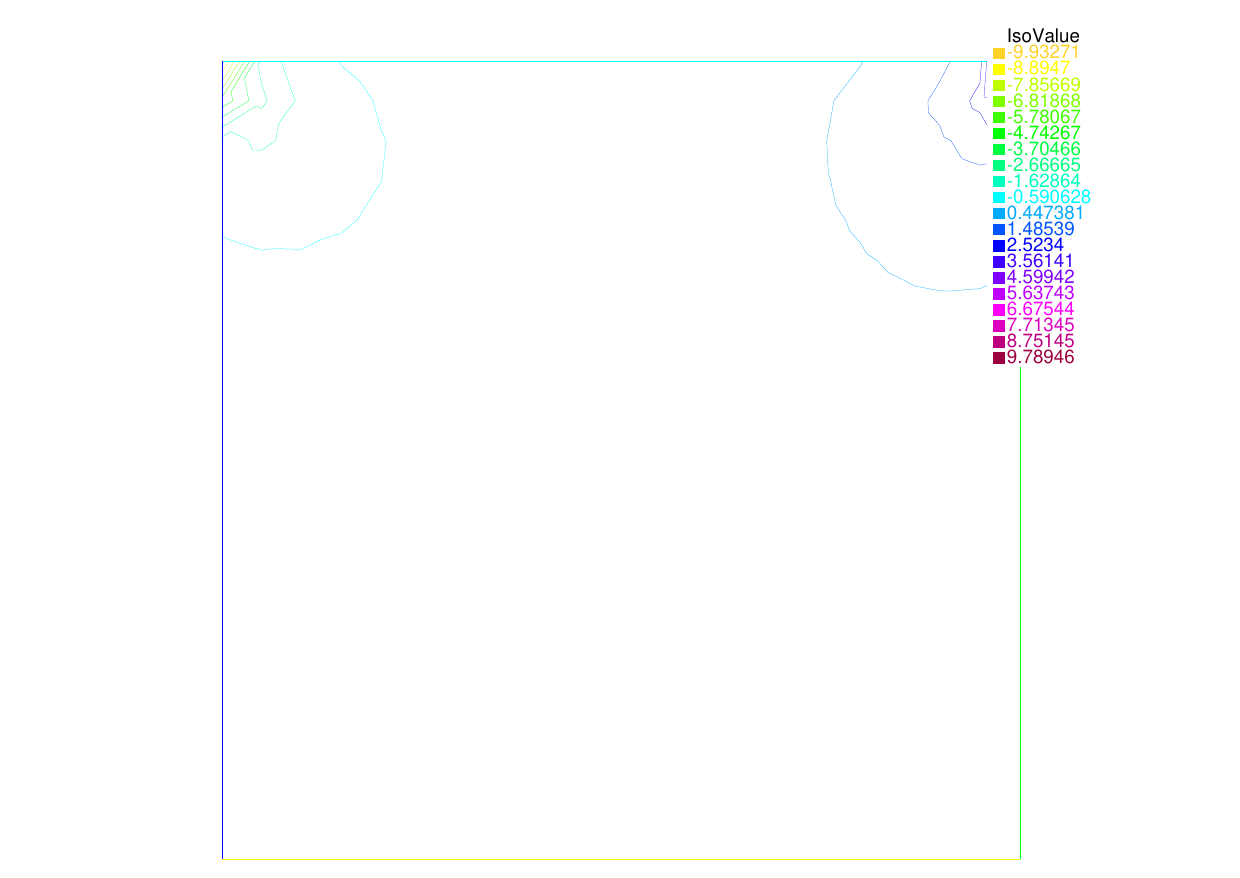}}
\subfigure[  pressure: $\alpha=5$]
{\includegraphics[width=5 cm]{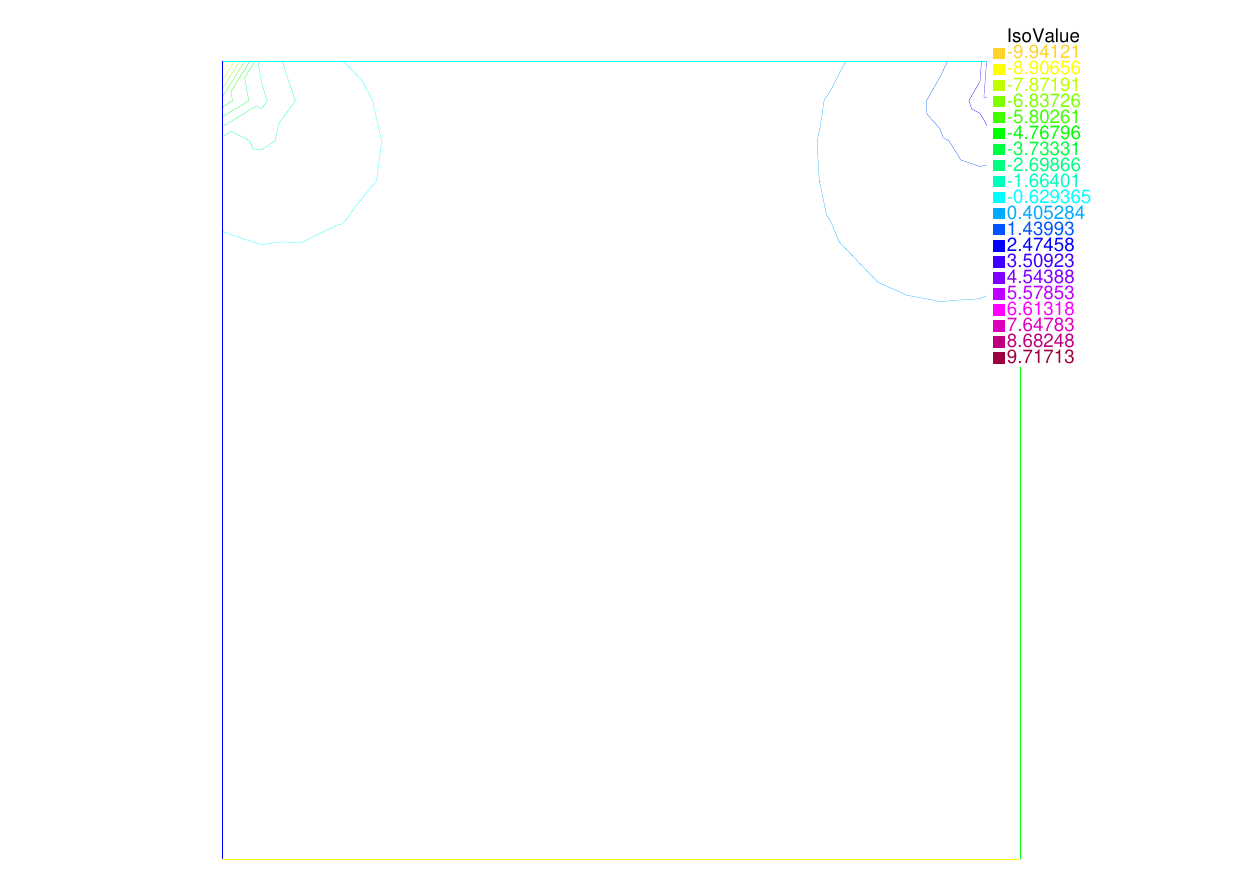}}
\subfigure[  pressure: $\alpha=50 $]
{\includegraphics[width=5 cm]{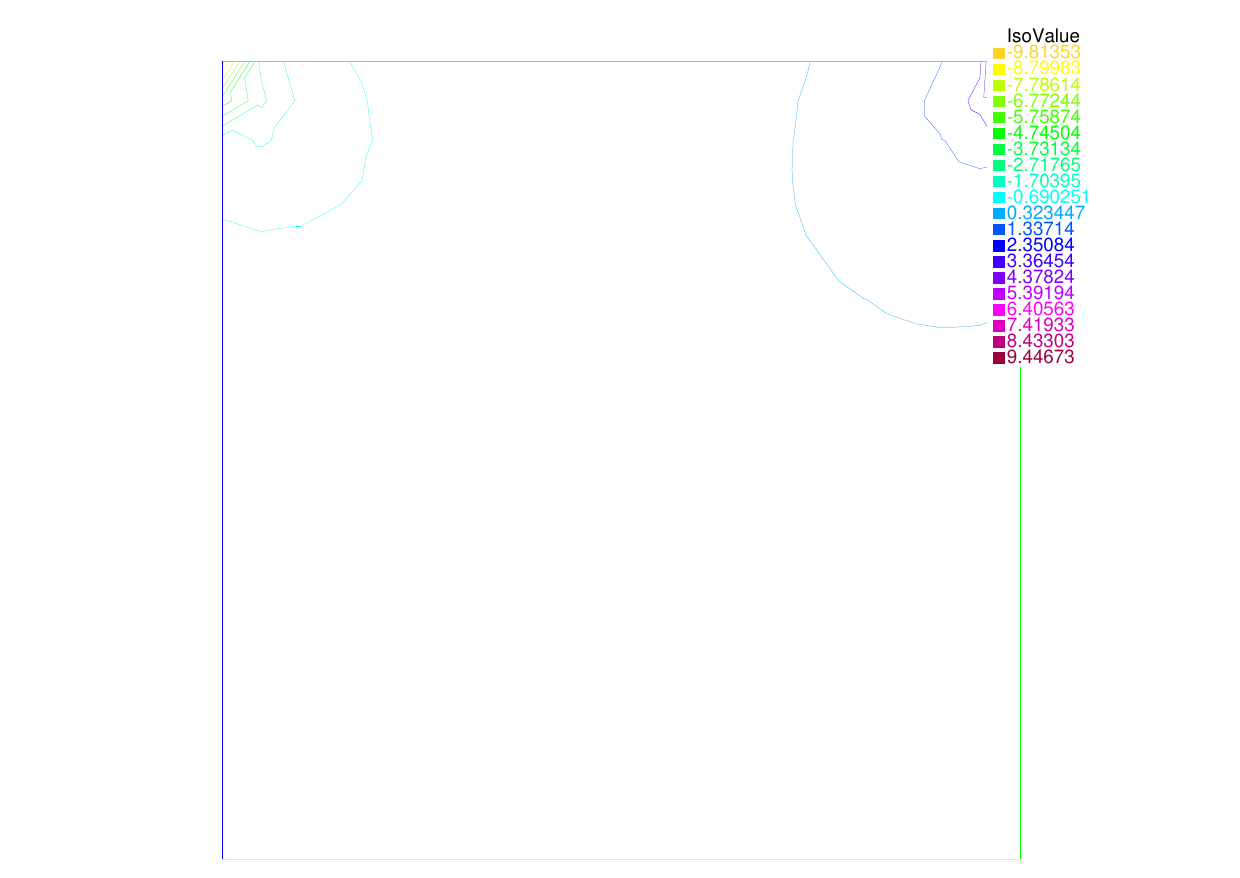}}
\caption{The   velocity streamlines and pressure contours  for Example \ref{EX7.3}:  $r=3$ and   $\alpha=1,5, 50$ at $T=0.5$}
\label{fig223:223}
\end{figure}

\begin{figure}[htbp!]%\label{Fig1}
\centering
\subfigure[velocity: $r=3$]
{\includegraphics[width=5cm]{u1305.eps}}
\subfigure[velocity: $r=5$]
{\includegraphics[width=5cm]{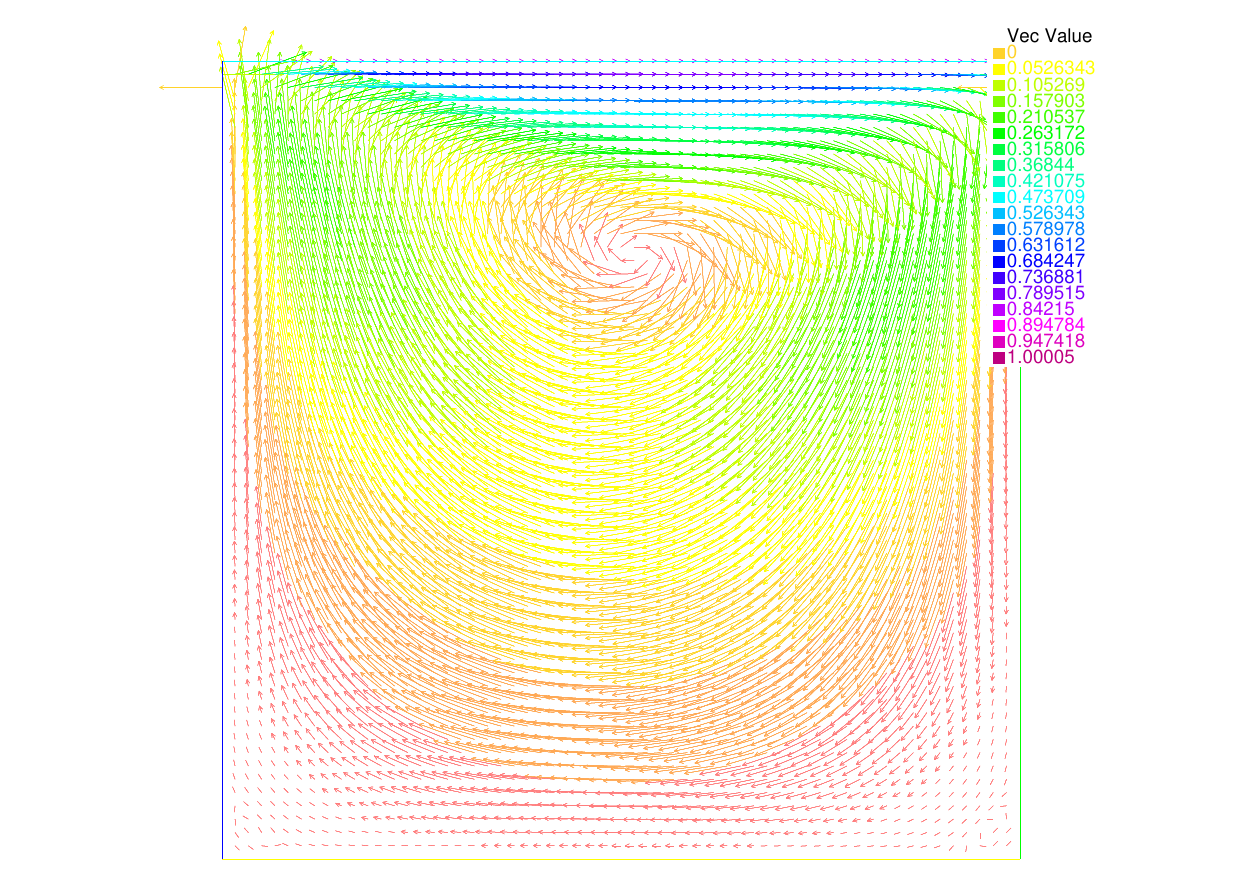}}%\\
\subfigure[velocity: $r=10$]
{\includegraphics[width=5cm]{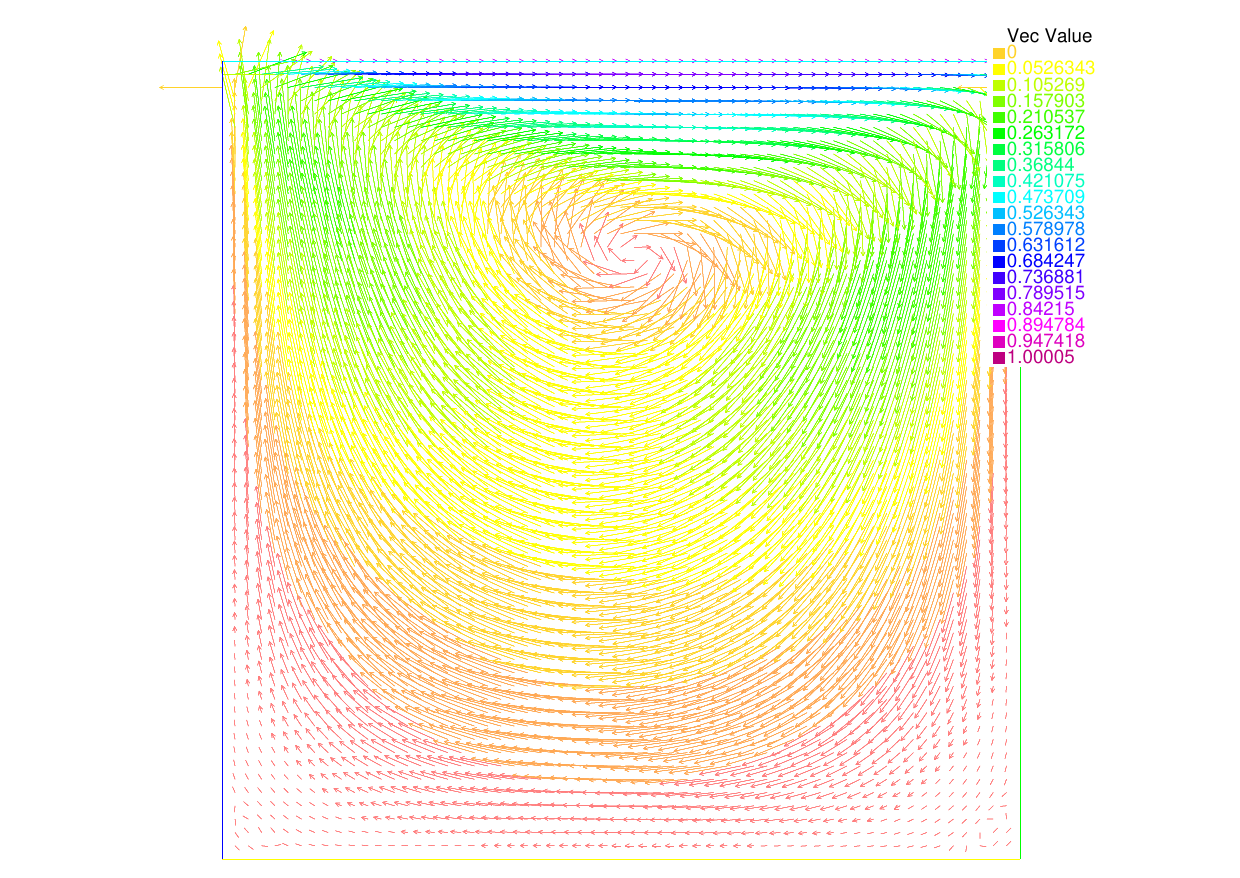}}
\quad
\subfigure[pressure: $r=3$]
{\includegraphics[width=5cm]{p1305.eps}}
\subfigure[ pressure: $r=5$]
{\includegraphics[width=5cm]{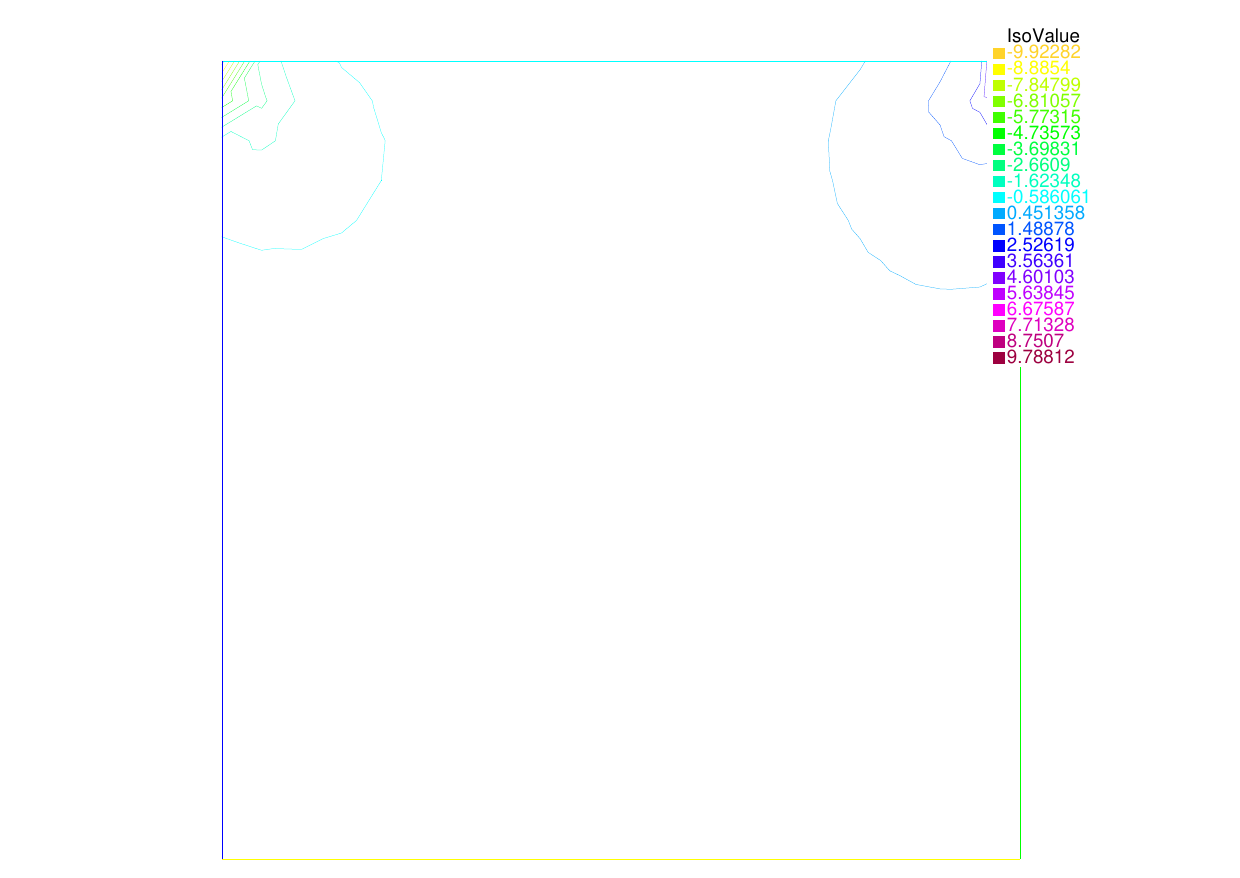}}%\\
\subfigure[ pressure: $r=10$]
{\includegraphics[width=5cm]{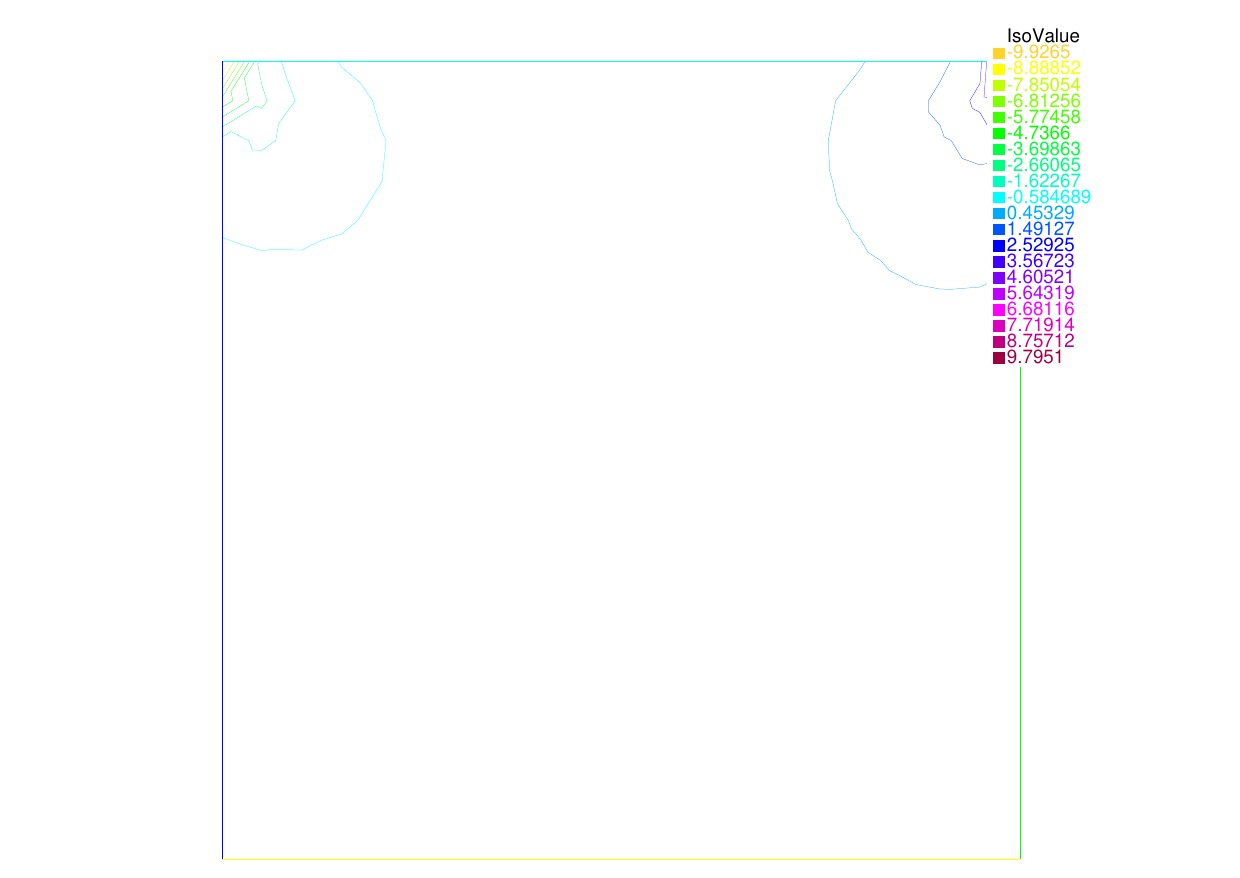}}
\caption{The   velocity streamlines and pressure contours  for Example \ref{EX7.2}:  $\alpha=1$ and    $r=3,5,10$ at $T=0.5$}
\label{fig233:233}
\end{figure}

\begin{exam}[The backward-facing step flow problem]\label{EX7.4}
We consider a backward-facing step flow problem in  $\Omega=\Omega_1\setminus \Omega_2$, with  $\Omega_1=[-4, 16]\times[-1, 2]$ and  $\Omega_2=[-4, 0]\times[-1, 0]$; see Figure \ref{fig4:mesh2} for the   domain and its finite element mesh. We take $\nu=0.01$   and $\bm{f}=\bm{0}$.  The boundary conditions are as follows:
$$\bm{u}|_{y=-1}=\bm{u}|_{y=2}=\bm{u}|_{-4\leq x\leq 0,y=0}=\bm{u}|_{x=0, -1\leq y\leq 0}=\bm{0},  $$
$$\bm{u}|_{x=-4}=(y(2-y), 0 )^T, \quad  \left(-p +\nu\frac{\partial u_1}{\partial x}\right) |_{x=16}=0, \quad u_{2}|_{x=16}=0.$$

We compute the fully discrete WG scheme \eqref{fullwg} with  $m=l=2$ in the following cases:
\begin{itemize}
\item [ I]. $\alpha=0$, i.e. the case of the Navier-Stokes  equations;

\item [ II].   $r=3$ and $\alpha=0.01, 0.1, 1$;

\item [ III]. $\alpha=1$ and $r=3, 10, 15$.
\end{itemize}
The obtained velocity and pressure approximations are shown   in Figures \ref{fig41:40} - \ref{fig41:45} at  time $T=0.5$.  As a comparison,  the    numerical solutions  obtained with the Taylor-Hood element %\cite{2012FREE}
 are also shown for $\alpha=0$; see (a), (b)  and (c) in Figure \ref{fig41:40}. Similar to Example \ref{EX7.3}, we can see that  our method is effective and the damping effect is gradually enhanced as the parameters $\alpha$ and $r$   increase.

\end{exam}

\vspace{0cm}
\begin{figure}[htbp!]
\centering
\setlength{\abovecaptionskip}{-2cm}
{\includegraphics[width=10cm]{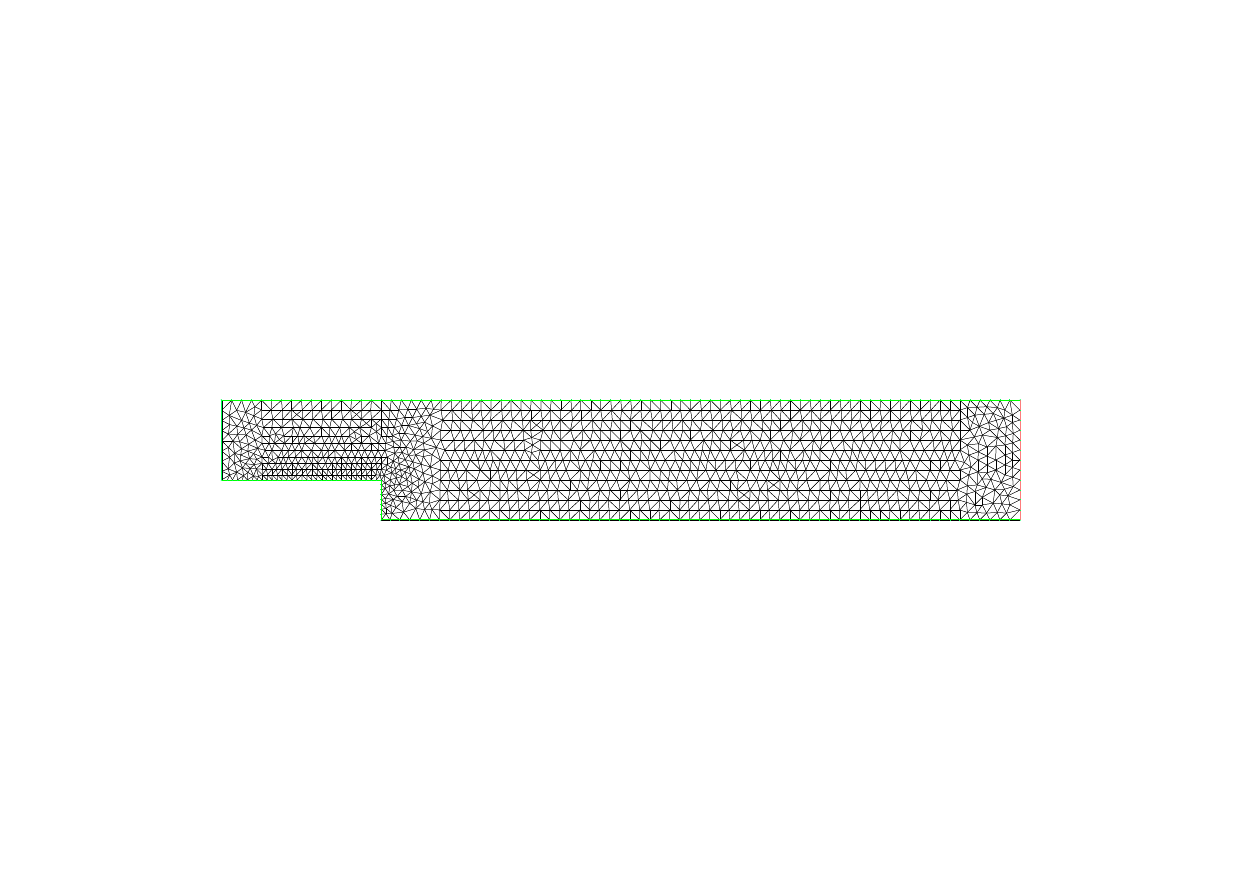}}%\\
\caption{The domain and  finite element mesh for Example \ref{EX7.4}}
\label{fig4:mesh2}
\end{figure}
%\vspace{0cm}

\begin{figure}[htbp!]
\centering
\setlength{\abovecaptionskip}{0.cm}
\subfigure[ $u_1$ (Taylor-Hood) ]
{\includegraphics[width=5.5cm]{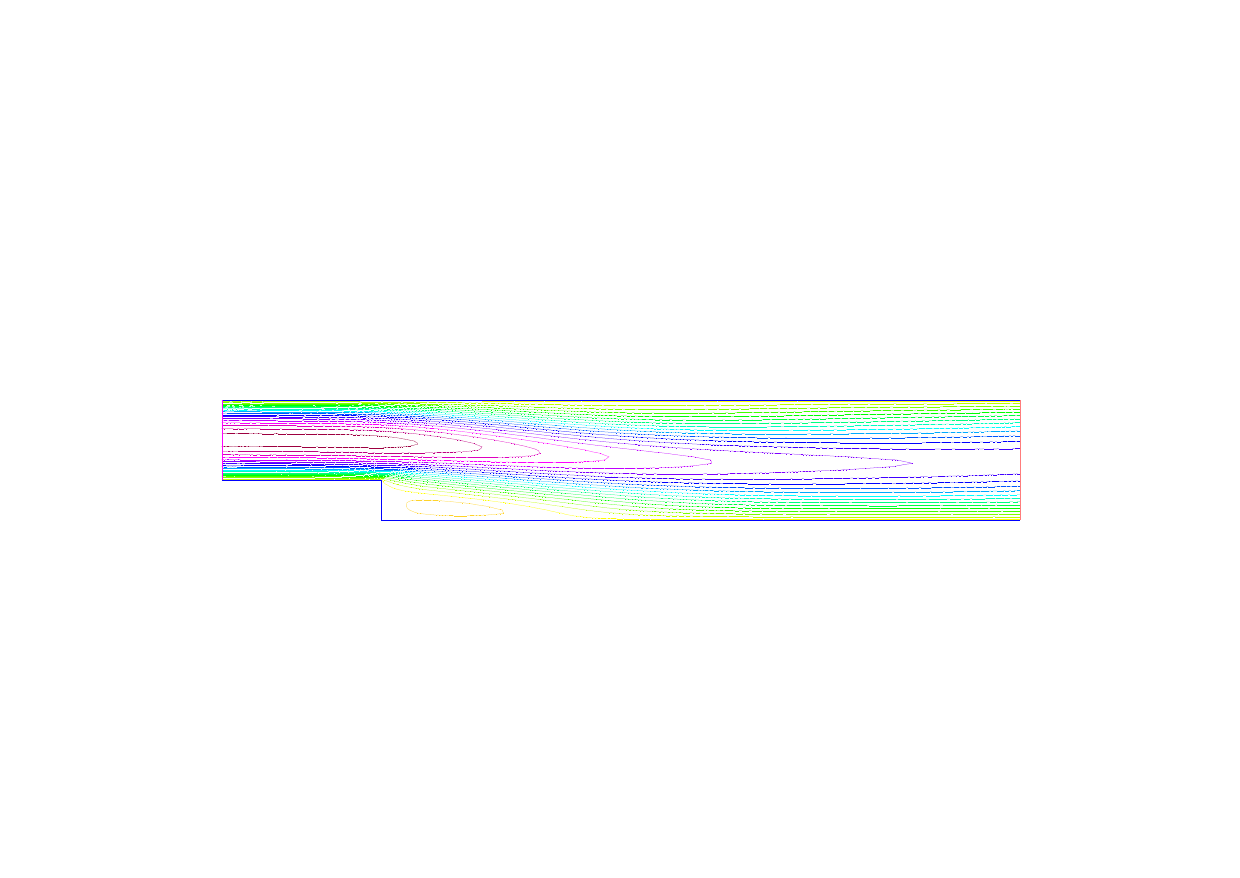}}
\subfigure[  $u_2$ (Taylor-Hood)]
{\includegraphics[width=5.5cm]{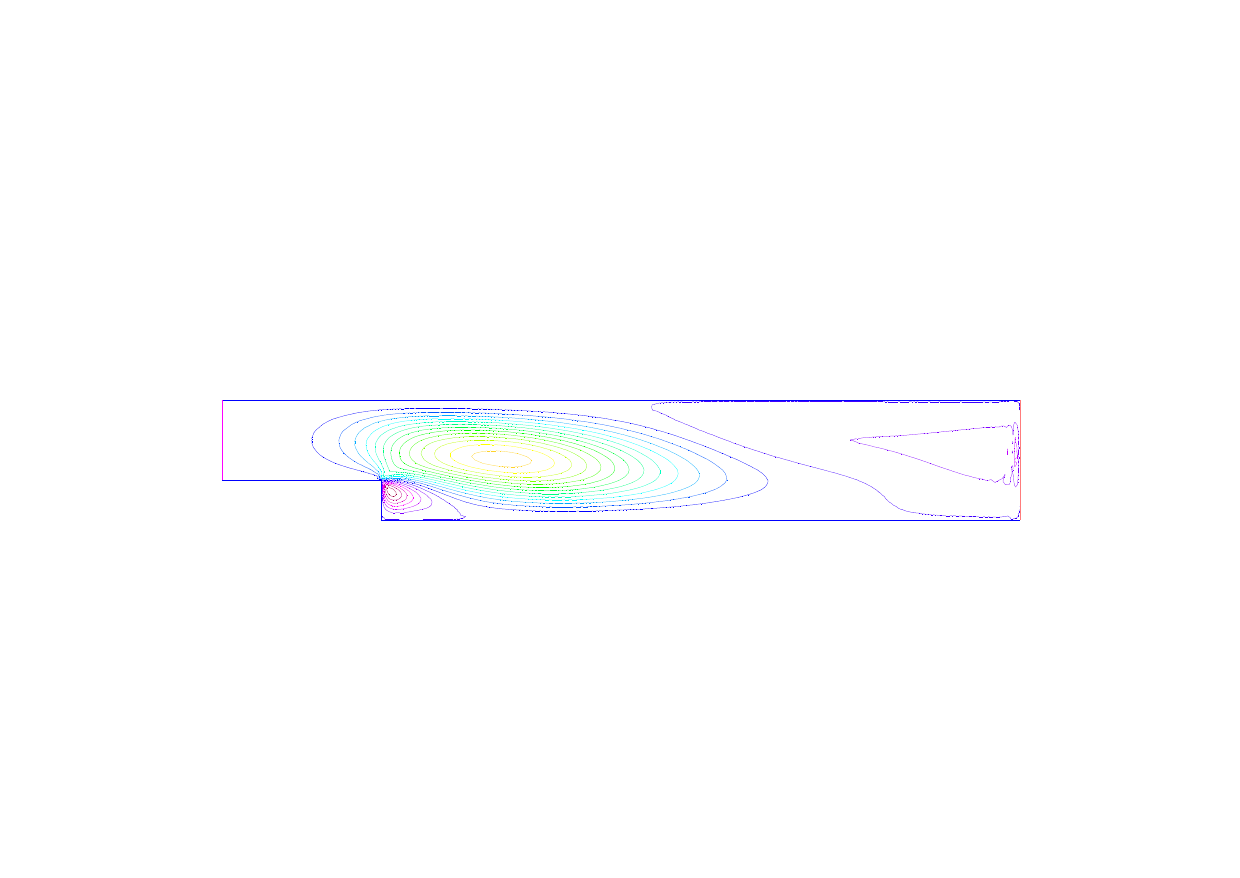}}%\\
\subfigure[pressure (Taylor-Hood)]
{\includegraphics[width=5.5cm]{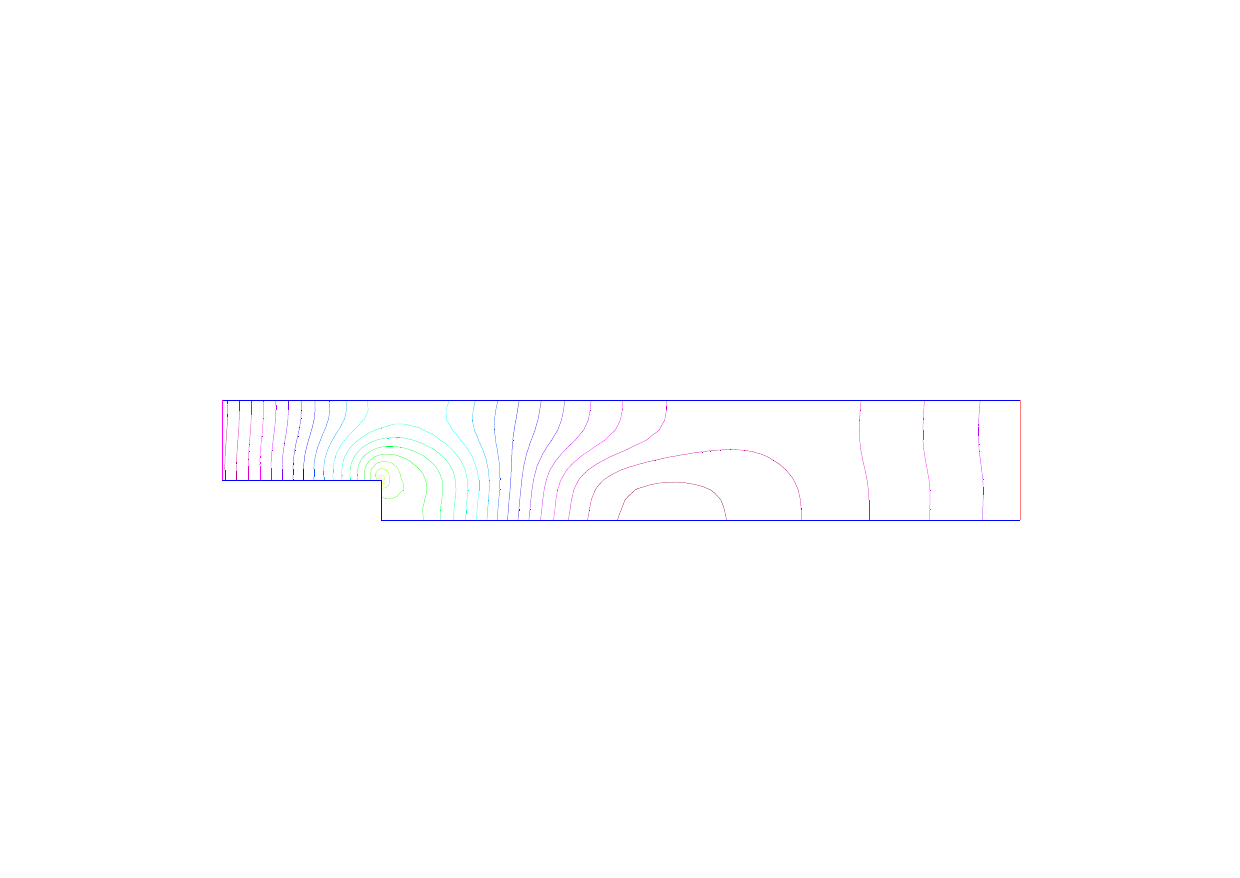}}\\
%\\
\subfigure[  $u_1$ (WG)]
{\includegraphics[width=5.5cm]{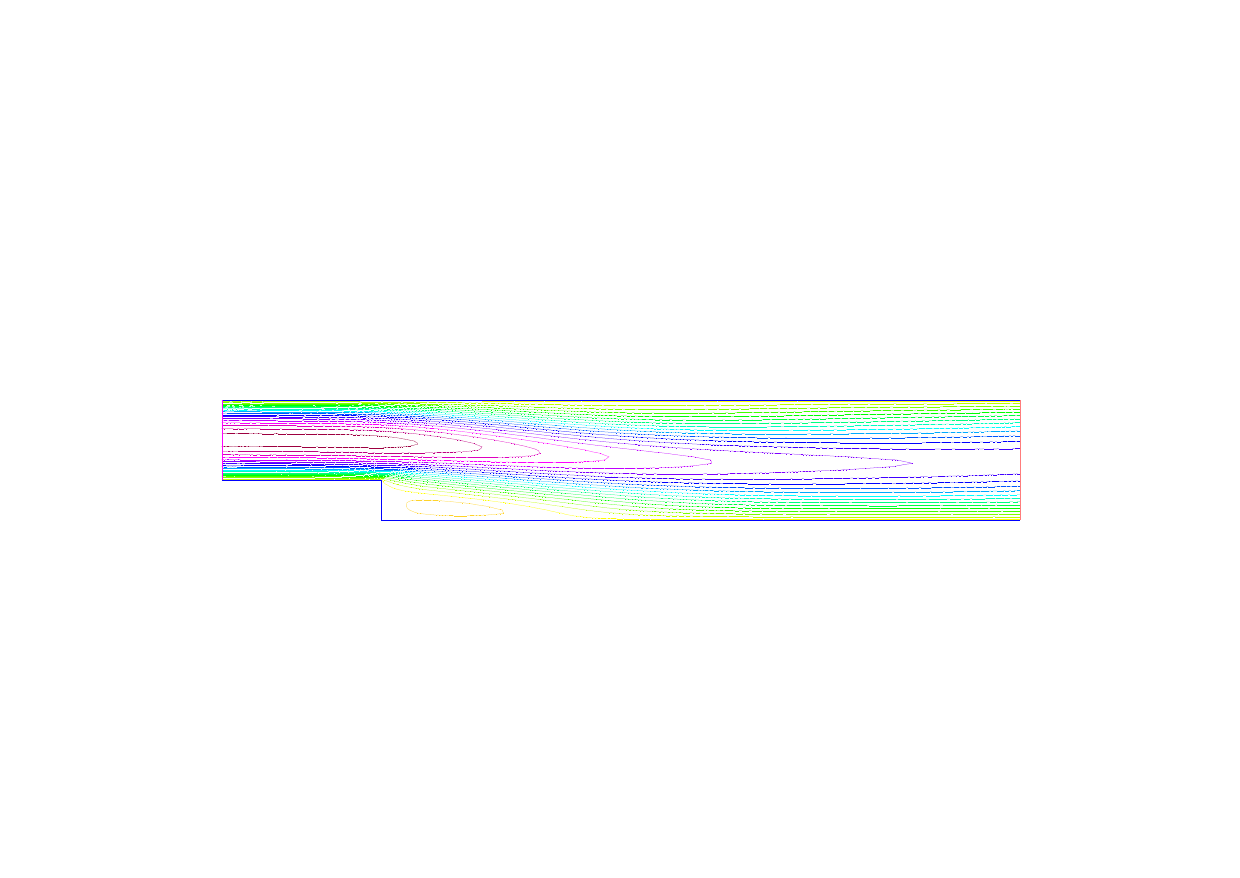}}
\subfigure[ $u_2$ (WG)]
{\includegraphics[width=5.5cm]{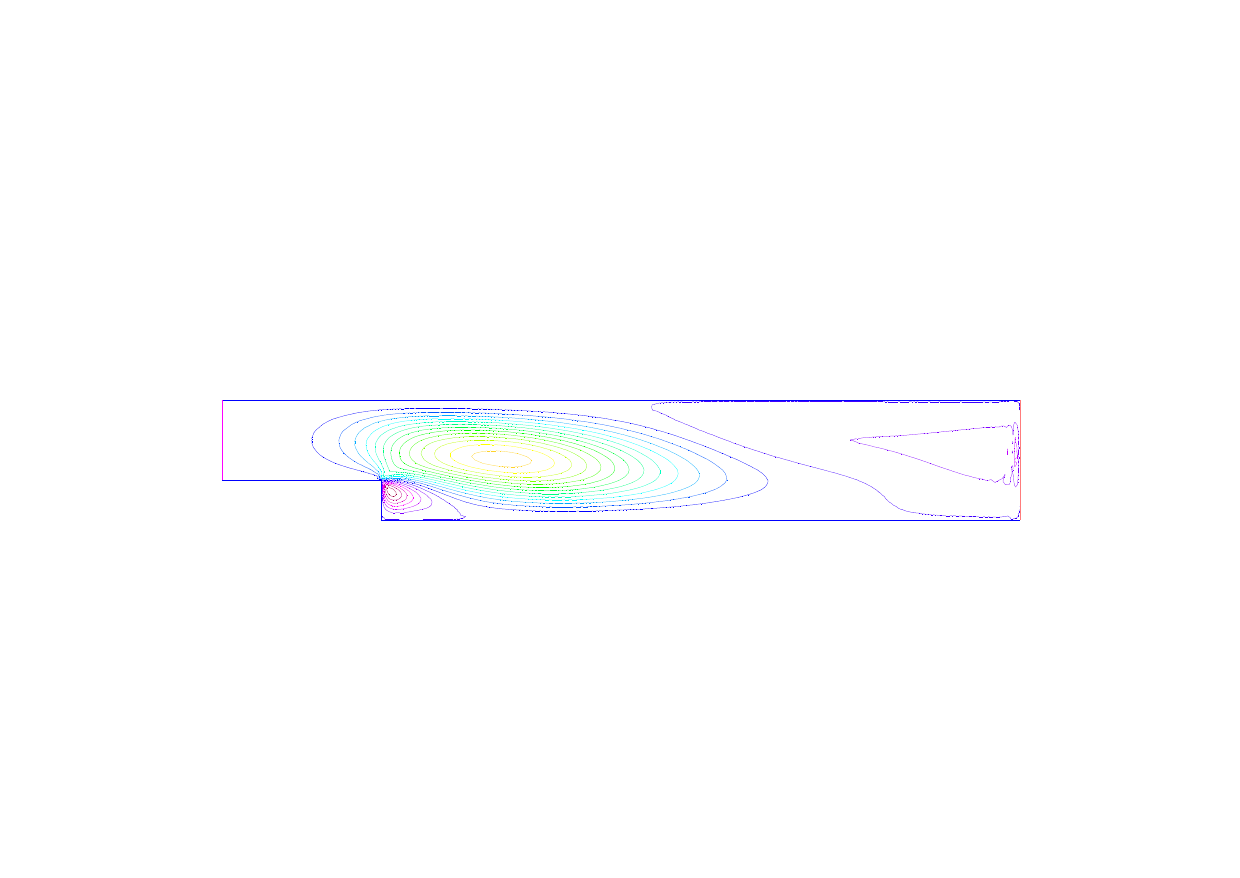}}
\subfigure[pressure (WG)]
{\includegraphics[width=5.5cm]{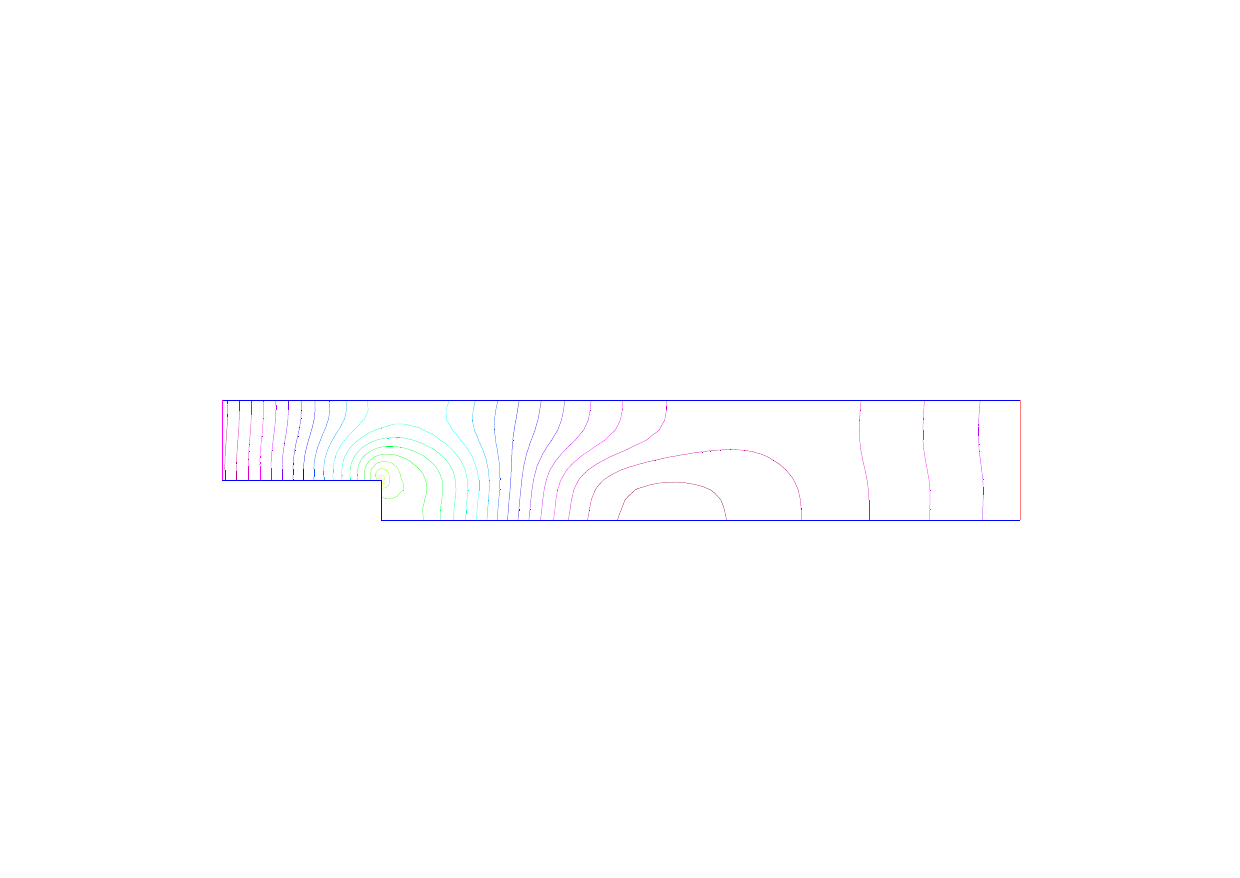}}%\\
 \caption{ The velocity $\bm u_h=(u_1,u_2)^T$   and pressure contours for Example \ref{EX7.4}: $\alpha=0$ at time $T=0.5$}
\label{fig41:40}
\end{figure}

\begin{figure}[htbp!]%\label{Fig1}
 \centering
\subfigure[  $u_1$: $\alpha=0.01$]
{\includegraphics[width=5.5cm]{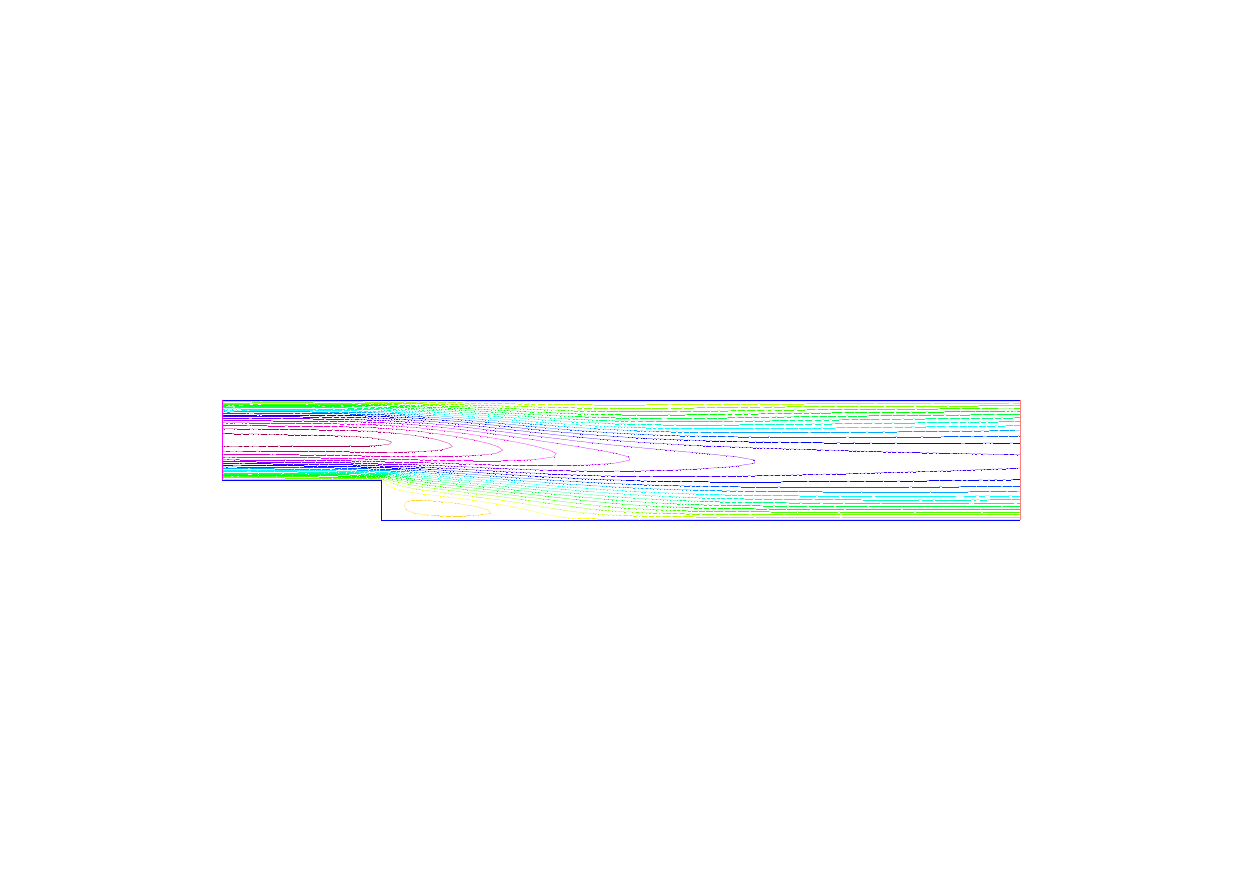}}
\subfigure[ $u_1$: $\alpha=0.1$]
{\includegraphics[width=5.5cm]{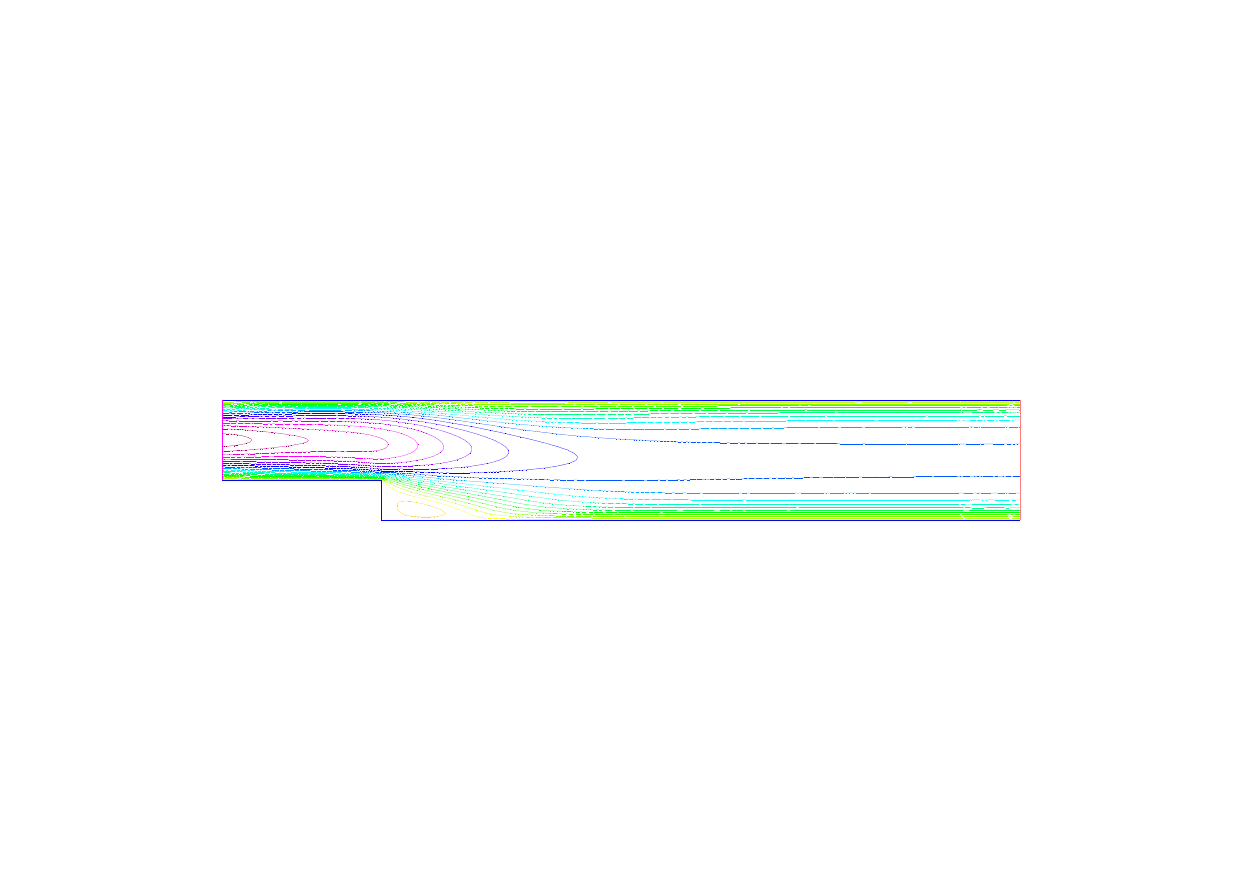}}%\\
\subfigure[  $u_1$: $\alpha=1$]
{\includegraphics[width=5.5cm]{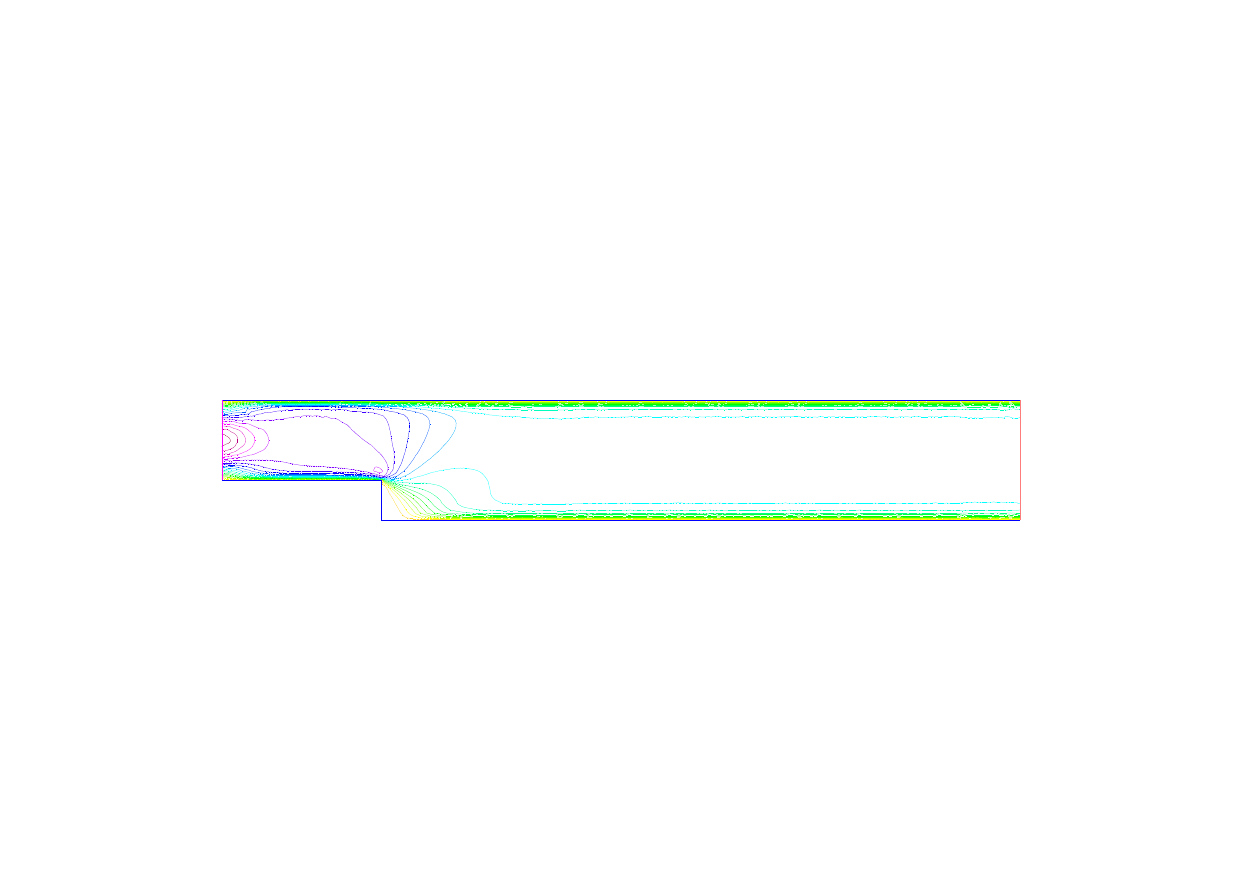}}\\

\subfigure[  $u_2$: $\alpha=0.01$]
{\includegraphics[width=5.5cm]{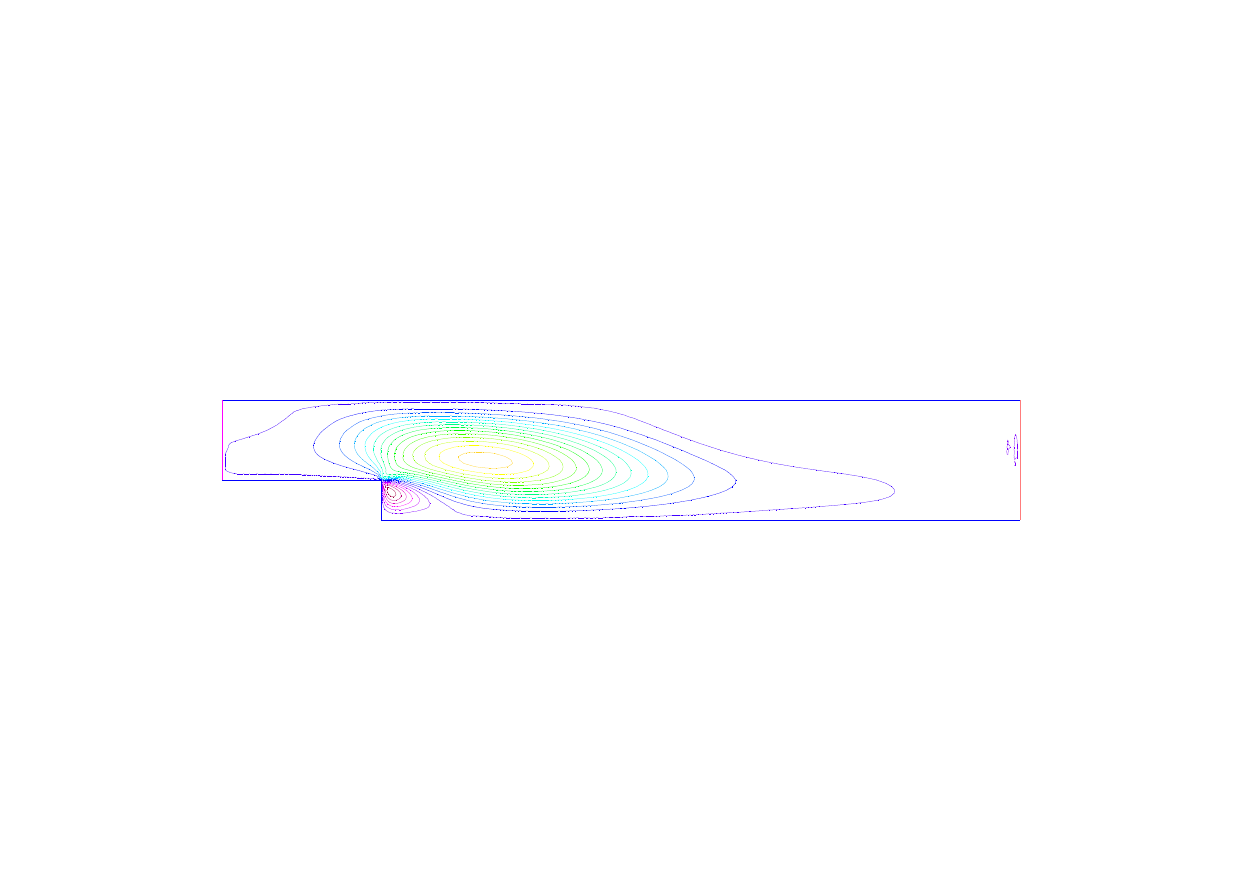}}
\subfigure[  $u_2$: $\alpha=0.1$]
{\includegraphics[width=5.5cm]{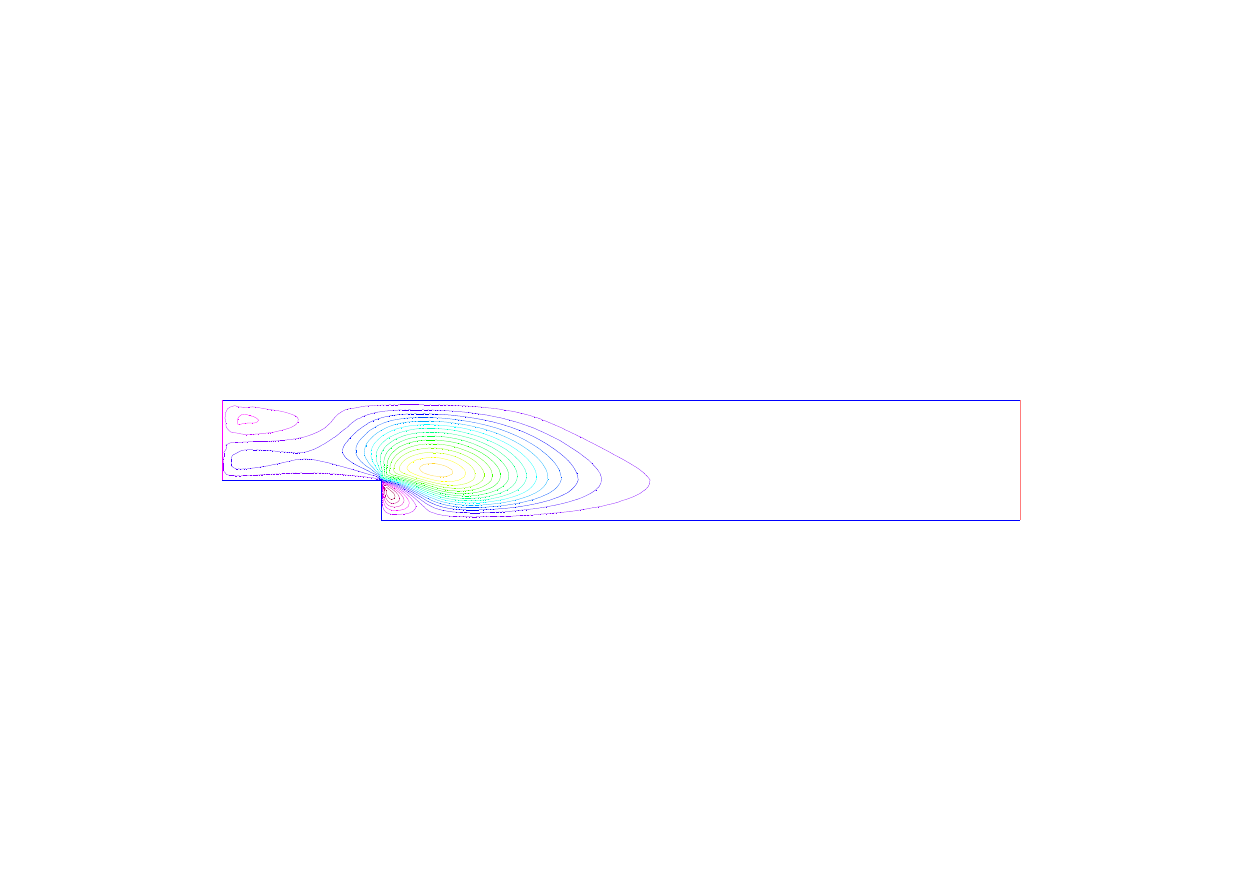}}%\\
\subfigure[  $u_2$: $\alpha=1$]
{\includegraphics[width=5.5cm]{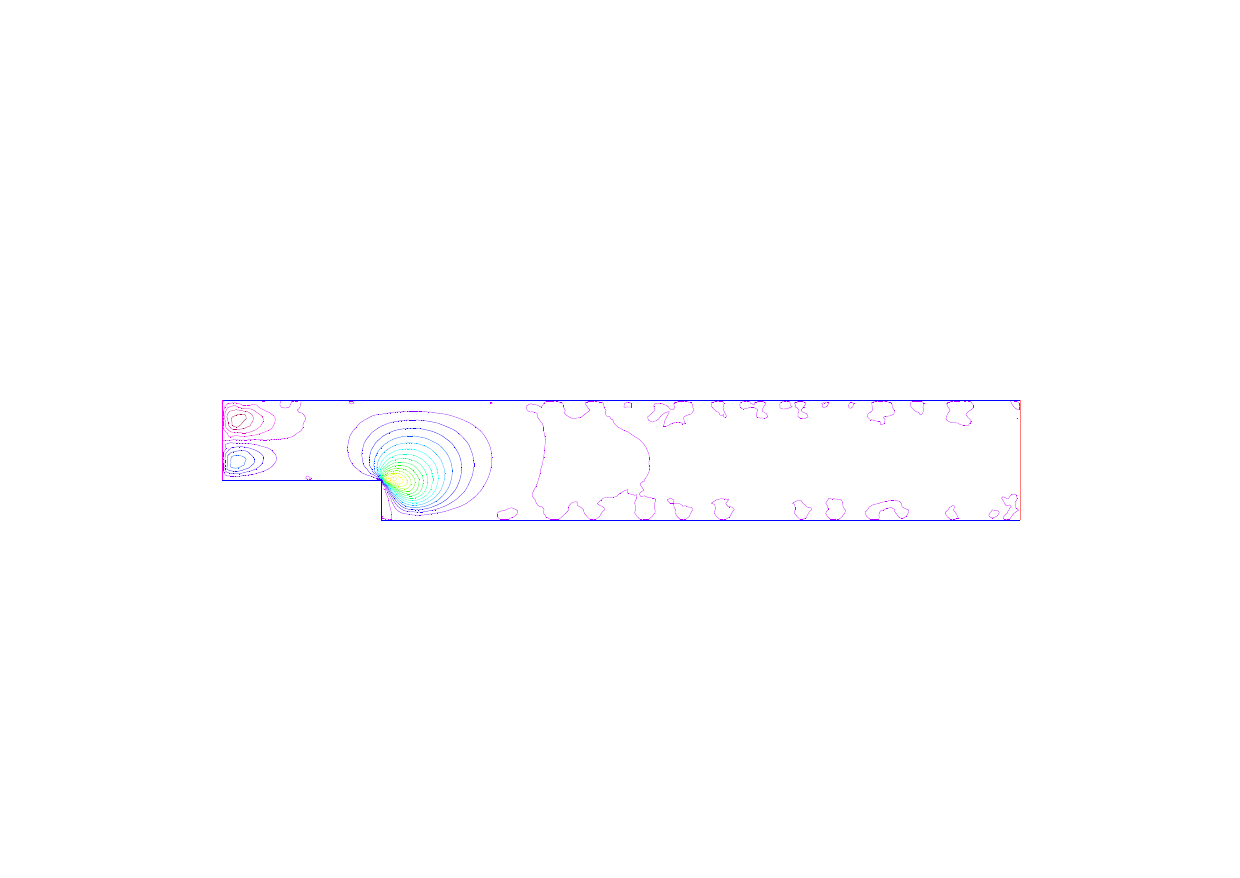}}

\subfigure[pressure: $\alpha=0.01$]
{\includegraphics[width=5.5cm]{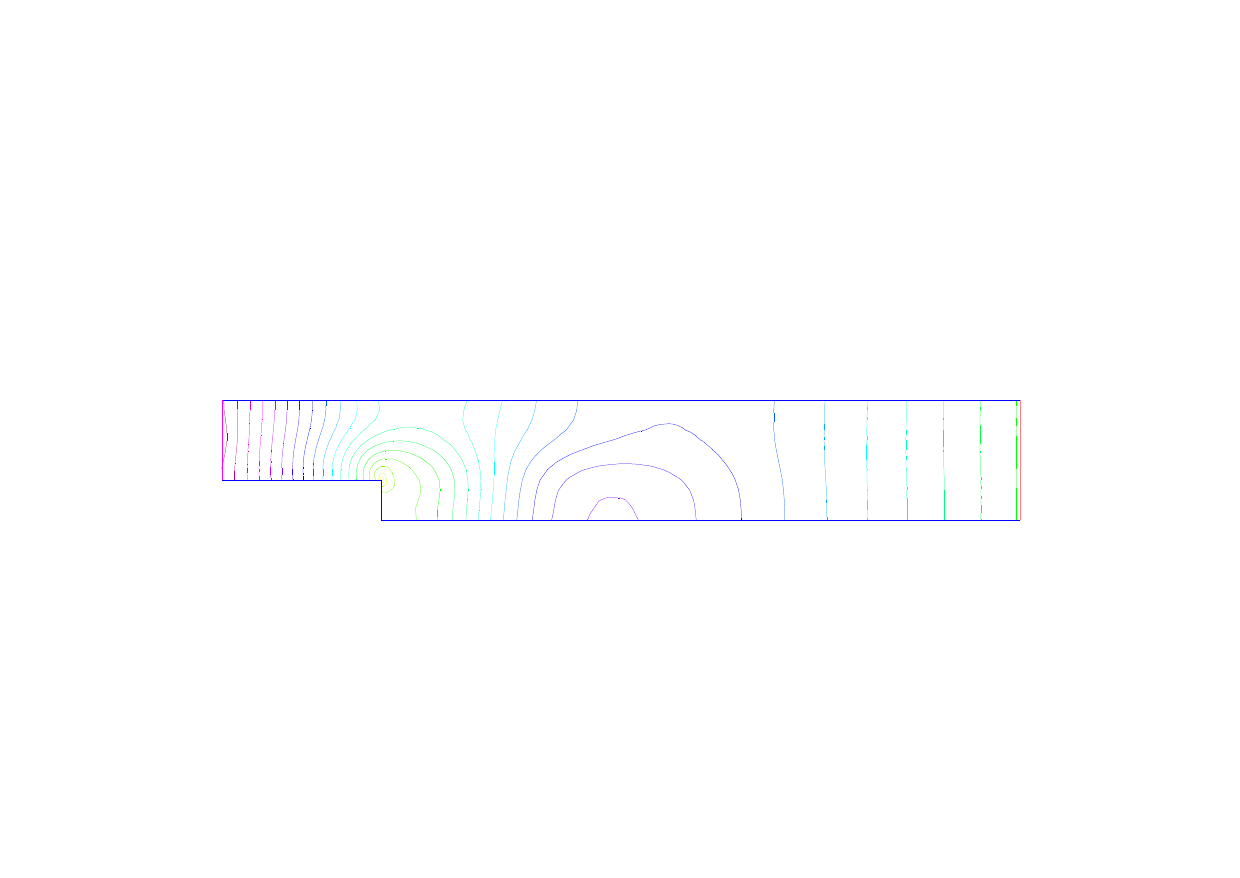}}
\subfigure[pressure: $\alpha=0.1$]
{\includegraphics[width=5.5cm]{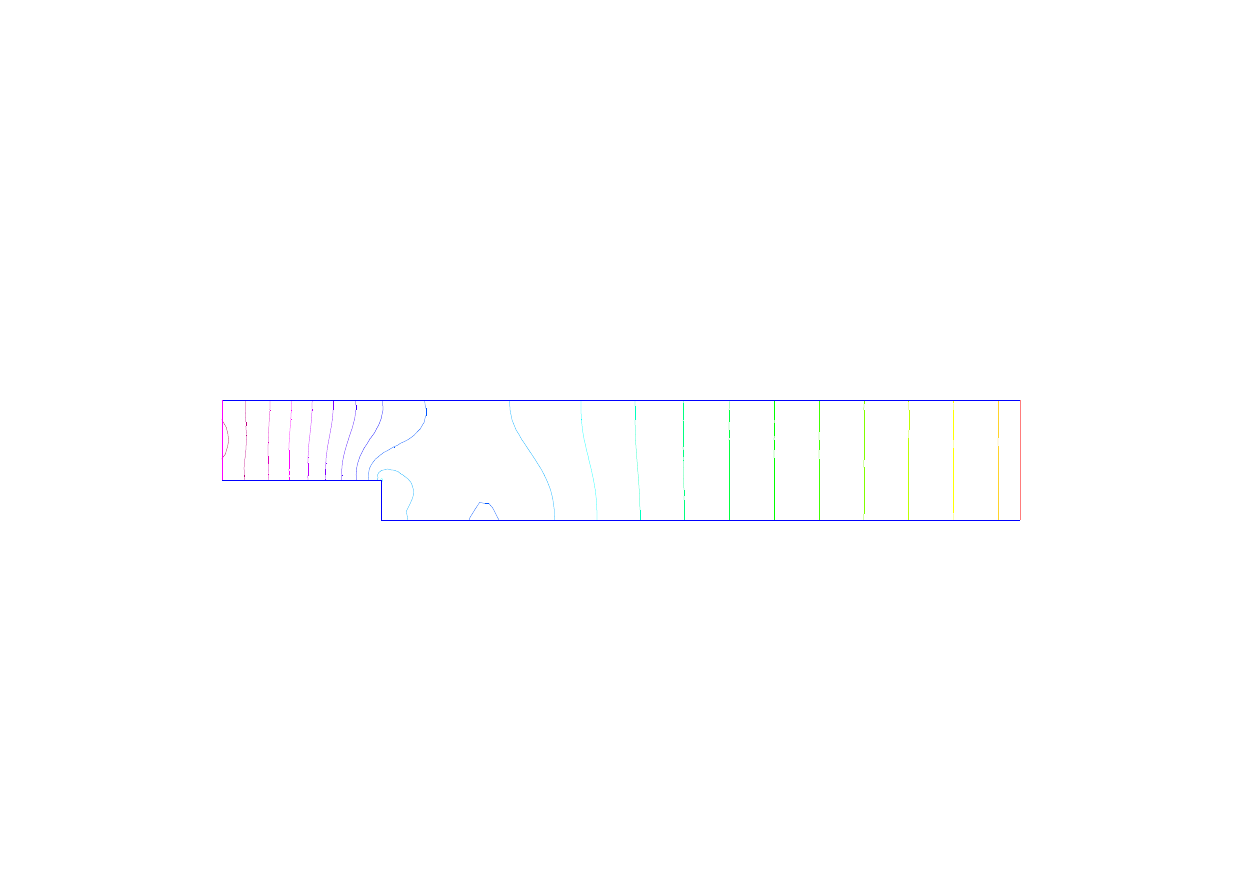}}%\\
\subfigure[pressure: $\alpha=1$]
{\includegraphics[width=5.5cm]{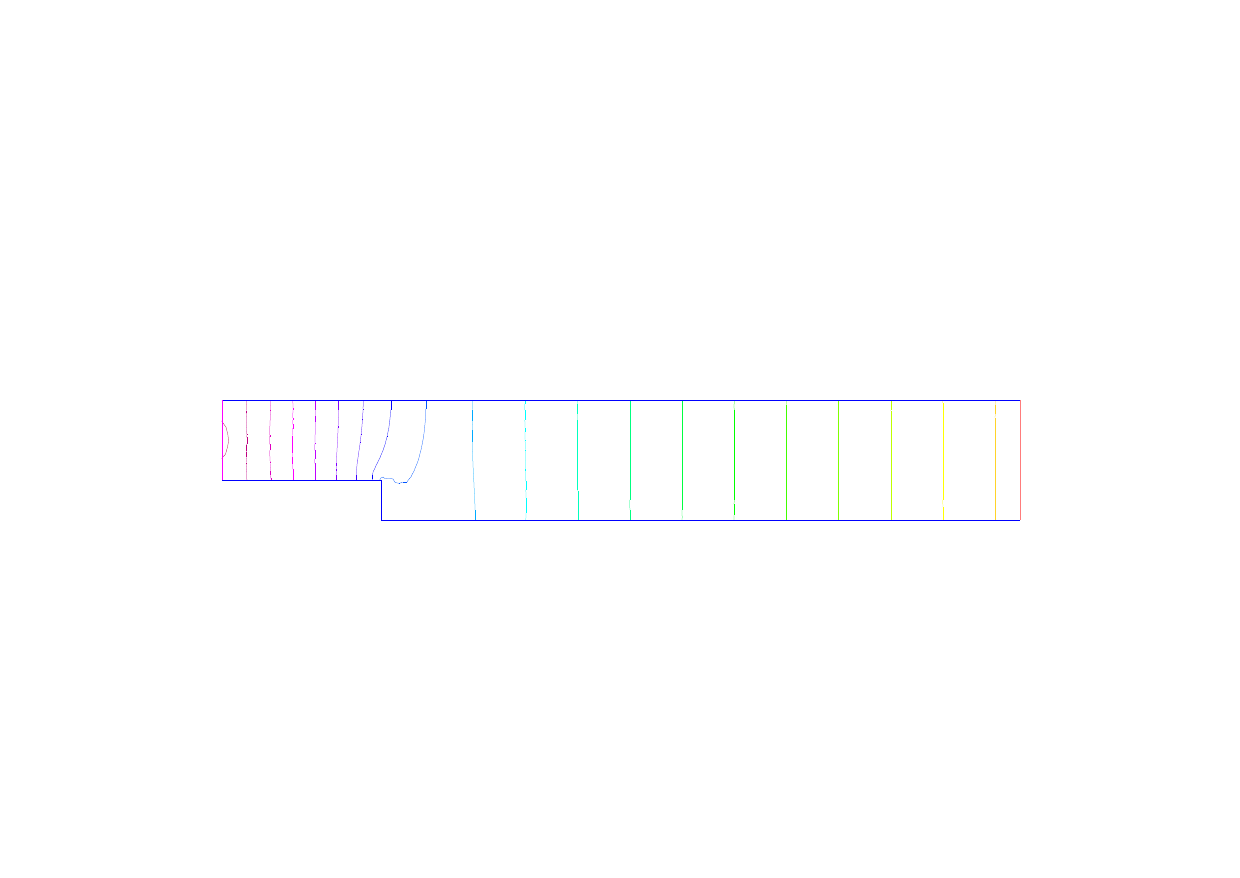}}
\caption{The velocity  $\bm u_h=(u_1,u_2)^T$   and pressure contours for Example \ref{EX7.4}: $r=3$ and $\alpha=0.01, 0.1, 1$ at time $T=0.5$ }
\label{fig41:42}
\end{figure}

\begin{figure}[htbp!]%\label{Fig1}
 \centering
\subfigure[  $u_1$: $r=3$]
{\includegraphics[width=5.5cm]{u1841305.eps}}
\subfigure[ $u_1$: $r=10$]
{\includegraphics[width=5.5cm]{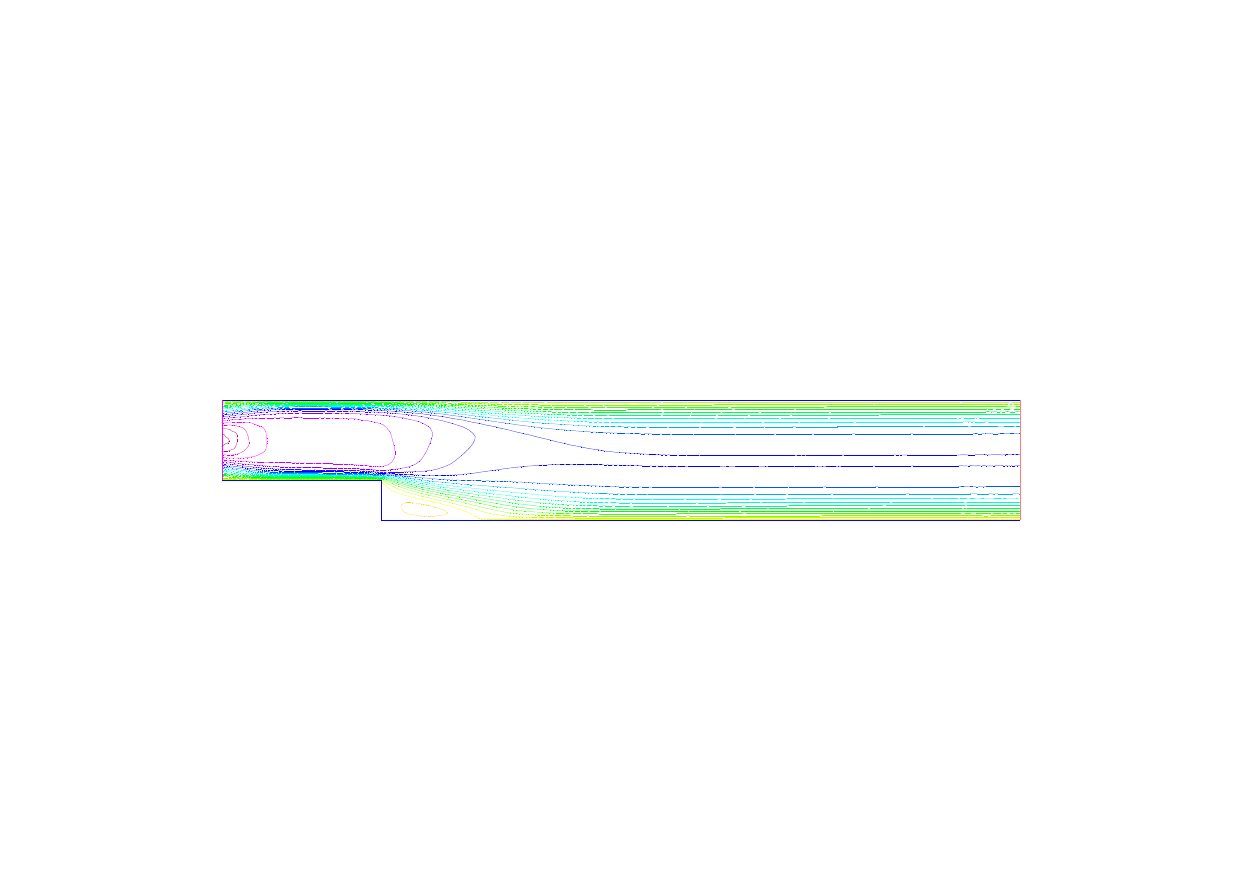}}%\\
\subfigure[  $u_1$: $r=15$]
{\includegraphics[width=5.5cm]{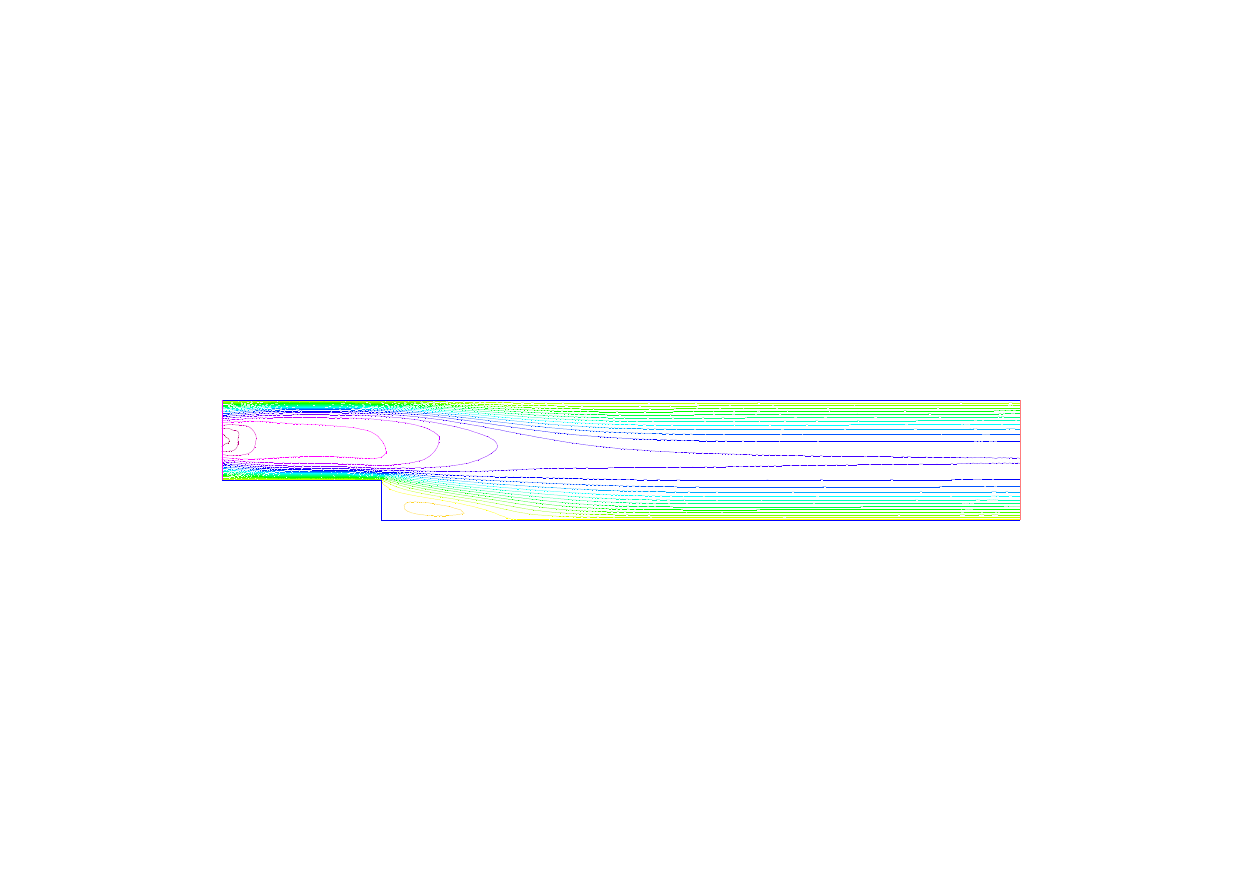}}\\

\subfigure[  $u_2$: $r=3$]
{\includegraphics[width=5.5cm]{u2841305.eps}}
\subfigure[  $u_2$: $r=10$]
{\includegraphics[width=5.5cm]{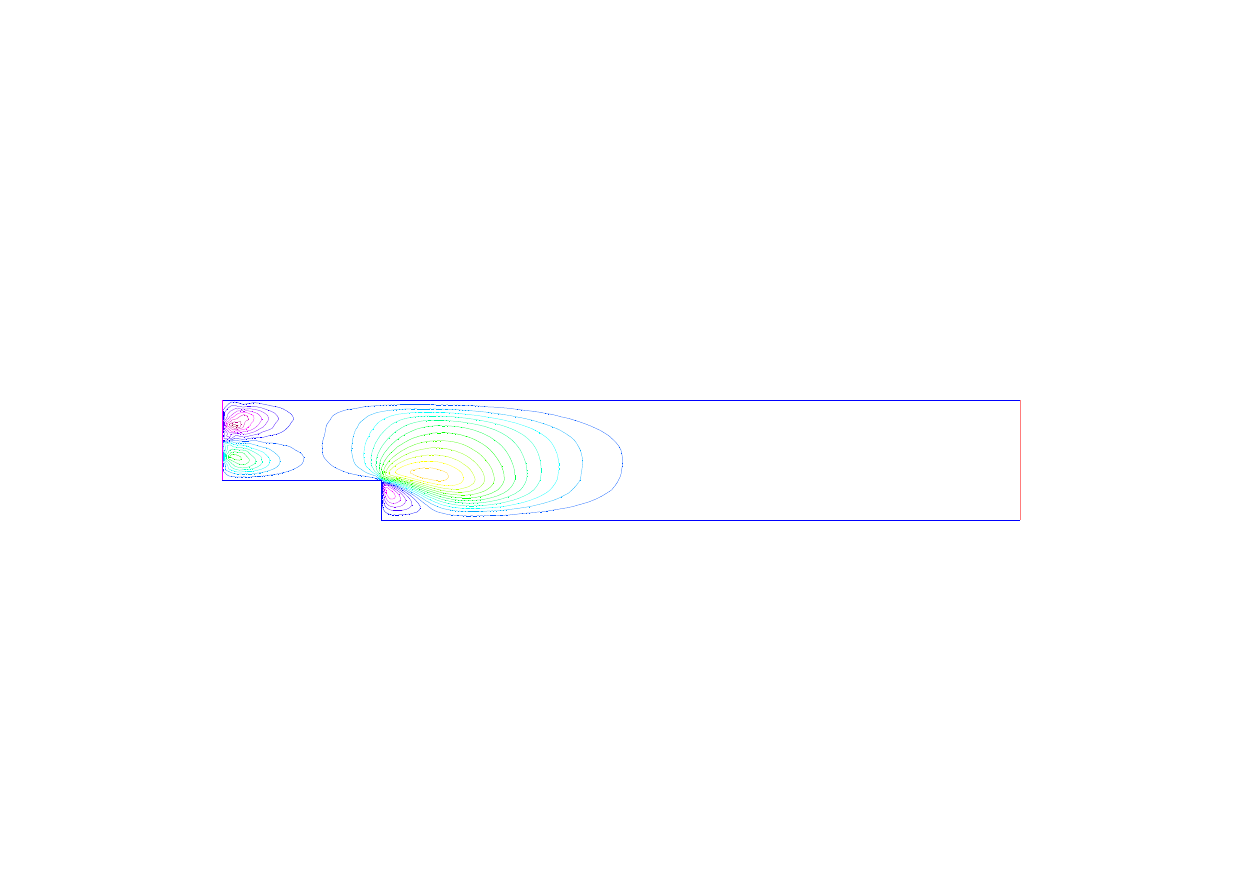}}%\\
\subfigure[  $u_2$: $r=15$]
{\includegraphics[width=5.5cm]{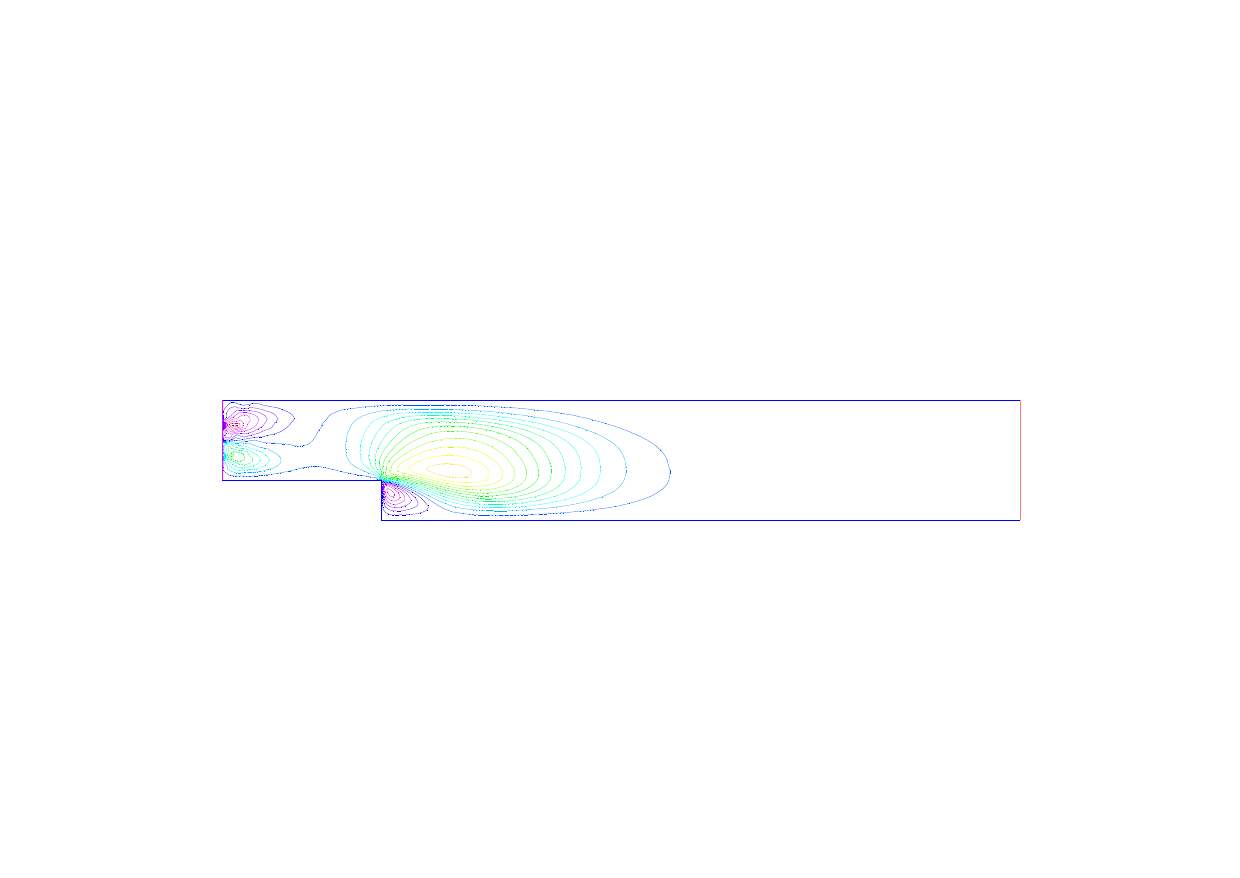}}

\subfigure[pressure: $r=3$]
{\includegraphics[width=5.5cm]{p841305.eps}}
\subfigure[pressure: $r=10$]
{\includegraphics[width=5.5cm]{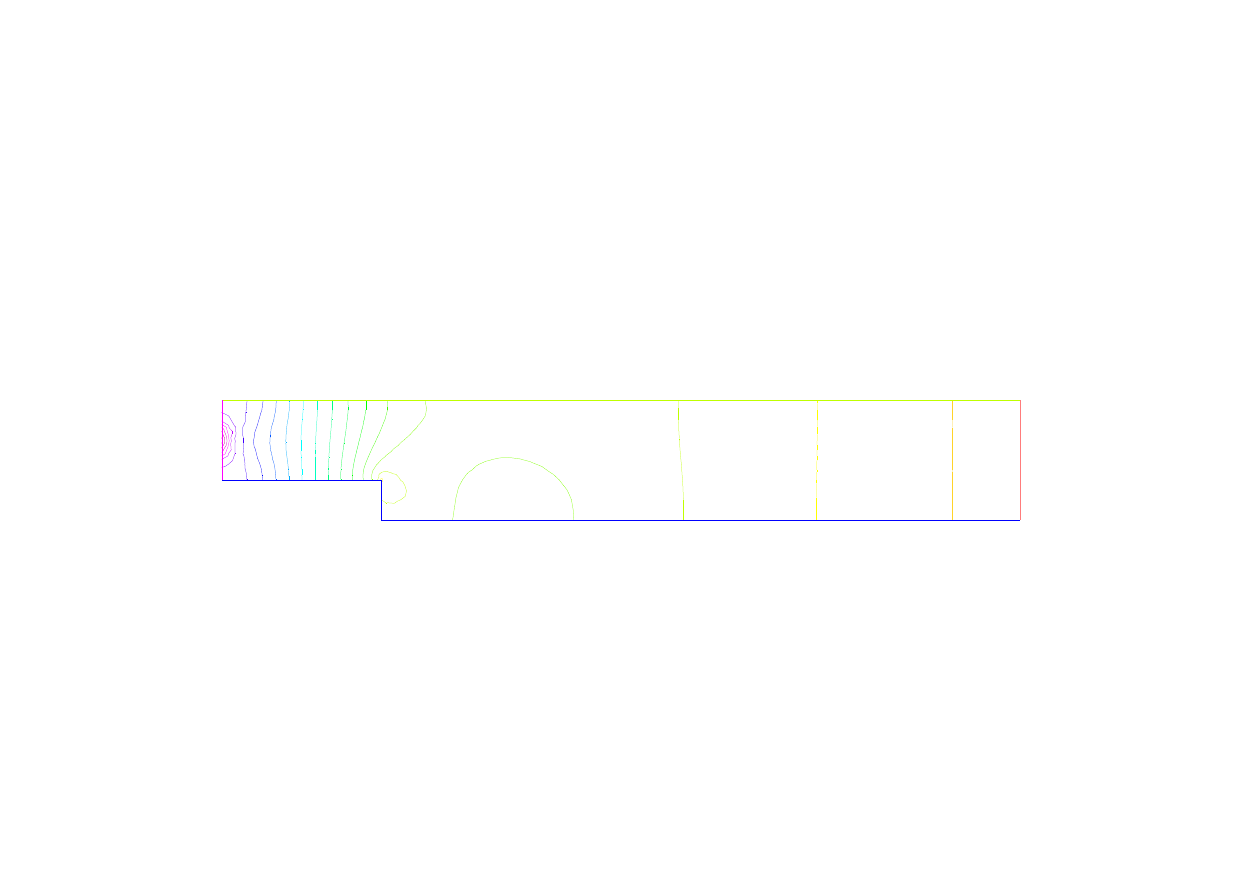}}%\\
\subfigure[pressure: $r=15$]
{\includegraphics[width=5.5cm]{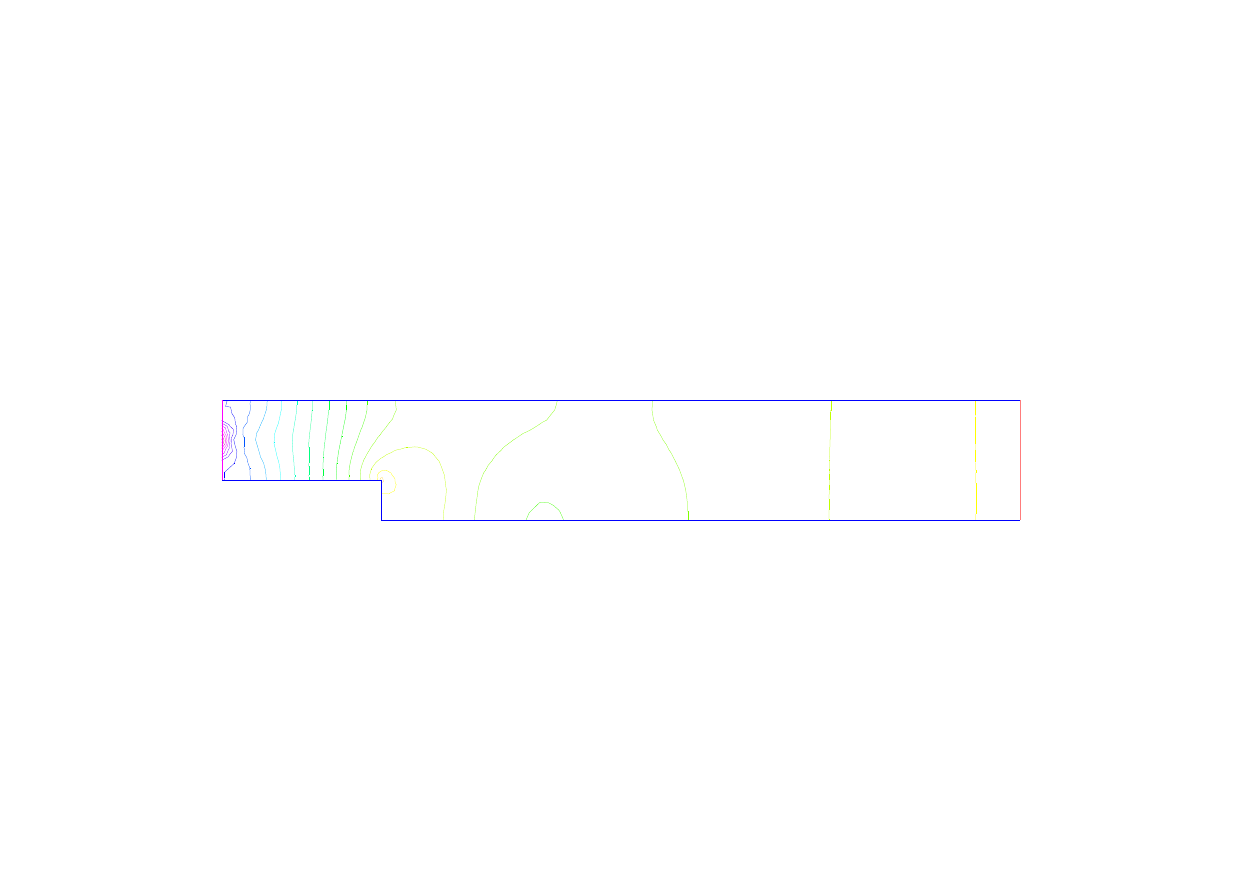}}
\caption{The velocity  $\bm u_h=(u_1,u_2)^T$   and pressure contours for Example \ref{EX7.4}: $\alpha=1$ and $r=3, 10, 15$ at time $T=0.5$ }
\label{fig41:45}
\end{figure}

\begin{exam}[The problem of flow around a   circular  cylinder]\label{EX7.5}
The flow around a circular cylinder   is examined with the Brinkman-Forchheimer model \eqref{BF0} and the the WG  method. We take
  $\Omega$ =$[0, 6]\times[0, 1] \setminus O_{d}(1,0.5)$, $\nu=0.005$   and $\bm{f}=\bm{0}$, where $O_{d}(1,0.5)$ is a disk with center $(1,0.5)$  and diameter $d=0.3$; see  Figure \ref{fig2:mesh2} for the domain and its finite element mesh. The boundary conditions are as follows:
$$\bm{u}|_{y=0}=\bm{u}|_{y=1}=\bm{u}|_{\partial O_{d}}=\bm{0}, \quad \bm{u}|_{x=0}=(6y(1-y), 0 )^T, $$
$$ \left(-p\bm{I}+\nu\nabla \bm{u}\right){  \bm{n}}\big|_{x=6}=0, $$
where $\bm{I}$ and $ \bm{n}$ are the unit matrix and the outward unit normal vector, respectively.
We compute the fully discrete WG scheme \eqref{fullwg}  with  $m=l=2$ in the following cases:
\begin{itemize}
\item [ I]. $\alpha=0$, i.e. the case of the Navier-Stokes  equations;

\item [ II].   $r=3.5$ and $\alpha=0.1, 1, 5$;

\item [ III]. $\alpha=2$ and $r=3, 5, 10$.
\end{itemize}
The obtained velocity, vorticity and pressure approximations are shown   in Figures \ref{fig21:2111}, \ref{fig32:32} - \ref{fig35:35} at time $T=0.5$, respectively.  As a comparison,  the  referenced  numerical solutions  obtained  with the Taylor-Hood  element
are also shown for $\alpha=0$; see (a), (b)  and (c) in Figure \ref{fig21:2111}.
Similar to Example \ref{EX7.3}, we also  see that our method is effective and the damping impact on velocity, vorticity and pressure is gradually enhanced with the parameters $\alpha$ and $r$ increasing.

\end{exam}
%\vspace{-1.5cm}
\begin{figure}[htbp!]
\centering
\setlength{\abovecaptionskip}{-2cm}
%\subfigure
{\includegraphics[width=10cm]{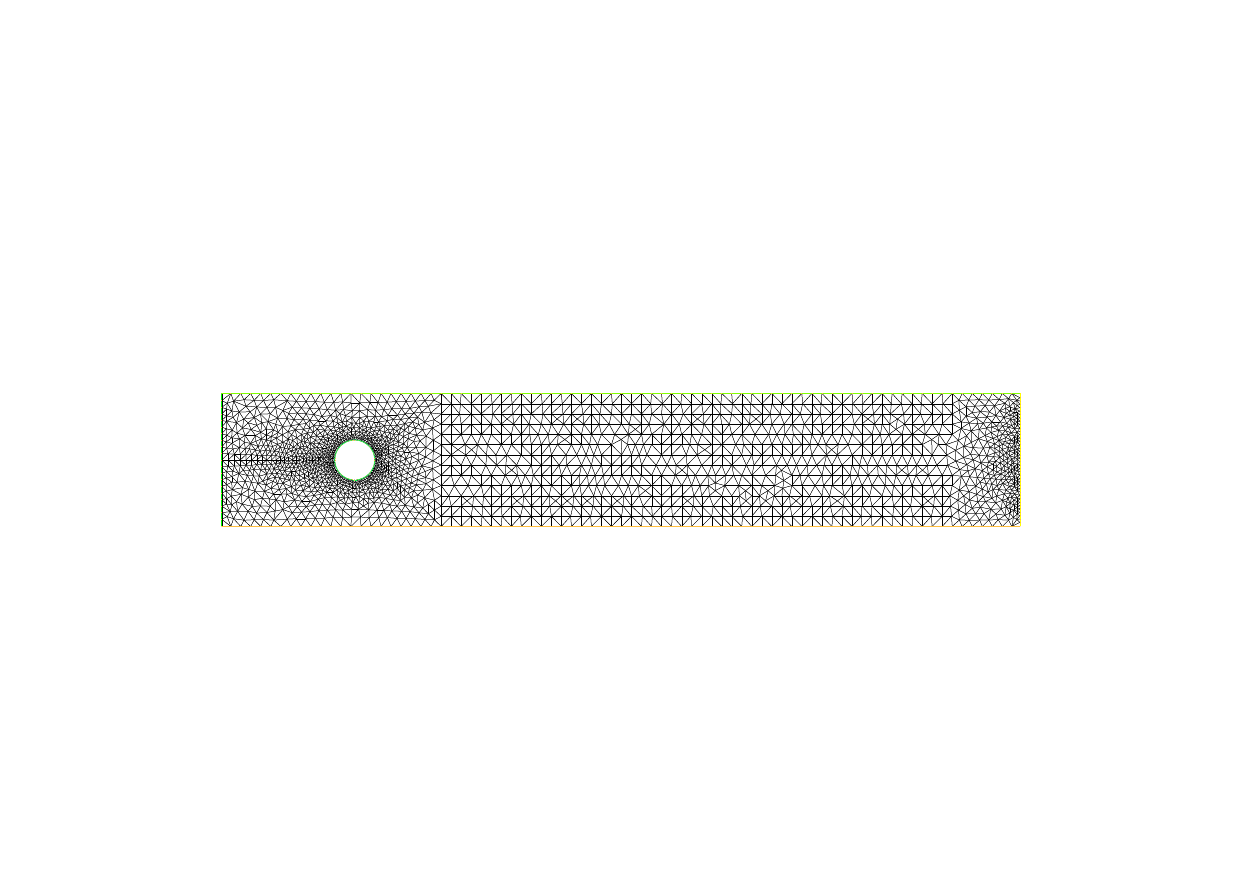}}%\\
\caption{The domain and  finite element mesh for Example \ref{EX7.5}}
\label{fig2:mesh2}
\end{figure}
%\vspace{0cm}

\begin{figure}[htbp!]
\centering
\setlength{\abovecaptionskip}{0.cm}
\subfigure[velocity (Taylor-Hood)]
{\includegraphics[width=5.5cm]{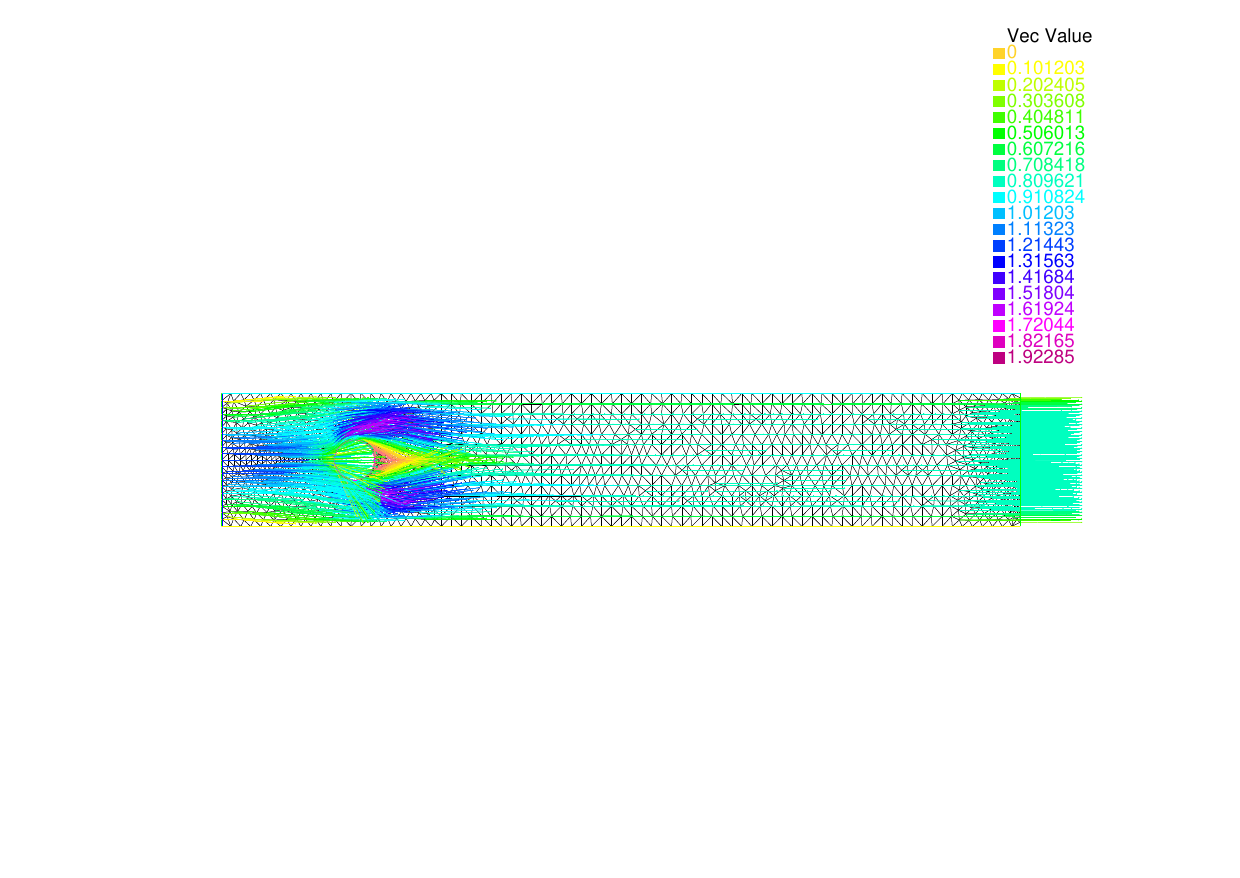}}
\subfigure[vorticity (Taylor-Hood) ]
{\includegraphics[width=5.5cm]{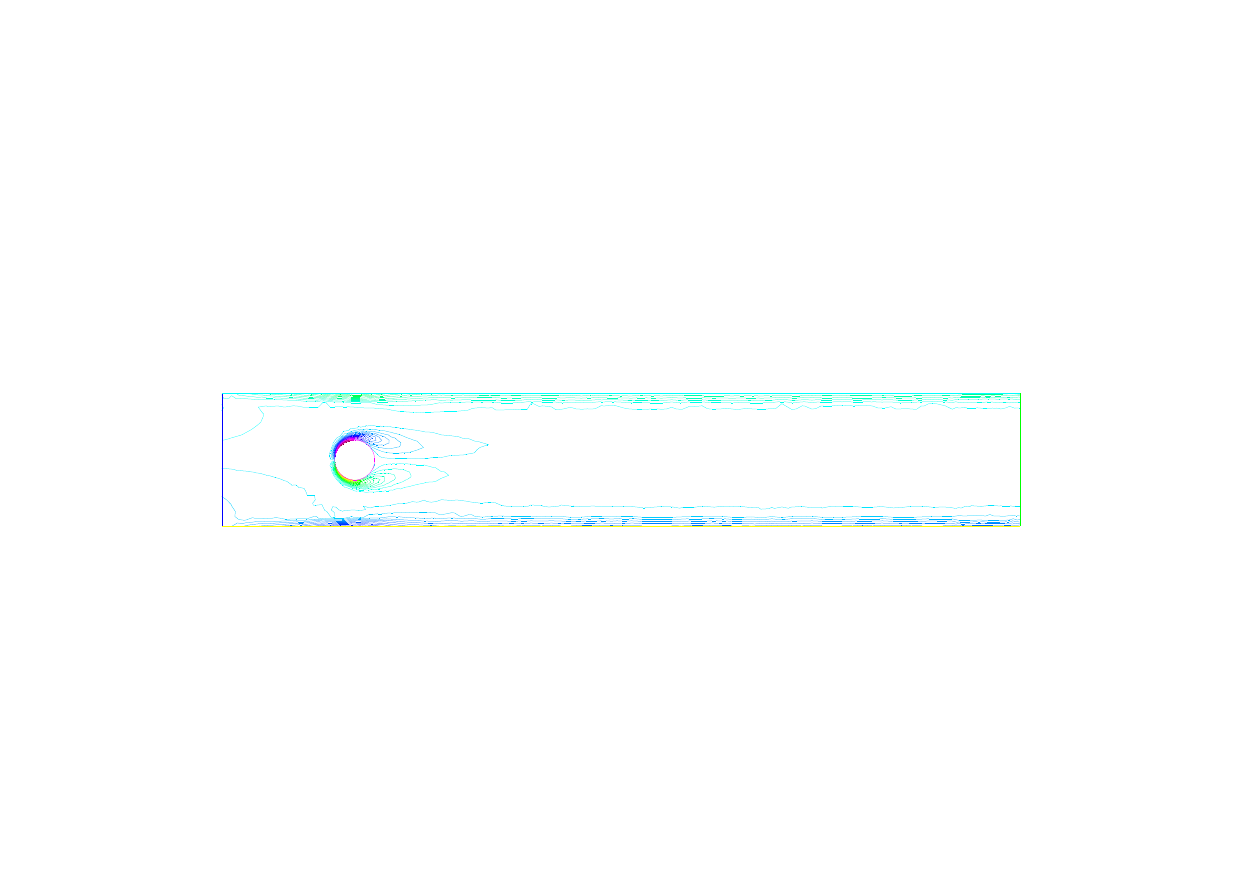}}%\\
\subfigure[pressure (Taylor-Hood) ]
{\includegraphics[width=5.5cm]{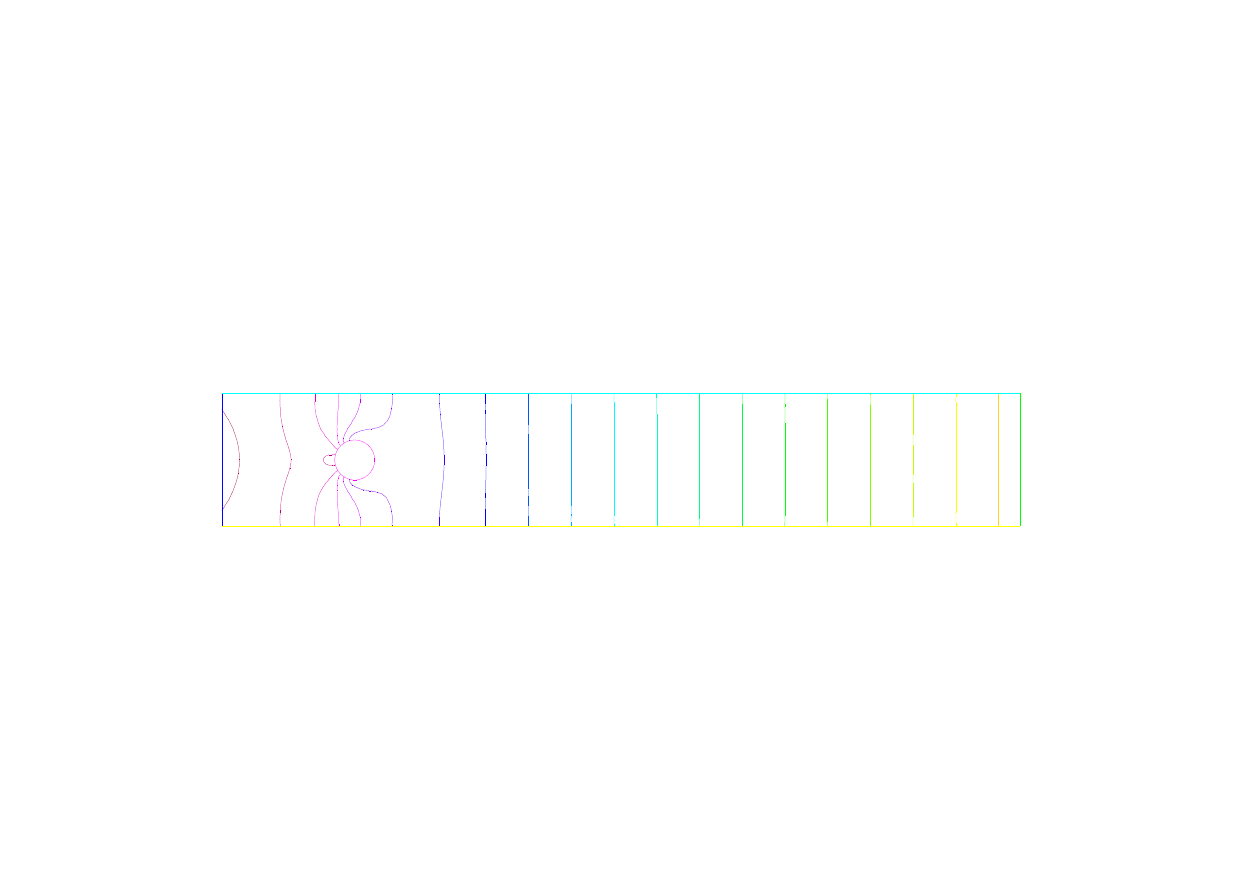}}\\%\\
%\\
\subfigure[velocity (WG)]
{\includegraphics[width=5.5cm]{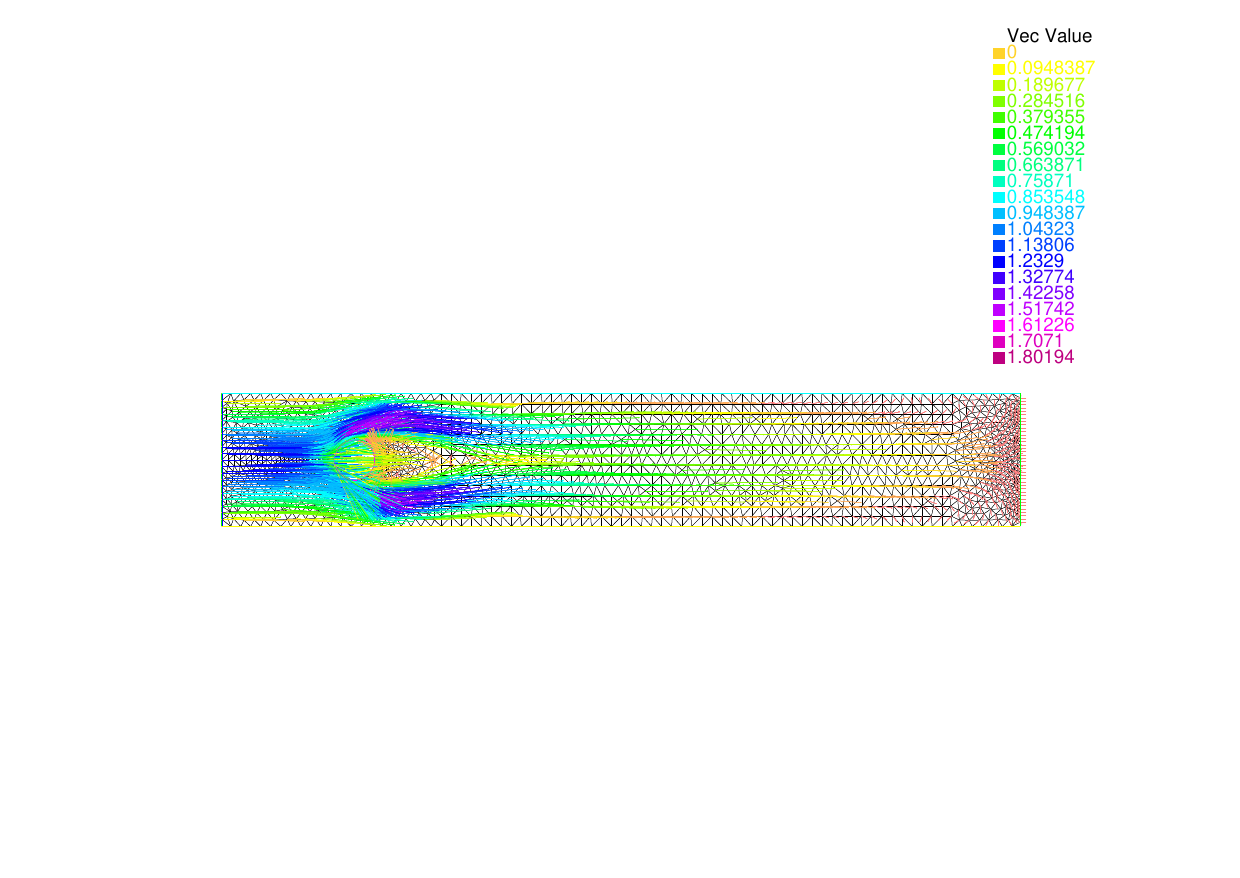}}%\\
\subfigure[vorticity (WG)]
{\includegraphics[width=5.5cm]{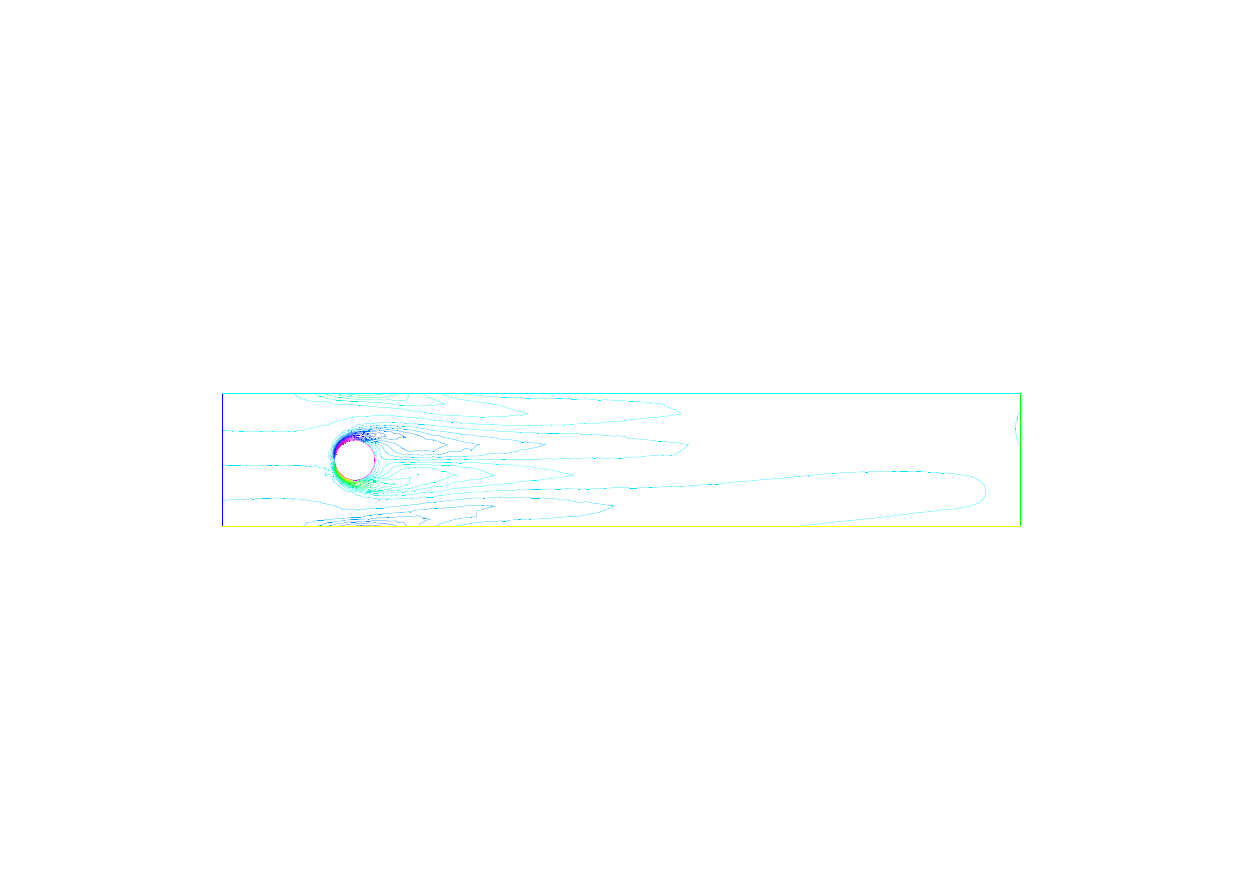}}
\subfigure[pressure (WG)]
{\includegraphics[width=5.5cm]{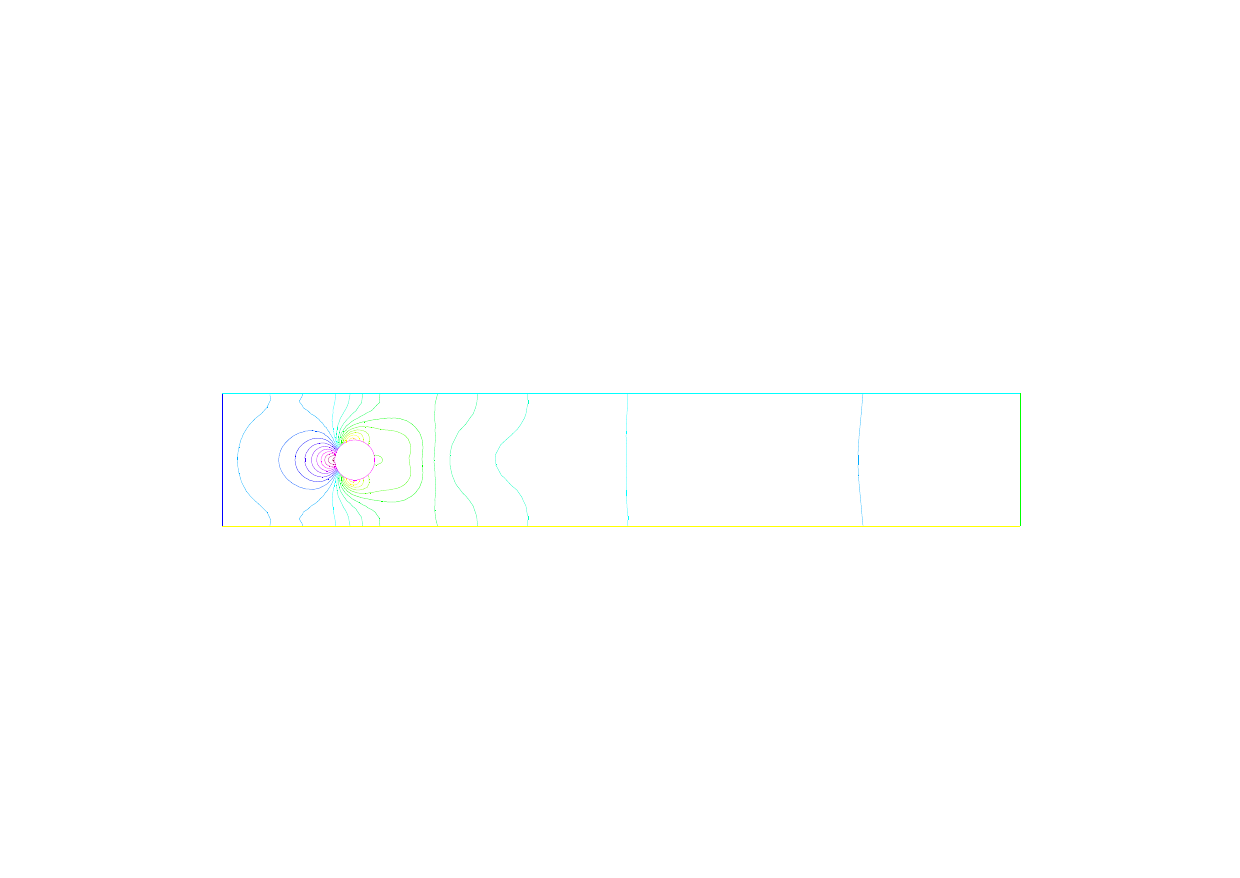}}%\hspace{-10mm}%\\
\caption{ The velocity streamlines, vortex lines  and pressure contours for Example \ref{EX7.5}: $\alpha=0$ at $T=0.5$}
\label{fig21:2111}
\end{figure}

\begin{figure}[htbp!]%\label{Fig1}
\centering
\subfigure[velocity: $\alpha=0.1$]
{\includegraphics[width=5.5cm]{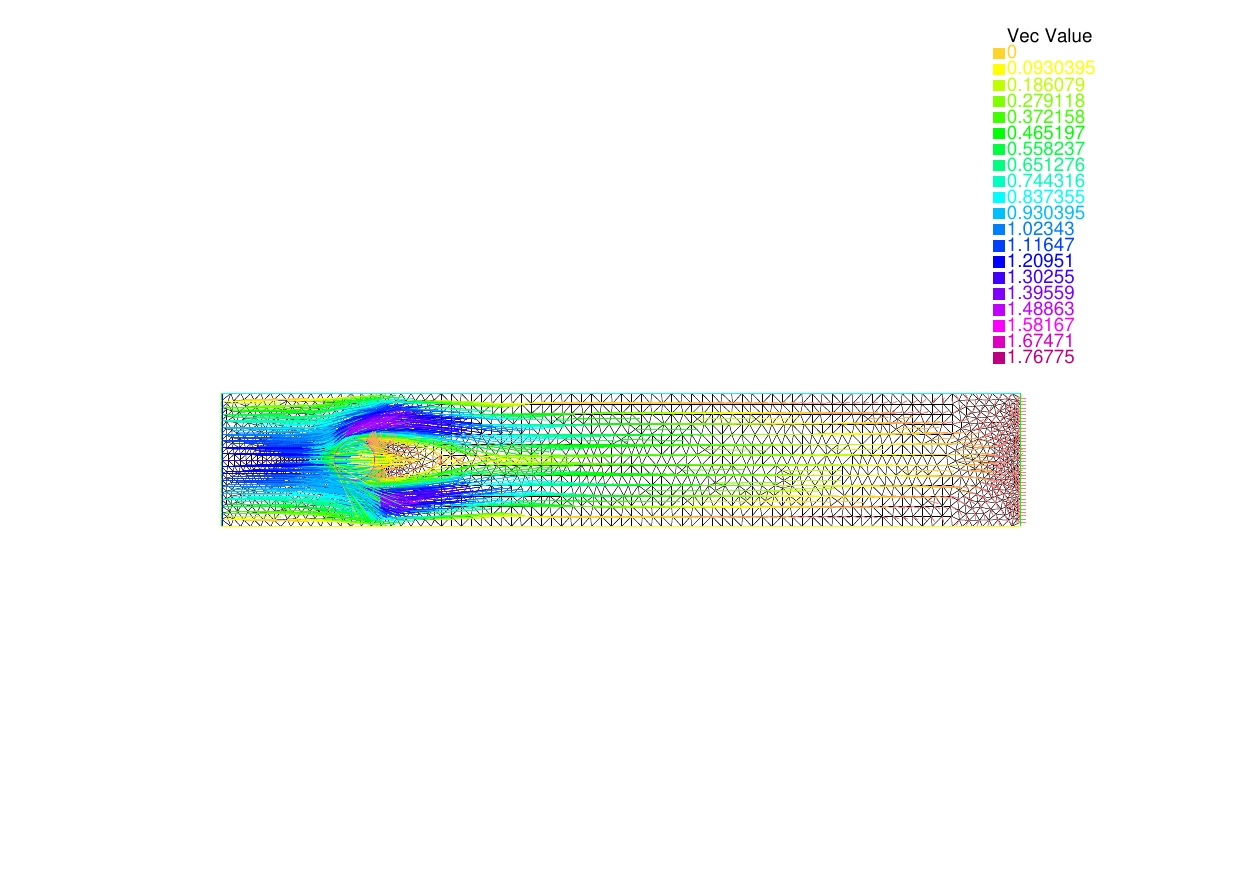}}
\subfigure[velocity: $\alpha=1$]
{\includegraphics[width=5.5cm]{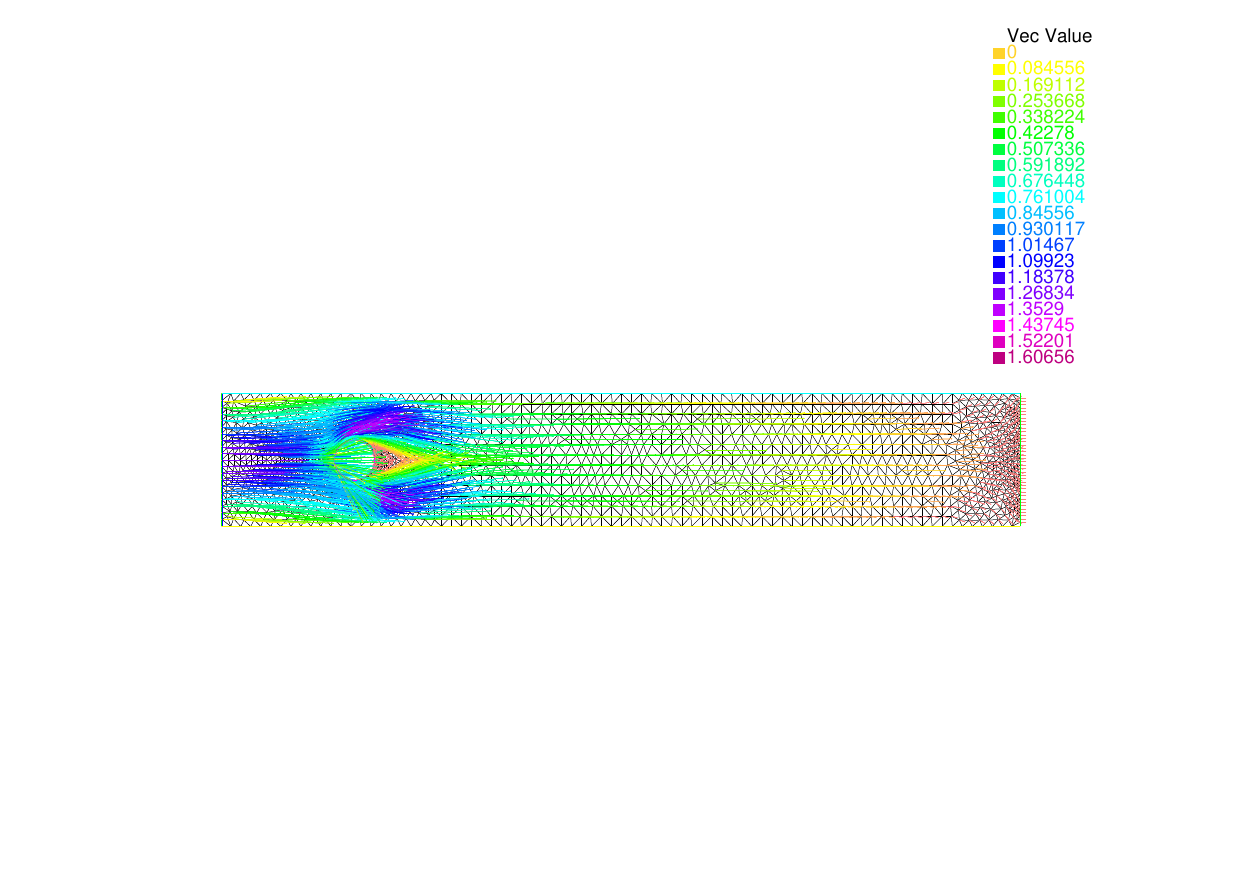}}%\\
\subfigure[velocity: $\alpha=5$]
{\includegraphics[width=5.5cm]{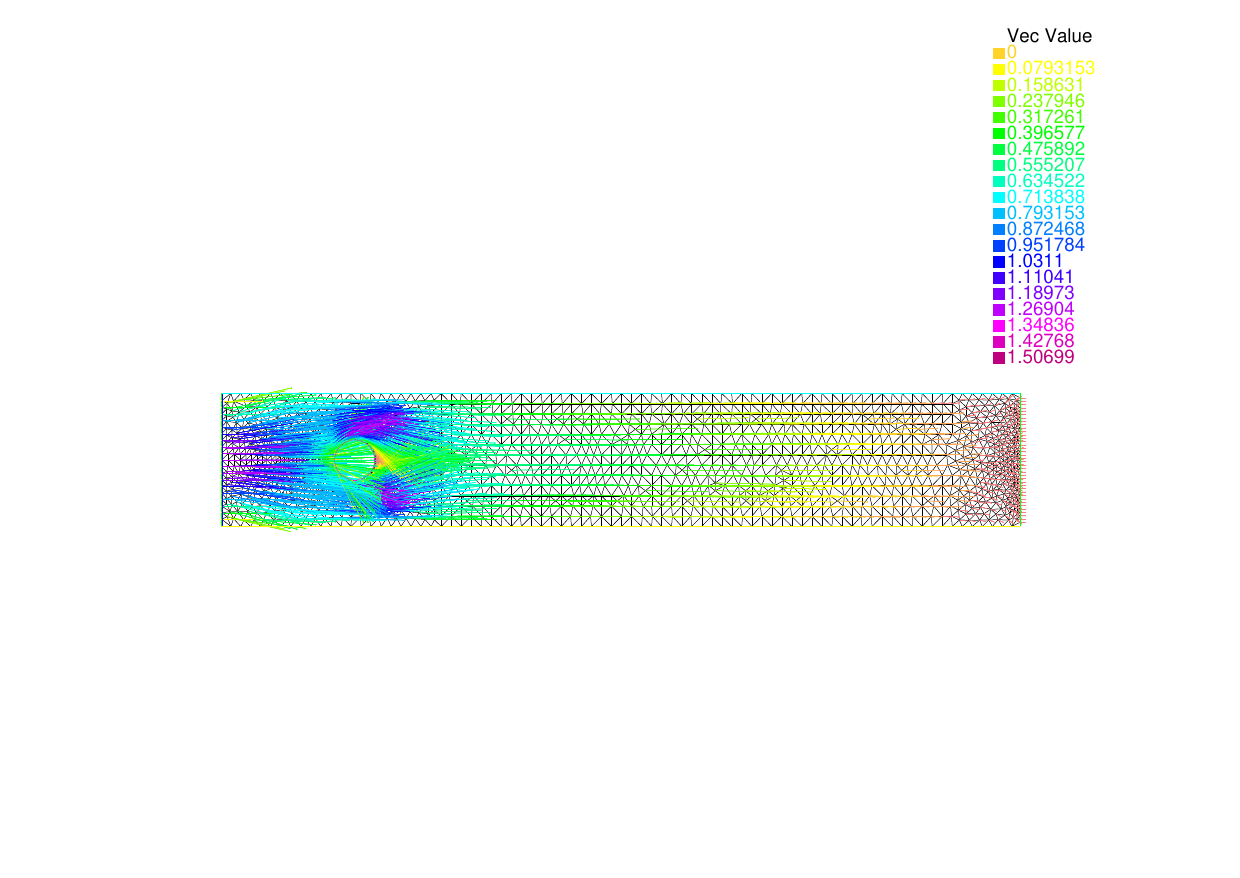}}
\quad
\quad
\subfigure[vorticity: $\alpha=0.1$]
{\includegraphics[width=5.5cm]{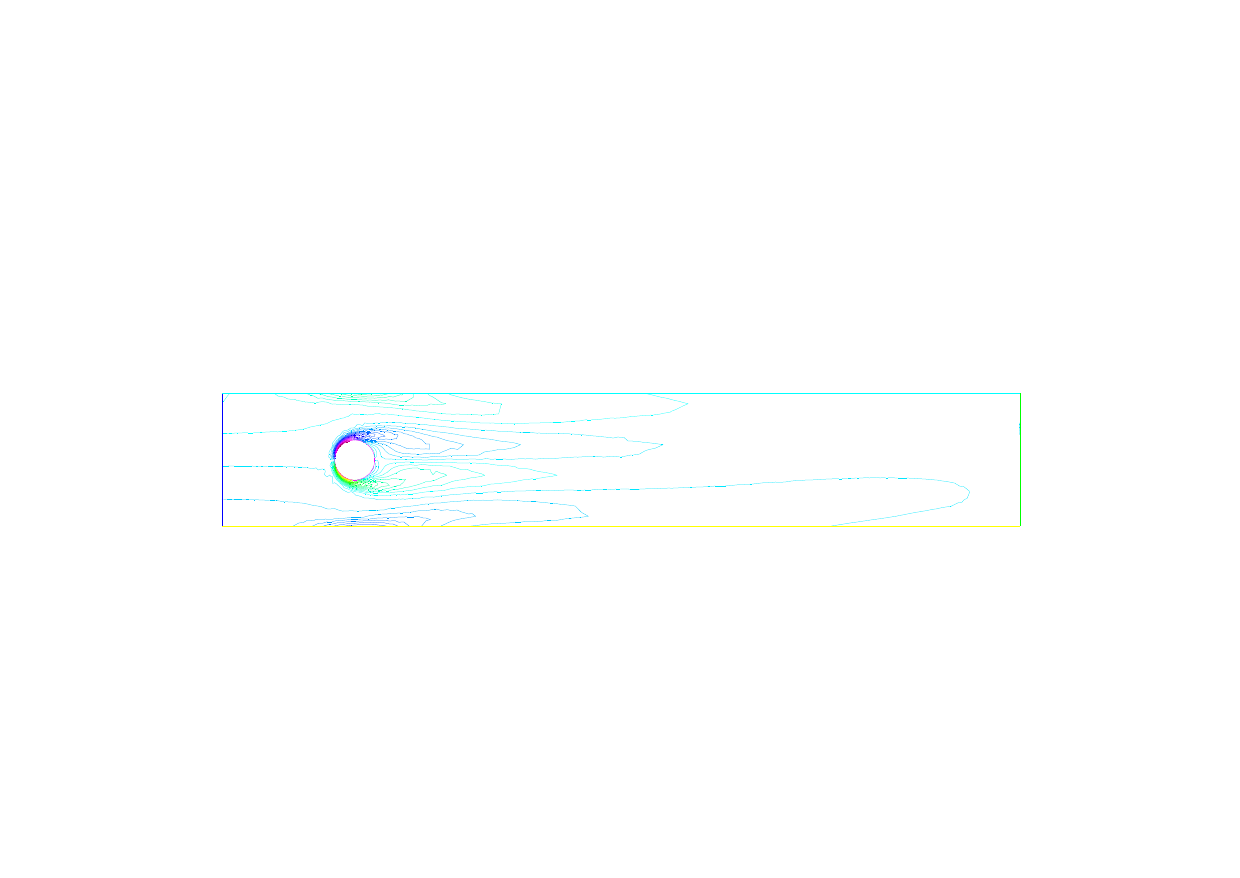}}
\subfigure[vorticity: $\alpha=1$]
{\includegraphics[width=5.5cm]{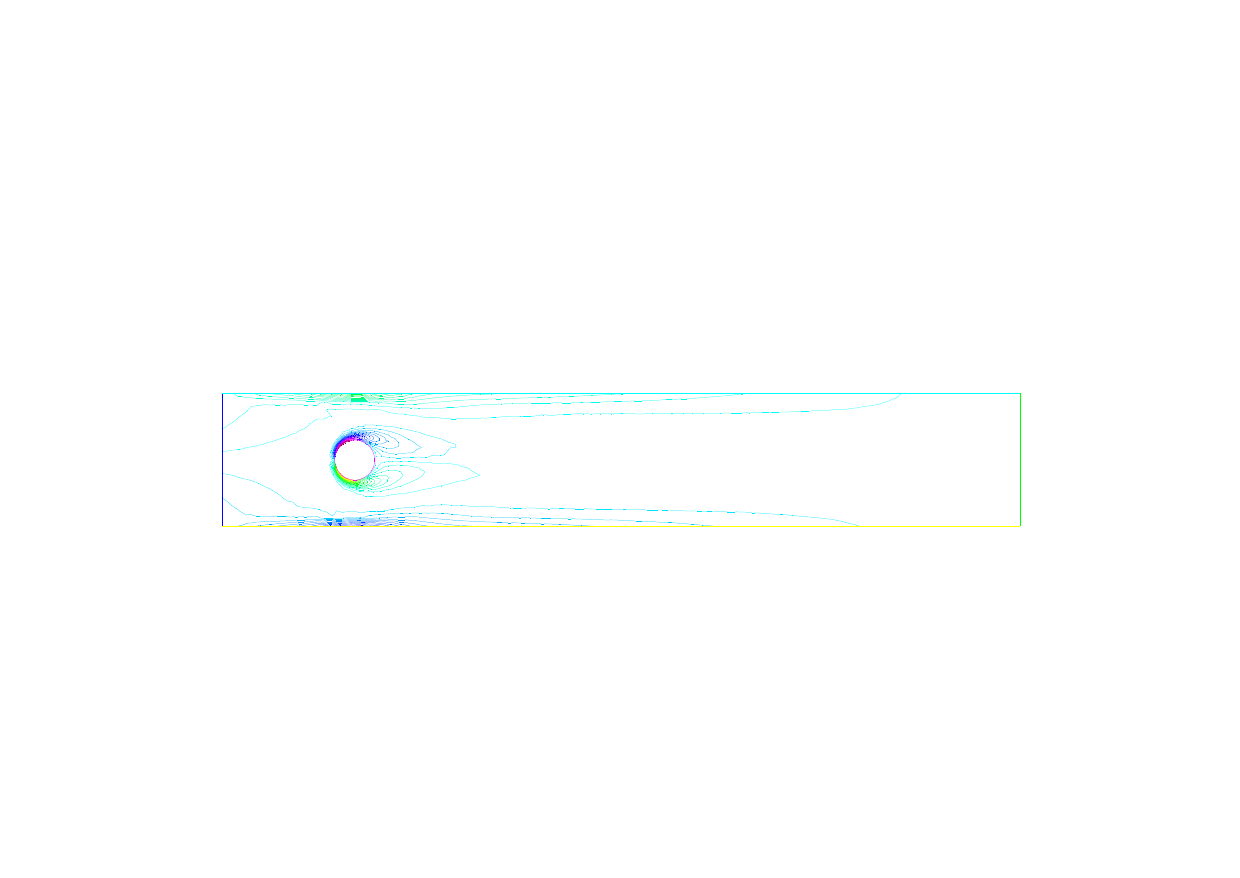}}%\\
\subfigure[vorticity: $\alpha=5$]
{\includegraphics[width=5.5cm]{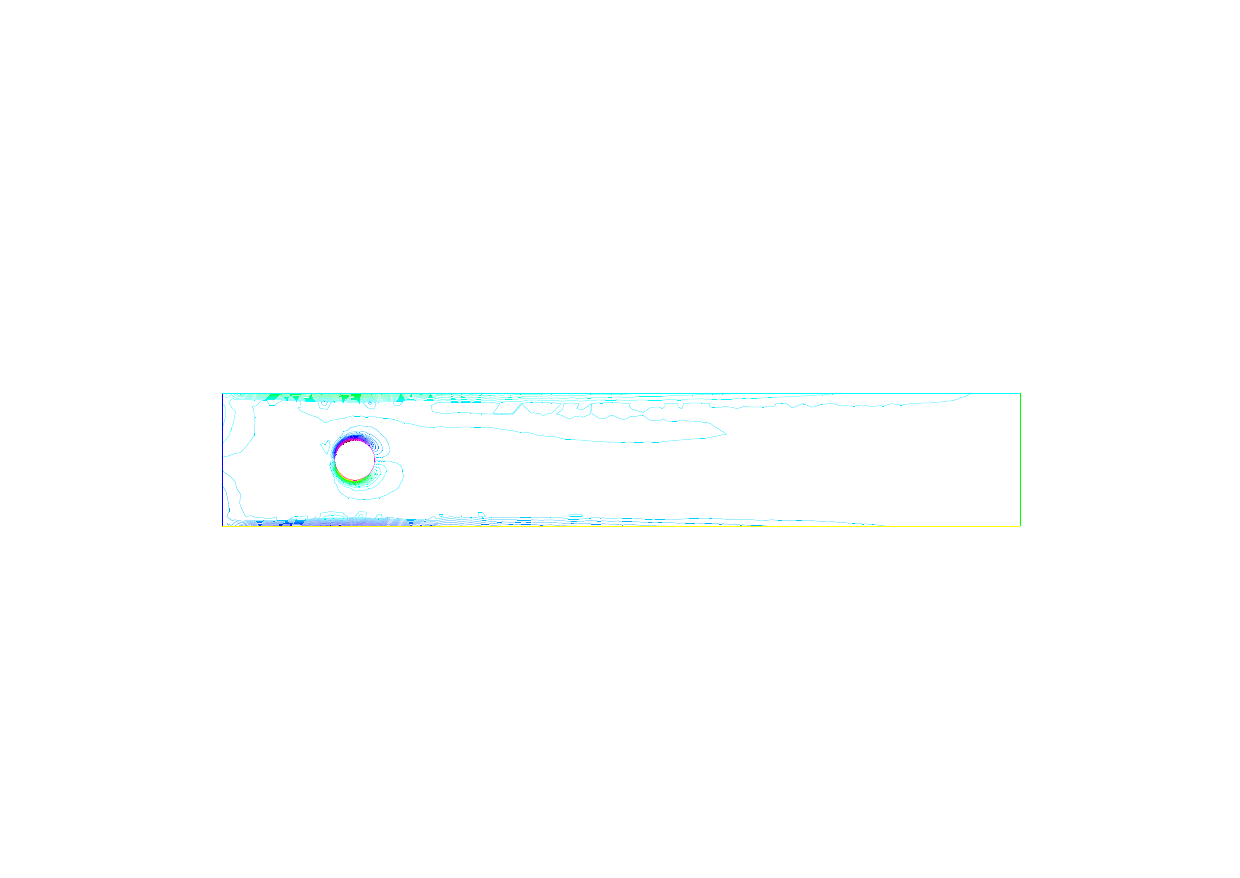}}\hspace{-10mm}
\quad
\subfigure[pressure: $\alpha=0.1$]
{\includegraphics[width=5.5cm]{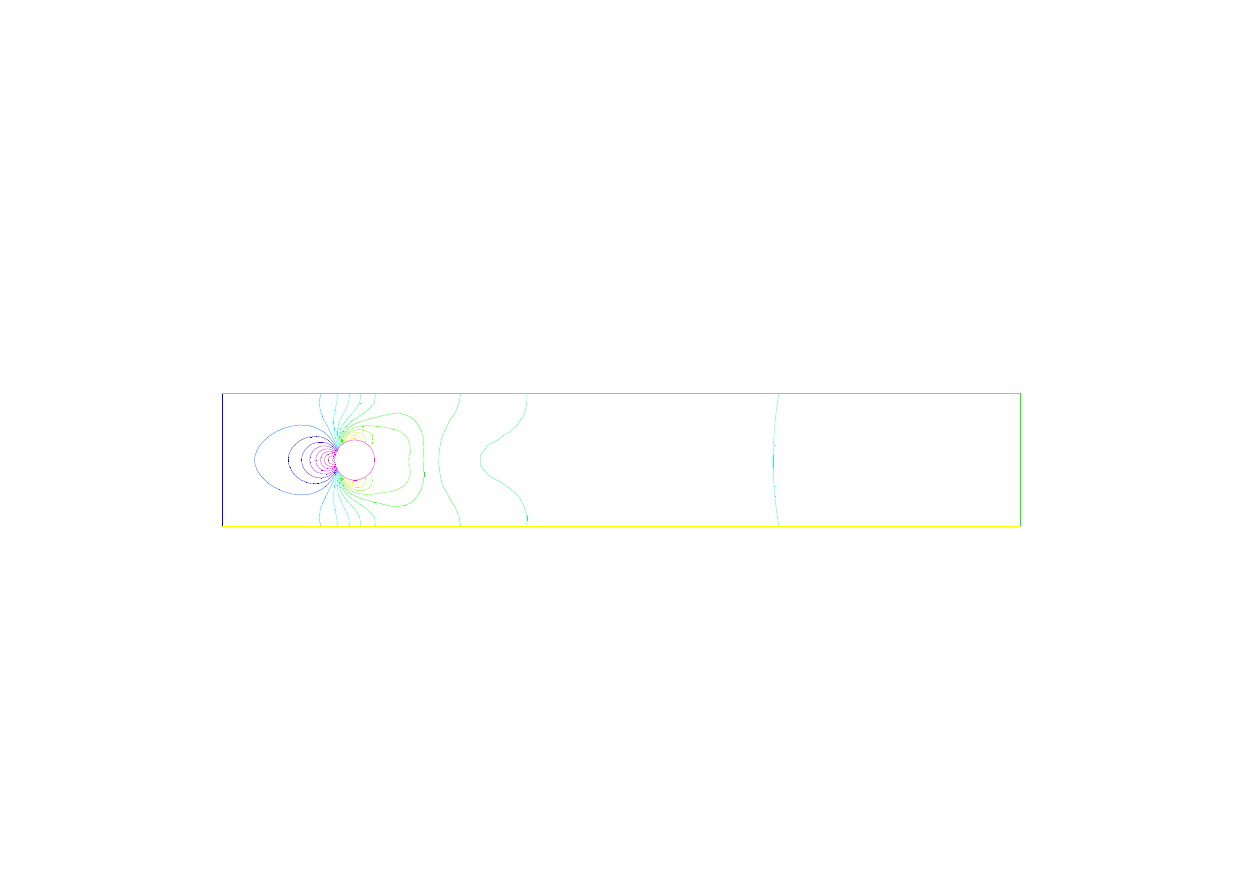}}
\subfigure[pressure: $\alpha=1$]
{\includegraphics[width=5.5cm]{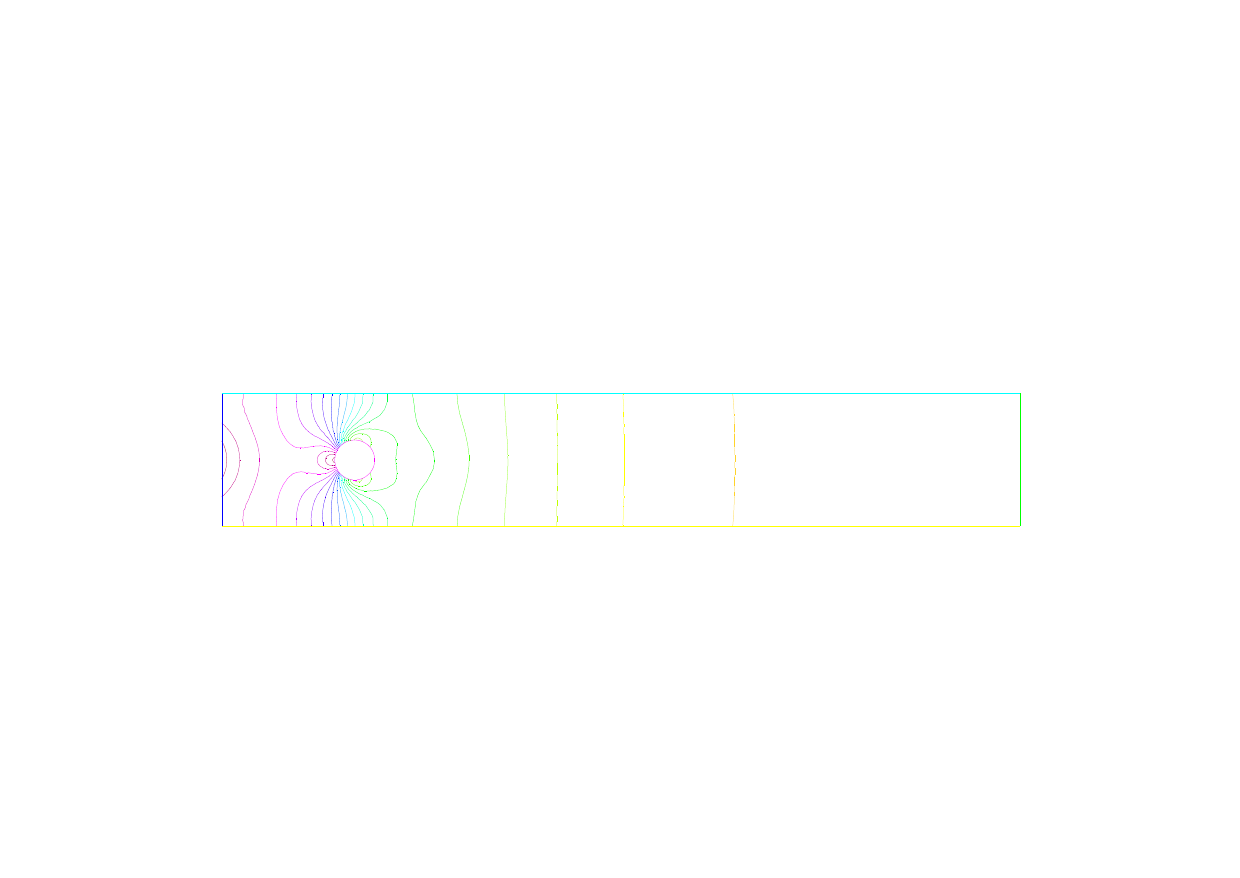}}%\\
\subfigure[pressure: $\alpha=5$]
{\includegraphics[width=5.5cm]{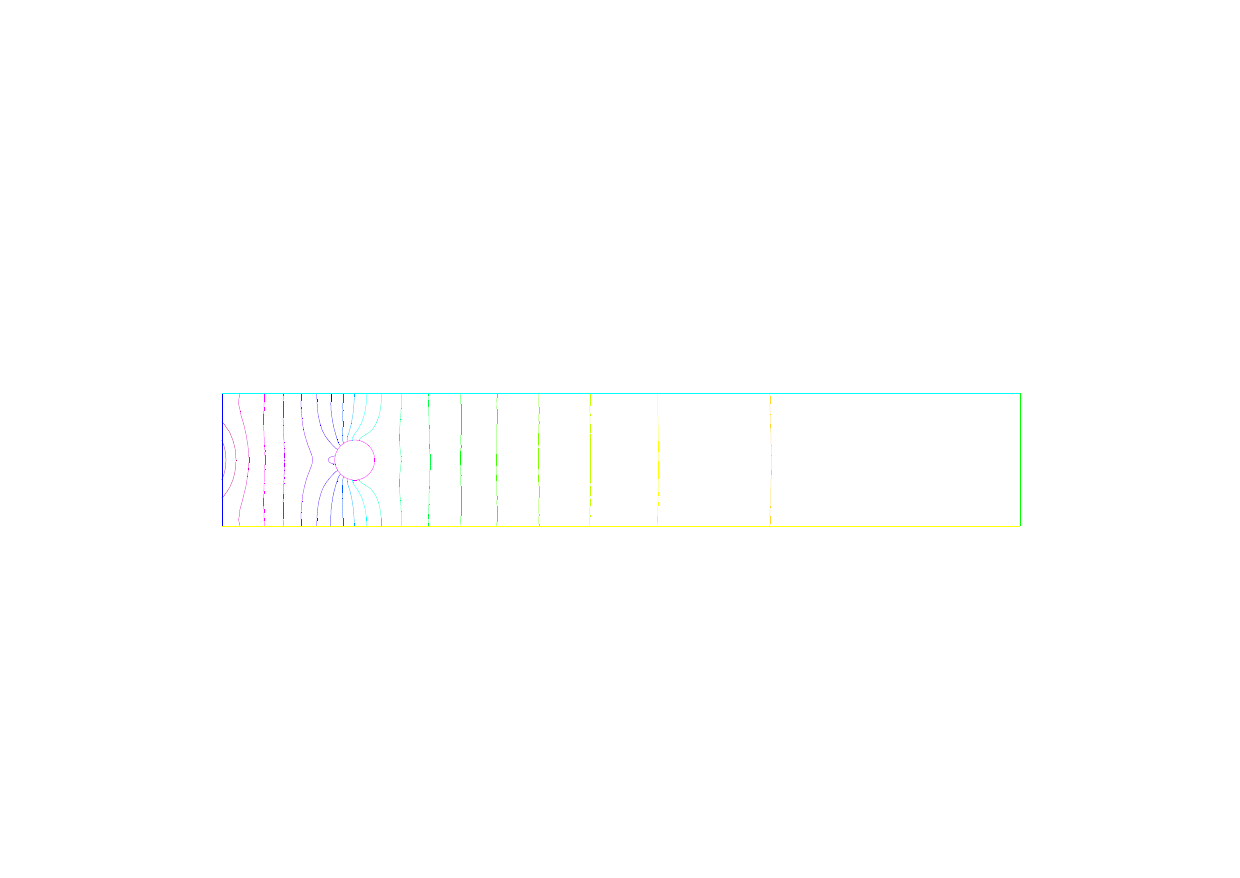}}%\hspace{-10mm}
\caption{The velocity streamlines, vortex lines  and pressure contours for Example \ref{EX7.5}: $r=3.5$ and $\alpha=0.1, 1, 5$ at $T=0.5$}
\label{fig32:32}
\end{figure}

%\vspace{-0.3cm}

%\vspace{0cm}

\begin{figure}[htbp!]
\centering
\setlength{\abovecaptionskip}{0.cm}
\subfigure[velocity: $r=3$]
{\includegraphics[width=5.5cm]{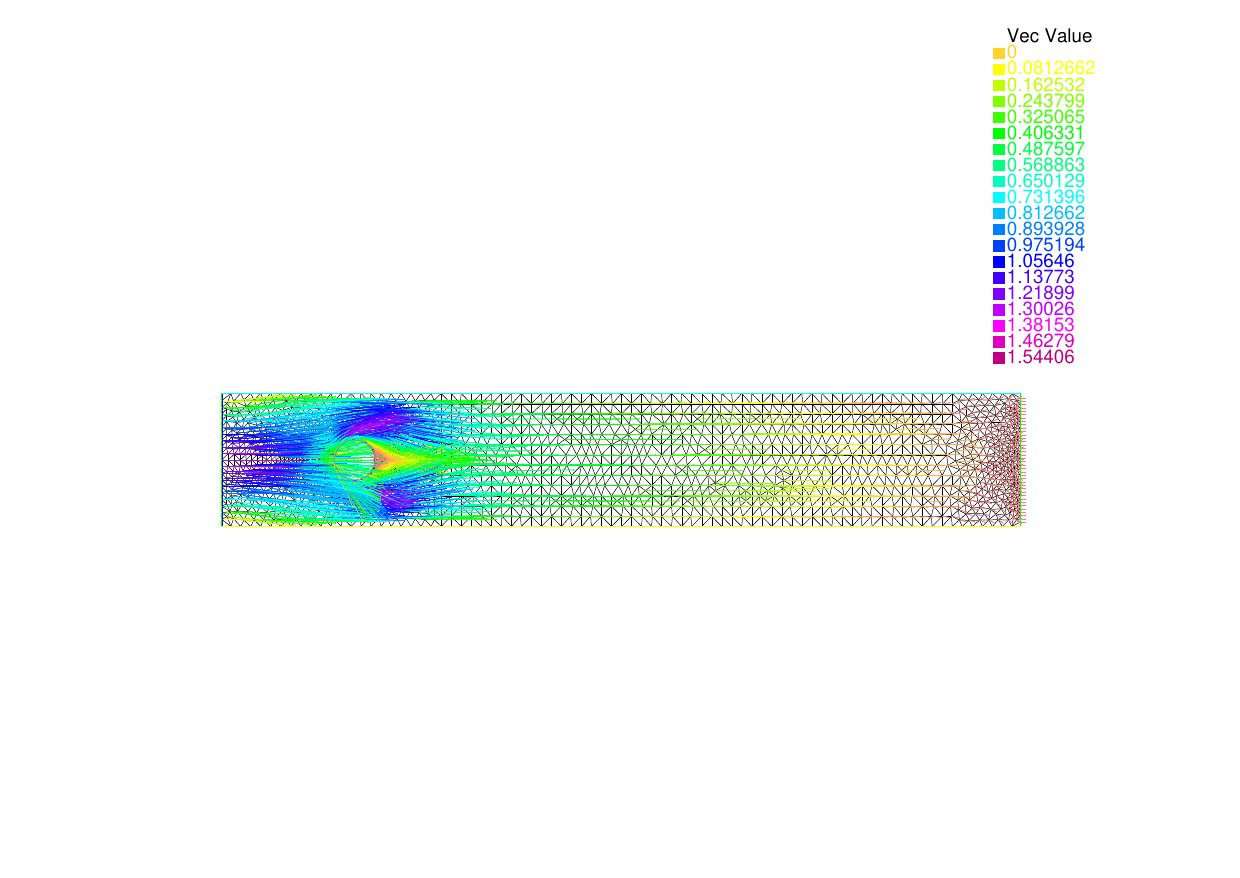}}
\subfigure[velocity: $r=5$]
{\includegraphics[width=5.5cm]{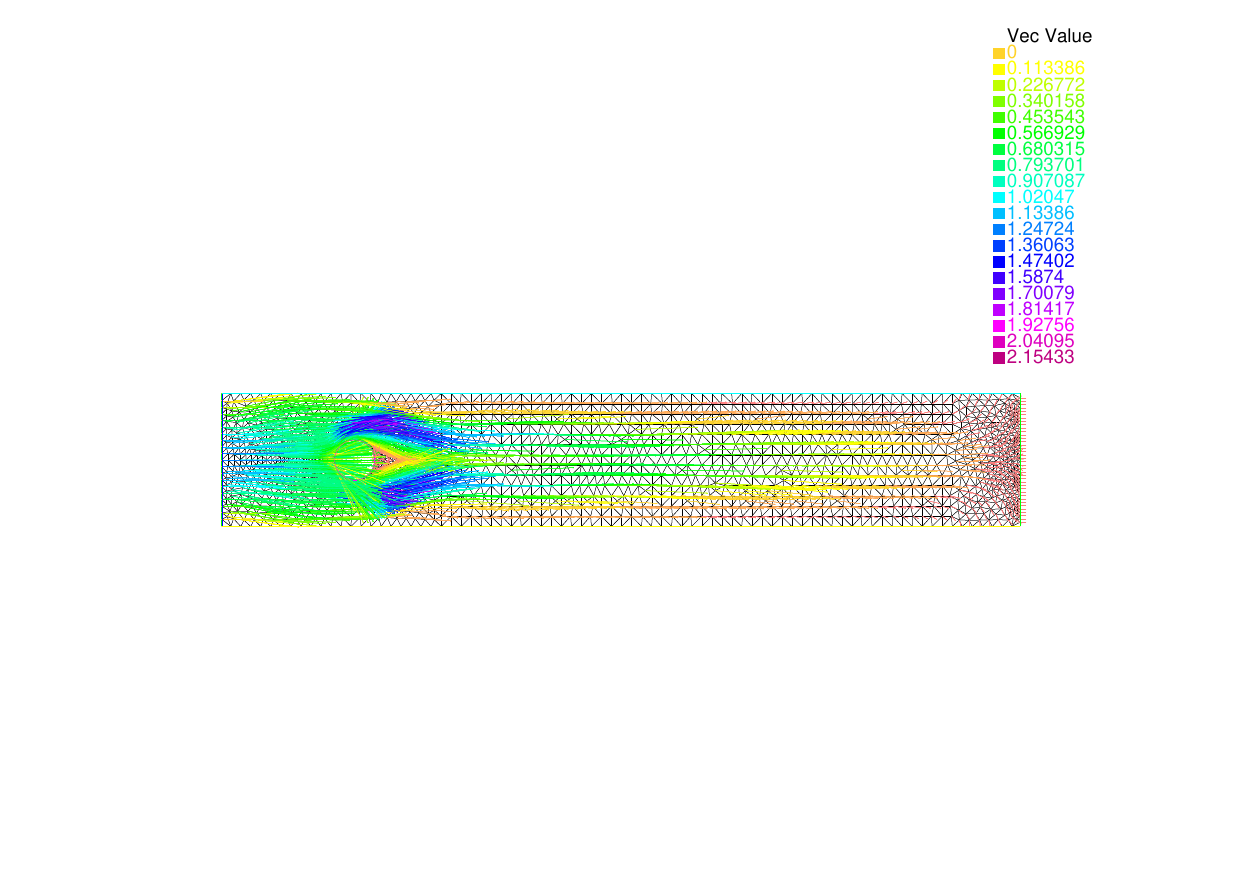}}%\\
\subfigure[velocity: $r=10$]
{\includegraphics[width=5.5cm]{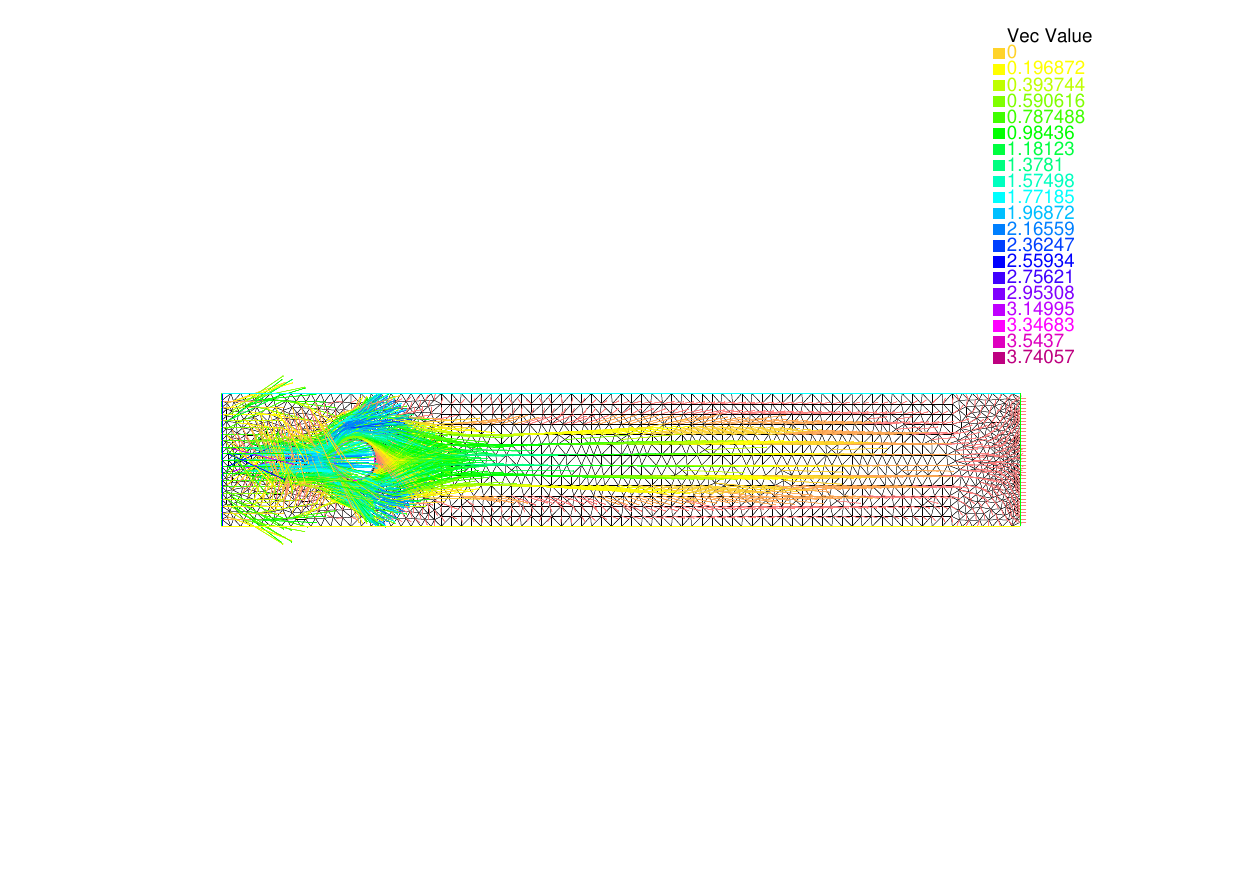}}
\quad
\subfigure[vorticity: $r=3$]
{\includegraphics[width=5.5cm]{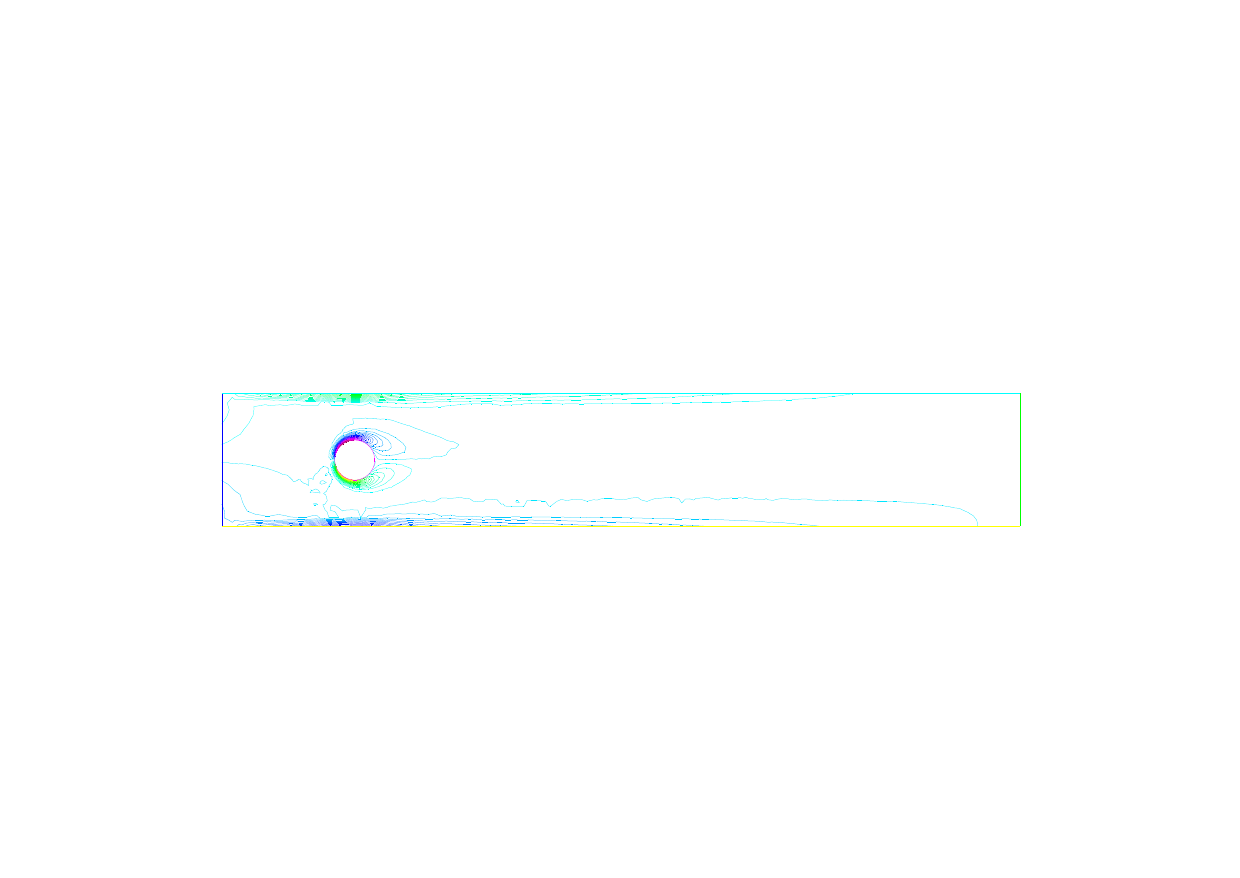}}
\subfigure[vorticity: $r=5$]
{\includegraphics[width=5.5cm]{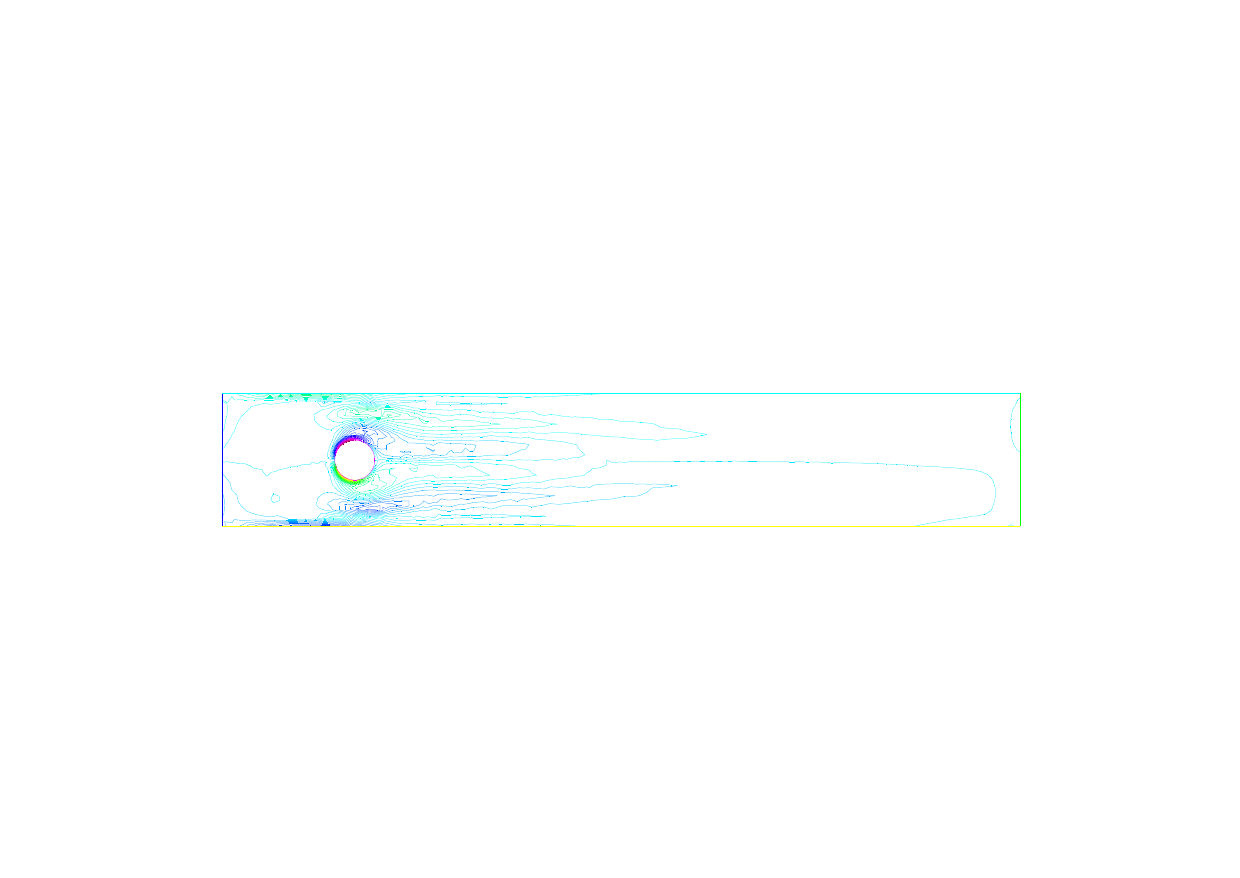}}%\\
\subfigure[vorticity: $r=10$]
{\includegraphics[width=5.5cm]{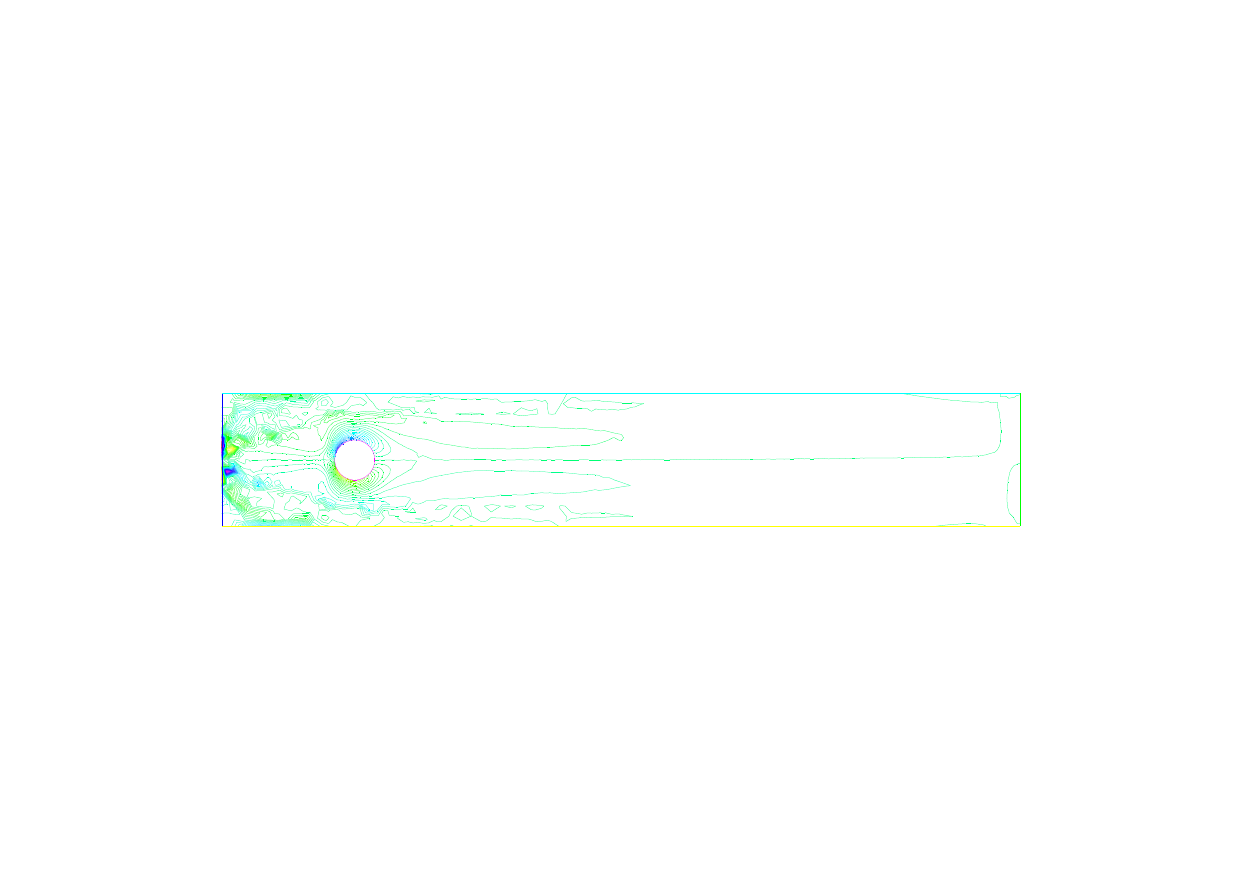}}
\quad
\subfigure[pressure: $r=3$]
{\includegraphics[width=5.5cm]{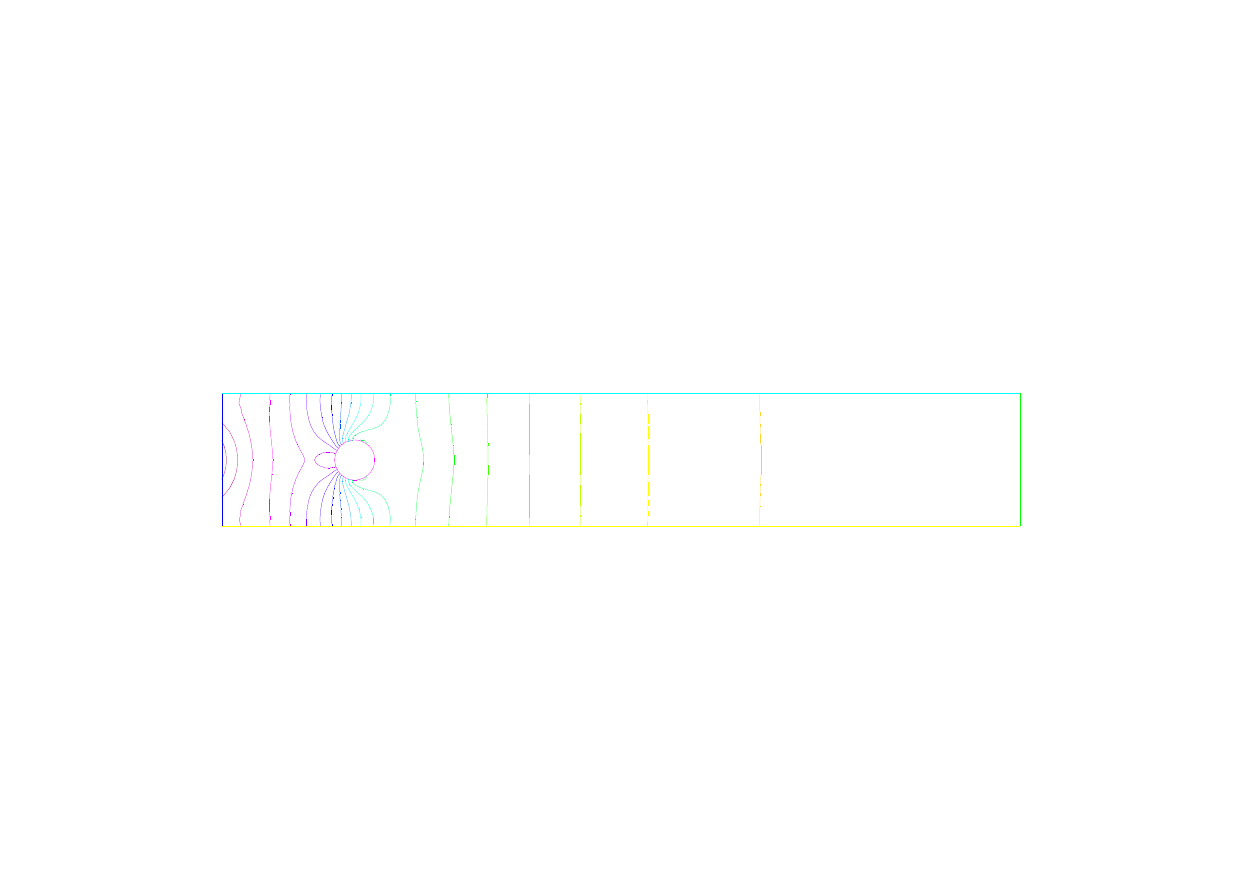}}
\subfigure[pressure: $r=5$]
{\includegraphics[width=5.5cm]{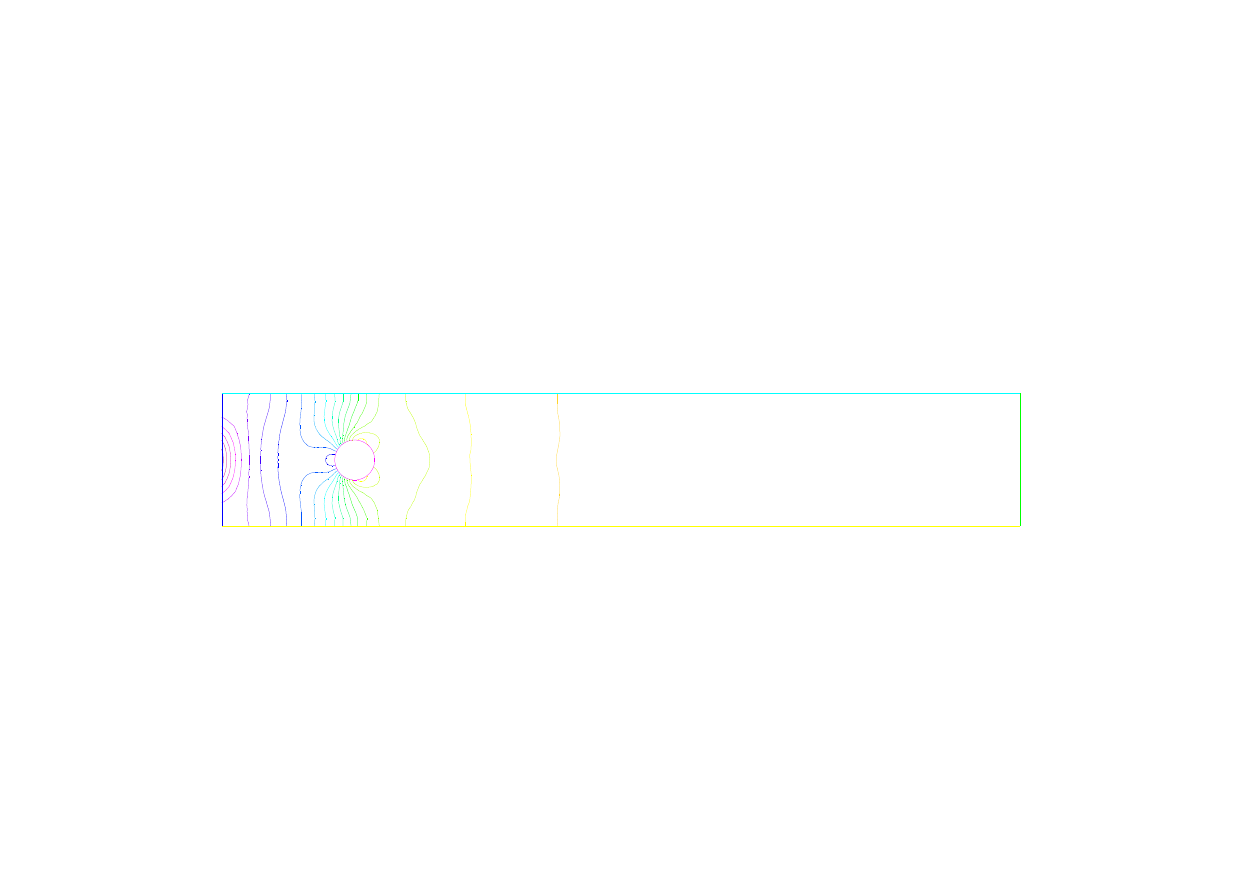}}%\\
\subfigure[pressure: $r=10$]
{\includegraphics[width=5.5cm]{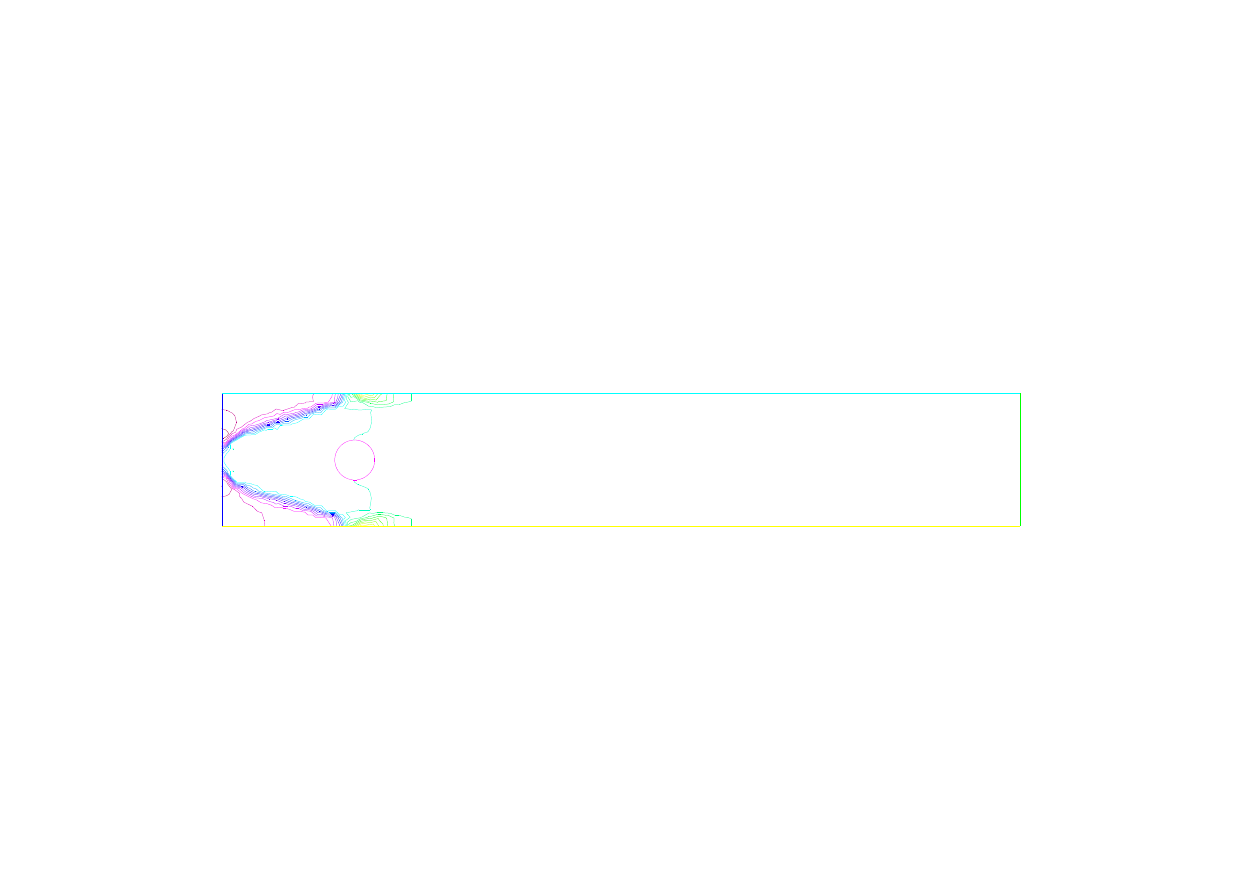}}
\caption{The velocity streamlines, vortex lines  and pressure contours for Example \ref{EX7.5}: $\alpha=2$ and  $r=3,5,10$ at $T=0.5$ }
\label{fig35:35}
\end{figure}

\section{Conclusion}
We proposed a class of globally divergence-free semi-discrete  and fully discrete WG scheme for the unsteady incompressible convective Brinkman-Forchheimer equations, which are pressure robust.
The stability, existence and uniqueness of  solution, and  optimal error estimations of semi-discrete and fully discrete WG schemes for the velocity and pressure approximations have been demonstrated. The developed linearized iteration algorithm is convergent.
Numerical tests verified the theoretical results and demontrated the robustness of the proposed methods.

%% \text{Re}ferences

%% \text{Re}ferences with bibTeX database:

% \color{black}A uniform posteriori error estimates with respect to the \text{Re}ynolds number \color{black}will be given in the future work.

%\section*{Acknowledgments.}
%\addcontentsline{toc}{section}{Acknowledgments.}
%This work is supported in part by National Natural Science Foundation of
%China (11771312, 11271273) and Major \text{Re}search
%Plan of  National Natural Science Foundation of China (91430105).
%The authors thank the editor and referees for their criticism, valuable comments and suggestions which helped to improve this paper.

%%\section*{\text{Re}ferences.}
%%\bibliographystyle{elsarticle-num}
%%\bibliography{<mydatabase>}


\color{black}
%% \text{Re}ferences without.} bibTeX database:
\bibliographystyle{siam}
\bibliography{mybib}{}
\color{black}

%%\end{thebibliography}
% \end{thebibliography}

%%\end{thebibliography}
% \end{thebibliography}

\end{document}